\documentclass[letterpaper,11pt,reqno]{amsart}
\usepackage[margin=1.2in]{geometry}
\usepackage{amsthm}
\usepackage{eucal}
\usepackage{tikz}
\usepackage{stackrel}
\usepackage{xypic}
\usepackage{amsfonts,amssymb}
\usepackage{epsfig}
\usepackage{enumerate}
\usepackage{enumitem}
\usepackage{ mathrsfs }
\usepackage{amsthm}
\usetikzlibrary{decorations.pathreplacing,calligraphy}
\usepackage[breaklinks,colorlinks,citecolor=teal,linkcolor=teal,urlcolor=teal,pagebackref,hyperindex]{hyperref}

\newtheorem{theorem}{Theorem}[section]
\newtheorem{lemma}[theorem]{Lemma}

\newtheorem*{question*}{Question}
\newtheorem{prop}[theorem]{Proposition}
\newtheorem{corollary}[theorem]{Corollary}

\newtheorem{conjecture}[theorem]{Conjecture}
\theoremstyle{definition}
\newtheorem{defn}[theorem]{Definition}
\newtheorem{remark}[theorem]{Remark}
\newtheorem{example}[theorem]{Example}

\newcommand{\Q}{\mathbb Q}
\newcommand{\C}{\mathbb C}
\newcommand{\R}{\mathbb R}
\newcommand{\K}{\mathbb K}
\renewcommand{\P}{\mathbb P}

\newcommand{\T}{\mathbb T}
\newcommand{\D}{\mathbb D}
\newcommand{\F}{\mathbb F}

\newcommand{\Z}{\mathbb Z}
\newcommand{\N}{\mathbb N}

\newcommand{\cA}{\mathcal A}

\newcommand{\cO}{\mathcal O}
\newcommand{\ccH}{\mathcal H}
\newcommand{\ccL}{\mathcal L}
\newcommand{\ccJ}{\mathcal J}
\newcommand{\ccM}{\mathcal M}
\newcommand{\bbI}{\mathbb I}
\newcommand{\reg}{\textnormal{reg}}

\newcommand{\ls}{\textnormal{l.s.}}
\newcommand{\op}{\textnormal{op}}
\newcommand{\id}{\textnormal{id}}
\newcommand{\SH}{{\mathbb S}H}
\newcommand{\Sc}{\textnormal{Sc}}

\def\End{\textnormal{End}}
\def\ev{\textnormal{ev}}

\def\Hom{\textnormal{Hom}}
\def\id{\textnormal{id}}
\def\Im{\textnormal{Im}}

\def\std{\textnormal{std}}

\def\id{\textnormal{id}}

\def\lra#1{\overset{#1}{\longrightarrow}}
\def\mod#1{#1\text{-mod}}
\def\sys#1{\textnormal{2-sys-}#1}
\def\Ind#1{\textnormal{Ind-}#1}
\def\Pro#1{\textnormal{Pro-}#1}

\newlist{HF}{enumerate}{1}
\setlist[HF]{label=(HF\arabic*)}

\newlist{PB}{enumerate}{1}
\setlist[PB]{label=(PB\arabic*)}

\newlist{CZ}{enumerate}{1}
\setlist[CZ]{label=(CZ\arabic*)}

\newlist{S1HF}{enumerate}{1}
\setlist[S1HF]{label=(S1HF\arabic*)}

\numberwithin{equation}{section}

\title{Birational Calabi-Yau manifolds have the same small quantum products}

\author{Mark McLean}

\begin{document}

\begin{abstract}
We show that any two birational projective Calabi-Yau
manifolds have isomorphic small quantum cohomology algebras after a certain change of Novikov rings.
The key tool used is a version of an algebra called symplectic cohomology, which is constructed using Hamiltonian Floer cohomology.
Morally, the idea of the proof is to show that
both small quantum products are
identical deformations of symplectic cohomology of some common
open affine subspace.
%
%
\end{abstract}

\maketitle

\bibliographystyle{mcleanalpha}

\tableofcontents

\section{Introduction}

We are interested in the following broad question.
\begin{question*}
What properties do birationally equivalent Calabi-Yau manifolds have in common?
\end{question*}
Recall that two smooth projective varieties $X$, $\widehat{X}$ are {\it birationally equivalent}
if there exist
Zariski dense open subsets
$A \subset X$, $\widehat{A} \subset \widehat{X}$
and an isomorphism
$A \lra{\cong} \widehat{A}$.
In this paper, by {\it Calabi-Yau manifold}, we will mean a smooth projective variety with trivial first Chern class.
Batyrev showed in \cite{batyrev1999birational}
that birational Calabi-Yau manifolds have equal Betti numbers.
More generally,
by combining ideas from
\cite{Kontsevich:1995motivic},\cite{denef2001geometry}
with ideas in \cite[Section 3.3]{gillet1996descent}, it can be shown that they have identical integral cohomology groups.
One can ask, do their cup product structures agree?
It turns out that this is false (\cite[Example 7.7]{friedman1991threefolds}).
However,
Morrison in \cite{morrisonbeyond}
conjectured
that birational Calabi-Yau threefolds
have identical small quantum cohomology rings, and this was proven in \cite{liruansurgerybirational}.
This leads to the following conjecture.
\begin{conjecture} \label{conjecture birational induces isomorphism} (\cite{RuanSurgery})
Any two birational Calabi-Yau manifolds have isomorphic big quantum cohomology rings up to analytic continuation.
\end{conjecture}

It was shown in
\cite{lee2011invariance},
\cite{lee2013invariance}
and
\cite{lee2014invariance}
that certain birational isomorphisms called {\it ordinary flops} induce isomorphisms between big quantum cohomology rings up to analytic continuation.
In \cite[Conjecture IV]{wang2002k}, it was conjectured that a small perturbation of a birational isomorphism between Calabi-Yau manifolds is a sequence of ordinary flops.
Hence the work above tells us that \cite[Conjecture IV]{wang2002k} implies Conjecture \ref{conjecture birational induces isomorphism}.
Kawamata in \cite{kawamata2008flops}
showed that birational morphisms between Calabi-Yau manifolds
can be decomposed into sequences of flops, however the structure of these flops in general is unknown.

In order to state our main theorem precisely, we need to set up some notation.
Let $\widehat{\Phi} : X \dashrightarrow \widehat{X}$
be a birational isomorphism between Calabi-Yau manifolds.
Fix a field $\K$ and fix K\"{a}hler forms
$\omega_X$, $\omega_{\widehat{X}}$
on $X$ and $\widehat{X}$ respectively so that the de Rham cohomology classes $[\omega_X] \in H^2(X;\R)$, $[\omega_{\widehat{X}}] \in H^2(\widehat{X};\R)$ lift to integral cohomology classes.
Then a standard argument (see 
\cite[Lemma 4.2]{kawamata2002d} or
Corollary \ref{corollary identification of H2})
tells us that the map $\widehat{\Phi}$
gives us natural identifications
$H^2(X;\R) \cong H^2(\widehat{X};\R)$
and $H_2(X;\Z) \cong H_2(\widehat{X};\Z)$.
As a result, we will not distinguish between these groups.
Therefore we can define the following Novikov rings:

\begin{equation} \label{equation common novikov ring}
\Lambda^{\omega_X,\omega_{\widehat{X}}}_\K = \left\{\sum_{i \in \N} b_i t^{a_i} \ : \ b_i \in \K, \ a_i \in H_2(X;\Z), \quad \min(\omega_X(a_i),\omega_{\widehat{X}}(a_i)) \to \infty
\right\},
\end{equation}
\begin{equation} \label{equation novikov ring omegaX}
\Lambda^{\omega_X}_\K = \left\{\sum_{i \in \N} b_i t^{a_i} \ : \ b_i \in \K, \ a_i \in H_2(X;\Z), \quad \omega_X(a_i) \to \infty
\right\},
\end{equation}
$$
\Lambda^{\omega_{\widehat{X}}}_\K = \left\{\sum_{i \in \N} b_i t^{a_i} \ : \ b_i \in \K, \ a_i \in H_2(X;\Z), \quad \omega_{\widehat{X}}(a_i) \to \infty
\right\}.
$$
Here, the first Novikov ring is the intersection of the other two.
The aim of this paper is to prove the following theorem.
\begin{theorem} \label{theorem main theorem}
	Let $\widehat{\Phi} : X \dashrightarrow \widehat{X}$ be a birational equivalence between Calabi-Yau manifolds and let
	$\omega_X$ and $\omega_{\widehat{X}}$
	be K\"{a}hler forms on $X$ and $\widehat{X}$ respectively whose cohomology classes
	lift to integer cohomology classes.
	Then there exists a graded
	$\Lambda^{\omega_X,\omega_{\widehat{X}}}_\K$-algebra
	$Z$ and algebra isomorphisms
	\begin{equation} \label{equation quantum cohomoogy isomorphisms}
	Z \otimes_{\Lambda_\K^{\omega_X,\omega_{\widehat{X}}}} \Lambda^{\omega_X}_\K \lra{\cong}
	QH^*(X;\Lambda^{\omega_X}_\K), \quad Z \otimes_{\Lambda_\K^{\omega_X,\omega_{\widehat{X}}}} \Lambda_\K^{\omega_{\widehat{X}}} \lra{\cong}
	QH^*(X;\Lambda_\K^{\omega_{\widehat{X}}})
	\end{equation}
	over the Novikov rings $\Lambda^{\omega_X}_\K$ and $\Lambda^{\omega_{\widehat{X}}}_\K$
	respectively
	where $QH^*$ means small quantum cohomology.
\end{theorem}

Note that small quantum cohomology of a Calabi-Yau manifold can be defined over any field $\K$ (see \cite{Ruan:topologicalsigma}).
Previous theorems identifying quantum cohomology rings of birational Calabi-Yau manifolds were proven using a degeneration argument (along with many additional ideas, such as a quantum Leray-Hirsch theorem).
We will prove our theorem using
Hamiltonian Floer cohomology.
This proof has the advantage that it works in greater generality and it explains in some sense why the result should hold.
The downside, in comparison to previous results, is that the isomorphisms
(\ref{equation quantum cohomoogy isomorphisms})
are not explicit.
In fact, the algebra $Z$ isn't explicit either.
Also, the results
\cite{RuanSurgery}, \cite{lee2011invariance},
\cite{lee2013invariance}
and
\cite{lee2014invariance}
identify actual enumerative
invariants through analytic continuation dictated by the identification $H_2(X;\Z) \cong H_2(\widehat{X};\Z)$.
It might be possible to use ideas from \cite{seidel2016connections} and
\cite{ganatra2015mirror}
 to
extend Theorem \ref{theorem main theorem} above so that we can identify such enumerative invariants in some special cases.
The results \cite{lee2011invariance},
\cite{lee2013invariance}
and
\cite{lee2014invariance}
 also identify big quantum cohomology rings, whereas we only identify small quantum cohomology rings.
Finally, one may ask if higher genus invariants of birational Calabi-Yau's are related (see \cite{iwaoleelinwanginvariantsunderasimpleflop}).
We currently do not know if the techniques in this paper will be useful in answering this question.

We will give a sketch of the proof of Theorem \ref{theorem main theorem} in Subsection \ref{subsection sketch of proof}.
The ideas of this proof were inspired by ongoing work of Borman and Sheridan
and it uses a version of symplectic cohomology defined in
\cite{CieliebakFloerHofersymplectichomologyII},
\cite{cieliebak2015symplectic},
\cite{groman2015floer},
\cite{venkatesh2017rabinowitz} and
 \cite{umutvarolgunessymplecticcohomology}.
As a corollary of the theorem above, we provide an alternative proof of the fact that birational Calabi-Yau manifolds have the same cohomology groups over any field.

\begin{corollary}
Let $X$, $\widehat{X}$ be birational Calabi-Yau manifolds. Then they have isomorphic cohomology groups over any field.
\end{corollary}
\proof
We wish to show that
$H^p(X;\K) \cong H^p(\widehat{X};\K)$
for any field $\K$ and
for each $p \in \Z$.
Fix such a $\K$ and $p \in \Z$.
%
%
%
%
Let $\check{\F}$, $\F$, $\widehat{\F}$
be the field of fractions of $\Lambda^{\omega_X,\omega_{\widehat{X}}}_\K$, $\Lambda^{\omega_X}_\K$ and $\Lambda^{\omega_{\widehat{X}}}_\K$ respectively.
Then $\K \subset \check{\F}$, $\check{\F} \subset \F$
and $\check{\F} \subset \widehat{\F}$.
Let $Z$ be as in Theorem \ref{theorem main theorem}, define
$\check{Z} := Z \otimes_{\Lambda_\K^{\omega_X,\omega_{\widehat{X}}}} \check{\F}$ and let $\check{Z}_p$ be the degree $p$ part of this graded algebra.
By the universal coefficient theorem \cite[3.6.1]{weibel1995introduction}, we have the following isomorphisms of vector spaces:
$$H^p(X;\Lambda^{\omega_X}_\K) \otimes_{\Lambda^{\omega_X}_\K} \F \cong H^p(X;\F), \quad
H^p(\widehat{X};\Lambda^{\omega_{\widehat{X}}}_\K) \otimes_{\Lambda^{\omega_{\widehat{X}}}_\K} \widehat{\F} \cong H^p(\widehat{X};\widehat{\F}).$$
Hence by Equation (\ref{equation quantum cohomoogy isomorphisms}),
we have isomorphisms
\begin{equation} \label{equation isomorphism with cohomology groups}
\check{Z}_p \otimes_{\check{\F}} \F \cong H^p(X;\F), \quad \check{Z}_p \otimes_{\check{\F}} \widehat{\F} \cong H^p(\widehat{X};\widehat{\F}).
\end{equation}
By the universal coefficient theorem,
we also have isomorphisms
\begin{equation} \label{equation universal coeff with K}
H^p(X;\K) \otimes_{\K} \F \cong H^p(X;\F), \quad H^p(\widehat{X};\K) \otimes_{\K} \widehat{\F} \cong H^p(\widehat{X};\widehat{\F}).
\end{equation}
Therefore by Equations
(\ref{equation isomorphism with cohomology groups}) and
(\ref{equation universal coeff with K}),
the dimension of $H^p(X;\K)$
is equal to the dimension of $H^p(\widehat{X};\K)$.
Hence $H^p(X;\K) \cong H^p(\widehat{X};\K)$.
%
\qed

\subsection{Example} \label{subsection example computation}

Even though we do
not know what the isomorphisms (\ref{equation quantum cohomoogy isomorphisms})
look like explicitly
we will speculate what they should look like in a particular example when restricted to even degrees (see \cite[Section 4.3]{morrisonbeyond}).
Suppose that $X$, $\widehat{X}$ are of dimension $3$ and that there exists
a class $\Gamma \in H_2(X;\Z)$
so that
\begin{itemize}
\item every connected one dimensional subvariety
in $X$ representing $\Gamma$
or in $\widehat{X}$ representing $-\Gamma$
is isomorphic to $\P^1$ with normal bundle $\cO(-1) \oplus \cO(-1)$ and
%
%
%
\item $X$ and $\widehat{X}$ are related by an Atiyah flop along all of these curves.
\end{itemize}
Since $H_2(X;\Z)$ is naturally
identified with $H_2(\widehat{X};\Z)$,
we have by Poincar\'{e} duality
a natural
identification
$H^k(X;\Q) = H^k(\widehat{X};\Q)$ for each even $k \in \Z$. Hence
from now on we will identify these cohomology groups. 
%
Let $\widehat{A}_0,\cdots,\widehat{A}_l \in H^4(X;\Q)$
be a basis so that $\widehat{A}_0$ is Poincar\'{e} dual to $\Gamma$
and let $A_0,\cdots,A_l \in H^2(X;\Q)$ be the dual basis
 with respect to the pairing $(\alpha,\beta) \to \int_X \alpha \cup \beta$.
%
%
We speculate that the even degree part of the algebra $Z$ from Theorem \ref{theorem main theorem} is isomorphic as a $\Lambda^{\omega_X,\omega_{\widehat{X}}}_\K$-module to $H^{\textnormal{even}}(X;\Lambda^{\omega_X,\omega_{\widehat{X}}}_\K)$ and the product $*_Z$ on this module is uniquely determined by the structure constants:
\begin{equation} \label{equation s structure}
A_i *_Z A_j =  A_i \cup_X A_j +
l \delta_{0i}\delta_{0j} \widehat{A}_0 t^\Gamma
+
\sum_{k=0}^l \sum_{\beta \notin \Z \Gamma} GW^{X,\beta}_{0,3}(A_i,A_j,A_k) \widehat{A}_k t^{\beta}, \ i,j \in \{0,\cdots,l\}.
\end{equation}
By
replacing the class $A_0$ in (\ref{equation s structure}) with $\frac{1}{1-t^\Gamma}A_0$ and $-\frac{t^{-\Gamma}}{1-t^{-\Gamma}} A_0$
respectively
and the class
$\widehat{A}_0$ with
$(1-t^\Gamma)\widehat{A}_0$ and $-\frac{1-t^{-\Gamma}}{t^{-\Gamma}} \widehat{A}_0$ respectively, we get
the isomorphisms in Equation (\ref{equation quantum cohomoogy isomorphisms}).

\subsection{Sketch of Proof} \label{subsection sketch of proof}

Theorem \ref{theorem main theorem}
is proven using {\it Hamiltonian Floer cohomology}.
Very roughly, Hamiltonian Floer cohomology
$HF^*(H)$ is a cohomology ring whose chain complex is generated by $1$-periodic orbits of a Hamiltonian $H$.
The key property of Hamiltonian Floer cohomology is that
it is isomorphic to quantum cohomology and
hence it is sufficient for us to show that the Hamiltonian Floer cohomology algebras of appropriate Hamiltonians
on $(X,\omega_X)$ and $(\widehat{X},\omega_{\widehat{X}})$
respectively
are related via equations similar to (\ref{equation quantum cohomoogy isomorphisms}).
In order to do this, we choose Zariski dense affine open subsets $A \subset X$,
$\widehat{A} \subset \widehat{X}$
so that $\widehat{\Phi}$ maps $A$ isomorphically to $\widehat{A}$.
The key idea now is to choose Hamiltonians $H$ on $X$ and $\widehat{H}$ on $\widehat{X}$
so that $H|_A = \widehat{\Phi}^* \widehat{H}|_{\widehat{A}}$
and which are constant outside a large compact subset $K$ of $A$.
One also has to modify $\omega_X$ and $\omega_{\widehat{X}}$ so that these K\"{a}hler forms agree near $K$.
If one could ignore all $1$-periodic orbits of $H$ and $\widehat{H}$ outside $K$ then their Hamiltonian Floer groups would be `identical' and hence we would be done.
However it turns out that if one ignores these orbits,
one gets groups that are no longer isomorphic to quantum cohomology.
In order to get around this problem, we consider a sequence of such Hamiltonians tending to infinity outside $K$. We package all of this data into a group called {\it symplectic cohomology} and show that these groups are in fact isomorphic to quantum cohomology.
In the following subsections, we provide a slightly more detailed sketch of the proof.

\subsubsection{Hamiltonian Floer Cohomology with Alternative Filtrations.}

In this subsection, we summarize the results of Section \ref{section Hamitlonian Floer cohomology and Filtrations}.
Let $(M,\omega)$ be a symplectic manifold with trivial first Chern class and fix an $\omega$-tame almost complex structure $J$.
A {\it contact cylinder}
is a codimension $0$ symplectic embedding of a subset
$\check{C} = [1-\epsilon,1+\epsilon] \times C$
of a symplectization of a contact manifold $C$  which bounds a Liouville domain $D$ (see Definition \ref{definition contact cylinder}).
Let $H$ be a time dependent Hamiltonian
on $M$
which is {\it compatible}
with this contact cylinder
(see Definition \ref{definition Hamiltonians compatible with contact cylinder})
whose $1$-periodic orbits are non-degenerate.
A {\it capped $1$-periodic orbit}
of $H$ is a $1$-periodic orbit $\gamma$ together with a certain equivalence class of surfaces $\widetilde{\gamma} : \Sigma \lra{} M$ with boundary equal to $\gamma$
(see Definition \ref{defn capped 1-periodic orbit}).
Now  to each capped $1$-periodic orbit $\gamma$ and to each closed $2$-form
$\widetilde{\omega}$ which is `compatible' with $\check{C}$ and $J$
we can assign an {\it action}.
This action depends on $\gamma$, $H$ and the cohomology class of $\widetilde{\omega}$ together with two additional parameters (Corollary \ref{corollary action of capped loop disjoint from contact cylinder} and Remark \ref{remark only depends on cohomology class}).
Therefore, one can think of the action of $\gamma$ as a particular function $\cA_{\check{C},H}(\gamma)$
from
(a certain subset $Q_{\check{C}}$ of)
$H^2(M,D;\R) \times \R \times \R$ to $\R$.
Morally, one should think of $D$ as the `complement' of a particular ample divisor and that the action function
$\cA_{\check{C},H}(\gamma)$
records the usual action together with the intersection numbers of the capping surface $\widetilde{\gamma}$ with various components of this divisor.
Let $a_\pm : Q_\pm \lra{} \R$
be continuous functions where $Q_\pm \subset Q_{\check{C}}$ are certain cones in $Q_{\check{C}}$.
We define the chain complex
$CF^*_{\check{C},a_-,a_+}(H)$ to be
the free abelian group generated by
capped $1$-periodic orbits $\gamma$ satisfying $a_- \leq \cA_{\check{C},H}(\gamma)|_{Q_-}$
and $a_+ \nleq \cA_{\check{C},H}(\gamma)|_{Q_+}$
(see Definition \ref{definition Floer chain complex} for more details).
The differential is a matrix with respect to the above basis of capped $1$-periodic orbits whose entries `count' solutions to a particular PDE with boundary conditions given by these orbits (Definition \ref{defn of differential}).
We define the {\it Hamiltonian Floer cohomology group}
$HF^*_{\check{C},a_-,a_+}(H)$
to be the homology of this chain complex.
This is a module over a particular Novikov ring $\Lambda_\K^{Q_+,+}$ (Definition \ref{definition convex cone associated to contact cylinder}).
Here are some key examples.
\begin{enumerate}[label=\{\Alph*\}]
\item
\label{item one dimensional cones}
 The case when $\check{C}$ is the empty set and
$Q_- = Q_+$ are $1$-dimensional
cones spanned by $([\omega],1,1) \in H^2(M;\R) \times \R^2$.
In this case, one gets
the usual Hamiltonian Floer groups
and the usual action filtration
(used in, say, \cite{FloerHofer:SymhomI}).
The Novikov ring $\Lambda_\K^{Q_+,+}$ in this case is equal to (\ref{equation novikov ring omegaX}) where $\omega_X$ is replaced by $\omega$ and where only non-negative exponents are allowed (I.e. $\omega(a_i) \geq 0$).
\item \label{item two dimensional cone}
The case when $\check{C}$ is non-empty, $Q_+$ is a one dimensional cone and $Q_-$ is a two dimensional cone containing $Q_+$ which projects to the cone spanned by $[\omega] \in H^2(M;\R)$.
In this case the corresponding Floer cohomology group is defined over the same Novikov ring, but
since $Q_-$ is larger, one can use it to ignore certain $1$-periodic orbits.
\item \label{item four dimensional cone}
Finally there is the case when $Q_+$ is a certain
two dimensional cone and $Q_-$ is a four dimensional cone.
The Novikov ring in this case is a subring of
(\ref{equation common novikov ring}) where only non-negative exponents are allowed (I.e. $\omega_X(a_i), \omega_{\widehat{X}}(a_i) \geq 0$).
We will use this case to define the algebra $Z$.
\end{enumerate}

These Hamiltonian Floer groups
satisfy the following properties:
\begin{enumerate}[label=(HF\arabic*)]
\item  \label{item continuation map property}
(Definition \ref{definition chain level continuation map}).
If $H_- \leq H_+$ (plus some other conditions) then there is a natural {\it continuation map}
$\Phi^p_{H^-,H^+} : HF^p_{\check{C},a_-,a_+}(H^-) \lra{}
HF^p_{\check{C},a_-,a_+}(H^+)$ which is functorial.
\item (Definition \ref{definition action map}). If $Q^1_\pm \subset Q^0_\pm$ and $a^1_\pm \leq a^0_\pm|_{Q^1_\pm}$ then there is a natural {\it action map}
$HF^*_{\check{C},a^0_-,a^0_+}(H) \lra{}
HF^*_{\check{C},a^1_-,a^1_+}(H)$
and these maps are functorial and
commute with continuation maps.
\item \label{item pair of pants product property}
(Definition \ref{defnition pair of pants product}). 
Suppose 
\begin{itemize}
\item $(H^j)_{j=0,1,2}$ are Hamiltonians satisfying $H^0,H^1 < \frac{1}{2} H^2$,
\item $(Q^j_\pm)_{j=0,1,2}$ are certain cones in $H^2(M,D;\R) \times \R \times \R$ satisfying $Q^2_\pm \subset Q^j_\pm$ and
\item $a^j_\pm : Q^j_\pm \lra{} \R$ are certain continuous functions satisfying $a^2_- \leq a^0_- + a^1_-$  and $a^2_+ \leq \min(a^0_+ + a^1_-,a^0_- + a^1_+)$.
\end{itemize}
Then there is a {\it pair of pants product map} $HF^{p_0}_{\check{C},a^0_-,a^0_+}(H^0) \otimes_{\Lambda_\K^{Q,+}} HF^{p_1}_{\check{C},a^1_-,a^1_+}(H^1)
\lra{}
HF^{p_0+p_1}_{\check{C},a^2_-,a^2_+}(H^2)$ commuting with all of the maps above.
\end{enumerate}

\subsubsection{Lower semi-continuous Hamiltonians.}

The following subsection summarizes Section \ref{section floer cohomology for lower semi-continuous Hamitlonians}.
A {\it lower semi-continuous Hamiltonian}
is just a function $S^1 \times M \lra{} \R \cup \{\infty\}$ which is lower semi-continuous.
The good thing about this condition is that the set of smooth Hamiltonians smaller than $H$ form a directed system with respect to the usual ordering $\leq$.
For a lower semi-continuous Hamiltonian $H$ compatible with a contact cylinder $\check{C}$, we can define
$HF^*_{\check{C},a_-,a_+}(H)$
to be the direct limit of
$HF^*_{\check{C},a_-,a_+}(\check{H})$
for all smooth $\check{C}$ compatible Hamiltonians smaller than $H$.
These satisfy the same properties
\ref{item continuation map property}-\ref{item pair of pants product property}
above.

\subsubsection{Symplectic Cohomology}

The problem with the Floer groups above is that they do not have the correct invariance properties and they are not algebras.
We resolve these issues in Sections
\ref{section definition of symplectic cohomology} and \ref{section properties of symplectic cohomology}.
Let $M$, $\check{C}$, $D$ be as above.
Let $Q_- \subset Q_+$ be two cones in $Q_{\check{C}}$.
For a closed set $K \subset D$ 
we define the {\it symplectic cohomology algebra}
$$SH^*_{\check{C},Q_-,Q_+}(K \subset M) := \varinjlim_{a_-} \varprojlim_{a_+} HF^*_{\check{C},a_-,a_+}(H_K)$$
where we are using the directed system of continuous functions
$a_\pm : Q_\pm \lra{} \R$ with the usual ordering $\leq$ for $a_+$ and the opposite ordering for $a_-$ and where $H_K$ is the lower-semi-continuous Hamiltonian
\begin{equation} \label{equation index function Hamiltonian}
H_K : M \lra{} \R \cup \{\infty\}, \quad H_K(x) := \left\{
\begin{array}{ll}
0 & \text{if} \ x \in K \\
\infty & \text{otherwise.}
\end{array}
\right.
\end{equation}
This is defined over a particular Novikov ring $\Lambda_\K^{Q_+}$ and has a product induced by the pair of pants product.
The papers 
\cite{CieliebakFloerHofersymplectichomologyII},
\cite{cieliebak2015symplectic},
\cite{groman2015floer},
\cite{venkatesh2017rabinowitz} and
\cite{umutvarolgunessymplecticcohomology}
have a similar definition of symplectic cohomology.
However there are slight differences between all of these definitions (which potentially could lead to different algebras).
One main difference is that some of the definitions above
involve building a chain complex first,
and then taking homology.
Our definition does not do this, but only for the sake of ease.
The symplectic cohomology algebra satisfies the following properties (see the cited definitions and Propositions/Theorems for more accurate statements):
\begin{enumerate}[label=(SH\arabic*)]
\item \label{item transfer map property}
(Definition \ref{definition transfer map}). If $K_+ \subset K_- \subset D$ are closed subsets then
we have a {\it transfer map}
$SH^*_{\check{C},Q_-,Q_+}(K_- \subset M) \lra{} SH^*_{\check{C},Q_-,Q_+}(K_+ \subset M)$
which is functorial.
\item (Definition \ref{action maps for symplectic cohomology}).
If $Q^1_\pm \subset Q^0_\pm$ then there is an {\it action map}
$SH^*_{\check{C},Q^0_-,Q^0_+}(K \subset M) \lra{} SH^*_{\check{C},Q^1_-,Q^1_+}(K \subset M)$. These maps commute with continuation maps and are functorial.
\item \label{item isomorphic to quantum cohomology}
(Theorem \ref{theorem isomorphic to quantum cohomology}). If $K = M$ and
$\check{C}$, $Q_\pm$ are as in \ref{item one dimensional cones} then
$SH^*_{\check{C},Q_-,Q_+}(K \subset M)$
is isomorphic to quantum cohomology.
Also the `derived' version of symplectic cohomology $\varinjlim_{a_-} \varprojlim^1_{a_+} HF^*_{\check{C},a_-,a_+}(H_K)$ vanishes.
\item \label{item stably displaceable property}
(Theorem \ref{theorem stably displaceable complement}).
If the complement $M - K$ is stably displaceable (I.e. $(M- K) \times S^1 \subset M \times T^* S^1$ is displaceable by a Hamiltonian symplectomorphism) 
and $\check{C}$, $Q_\pm$ are as in \ref{item one dimensional cones}
then the transfer map
$SH^*_{\emptyset,Q_-,Q_+}(M \subset M) \lra{} SH^*_{\emptyset,Q_-,Q_+}(K \subset M)$ is an isomorphism.
\item \label{item alternative filtration property}
(Proposition \ref{proposition alternative filtrations defining symplectic cohomology}). If $\check{C}$, $Q_-$, $Q_+$
are as in \ref{item two dimensional cone}
and the Liouville domain $D$ is {\it index bounded} (Definition \ref{definition index bounded contact cylinder})
then the
action map
$$SH^*_{\check{C},Q_-,Q_+}(D \subset M) \lra{} SH^*_{\check{C},Q_+,Q_+}(D \subset M)$$ is an isomorphism.
\item \label{item index bounded transfer isomorphism}
(Proposition \ref{proposition transfer isomorphism between index bounded Liouville domains}).
Suppose that $\check{C}_0$, $\check{C}_1$ are index bounded contact cylinders with associated Liouville domains $D_0$ and $D_1$ respectively satisfying $D_1 \subset D_0$ along with some other conditions (essentially $D_0$ and $D_1$ need to be `large' in some sense). 
Then the transfer map
$SH^*_{\check{C},Q_-,Q_+}(D_0 \subset M) \lra{} SH^*_{\check{C},Q_-,Q_+}(D_1 \subset M)$ is an isomorphism
where $(Q_-,Q_+)$ is as in \ref{item two dimensional cone}.
\item \label{item changing novikov ring}
(Theorem \ref{theorem changing novikov ring} and Proposition \ref{flatness of rational polyhedral novikov rings}).
Let $\check{C}$ be an index bounded contact cylinder with associated Liouville domain $D$.
Suppose we have inclusions of rational polyhedral cones $Q^1_\pm \subset Q^0_\pm$ where $Q^1_-$ has dimension at least $2$ (e.g. case \ref{item two dimensional cone} or
\ref{item four dimensional cone}).
Also suppose
$\varinjlim_{a_-} \varprojlim^1_{a_+} HF^*_{\check{C},a_-,a_+}(H_D) = 0$
where $H_D$ is defined in (\ref{equation index function Hamiltonian}).
Then we have an isomorphism
$$SH^*_{\check{C},Q^0_-,Q^0_+}(D \subset M) \otimes_{\Lambda_\K^{Q^0_+}} \Lambda_\K^{Q^1_+}
\lra{\cong} SH^*_{\check{C},Q^1_-,Q^1_+}(D \subset M)$$
induced by the action map.
\item \label{item avoiding submanifolds}
(Proposition \ref{proposition regular subset for surface} and Lemma \ref{lemma regular Hamiltonians}).
Let $\check{C}$ be a contact cylinder with associated Liouville domain $D$.
Suppose $V \subset M - D - \check{C}$ is a union of real codimension $\geq 4$ submanifolds.
Let $Q_\pm$ be cones so that $Q_-$ is of dimension $\geq 2$.
Then the Floer trajectories and orbits defining $SH^*_{\check{C},Q_-,Q_+}(D \subset M)$ can be made to avoid $V$.
Hence this group only depends on these structures restricted to $M -V$.
\end{enumerate}

\subsubsection{Sketch of Proof of Main Theorem \ref{theorem main theorem}.}

Here we summarize the ideas behind the proof of Theorem \ref{theorem main theorem} coming from Sections
\ref{symplectic geometry of projective varieties} and \ref{section proof of main theorem}.
The proof has two parts.
In part (1), we modify the symplectic forms on $X$ and $\widehat{X}$ so that they agree on a certain large open subset and so that they admit certain index bounded contact cylinders.
In part (2), we use properties
\ref{item transfer map property}-\ref{item avoiding submanifolds} to finish our proof.

{\it Part (1)}:
First of all, we choose
Zariski dense affine subvarieties $A \subset X$, $\widehat{A} \subset \widehat{X}$ so that
\begin{enumerate}
\item the birational morphism
$\widehat{\Phi}$ induces an isomorphism
$\Phi : A \lra{} \widehat{A}$ and
\item $\omega_X$ and $\omega_{\widehat{X}}$ come from effective ample divisors with support equal to $X - A$ and $\widehat{X} - \widehat{A}$ respectively (after rescaling these forms).
\end{enumerate}
By Corollary \ref{proposition stably displaceable partly stratified symplectic subset}, we have that
$X - A$ is stably displaceable (by an $h$-principle).
Hence there is a compact subset
$K \subset A$ so that $X - K$
is stably displaceable.
By Proposition
\ref{proposition constructino of index bounded contact cylinder on appropriate affine variety}
we can construct an index bounded contact cylinder
$\check{C}$ on $X$ whose associated Liouville domain $D$ contains $K$ (here we use the fact that $\widehat{X}$ is Calabi-Yau).
By using the ample divisors above,
we can modify the K\"{a}hler form $\omega_{\widehat{X}}$
(without changing its cohomology class up to rescaling)
so that
$\omega_X$ and $\Phi^* \omega_{\widehat{X}}$ agree near $D$.
Again, by Corollary \ref{proposition stably displaceable partly stratified symplectic subset} we can find a compact subset $\widehat{K} \subset \widehat{A}$ containing $\Phi(D)$ so that $\widehat{X} - \widehat{K}$ is stably displaceable.
Also by Proposition
\ref{proposition constructino of index bounded contact cylinder on appropriate affine variety}
one can construct an index bounded contact cylinder 
$\widehat{C}$ in $\widehat{A}$ whose associated Liouville domain $\widehat{D} \subset \widehat{A}$ contains $\widehat{K}$.

{\it Part (2)}:
From now on we identify $H^2(X,D;\R) = H^2(\widehat{X},\widehat{D};\R)$.
We let $Q_{\omega_X}$ and $Q_{\omega_{\widehat{X}}}$ be $1$-dimension cones spanned by
$([\omega_X],1,1)$ and $([\omega_{\widehat{X}}],1,1)$ respectively as in \ref{item one dimensional cones}.
We let
$Q_\pm$, $\widehat{Q}_\pm$ be the corresponding enlarged cones as in \ref{item two dimensional cone}.
Finally we let
$\widetilde{Q}_\pm$ be the cones spanned by both $Q_\pm$ and $\widehat{Q}_\pm$
(these are cones as in \ref{item four dimensional cone}).
By \ref{item isomorphic to quantum cohomology}, \ref{item stably displaceable property}
and \ref{item alternative filtration property}
we have that
$SH^*_{\check{C},Q_-,Q_+}(D \subset X)$
and
$SH^*_{\widehat{C},\widehat{Q}_-,\widehat{Q}_+}(\widehat{D} \subset {\widehat{X}})$ are isomorphic to the quantum cohomology rings of $X$ and $\widehat{X}$ respectively.
By \ref{item index bounded transfer isomorphism}
we have that the transfer map
$
SH^*_{\widehat{C},\widehat{Q}_-,\widehat{Q}_+}(\widehat{D} \subset {\widehat{X}})
\lra{}
SH^*_{\Phi(\check{C}),\widehat{Q}_-,\widehat{Q}_+}(\Phi(D) \subset {\widehat{X}})$
is an isomorphism.
Define the $\Lambda_\K^{\omega_X,\omega_{\widehat{X}}}$-algebra
$Z := SH^*_{\check{C},\widetilde{Q}_-,\widetilde{Q}_+}(D \subset X)$.
Now since the regions $V_X \subset X$
and $V_{\widehat{X}} \subset \widehat{X}$ for which the birational morphisms $\widehat{\Phi}$
and $\widehat{\Phi}^{-1}$
are ill defined has real codimension $\geq 4$ by Lemma \ref{lemma identification of second homology groups}, we have by
\ref{item avoiding submanifolds}
an isomorphism of $\Lambda_\K^{\omega_X,\omega_{\widehat{X}}}$-algebras
\begin{equation} \label{equation Z isomorphism}
Z \cong SH^*_{\check{C},\widetilde{Q}_-,\widetilde{Q}_+}(\Phi(D) \subset \widehat{X}).
\end{equation}
The isomorphisms
(\ref{equation quantum cohomoogy isomorphisms})
now follow from Equation (\ref{equation Z isomorphism})
combined with the second part of
\ref{item isomorphic to quantum cohomology}
and
\ref{item changing novikov ring}.

\subsection{Notation Throughout this paper} \label{section notation}
\begin{itemize}
	\item We will fix a ring $\K$.
	\item $(M,\omega)$ will be a compact connected symplectic manifold of dimension $2n$ satisfying $c_1(\omega) = 0$
	and where $[\omega] \in H^2(M;\R)$ lifts to an integral cohomology class,
	\item $J_0$ is a fixed almost complex structure on $M$ taming $\omega$,
	\item $(V_i)_{i = 1}^l$ is a finite
	collection of
	 (not necessarily properly embedded) submanifolds of $M$
	of codimension $\geq 4$
	and where $V := \cup_{i = 1}^l V_i$ is compact,
	\item $\T := \R / \Z$, $\bbI_- := (-\infty,0]$, $\bbI_+ := [0,\infty)$,
	\item $(s,t)$ will be the natural coordinate system on $\bbI_\pm \times \T$
	or $\R \times \T$.
	\item 
	For any manifold $\Sigma$,
	$\ccJ^\Sigma(J_0,V,\omega)$ is the space of smooth families of $\omega$-tame almost complex structures $J := (J_\sigma)_{\sigma \in \Sigma}$
	smoothly parameterized by $\Sigma$
	 equipped with the $C^\infty$ topology
	so that all the derivatives of $J$
	and $J_0$ agree
	at $v$ for all $v \in V$.
	\item $\ccJ(J_0,V,\omega) := \ccJ^{\text{pt}}(J_0,V,\omega)$.
	\item If $I$ is a set and $W$ is a vector space then $W^I$ is the vector space of maps $I \lra{} W$
	or equivalently tuples $(w_j)_{j \in I}$ of elements in $W$.
\end{itemize}

\bigskip

{\it Acknowledgments}: I would like to thank Nick Sheridan for extremely helpful conversations concerning Novikov parameters, Strom Borman for helping me with transfer map properties in Section \ref{subsection Transfer Isomorphisms between Index bounded Liouville Domains} and Umut Varolgunes
for many helpful suggestions and corrections.
I would also like to thank Paul Seidel for his encouragement and suggestions for future work and Yongbin Ruan and Y.P. Lee
for their helpful comments.
Finally I would like to thank Chenyang Xu and Qizheng Yin for helpful discussions.
This paper is supported by the NSF grants DMS-1508207 and DMS-1811861.

\section{Hamiltonian Floer Cohomology and Filtrations} \label{section Hamitlonian Floer cohomology and Filtrations}

\subsection{Alternative Action Values of Periodic Orbits.}

\begin{defn}\label{defn capped 1-periodic orbit}
	A {\it Hamiltonian} is a smooth family of functions
	$H = (H_t)_{t \in \T}$ on a symplectic manifold (by default this is $(M,\omega)$ unless stated otherwise).
	It is {\it autonomous} if $H_t$ does not depend on $t$,
	and hence we usually express such a Hamiltonian as a single function $M \lra{} \R$.
	The {\it time $t$ flow
		$(\phi^H_t : M \lra{} M)_{t \in \R}$}
	of a Hamiltonian
	$H \equiv (H_t)_{t \in \T}$
	is the time $t$ flow of
	the unique time dependent
	vector field
	$(X^H_t)_{t \in \R}$
	satisfying
	$i_{X^H_t}\omega = -dH_t$ for all $t \in \R$.
	A {\it $1$-periodic orbit}
	is a smooth map
	$\gamma : \T \lra{} M$
	satisfying $\dot{\gamma} = X^H_t$ for all $t \in \T$.
	A $1$-periodic orbit $\gamma$
	is {\it non-degenerate}
	if the linearized return map
	$D\phi^H_1 : T_{\gamma(0)}M \lra{} T_{\gamma(0)}M$
	has no eigenvalue equal to $1$.
	A {\it capped loop} is an
	equivalence class of pairs $(\widetilde{\gamma},\check{\gamma})$
	of smooth maps
	\begin{equation} \label{equation pair of maps}
	\widetilde{\gamma} : S \lra{} M, \quad \check{\gamma} : \T \lra{} \partial S
	\end{equation}
	where $S$ is a smooth oriented surface with boundary,  $\check{\gamma}$ is an orientation preserving diffeomorphism
	and where any two such pairs
	$ (\widetilde{\gamma}_0,\check{\gamma}_0)$, $(\widetilde{\gamma}_1,\check{\gamma}_1)$
	are equivalent if $\widetilde{\gamma}_0 \circ \check{\gamma}_0 = \widetilde{\gamma}_1 \circ \check{\gamma}_1$
	and if the surface obtained by gluing $\widetilde{\gamma}_0$ and $\widetilde{\gamma}_1$ along the boundary via the map $\check{\gamma}_0 \circ \check{\gamma}_1^{-1}$ is null homologous.
	More precisely, this gluing is defined to be the continuous map
	\begin{equation} \label{equation join of two cappings}
	\widetilde{\gamma}_0 \star
	 \widetilde{\gamma}_1 : S_0 \sqcup \overline{S}_1 / \sim \lra{} M, \quad
	\gamma_0 \star \gamma_1(\sigma) :=
	\left\{
	\begin{array}{ll}
	\widetilde{\gamma}_0 & \text{if} \ z \in S_0 \\
	\widetilde{\gamma}_1 & \text{otherwise}
	\end{array}
	\right.
		\end{equation}
	where $S_0$ is the domain of $\widetilde{\gamma}_0$,
	$\overline{S}_1$ is the domain
	of $\widetilde{\gamma}_1$ with the opposite orientation
	 and where the identification $\sim$ is defined to be
	$$\partial S_0 \ni \check{\gamma}_0(t) \sim \check{\gamma}_1(t) \in \partial S_1, \ t \in \T.
	$$
	If $(\widetilde{\gamma},\check{\gamma})$
	is a capped loop, then the
	{\it associated loop} of $(\widetilde{\gamma},\check{\gamma})$
	is the map
	$\widetilde{\gamma} \circ \check{\gamma} : \T \lra{} M$.
	We define
	$\widetilde{\ccL}(M)$ to be
	the space of capped loops
	equipped with the quotient topology induced from the $C^\infty$ topology on the space of pairs of maps as in Equation (\ref{equation pair of maps}).
	
	A {\it capped $1$-periodic orbit $\gamma$} of a Hamiltonian $H$
	is a capped loop, whose associated loop $\overline{\gamma} : \T \lra{} M$
	is a $1$-periodic orbit of $H$.
	We call $\overline{\gamma}$ the {\it associated $1$-periodic orbit} of $\gamma$.
	A capped $1$-periodic orbit $\gamma$ is
	{\it non-degenerate}
	if the associated $1$-periodic orbit
	is non-degenerate.
\end{defn}

\begin{defn} \label{definition action}
	Let $K = (K_\sigma)_{\sigma \in \Sigma}$ be a smooth family of functions on $M$ parameterized by a manifold $\Sigma$.
	Let $\widetilde{\omega}$
	be a closed $2$-form on $M - V$.
	We say that $K$ is {\it $\widetilde{\omega}$-compatible} if
	there is a smooth family of functions
	$G = (G_\sigma)_{\sigma \in \Sigma}$ on $M - V$
	so that
	\begin{equation} \label{equation H compatible with omega}
	i_{X_{K_\sigma}} \widetilde{\omega}
	= -dG_\sigma, \quad \forall \ \sigma \in \Sigma.
	\end{equation}
	We will call $G$ {\it a primitive associated to $(K,\widetilde{\omega})$}.

	Now let $H := (H_t)_{t \in \T}$
	be a Hamiltonian which is $\widetilde{\omega}$-compatible and $F = (F_t)_{t \in \T}$ a primitive associated to
	$(H,\widetilde{\omega})$.
	The {\it $(H,\widetilde{\omega},F)$-action
	of a capped loop $\gamma := (\widetilde{\gamma},\check{\gamma})$
	} on $M$ 
where the associated loop
	$\overline{\gamma} : \T \lra{} M$ is disjoint from $V$
	is defined to be
	\begin{equation} \label{eqn:fomegahactionfunctional}
	\ \cA_{H,\widetilde{\omega},F}(\gamma) := -\int_S (\widetilde{\gamma}')^* \widetilde{\omega} + \int_0^1 F_t(\overline{\gamma}(t)) dt
	\end{equation}
	where $\widetilde{\gamma}' : S \lra{} M$ is some $C^\infty$ small perturbation of $\widetilde{\gamma}$ away from  $\partial S$
	so that its image is disjoint from $V$.
	If $\widetilde{\omega}$
	extends to a smooth $2$-form on $M$
	 and $F$ extends to a smooth family of functions on $M$ then we define
	the {\it $(H,\widetilde{\omega},F)$-action}
	of any capped loop $\gamma$
	by Equation (\ref{eqn:fomegahactionfunctional}) with $\widetilde{\gamma}'$ replaced by $\gamma$.
	If $(\widetilde{\gamma},\check{\gamma})$
	is any capped loop whose associated loop is constant then we define
	the {\it $(H,\widetilde{\omega},F)$-action}
	$\cA_{H,\widetilde{\omega},F}(\gamma)$
	to be 	$\cA_{H,\widetilde{\omega},F}(\gamma')$
	where $\gamma'$ is a capped loop disjoint from $V$ which is smoothly isotopic to $\gamma$ through capped loops with constant associated loops.
\end{defn}

Note that the perturbations $\widetilde{\gamma}'$ and $\gamma'$ above exist since $V$ is a finite union of codimension $\geq 4$ submanifolds of $M$.
Note also that such an action will usually be computed for
capped $1$-periodic orbits of
$H$.

We will only deal with very specific
closed $2$-forms $\widetilde{\omega}$
associated to certain contact hypersurfaces
inside $(M,\omega)$.
We will now introduce such closed $2$-forms.

\begin{defn} \label{definition contact cylinder} (See Figure \ref{fig:contactcylinder})
	A {\it contact cylinder in $M$}
	consists of a codimension $0$ submanifold
	$$\check{C} := [1-\epsilon,1+\epsilon] \times C \subset M - V$$
	so that $\omega|_{\check{C}} = d(r_C \alpha_C)$ where
	$r_C : \check{C} \lra{} [1-\epsilon,1+\epsilon]$ is the natural projection map, and $\alpha_C$ is a contact form on $C$.
	We also require that $\{0\} \times C$ is the boundary of a compact codimension $0$ submanifold $D$ 
	so that $\omega|_D = d\theta$ for some $\theta \in \Omega^1(D)$ satisfying $\theta|_{D \cap \check{C}} = r_C \alpha_C$.
	Here $r_C$ is called
	the {\it radial coordinate associated to $\check{C}$}, 
	$\alpha_C$ is called the {\it contact form associated to $\check{C}$}
	and $D$ is
	called the {\it Liouville domain associated to $\check{C}$}.
	A $2$-form 
	$\widetilde{\omega} \in \Omega^2(M-V)$
	is {\it $\check{C}$-compatible}
	if 
	\begin{enumerate}[label=(\alph*)]
		\item $\widetilde{\omega}$ is
		closed and $J_0$-tame outside
		$D \cup \check{C}$,
		\item 
		$\widetilde{\omega}|_{\check{C}} = 
		d(f_{\widetilde{\omega}}(r_C)r_C \alpha)$ where $f_{\widetilde{\omega}} : \R \lra{} \R$ is a smooth function satisfying:
	$$f_{\widetilde{\omega}}|_{(-\infty,1+\epsilon/4]} = \lambda^-_{\widetilde{\omega}}, \quad
	f_{\widetilde{\omega}}|_{[1+\epsilon/2,\infty)} = \lambda^+_{\widetilde{\omega}},
	\quad f_{\widetilde{\omega}}' \geq 0.
	$$
	for some constants $\lambda^\pm_{\widetilde{\omega}} \geq 0$,	
	\item $\widetilde{\omega}|_D = \lambda_{\widetilde{\omega}}^-\omega$ and
	\item $\widetilde{\omega} = \lambda^+_{\widetilde{\omega}} \omega$ if $\check{C}$ is the empty set. \end{enumerate}
	We call
	$\lambda^\pm_{\widetilde{\omega}}$ the {\it scaling constants for $\widetilde{\omega}$} and
	$f_{\widetilde{\omega}}$ the
	{\it scaling function for $\widetilde{\omega}$}.
	A family of $2$-forms
	$\widetilde{\omega}^\bullet = (\widetilde{\omega}^\sigma)_{\sigma \in \Sigma}$ is {\it $\check{C}$-compatible}
	if $\widetilde{\omega}_\sigma$ is $\check{C}$-compatible for each $\sigma \in \Sigma$.
	If $\check{C}$ is the empty contact cylinder then $\lambda^-_{\widetilde{\omega}}$ is defined to be arbitrary
	(I.e. we can choose $\lambda^-_{\widetilde{\omega}}$ to be anything we like and it is considered as part of the data defining $\widetilde{\omega}$).
\end{defn}

\begin{center}
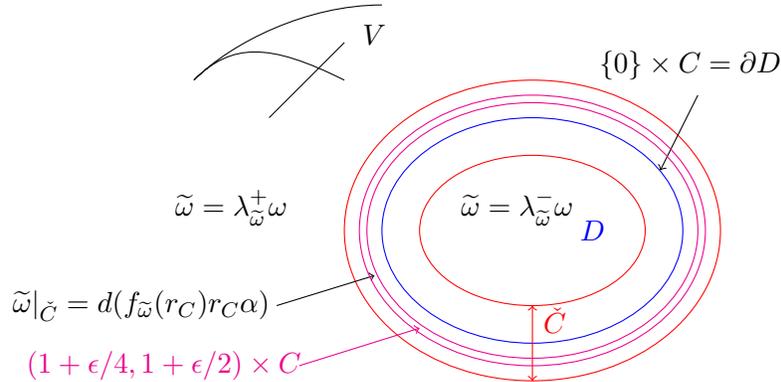
\begin{figure}[h]
\begin{tikzpicture}

\draw [blue] (1,-2) node (v1) {} ellipse (2 and 1.5);
\draw [red] (v1) node (v2) {} ellipse (1.5 and 1);
\draw [red] (v2) ellipse (2.5 and 2);
\draw (-1,1) .. controls (-2,1) and (-3,0.5) .. (-3.5,0) .. controls (-3,0.5) and (-2.5,0.5) .. (-1.5,0);
\draw (-1.5,0.5) -- (-2.5,-0.5);
\node at (1.8,-2) {\color{blue} $D$};
\draw [red,<->](1,-3) -- (1,-4);
\node at (1.3,-3.2) {\color{red} $\check{C}$};
\draw [->](3.2,-0.2) -- (2.7,-1.2);
\node at (3.1,0.2) {$\{0\} \times C = \partial D$};
\node at (0.8,-1.7) {$\widetilde{\omega} = \lambda^-_{\widetilde{\omega}} \omega$};
\node at (-3,-1.7) {$\widetilde{\omega} = \lambda^+_{\widetilde{\omega}} \omega$};
\draw[->] (-2.4,-3) -- (-1.1,-2.6);
\node at (-4.2,-3) {$\widetilde{\omega}|_{\check{C}} = 
	d(f_{\widetilde{\omega}}(r_C)r_C \alpha)$};
\node at (-1.1,0.6) {$V$};
\draw [magenta] (v1) node (v3) {} ellipse (2.2 and 1.7);
\draw [magenta] (v3) ellipse (2.3 and 1.8);
\draw [magenta,->](-2.1,-3.8) -- (-0.5,-3.3);
\node at (-3.9,-3.8) {\color{magenta} $(1+\epsilon/4,1+\epsilon/2) \times C$ };
\end{tikzpicture}
\caption{A contact cylinder} \label{fig:contactcylinder}
\end{figure}
\end{center}

\begin{defn} \label{definition Hamiltonians compatible with contact cylinder}
	An autonomous Hamiltonian $K : M \lra{} \R$
	is {\it weakly $\check{C}$-compatible}
	if
	$K|_{[1+\epsilon/8,1+\epsilon/2] \times C} = \lambda_K r_C + m_K$
	for some constants $\lambda_K$ and $m_K$.
	The constant
	$\lambda_K$ is called the {\it slope of $K$ along $\check{C}$} and
	$m_K$ is called the {\it height of $K$ at $\check{C}$}.
	Also if $\check{C}$ is the empty contact
	cylinder then we define the slope and height of $K$ to be $0$.
	We say that $K$ is
	{\it $\check{C}$-compatible}
	if it is weakly $\check{C}$-compatible and
	if $K|_{M - (D \cup \check{C})}$ is constant.
	A smooth family of autonomous Hamiltonians
	$K_\bullet := (K_\sigma)_{\sigma \in \Sigma}$ parameterized by a manifold
	$\Sigma$ is {\it (weakly) $\check{C}$-compatible}
	if $K_\sigma$ is (weakly) $\check{C}$-compatible for each $\sigma \in \Sigma$.
	An almost complex structure $J$ on $M$
	is {\it $\check{C}$-compatible}
	if
	$J \in \ccJ(J_0,V,\omega)$,
	and $dr_C \circ J = -\alpha_C$ inside
	$[1+\epsilon/8,1+\epsilon/2] \times C$.
	A smooth family of almost complex structures
	$J_\bullet = (J_\sigma)_{\sigma \in \Sigma}$
	is {\it $\check{C}$-compatible}
	if $J_\sigma$ is $\check{C}$-compatible
	for each $\sigma \in \Sigma$.
\end{defn}
	
\begin{remark}
The space of $2$-forms
which are $\check{C}$-compatible
is weakly contractible
(in fact it forms a convex subset of $\Omega^2(M)$).
Also the space of
(weakly) $\check{C}$-compatible
Hamiltonians (resp. almost complex structures)
is weakly contractible.

\end{remark}

By a direct calculation, we have the following lemma and corollary:

\begin{lemma} \label{lemma Hamiltonian admissible with two form associated to contact cylinder}
	Let 
	\begin{itemize}
	\item $\check{C}$ be a contact cylinder with cylindrical coordinate $r_C$ and associated Liouville domain $D$,
	\item $\widetilde{\omega}$ be a $2$-form compatible with $\check{C}$
	where $\lambda^\pm_{\widetilde{\omega}}$
	(resp. $f_{\widetilde{\omega}}$)
	are the scaling constants (resp.  scaling function) for $\widetilde{\omega}$ and
	\item let $H = (H_t)_{t \in \T}$ be a weakly $\check{C}$-compatible Hamiltonian which is $\check{C}$-compatible
	if $\widetilde{\omega}$
	is not a locally constant multiple of $\omega$ outside $[1+\epsilon/4,1+\epsilon/2] \times C$
	and
	 let
	$(\lambda_{H_t})_{t \in \T}$
	and $(m_{H_t})_{t \in \T}$ be the slopes and heights of $(H_t)_{t \in \T}$ along $\check{C}$ respectively.
	\end{itemize}
	Then the smooth family of
	functions $F^{H,\check{C},\widetilde{\omega}} := (F^{H_t,\check{C},\widetilde{\omega}})_{t \in \T}$
	defined by
	\begin{equation} \label{equation primitive assocatiated to contact cylinder}
	F^{H_t,\check{C},\widetilde{\omega}} : M \lra{} \R, \ \
	F^{H_t,\check{C},\widetilde{\omega}} :=
	\left\{
	\begin{array}{ll}
	\lambda^-_{\widetilde{\omega}} H_t & \text{in} \ D \cup ([1,1+\epsilon/4]\times C) \\
	\lambda_{H_t} f_{\widetilde{\omega}}(r_C)r_C + \lambda^-_{\widetilde{\omega}} m_{H_t} & \text{in} \ [1+\epsilon/4,1+\epsilon/2] \times C \\
	\lambda^+_{\widetilde{\omega}}(H_t - m_{H_t}) + \lambda^-_{\widetilde{\omega}}m_{H_t} & \text{otherwise}
	\end{array}
	\right.
	\end{equation}
	is a primitive associated to $(H,\widetilde{\omega})$
	as in Definition \ref{definition action}
	for all $t \in \T$ (See Figure \ref{fig:Primitive}).
\end{lemma}

\begin{center}
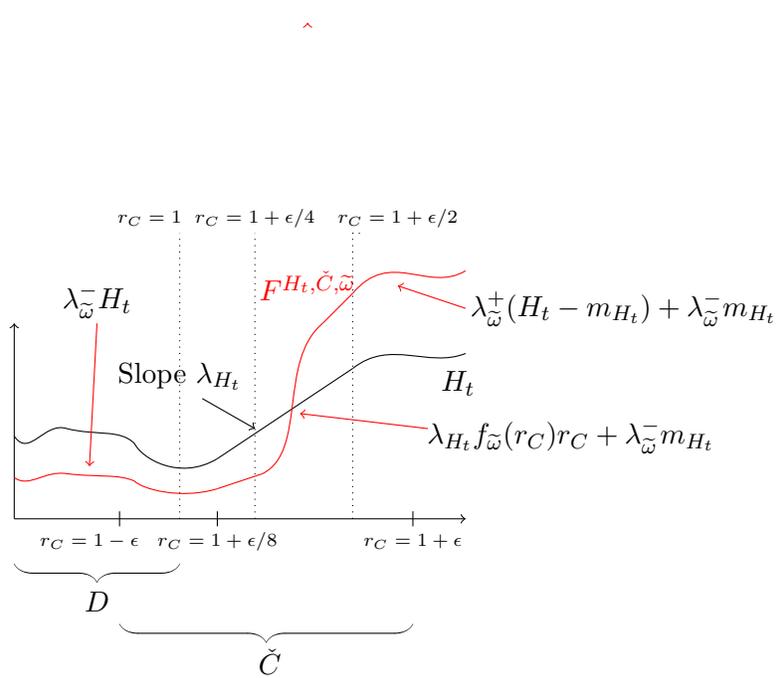
\begin{figure}[h]
\begin{tikzpicture}

\draw [->](-2,-2.5) -- (4,-2.5);
\draw [->](-2,-2.5) -- (-2,0.1);
\draw (0.7,-2.4) -- (0.7,-2.6) node (v1) {};
\draw (0.7,-1.7) -- (2.5,-0.5);
\draw (0.7,-1.7) .. controls (0.2,-2) and (-0.3,-1.7) .. (-0.4,-1.5) .. controls (-0.6,-1.3) and (-0.9,-1.4) .. (-1.3,-1.3) .. controls (-1.6,-1.2) and (-1.8,-1.7) .. (-2,-1.4);
\draw (2.5,-0.5) .. controls (3,-0.1) and (3.5,-0.5) .. (4,-0.3);

\draw[red] (0.7,-1.7*0.5-1.25) .. controls (0.2,-2*0.5-1.25) and (-0.3,-1.7*0.5-1.25) .. (-0.4,-1.5*0.5-1.25) .. controls (-0.6,-1.3*0.5-1.25) and (-0.9,-1.4*0.5-1.25) .. (-1.3,-1.3*0.5-1.25) .. controls (-1.6,-1.2*0.5-1.25) and (-1.8,-1.7*0.5-1.25) .. (-2,-1.4*0.5-1.25);
\draw[red] (2.5,-0.5*1.5+1.25) .. controls (3,-0.1*1.5+1.25) and (3.5,-0.5*1.5+1.25) .. (4,-0.3*1.5+1.25);
\draw[red] (0.7,-1.7*0.5-1.25) -- (1.3,-1.9);
\draw[red] (2.1,0.1) -- (2.5,-0.5*1.5+1.25);

\draw [red](2.1,0.1) .. controls (1.5,-0.4) and (1.9,-1.6) .. (1.3,-1.9);
\node at (3.9,-0.7) {$H_t$};
\draw[decoration={calligraphic brace,amplitude=7pt},decorate] (0.2,-3.1) -- (-2,-3.1);
\node at (-0.9,-3.6) {$D$};
\node at (-0.2,1.5) {\tiny $r_C=1$};
\node at (0.7,-2.8) {\tiny $r_C=1+\epsilon/8$};
\node at (1.2,1.5) {\tiny $r_C=1+\epsilon/4$};
\node at (3.1,1.5) {\tiny $r_C=1+\epsilon/2$};
\node at (1.9,0.6) {\color{red}$F^{H_t,\check{C},\widetilde{\omega}}$};
\draw [->,red](-0.9,0.1) -- (-1,-1.8);
\node at (-0.9,0.4) {$\lambda^-_{\widetilde{\omega}} H_t$};
\node at (1.3,-0.7) {};
\draw [->,red](4,0.3) -- (3.1,0.6);
\node at (6.1,0.3) {$\lambda^+_{\widetilde{\omega}}( H_t - m_{H_t}) + \lambda^-_{\widetilde{\omega}} m_{H_t}$};
\draw [->](0.5,-0.9) -- (1.2,-1.3);
\node at (0.2,-0.6) {Slope $\lambda_{H_t}$};
\draw [->,red](3.5,-1.3) -- (1.8,-1.1);
\node at (5.4,-1.4) {$\lambda_{H_t} f_{\widetilde{\omega}}(r_C)r_C + \lambda^-_{\widetilde{\omega}} m_{H_t}$};
\draw [->,red](1.9,4.1);
\draw [dotted](1.2,-2.5) -- (1.2,1.3);
\draw [dotted](2.5,-2.5) -- (2.5,1.3) node (v2) {} -- (v2);
\draw [dotted](0.2,-2.5) -- (0.2,1.3);
\draw [](3.3,-2.4) -- (3.3,-2.6);
\node at (3.3,-2.8) {\tiny $r_C=1+\epsilon$};
\draw [](-0.6,-2.4) -- (-0.6,-2.6);
\node at (-1,-2.8) {\tiny $r_C=1-\epsilon$};
\draw [decoration={calligraphic brace,amplitude=7pt},decorate](3.3,-3.9) -- (-0.6,-3.9);
\node at (1.4,-4.4) {$\check{C}$};
\end{tikzpicture}
\caption{Primitive associated to $(H,\widetilde{\omega})$} \label{fig:Primitive}
\end{figure}
\end{center}

\begin{remark} \label{remark mostly flat}
Most of the important Hamiltonians
in subsections \ref{section alternative filtratrions},
\ref{subsection Transfer Isomorphisms between Index bounded Liouville Domains}
and
\ref{subsection chain complex for sh}
 will have small derivatives outside
$D \cup ([1,1+\epsilon/8] \times C)$.
This makes our calculations easier.
It is good to keep such Hamiltonians in mind throughout this paper since
they appear in many of the most important calculations.
In this special case, Figure \ref{fig:Primitive} might look like Figure \ref{fig:specialPrimitive} instead.
\begin{center}
\begin{figure}[h]
	\begin{tikzpicture}
	
	\draw [->](-2,-2.5) -- (4,-2.5);
	\draw [->](-2,-2.5) -- (-2,0.1);
	\draw (0.7,-2.4) -- (0.7,-2.6){};
	\draw (-0.4,-1.3) -- (0,-0.4);
	\draw (-0.4,-1.3) .. controls (-0.6,-1.8) and (-0.9,-2) .. (-1.2,-1.5) .. controls (-1.5,-1.3) and (-1.4,-1.5) .. (-1.6,-1.3) .. controls (-1.8,-1.2) and (-1.8,-1.7) .. (-2,-1.4);
	\draw (0,-0.4) .. controls (0.2,0) and (0.3,-0.1) .. (1.1,-0.1);
	\draw (1.1,-0.1) -- (2.7,-0.2);
	\draw (2.7,-0.2) .. controls (3.7,-0.3) and (3.5,-0.2) .. (3.9,-0.2);
	
	\draw[red] (-0.4,-1.3*0.5-1.25) .. controls (-0.6,-1.8*0.5-1.25) and (-0.9,-2*0.5-1.25) .. (-1.2,-1.5*0.5-1.25) .. controls (-1.5,-1.3*0.5-1.25) and (-1.4,-1.5*0.5-1.25) .. (-1.6,-1.3*0.5-1.25) .. controls (-1.8,-1.2*0.5-1.25) and (-1.8,-1.7*0.5-1.25) .. (-2,-1.4*0.5-1.25);
	\draw[red] (-0.4,-1.3*0.5-1.25) -- (0,-0.4*0.5-1.25);
	\draw[red] (1.1,-0.1*0.5-1.25) -- (2.7,-0.2*0.5-1.25);
	\draw[red] (0,-0.4*0.5-1.25) .. controls (0.2,0*0.5-1.25) and (0.3,-0.1*0.5-1.25) .. (1.1,-0.1*0.5-1.25);
	\draw[red] (1.1,-0.1*0.5-1.25) -- (2.7,-0.2*0.5-1.25);
	\draw[red] (2.7,-0.2*0.5-1.25) .. controls (3.7,-0.3*0.5-1.25) and (3.5,-0.2*0.5-1.25) .. (3.9,-0.2*0.5-1.25);
	
	\node at (3,0) {$H_t$};
	\draw[decoration={calligraphic brace,amplitude=7pt},decorate](0.2,-3.1) -- (-2,-3.1);
	\node at (-0.9,-3.6) {$D$};
	\node at (-0.2,1.5) {\tiny $r_C=1$};
	\node at (0.7,-2.8) {\tiny $r_C=1+\epsilon/8$};
	\node at (1.2,1.5) {\tiny $r_C=1+\epsilon/4$};
	\node at (3.1,1.5) {\tiny $r_C=1+\epsilon/2$};
	\draw [->,red](0.7,-1.8) -- (0.1,-1.6);
	\node at (2.2,-1.8) {\color{red} $F^{H_t,\check{C},\widetilde{\omega}} \approx \lambda^-_{\widetilde{\omega}} H_t$};

	\draw [dotted](1.2,-2.5) -- (1.2,1.3);
	\draw [dotted](2.5,-2.5) -- (2.5,1.3) node (v2) {} -- (v2);
	\draw [dotted](0.2,-2.5) -- (0.2,1.3);
	\draw [](3.3,-2.4) -- (3.3,-2.6);
	\node at (3.3,-2.8) {\tiny $r_C=1+ \epsilon$};
	\draw [](-0.6,-2.4) -- (-0.6,-2.6);
	\node at (-1,-2.8) {\tiny $r_C=1-\epsilon$};
	\draw[decoration={calligraphic brace,amplitude=7pt},decorate] (3.3,-3.9) -- (-0.6,-3.9);
	\node at (1.4,-4.4) {$\check{C}$};
	\end{tikzpicture}
\caption{Primitive associated to $(H,\widetilde{\omega})$, where $H$ has small derivatives outside $D \cup ([1,1+\epsilon/8] \times C)$.} \label{fig:specialPrimitive}
\end{figure}
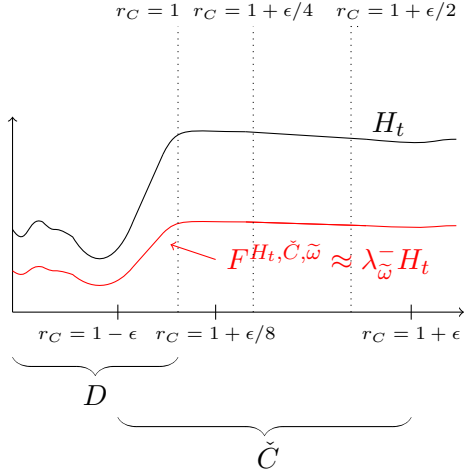
\end{center}
\end{remark}

We have the following corollary
of Lemma \ref{lemma Hamiltonian admissible with two form associated to contact cylinder}.

\begin{corollary} \label{corollary action of capped loop disjoint from contact cylinder}
	Let $\check{C}$, $\widetilde{\omega}$, $H = (H_t)_{t \in \T}$, $\lambda^\pm_{\widetilde{\omega}}$, $(\lambda_{H_t})_{t \in \T}$,
	$(m_{H_t})_{t \in \T}$, $F^{H,\check{C},\widetilde{\omega}} := (F^{H_t,\check{C},\widetilde{\omega}})_{t \in \T}$ be as in Lemma \ref{lemma Hamiltonian admissible with two form associated to contact cylinder}.
	Then for every capped loop
	$\gamma = (\widetilde{\gamma},\check{\gamma})$ whose associated loop
	$\overline{\gamma} : \T \lra{} M$
	satisfies $\overline{\gamma}(\T) \subset M - ([1+\epsilon/4,1+\epsilon/2] \times C)$,
	we have
	\begin{equation} \label{equation contact cylinder action}
	\cA_{H,\check{C},F^{H,\check{C},\widetilde{\omega}}}(\gamma) =
	\left\{
	\begin{array}{ll}
	 I_0 + \lambda^-_{\widetilde{\omega} }I_1 & \text{if} \ \overline{\gamma}(\T) \subset D \cup ([1,1+\epsilon/4] \times C) \\
	I_0 + \lambda^+_{\widetilde{\omega}}(I_1 - I_2) + \lambda^-_{\widetilde{\omega}}I_2
	&
	\text{otherwise}
	\end{array}
	\right.
	\end{equation}
	where
	$$
	I_0 := -\int_S \widetilde{\gamma}^* \widetilde{\omega}, \quad
	I_1 :=  \int_0^1 H_t(\overline{\gamma}(t))) dt, \quad I_2 := \int_0^1 m_{H_t} dt.
	$$
\end{corollary}

\begin{remark} \label{remark only depends on cohomology class}
The $(H,\widetilde{\omega},F^{H_t,\check{C},\widetilde{\omega}})$-action of a capped loop $\gamma$ satisfying the conditions
of Corollary \ref{corollary action of capped loop disjoint from contact cylinder},
only depends on $\omega$, $H$, $\check{C}$, $\lambda^\pm_{\widetilde{\omega}}$ and the relative cohomology class
$[\widetilde{\omega} - \lambda^-_{\widetilde{\omega}}\omega] \in H^2(M,D;\R)$.
\end{remark}

The following definition packages together
all the necessary action values
stated above.

\begin{defn} \label{definition action functionals}

Let $\omega_{\check{C}}$ be a $\check{C}$-compatible $2$-form
with scaling constants $0$ and $1$ and which is equal to $\omega$ outside $D \cup ([1,1+\epsilon/2] \times C)$.
Let $Q_{\check{C}} \subset H^2(M,D;\R) \times \R \times \R$
be the subset consisting of all triples $(q,\lambda^-,\lambda^+)$ satisfying
$q = [\widetilde{\omega}-\lambda^- \omega + \lambda^- \omega_{\check{C}}]$ for some
$\check{C}$-compatible $2$-form
$\widetilde{\omega} \in \Omega^2(M)$
whose scaling constants are $\lambda^\pm$ satisfying $\lambda^- \leq \lambda^+$.
For any capped loop
$\gamma = (\widetilde{\gamma},\check{\gamma})$ whose associated loop
has image disjoint from $ [1-\epsilon/4,1+\epsilon/2] \times C$,
we define the {\it $(H,\check{C})$-action}
of $\gamma$ to be the function
$$\cA_{H,\check{C}}(\gamma) : Q_{\check{C}}  \lra{} \R, \quad \cA_{H,\check{C}}(\gamma)(q,\lambda^-,\lambda^+) :=
\cA_{H,\check{C},F^{H,\check{C},\widetilde{\omega}}}(\gamma)$$
where
$\widetilde{\omega} \in \Omega^2(M)$
is a $\check{C}$-compatible $2$-form
with scaling constants $\lambda^\pm$
satisfying
$[\widetilde{\omega}-\lambda^- \omega + \lambda^- \omega_{\check{C}}] = q$
and $\lambda^- \leq \lambda^+$.
\end{defn}

The $(H,\check{C})$-action of $\gamma$
is well defined by Remark
\ref{remark only depends on cohomology class} and it does not depend on the choice of $\omega_{\check{C}}$.

\begin{example} \label{example usual action filtraiton}
	Let $\check{C}$, $\widetilde{\omega}$, $\lambda^\pm_{\widetilde{\omega}}$,  $(\lambda_{H_t})_{t \in \T}$, 
	$(m_{H_t})_{t \in \T}$,  $F^{H,\check{C},\widetilde{\omega}} := (F^{H_t,\check{C},\widetilde{\omega}})_{t \in \T}$ be as in Lemma \ref{lemma Hamiltonian admissible with two form associated to contact cylinder}.
	Suppose that the contact cylinder $\check{C}$ is the empty set and $\widetilde{\omega} =  \omega$.
	Then
	for each capped loop $\gamma = (\widetilde{\gamma},\check{\gamma})$, we get
	\begin{equation} \label{equation action functional}
	\cA_{H,\emptyset}(\gamma)([\widetilde{\omega}],a,1) = -\int_S \widetilde{\gamma}^* \omega + \int_0^1 H_t(\overline{\gamma}(t)) dt
	\end{equation}
	for all $a \in \R$
	which is the usual action functional
	defined in, say, \cite{Floer:gradientflow}
	(with different sign conventions)
	where $S$ is the domain of $\widetilde{\gamma}$
	and $\overline{\gamma}$ is the associated loop of $\gamma$.
\end{example}

\subsection{Floer Trajectories.}

In this section we will give a definition of a Floer trajectory converging to a collection of capped $1$-periodic orbits
and state some results concerning spaces of Floer trajectories.
Throughout this subsection,
we will fix a (possibly empty) contact
cylinder $\check{C} = [1-\epsilon,1+\epsilon] \times C \subset M$.

\begin{defn} \label{definition wspaces of compatible objects}
Let $\ccH^\Sigma(\check{C})$
be the space of smooth families
of Hamiltonians
$H^\bullet = (H^\sigma)_{\sigma \in \Sigma}$ parameterized by a manifold $\Sigma$
which are weakly $\check{C}$-compatible and equipped with the $C^\infty_{\textnormal{loc}}$ topology.
Let
$\overline{\ccH}^\Sigma(\check{C}) \subset \ccH^\Sigma(\check{C})$
be the subspace consisting of those $H$ which are $\check{C}$-compatible.
Let $\ccJ^\Sigma(\check{C})$
be the space of smooth families
$J^\bullet = (J^\sigma)_{\sigma \in \Sigma}$ of almost complex structures on $M$
which are $\check{C}$-compatible equipped with the $C^\infty_{\textnormal{loc}}$ topology.
\end{defn}

\begin{defn} \label{defn Riemann Surface}
(See Figure \ref{fig:riemannsurface}.)
A {\it Riemann surface with $n_-$ negative cylindrical ends and $n_+$ positive cylindrical ends}
is a Riemann surface $\Sigma$
together with a collection of proper embeddings
$$(\iota_j : \bbI_\pm \times \T \lra{} \Sigma)_{j \in I_- \sqcup I_+}$$
where
\begin{itemize}
\item $I_-, I_+$ are finite indexing sets and the images of $\iota_j$ are all disjoint from each other,
\item  $\Sigma$ is biholomorphic to
$\check{\Sigma} - \cup_{j \in I_- \sqcup I_+} \{p_j\}$ where $\check{\Sigma}$ is a closed Riemann surface and
$(p_j)_{j \in I_- \sqcup I_+}$ are distinct points in $\check{\Sigma}$ and
\item for each $j \in I_\pm$
there is a holomorphic chart $U_j$ in $\check{\Sigma}$ centered at $p_j$
so that $\Im(\iota_j) \subset U_j$ and
$$\iota_j(s,t) = e^{\mp  2\pi(s + i t)}$$
with respect to this chart.

\end{itemize}

The map $\iota_j$ is called the {\it negative } (resp. {\it positive}) {\it cylindrical end associated to $j \in I_-$} (resp. $j \in I_+$).
The coordinate $(s+it) \circ \iota_j^{-1} : \Im(\iota_j) \lra{} \C$ on the negative (resp. positive) cylindrical end $\iota_j \in I_\pm$ is called the {\it cylindrical coordinate associated to $j \in I_- \sqcup I_+$}.
\end{defn}

\begin{center}
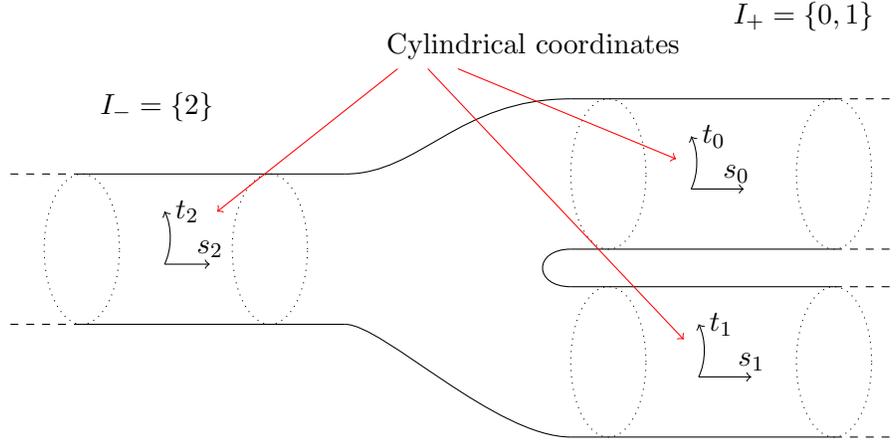
\begin{figure}[h]
	\begin{tikzpicture}
	
	\draw [dotted] (4,-1) ellipse (0.5 and 1);
	\draw [dashed](4,0) -- (5,0);
	\draw [dashed](4,-2) -- (5,-2);
	\draw (4,0) -- (0.5,0);
	\draw (4,-2) -- (0.5,-2);
	\draw [dotted] (4,1.5) ellipse (0.5 and 1);
	\draw [dashed](4,2.5) -- (5,2.5);
	\draw [dashed](4,0.5) -- (5,0.5);
	\draw [](4,2.5) -- (0.5,2.5);
	\draw (4,0.5) -- (0.5,0.5);
	\draw [dotted] (-6,0.5) ellipse (0.5 and 1);
	\draw [dashed](-6,1.5) -- (-7,1.5);
	\draw [dashed](-6,-0.5) -- (-7,-0.5);
	\draw (-6,1.5) -- (-2.5,1.5);
	\draw (-6,-0.5) -- (-2.5,-0.5);
	\draw (-2.5,1.5) .. controls (-1.5,1.5) and (-1,2.5) .. (0.5,2.5);
	\draw (-2.5,-0.5) .. controls (-2,-0.5) and (-0.5,-2) .. (0.5,-2);
	\draw (0.5,0) .. controls (0,0) and (0,0.5) .. (0.5,0.5);
	\draw [dotted] (1,1.5) ellipse (0.5 and 1);
	\draw [dotted] (1,-1) ellipse (0.5 and 1);
	\draw [dotted] (-3.5,0.5) ellipse (0.5 and 1);
	\node at (-5,2.4) {$I_-=\{2\}$};
	\node at (3.6,3.6) {$I_+=\{0,1\}$};
	\draw [->](2.1,1.3) -- (2.8,1.3);
	\draw [->](2.1,1.3) .. controls (2.2,1.5) and (2.2,1.8) .. (2.1,2);
	\node at (2.7,1.5) {$s_0$};
	\node at (2.4,2) {$t_0$};
	\draw [->](2.2,-1.2) -- (2.9,-1.2);
	\draw [->](2.2,-1.2) .. controls (2.3,-1) and (2.3,-0.7) .. (2.2,-0.5);
	\node at (2.9,-1) {$s_1$};
	\node at (2.5,-0.5) {$t_1$};
	\draw [->](-4.9,0.3) -- (-4.3,0.3);
	\draw [->](-4.9,0.3) .. controls (-4.8,0.5) and (-4.8,0.8) .. (-4.9,1);
	\node at (-4.3,0.5) {$s_2$};
	\node at (-4.6,1) {$t_2$};
	\node at (-5.1,-2.2) {};
	\draw [->,red](-1.8,2.9) -- (-4.2,1);
	\draw [->,red](-1.4,2.9) -- (2,-0.7);
	\draw [->,red](-1,2.9) -- (1.9,1.7);
	\node at (-0,3.2) {Cylindrical coordinates};
	\end{tikzpicture}
\caption{Riemann surface with one negative cylindrical end and two positive cylindrical ends} \label{fig:riemannsurface}
\end{figure}
\end{center}

\begin{defn} \label{definition riemann surface admissible}
Let $\Sigma$ be a Riemann surface with
$n_-$ negative cylindrical ends and $n_+$ positive cylindrical ends labeled by
finite sets $I_-$ and $I_+$ respectively.
A $1$-form $\beta \in \Omega^1(\Sigma)$
is {\it $\Sigma$-compatible} if $\iota_j^*\beta = \kappa_j dt$ for all $j \in I_- \sqcup I_+$ where $(\kappa_i)_{i \in I_- \sqcup I_+}$ are positive constants.
Here $\kappa_j$ is called the {\it weight of $\beta$} at the cylindrical end corresponding to $j$.
A smooth family of tensor fields $\alpha := (\alpha_z)_{z \in \Sigma}$ on $M$ (e.g. functions, differential forms, almost complex structures)
is $\Sigma$-compatible if there is a compact subset $K_\alpha \subset \Sigma$
and a smooth family of tensors $\alpha^j := (\alpha^j_t)_{t \in \T}$ on $M$ for each $j \in I_- \sqcup I_+$
so that 
$\alpha_{\iota_j(s,t)} = \alpha^j_t$ for all $(s,t) \in \bbI_\pm \times \T$
satisfying $\iota_j(s,t) \notin K_\alpha$ and all $j \in I_- \sqcup I_+$.
Here $\alpha^j$ is the {\it limit of $\alpha$ corresponding to $j \in I_- \sqcup I_+$} and $\alpha^\# := (\alpha^j)_{j \in I_- \sqcup I_+}$ are the {\it limits of $\alpha$}.
Let $\widetilde{\omega} \in \Omega^2(M)$ be $\check{C}$-compatible.
A smooth family of autonomous Hamiltonians
$H := (H_z)_{z \in \Sigma}$ on $M$
is {\it $(\Sigma,\check{C})$-admissible}
if
\begin{enumerate}
\item $H \in \ccH^\Sigma(\check{C})$ and $H$ is $\Sigma$-compatible,
\item \label{equation forms are negative}
$d(h^x \beta), d(\lambda \beta), d(m \beta), d((h^x - \lambda) \beta) \leq 0$
for all $x \in M$ where $h^x$, $\lambda$ and $m$ are maps from $\Sigma$ to $\R$ defined by:
$$h^x(\sigma) := H_\sigma(x), \quad \lambda(\sigma) := \lambda_{H_\sigma}, \quad m(\sigma) := m_{H_\sigma} \quad \ \forall \ \sigma \in \Sigma$$
where $(\lambda_{H_\sigma})_{\sigma \in \Sigma}$ and $(m_{H_\sigma})_{\sigma \in \Sigma}$
are the slopes and heights of $H = (H_\sigma)_{\sigma \in \Sigma}$ respectively.
\end{enumerate}
Let $H^\# := (H^j)_{j \in I_- \sqcup I_+}$ be a collection of Hamiltonians.
We define $\ccH^\Sigma(H^\#,\check{C})$
to be the set of $(\Sigma,\check{C})$-admissible smooth families of Hamiltonians whose limits are $H^\#$.
We also define
$\overline{\ccH}^\Sigma(H^\#,\check{C}) :=
\overline{\ccH}^\Sigma(\check{C}) \cap \ccH^\Sigma(H^\#,\check{C})$.

Let $J^j := (J^j_t)_{t \in \T}$ be a smooth family of almost complex structures in $\ccJ^\T(\check{C})$ for each $j \in I_- \sqcup I_+$
and let $J^\# := (J^j)_{j \in I_- \sqcup I_+}$.
Define $\ccJ^\Sigma(J^\#,\check{C}) \subset \ccJ^\Sigma(\check{C})$
to be the subspace of $\Sigma$-compatible
families
$J = (J_z)_{z \in \Sigma}$
of almost complex structures 
whose limits are $J^\#$.

Now let $H \in \ccH^\Sigma(H^\#,\check{C})$ and
$J \in \ccJ^\Sigma(J^\#,\check{C})$ for some $H^\#$,$J^\#$ as above and
let ${\bf j}$ be the natural complex structure on $\Sigma$.
We say that $u : \Sigma \lra{} M$ satisfies the {\it Floer equation}
with respect to
$(H,J)$ if
\begin{equation} \label{eqn:floer}
(du + X_{H_\sigma} \otimes \beta) + J_\sigma \circ (du + X_{H_\sigma} \otimes \beta) \circ {\bf j} = 0
\end{equation}
at each point $\sigma \in \Sigma$.
A continuous map $u : \Sigma \lra{} M$ {\it converges to capped $1$-periodic orbits
$\gamma^\# := ((\widetilde{\gamma}^j,\check{\gamma}^j))_{j \in I_- \sqcup I_+}$}
of $(\kappa_j H^j)_{j \in I_- \sqcup I_+}$ respectively
if
\begin{itemize}
\item $\lim_{s \to \pm \infty} u( \iota_j(s,t)) = \overline{\gamma}^j(t)$ for all $j \in I_\pm$ where $\overline{\gamma}^j$ is the associated loop of
$(\widetilde{\gamma}^j,\check{\gamma}^j)$
and
\item the surface obtained by gluing the ends of $u$ with the surfaces $(\widetilde{\gamma}^j)_{j \in I_- \sqcup I_+}$
is null-homologous in $H_2(M;\Z)$.
\end{itemize}
We let $\ccM(H,J,\gamma^\#)$
be the space of maps $u : \Sigma \lra{} M$ satisfying the Floer equation with respect to $(H,J)$ and converging to $\gamma^\#$ equipped with the $C^\infty_{\text{loc}}$ topology.
\end{defn}

\begin{remark}
The space $\ccM(H,J,\gamma^\#)$
also depends on $\beta$ but we omit
this from the notation as it is
either clear which $\beta$ we are using, or if $\beta$ isn't mentioned then we will assume some $\beta$ has been chosen.
\end{remark}

The motivation for part (\ref{equation forms are negative}) of the definition of a $(\Sigma,\check{C})$-admissible Hamiltonian above is that it ensures that, roughly,
that the Floer complex (defined in Section \ref{definition Floer alternative filtration}) is filtered by the
$(H,\check{C})$-action
from Definition \ref{definition action functionals}.
See also Lemma \ref{lemma filtration} combined with Equation (\ref{equation primitive assocatiated to contact cylinder}).

\begin{defn} \label{definition conley zehnder index}
	
	Each capped $1$-periodic orbit
	$\gamma$ can be assigned an
	index $CZ(\gamma)$
	called the {\it Conley-Zehnder index}.
	Such an index is defined in the following way:
	To any path $A := (A_t)_{t \in [a,b]}$
	of symplectic matrices,
	we can assign an index $CZ(A)$
	called its {\it Conley-Zehnder index}
	(\cite{ConleyZehnder:Index},
	\cite{RobbinSalamon:maslov}
	and
	\cite{Gutt:GeneralizedConleyZehnder}).
	We will not give a definition here,
	but state some important properties
	(See \cite[Proposition 6]{Gutt:GeneralizedConleyZehnder},
	\cite[Lemma 26]{Gutt:GeneralizedConleyZehnder}
	and
	\cite[Corollary 4.9]{McLean:isolated}):
	\begin{CZ}
		\item \label{item:cznormalization}
		$CZ((e^{it})_{t \in [0,2\pi]}) = 2$.
		\item \label{item:czadditive}
		$CZ(A \oplus B) = CZ(A) + CZ(B)$ where $A = (A_t)_{t \in [a,b]}$,
		$B = (B_t)_{t \in [a,b]}$ are paths of symplectic matrices and
		$A \oplus B := (A_t \oplus B_t)_{t \in [a,b]}$.
		\item  \label{item:catenationczproperty}
		The Conley-Zehnder index of the catenation of two paths is the sum of their Conley-Zehnder indices.
		\item \label{item:homotopyinvariance}
		If $A$ and $B$ are two paths of symplectic matrices which are homotopic relative to their endpoints
		then they have the same Conley-Zehnder index.
		Also such an index only depends on the path up to orientation
		preserving reparameterization.
		\item \label{item:sheartransformation}
%
		Let
		$$A_t = \left(
		\begin{array}{cc}
		\id & -tB \\
		0 & \id
		\end{array}
		\right)\in GL(2n;\R), \quad \forall t \in [0,1]$$
		be a family of $2n \times 2n$ matrices where $\id$ is the identity $n \times n$ matrix
		and $B$ is a symmetric $n \times n$ matrix.
		Let $A := (A_t)_{t \in [0,1]}$ be a path of symplectic matrices with respect to the
		linear symplectic form
		$\Omega = \sum_{i=1}^n x_i^* \wedge y_i^*$
		where $x_1^*,\cdots,x_n^*,y_1^*,\cdots,y_n^*$ are the dual basis vectors
		of the standard basis $x_1,\cdots,x_n,y_1,\cdots,y_n$
		of $\R^{2n}$.
		Then $CZ(A) = \frac{1}{2}\text{Sign}(B)$.
		\item \label{item:constantrankpath}
		Let $(A_t)_{t \in [0,1]}$
		be a path of symplectic matrices
		so that $\dim(\ker(A_t - \id))$
		is independent of $t$.
		Then $CZ((A_t)_{t \in [0,1]}) = 0$.
		\item \label{item:shortpath}
		Let $Sp(2n)$ be the space of symplectic $2n \times 2n$ matrices and let $A \in Sp(2n)$.
		Then there is a neighborhood $N_A$ of $A$
		so that any path $(A_t)_{t \in [0,1]}$ in $N_A$ with $A_0 = A$
		satisfies $CZ((A_t)_{t \in [0,1]}) \in [-\frac{k}{2},\frac{k}{2}]$
		where $k = \text{dim ker}(A - \text{id})$.
	\end{CZ}
	The {\it Conley-Zehnder index} of a capped $1$-periodic orbit $\gamma = (\widetilde{\gamma},\check{\gamma})$ of $H$
	is given by the Conley-Zehnder index of
	$$\tau|_{\check{\gamma}(t)} \circ D\phi^H_t|_{\check{\gamma}(0)} \circ (\tau|_{\check{\gamma}(0)})^{-1}, \quad t \in [0,1]$$
	where
	$\tau : \widetilde{\gamma}^* TM \lra{} S \times \C^n$
	is a symplectic trivialization over the domain $S$ of $\widetilde{\gamma}$ and $\tau|_\sigma : \widetilde{\gamma}^*(TM)|_\sigma \lra{} \C^n$ is its restriction to the fiber $\sigma \in \Sigma$.
	We define the {\it index of $\gamma$} to be $|\gamma| := n-CZ(\gamma)$.
	If $(\gamma_j)_{j \in I}$
		is a finite collection of capped $1$-periodic orbits then we define
		$|(\gamma_j)_{j \in I}| := \sum_{j \in I} |\gamma_j|$.
\end{defn}

\begin{remark} \label{definition conley zehnder index of associated loop}
	Such an index 
	does not depend on the choice of trivialization $\tau$
	by \ref{item:homotopyinvariance} combined with the fact that $\pi_1(Sp(2n))$ is abelian.
	Also since $c_1(M) = 0$, the index only depends on the associated loop $\overline{\gamma}$.
	Therefore, we will define
	$|\overline{\gamma}| := |\gamma|$
	for any associated loop $\overline{\gamma}$ of a capped $1$-periodic orbit $\gamma$.
\end{remark}

\begin{defn} \label{defnition ubiquitous}
	Let $T$ be a topological space.
	A subset $S \subset T$ is
	{\it ubiquitous}
	if it contains a countable intersection of
	dense open sets.
\end{defn}

We have the following important proposition:
\begin{prop} \label{proposition regular subset for surface}
	Let $\Sigma$, $H$, $J^\#$, $(\kappa_j)_{j \in I}$
	be as in Definition \ref{definition riemann surface admissible}.
	Then there is a ubiquitous subset
	$\ccJ^{\Sigma,\reg}(H,J^\#,\check{C}) \subset \ccJ^\Sigma(J^\#,\check{C})$ so that
	a certain family of linearized operators is surjective and so certain other transversality conditions hold (see Definition \ref{definition H V regular}).
	
	Also, suppose that we have a collection
	$\gamma^\# := (\gamma^j)_{j \in I_- \sqcup I_+}$ of non-degenerate capped $1$-periodic orbits of $(\kappa_j H^j)_{j \in I_- \sqcup I_+}$ whose associated $1$-periodic orbits are disjoint from $V$
	and so that if $\Sigma \neq \R \times \T$ then at least two such $1$-periodic orbits of $\gamma^\#$
	have distinct images.
	Then for each $J \in \ccJ^{\Sigma,\reg}(H,J^\#,\check{C})$,
	$\ccM(H,J,\gamma^\#)$
	is an oriented $k$-dimensional manifold where
	\begin{equation} \label{equation conley zehnder index sum}
		k := |(\gamma^j)_{j \in I_-}| - |(\gamma^j)_{j \in I_+}| + n(|I|_+ - |I|_- + \chi(\Sigma))
	\end{equation}
	where $\chi(\Sigma)$ is the Euler characteristic of $\Sigma$.
	Also if the dimension of $\ccM(H,J,\gamma^\#)$ is $\leq 1$
	then we can ensure that
	the image of each $u \in \ccM(H,J,\gamma^\#)$
	is disjoint from $V$.
\end{prop}

The proof is extremely standard using ideas from \cite[Chapters 3 and 6]{McduffSalamon:sympbook} and from \cite{schwarz1995cohomology}.
However for the sake of completeness,
we will prove this proposition in Appendix B.
Note that $\ccM(H,J,\gamma^\#)$
may not be a manifold at all if any of the capped $1$-periodic orbits $\gamma^j$ are degenerate even if $J \in \ccJ^{\Sigma,\reg}(H,J^\#,\check{C}^\#)$ by our current definition.
We also have a $1$-parameter version of this proposition stated below.
\begin{defn} \label{definition parameterized moduli spaces}
Let $\Sigma_\bullet := (\Sigma_s)_{s \in [0,1]}$
be a smooth family of Riemann surfaces
with $n_-$ negative cylindrical ends and $n_+$ positive cylindrical ends labeled by finite sets $I_-$ and $I_+$ respectively.
Let $(\kappa_j)_{j \in I_- \sqcup I_+}$ be positive numbers and let
$\beta_s$
be a $\Sigma_s$-compatible $1$-form
so that $\kappa_j$ is the weight of $\beta_s$ at the cylindrical end corresponding to $j \in I_- \sqcup I_+$
for each $j \in I_- \sqcup I_+$, $s \in [0,1]$ and so that
$(\beta_s)_{s \in [0,1]}$
is a smooth family of $1$-forms.

Let $H^\# := (H_i)_{i \in I_- \sqcup I_+}$
be elements of $\ccH^\T(\check{C})$.
We define
$\ccH^{\Sigma_\bullet}(H^\#,\check{C})$
(resp. $\overline{\ccH}^{\Sigma_\bullet}(H^\#,\check{C})$) to be the space of smooth families of Hamiltonians $H := (H_{s,\sigma})_{s \in [0,1], \sigma \in \Sigma_s}$
so that the subfamily $H_{s,\bullet} := (H_{s,\sigma})_{\sigma \in \Sigma_s}$ is an element of
$\ccH^{\Sigma_s}(H^\#,\check{C})$
(resp $\overline{\ccH}^{\Sigma_s}(H^\#,\check{C})$)
 for each $s \in [0,1]$.
Similarly let $\ccJ^{\Sigma_\bullet}(J^\#,\check{C})$
  be the space of smooth families of almost complex structures
$J := (J_{s,\sigma})_{s \in [0,1], \sigma \in \Sigma_s}$ so that the subfamily
$J_{s,\bullet} := (J_{s,\sigma})_{\sigma \in \Sigma_s}$ is an element of
$\ccJ^{\Sigma_s}(J^\#,\check{C})$
 for each $s \in [0,1]$.
For each $H \in \ccH^{ \Sigma_\bullet}(H^\#,\check{C})$ and each $Y_j \in \ccJ^{\Sigma_j,\reg}(H_{j,\bullet},J^\#,\check{C})$ for $j=0,1$, define
$\ccJ^{\Sigma_\bullet}((Y_0,Y_1),J^\#,\check{C}^\#)$ to be the subspace of
$\ccJ^{\Sigma_\bullet}(J^\#,\check{C})$ consisting of those $J$ as above
satisfying $J_{0,\bullet} = Y_0$
and $J_{1,\bullet} = Y_1$.
%
For each $J \in \ccJ^{\Sigma_\bullet}((Y_0,Y_1),J^\#,\check{C})$ as above
and each tuple of capped $1$-periodic orbits $\gamma^\# := (\gamma^j)_{j \in I_- \sqcup I_+}$ of $(\kappa_j H^j)_{j \in I_- \sqcup I_+}$,
define
$\ccM(H,J,\gamma^\#) := \sqcup_{s \in [0,1]} \ccM(H_{s,\bullet},J_{s,\bullet},\gamma^\#)$ with the induced $C^\infty_{\textnormal{loc}}$ topology.
\end{defn}

We then have the following parameterized version of Proposition \ref{proposition regular subset for surface}
above.
We will not prove this since the ideas used to prove it are exactly the same as those from Proposition \ref{proposition regular subset for surface}.
\begin{prop} \label{proposition parameterized regular subset for surface}

	Let $\Sigma_\bullet$, $H^\#$, $J^\#$, $(\kappa_j)_{j \in I}$
	be as in Definition \ref{definition parameterized moduli spaces}.
	Suppose $H \in \ccH^{\Sigma_\bullet}(H^\#,\check{C})$,
	$Y_j \in \ccJ^{\Sigma_j,\reg}(H_{j,\bullet},J^\#,\check{C})$ for $j=0,1$ where $H_{s,\bullet}$ is described in the previous definition.
	Then there is a ubiquitous subset
	$$\ccJ^{ \Sigma_\bullet,\reg}(H,(Y_0,Y_1),J^\#,\check{C}) \subset \ccJ^{\Sigma_\bullet}((Y_0,Y_1),J^\#,\check{C})$$
	so that
	a certain family of linearized operators is surjective  and so certain other transversality conditions hold.

	Also suppose
	$\gamma^\# := (\gamma^j)_{j \in I_- \sqcup I_+}$ are capped $1$-periodic orbits of $(\kappa_j H^j)_{j \in I_- \sqcup I_+}$ so that $\gamma^j$ is non-degenerate for each $j \in I_- \sqcup I_+$,
	whose associated $1$-periodic orbits are disjoint from $V$
	and so that if $\Sigma \neq \R \times \T$ then at least two associated $1$-periodic orbits of $\gamma^\#$
	have distinct images.
	Then for each
	$J \in \ccJ^{ \Sigma_\bullet,\reg}(H,(Y_0,Y_1),J^\#,\check{C})$,
	$\ccM(H,J,\gamma^\#)$
	is an oriented $k+1$-dimensional manifold with boundary 
	equal to $\sqcup_{j=0}^1 \ccM(H_{j,\bullet},Y_j,\gamma^\#)$
	where $k$ is defined as in equation
	(\ref{equation conley zehnder index sum}).
		Also if the dimension of $\ccM(H,J,\gamma^\#)$ is $\leq 1$
	then we can ensure that
	the image of each $u \in \ccM(H,J,\gamma^\#)$
	is disjoint from $V$.
\end{prop}


\begin{example} \label{example floer cylinders}
Let $\R \times \T = \C  / \Z$
be a Riemann surface with a positive cylindrical end indexed by $I_+ := \{+\}$
and a negative cylindrical end
indexed by $I_- := \{-\}$
given by the natural inclusion maps into $\R \times \T$.
Also let $\beta := dt$ and $\kappa_\pm := 1$
be the corresponding weights of $\beta$.
Here $\R \times \T$ along with $\beta$ and $\kappa_\pm$
is called a {\it Riemann cylinder}.
Let $H$ be $(\R \times \T,\check{C})$-compatible
with associated limits
$H^\pm$.
Also let $J^\pm \in \ccJ^\T(\check{C})$
and let $J \in \ccJ^{\R \times \T}((J^+,J^-),\check{C})$.
Then a map $u : \R \times \T \lra{} M$
satisfying the $(H,J)$-Floer equation and converging to capped $1$-periodic orbits $\gamma_-,\gamma_+$ of $H^-$, $H^+$ respectively
is called an
{\it $(H,J)$-Floer cylinder
connecting $\gamma_-$ and $\gamma_+$}.
\end{example}

\begin{defn} \label{definition space of Floer cylinders}
Let $\R \times \T$ be a Riemann cylinder.
Let
$$\iota_{\textnormal{Ham}} : \ccH^\T(\check{C}) \lhook\joinrel\lra{} \bigcup_{H^\# = (H^-,H^+) \in (\ccH^\T(\check{C}))^2} \ccH^{\R \times \T}(H^\#,\check{C})$$
 $$
 \iota_{\textnormal{cpx}} : \ccJ^\T(\check{C}) \lhook\joinrel\lra{} \bigcup_{J^\# = (J^-,J^+) \in (\ccJ^\T(\check{C}))^2} \ccJ^{\R \times \T}(J^\#,\check{C}) 
 $$
be the natural embeddings sending
$(H_t)_{t \in \T}$ 
to $(H_t)_{(s,t) \in \R \times \T}$
and
$(J_t)_{t \in \T}$
to $(J_t)_{(s,t) \in \R \times \T}$
respectively.

Also for each $(H,J) \in \ccH^\T(\check{C}) \times \ccJ^\T(\check{C})$ and for each pair $\gamma^\# := (\gamma^-,\gamma^+)$ of capped $1$-periodic orbits of $H$,
we define
\begin{equation} \label{equation quotient moduli space}
\ccM(H,J,\gamma^\#) := \ccM(\iota_{\textnormal{Ham}}(H),\iota_{\textnormal{cpx}}(J),\gamma^\#), \quad
\overline{\ccM}(H,J,\gamma^\#) := \ccM(H,J,\gamma^\#)/\R
\end{equation}
where the natural $\R$-action is given by translation in the $s$ coordinate.
Also for each $H \in \ccH^\T(\check{C})$,
define
$$\ccJ^{\T,\reg}(H,\check{C}) := \ccJ^{\T}(\check{C}) \cap \iota_{\textnormal{cpx}}^{-1}(\cup_{J \in \ccJ^\T(\check{C})}\ccJ^{\R \times \T,\reg}(\iota_{\text{Ham}}(H),(J,J),\check{C})).$$
%
\end{defn}

We also have the following proposition (see \cite[Proposition 3.4]{usher2011deformed}):
\begin{prop} \label{proposition regularity for translation invariant floer cylinders}
For each
$H \in \ccH^\T(\check{C})$,
$\ccJ^{\T,\reg}(H,\check{C})$
is a ubiquitous subset of
$\ccJ^\T(\check{C})$.
Also
for each pair of non-degenerate
capped $1$-periodic orbits
$\gamma_-$, $\gamma_+$ of $H$,
we have that
$\overline{\ccM}(H,J,\gamma^\#)$
is a manifold of dimension
$|\gamma_-| - |\gamma_+| -1$.
If the dimension of
$\overline{\ccM}(H,J,\gamma^\#)$
is $\leq 1$ then each element of
$\overline{\ccM}(H,J,\gamma^\#)$
has image disjoint from $V$.
\end{prop}

\begin{example} \label{example pair of pants}
Let $\Sigma := \P^1 - \{0,1,\infty\}$
and let
$\varpi : \Sigma \lra{} \C^*$
be a proper holomorphic map of degree $2$ with exactly one branch point at $1 \in \C^*$ and so that $\varpi^{-1}(\{z \in \C^* \ : \ |z|<1\})$ is connected (See Figure \ref{fig:riemannsurfacebranch}).
\begin{center}
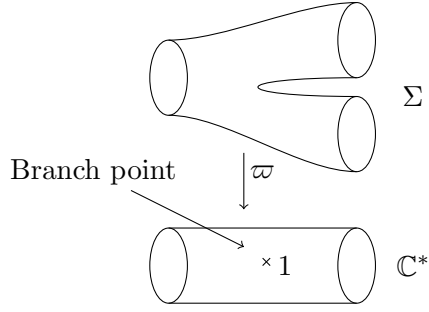
\begin{figure}[h]
	\begin{tikzpicture}
	\draw  (-3.75,2) ellipse (0.25 and 0.5);
	\draw  (-3.75,0.75) ellipse (0.25 and 0.5);
	\draw  (-6.25,1.5) ellipse (0.25 and 0.5);
	\draw  (-3.75,-1) ellipse (0.25 and 0.5);
	\draw  (-6.25,-1) ellipse (0.25 and 0.5);
	\draw (-6.25,2) .. controls (-5.5,2) and (-4.5,2.5) .. (-3.75,2.5);
	\draw (-6.25,1) .. controls (-5.5,1) and (-4.5,0.25) .. (-3.75,0.25);
	\draw (-3.75,1.5) .. controls (-5.5,1.5) and (-5.5,1.25) .. (-3.75,1.25);
	\draw (-6.25,-0.5) -- (-3.75,-0.5);
	\draw (-6.25,-1.5) -- (-3.75,-1.5);
	
	\draw (-5.0,-0.9) -- (-4.9,-1.0);
	\draw (-4.9,-0.9) -- (-5.0,-1.0);
	\draw[->] (-6.75,0) -- (-5.25,-0.75);
	\node at (-7.25,0.25) {Branch point};
	\node at (-3,1.25) {$\Sigma$};
	\node at (-3,-1) {$\C^*$};
	\draw[->] (-5.25,0.5) -- (-5.25,-0.25);
	\node at (-5,0.25) {$\varpi$};
	\node at (-4.7,-1) {$1$};
	\end{tikzpicture}
	\caption{Branched cover of pair of pants over a cylinder.} \label{fig:riemannsurfacebranch}
\end{figure}
\end{center}
Let $I_- := \{2\}$ and $I_+ := \{0,1\}$.
Choose cylindrical ends
$$\iota_j : \bbI_\pm \times \T \lra{} \Sigma, \quad j \in I_\pm$$
for $\Sigma$ so that
$\varpi(\iota_j(s,t)) = e^{2\pi (s + 1 + it)}$ for $j \in I_+$
and
$\varpi(\iota_2(s,t)) = e^{2\pi (s - 1 + 2it)}$.
Then
$\beta := \varpi^* d(\arg(z)/2\pi)$
is a $\Sigma$-compatible $1$-form
so that the weight $\kappa_2$ of $\beta$ is $2$
at the negative end and the weights $\kappa_0$, $\kappa_1$ are $1$ at each positive end.
We call $\Sigma$ with its negative and two positive cylindrical ends above together with the $1$-form $\beta$
the {\it pair of pants}. 

Let $H^\# = (H^j)_{j \in \{0,1,2\}}$ where $H^j \in \ccH^\T(\check{C})$ for $j=0,1,2$ and let $H \in \ccH^\T(H^\#,\check{C})$.
%
Similarly let $J^\# = (J^i)_{i \in \{0,1,2\}}$  where $J^j \in \ccJ^\T(\check{C})$ where $j=0,1,2$, define and let $J \in \ccJ(J^\#,\check{C})$.
Let $\gamma_j$ be a capped $1$-periodic orbit of $\kappa_j H_j$ for $j=0,1,2$ and define
$\gamma^\# := (\gamma_j)_{j\in \{0,1,2\}}$.
An {\it $(H,J)$-pair of pants connecting $\gamma^\#$}
is an $(H,J)$-Floer trajectory connecting the capped $1$-periodic orbits $\gamma^\#$.
\end{example}

%
%

\subsection{Directed and Inverse Systems}

In this subsection we will
give some definitions
concerning inverse/direct limits,
which will be used later to define
symplectic cohomology.
A small part of this section will also be used in the next section
to define Novikov rings.
Many of the main ideas come from
\cite{shapetheorybook1982} and
\cite{shapetheorybook2013}.
We will also introduce a category of ``double systems'' which, alternatively, is a category of appropriate ``ind-pro'' objects (I.e. direct limits of inverse limits of modules).
Such categories have appeared in
\cite[Section A.3]{beilinsonglueperverse}
and \cite{Kato:existencelocal}
for example.

\begin{defn} \label{definition directed set}
If $(I,\leq_I)$, $(K, \leq_K)$
are sets together relations $\leq_I$ and $\leq_K$ then we define
$(I \times K,\leq_{I \times K})$
to be the product $I \times K$
together with the relation $\leq_{I \times K}$
satisfying $(i,k) \leq_{I \times K} (i',k')$
if $i \leq_I i'$ and $k \leq_K k'$.
Also we define $(I^\op,\leq^\op)$
to be the set $I^\op := I$
together with the relation $\leq_I^\op$
satisfying $i \leq_I^\op i'$ if $i' \leq_I i$ for all $i,i' \in I^\op$.

A {\it directed set} $(I,\leq_I)$
is a set $I$ together with a reflexive
transitive binary relation
$\leq_I$ so that for all $i_0,i_1 \in I$,
there exists $i_2 \in I$ satisfying
$i_0,i_1 \leq_I i_2$.
An {\it inverse directed set}
$(K,\leq_K)$ is a set $K$ together with
a relation $\leq_K$ so that $(K^\op,\leq_K^\op)$ is a directed set.

If $(I,\leq_I)$ is a directed set,
then a subset $\check{I} \subset I$ is {\it cofinal}
if for each $a \in I$,
there exists $\check{a} \in \check{I}$
satisfying $a \leq_I \check{a}$.
A sequence $(a_k)_{k \in \N}$ of elements in $I$ is {\it cofinal}
if the set $\{a_k \ : \ k \in \N \} \subset I$
is cofinal.
A {\it directed subset} of $I$ is a subset
$\check{I} \subset I$ where the induced ordering
makes $\check{I}$ into a directed set.
A subset $\check{K} \subset K$ is {\it cofinal} inside an inverse directed set $(K,\leq_K)$
if $\check{K}^\op \subset K^\op$ is cofinal in $(K^\op,\leq^\op_K)$.
A {\it double set} is a pair
$(I \times K, \leq_{I \times K})$
where $(I \leq_I)$ is a directed set
and $(K, \leq_K)$ is an inverse directed set.
\end{defn}

\begin{remark}
We can think of a set with a reflexive transitive binary relation $(I,\leq)$
as a category with objects $I$
and with a unique morphism denoted by $i \to j$ for each $i \leq j$ and no other morphisms.
Since there is at most one morphism between any two objects, a functor
$F : (I_0,\leq) \lra{} (I_1,\leq)$
is characterized by the corresponding map
$F : I_0 \lra{} I_1$
on objects.
From now on, we will not distinguish between such functors and maps.

For ease of notation,
we will sometimes just write $I$ for a set with relation $(I,\leq)$ if it is clear what the relation is.
In particular, we will write
$I \times K$
instead of $(I \times K,\leq_{I \times K})$.
\end{remark}

\begin{defn} \label{definition double system}
Let $R$ be a commutative ring and let $\mod{R}$ be the category of $\Z$-graded $R$-modules. 
We write $Ob(\mod{R})$ and $Mor({\mod{R}})$
to be the class of objects and morphisms of $\mod{R}$.

A {\it directed system} is a functor
$D : I \lra{} \mod{R}$ where $(I,\leq_I)$
is a directed set.
An {\it inverse system} is a functor
$V : K \lra{} \mod{R}$ where $(K,\leq_K)$
is an inverse directed system.
A {\it double system of graded $R$-modules}
(or {\it double system})
is a functor
$W : I \times K \lra{} \mod{R}$
where $(I \times K, \leq_{I \times K})$ is a double set.
The morphisms $W((i,k)\to (i',k')) \in Mor(\mod{R})$, $(i,k) \leq_{I \times K} (i',k')$ are called
{\it double system morphisms} of $W$.
\end{defn}

\begin{remark}
We will sometimes just write down a
double system $W : I \times K \lra{} \mod{R}$ 
as a collection
$(W(i,k))_{(i,k) \in I \times K}$ of $R$-modules if it is clear what the
morphisms $W((i,k) \to (i',k'))$ are.
We will do the same for directed and inverse systems.
\end{remark}

\begin{remark}
A directed system $D : I \lra{} \mod{R}$
is equivalent to a double system
$D : I = I \times \{\star\} \lra{} \mod{R}$
where $\{\star\}$ is the single element (inverse) directed set.
Similarly an inverse system is a double system
$V : K = \{\star\} \times K \lra{} \mod{R}$.
We will call such double systems
the {\it double systems associated to $D$ (resp. $V$)}.
\end{remark}

\begin{example} \label{example trivial double system}
The {\it trivial double system}
is the double system
given by the functor
$\iota_R : \{\star\} \times \{\star\} \lra{} \mod{R}$
sending $(\star,\star)$
to $R$
where
$\{\star\}$ is the directed/inverse
system consisting of one point.
Trivial directed/inverse systems are defined in a similar way.
\end{example}

\begin{defn}
For any $\Z$ graded $R$-module $N$,
let $(N)_p$ be the degree $p$ part of $N$ for each $p \in \Z$.
If $W : I \times K \lra{} \mod{R}$
is a double system then we define
the {\it double system shifted by $m \in \Z$}
to be the unique double system
$W[m] : I \times K \lra{} \mod{R}$
satisfying
$(W[m](i,k))_p = (W(i,k))_{m+p}$
for all $p \in \Z$
with the natural induced morphisms.
\end{defn}

\begin{defn} \label{definition direct limit}
The {\it direct limit} of a directed system
$D : I \lra{} \mod{R}$
is defined to be the graded $R$-module
$$\varinjlim_{i \in I} D(i) := \sqcup_{i \in I}D(i) / \sim, \quad \ x_i \sim x_j \ \text{iff} \ \exists \ k \ \text{s.t.} \ i,j \leq k, \ \textnormal{and} \ f_{ik}(x_i) = f_{jk}(x_j)$$
where $f_{ij} := D(i \to j)$ for all $i \leq j$.
We will sometimes write
$\varinjlim_i D(i)$ for such a direct limit.
\end{defn}

\begin{defn} \label{definition inverse limit}
The {\it inverse limit} of an inverse system
$V : K \lra{} \mod{R}$
is defined to be the graded $R$-module
$\varprojlim V$ where
$$(\varprojlim V)_p := \left\{ (x_k)_{k \in K} \in  \prod_{k \in K}(V(k))_p \ : \
\forall \ k,\check{k} \in K
 \ \text{s.t.} \ k \leq_K \check{k}, \
  f_{k\check{k}}(x_k) = x_{\check{k}}
\right\}$$
for each $p \in \Z$
where $f_{k\check{k}} := V(k \to \check{k})$ for all $k \leq \check{k}$.
We will sometimes write
$\varprojlim_j V(j)$ for such an inverse limit.
\end{defn}

We now wish to describe an appropriate
category of directed/inverse/double systems.
This category should have the property that
$\varinjlim$ and $\varprojlim$ are functors
and so that
certain ``obvious'' endomorphisms of  directed/inverse/double systems inducing isomorphisms
on $\varinjlim$ and $\varprojlim$
are in fact isomorphisms in this category.
Another important property of this
category is that we require a
``derived'' version
of $\varprojlim$ to exist.

\begin{defn} \label{definition double morphism composition}
Let
$W : I \times K \lra{} \mod{R}$
be a double system and let
$\phi :  W(i,k) \lra{} A$,
$\phi' : A' \lra{} W(i',k')$
be morphisms in $\mod{R}$
for some $i,i' \in I$, $k,k' \in K$
satisfying $(i',k') \leq_{I \times K} (i,k)$.
Then we define the {\it $W$-composition} 
$\phi \circ_W \phi' : A' \lra{} A$
of $\phi'$ and $\phi$
to be the composition
$\phi \circ f_{i'k'ik} \circ \phi'$
where
$f_{i'k'ik} := W((i',k') \to (i,k))$.
\end{defn}

\begin{defn} \label{definition double system category}
Let $W_j : I_j \times K_j \lra{} \mod{R}$, $j=0,1$
be double systems.
Then we can define a new double system
${\mathcal Mor}(W_0,W_1) : (I_1 \times K_0^\op) \times (I_0^\op \times K_1) \lra{} \mod{R}$
where the object
$((i_1,k_0),(i_0,k_1))$ is sent to
$Hom_{\mod{R}}(W_0(i_0,k_0),W_1(i_1,k_1))$
and a morphism
$((i_1,k_0),(i_0,k_1)) \to ((i'_1,k'_0),(i'_0,k'_1))$
is sent to the morphism
$$
Hom_{\mod{R}}(W_0(i_0,k_0),W_1(i_1,k_1)) \lra{} Hom_{\mod{R}}(W_0(i'_0,k'_0),W_1(i'_1,k'_1)) \\
$$
$$
\Phi \lra{} f^1_{i_1k_1i'_1k'_1} \circ \Phi \circ f^0_{i'_0k'_0i_0k_0}
$$
where $f^j_{iki'k'} := W_j((i,k) \to (i',k'))$ for all $(i,k) \leq_{I_j \times K_j} (i',k')$, $j=0,1$.
We define
\begin{equation} \label{equation for sys morphism}
Mor(W_0,W_1) := \varprojlim_{i_0} \varinjlim_{i_1} \varprojlim_{k_1} \varinjlim_{k_0} {\mathcal Mor}(W_0,W_1)((i_1,k_0),(i_0,k_1)).
\end{equation}

We define the {\it category of double systems}
$\sys{R}$ to be the category
whose objects are double systems,
whose morphisms between objects $W_0$ and $W_1$ are elements of $Mor(W_0,W_1)$
and where composition is induced by composition of $R$-module morphisms
in the following way.
If $W_j : I_j \times K_j \lra{} \mod{R}$, $j=0,1,2$ are double
systems then since inverse and direct limits commute with finite products, we can define the composition maps as follows:
$$
Mor(W_1,W_2) \times Mor(W_0,W_1)
=$$
$$
\varprojlim_{i_0} \varprojlim_{i'_1}
\varinjlim_{i_1}  \varinjlim_{i_2}
\varprojlim_{k_1}\varprojlim_{k_2} 
\varinjlim_{k_0} \varinjlim_{k'_1}
Hom(W_1(i'_1,k'_1),W_2(i_2,k_2)) \times
Hom(W_0(i_0,k_0),W_1(i_1,k_1)) 
$$
$$
=
\varprojlim_{(i_0,i'_1)}
\varinjlim_{(i_1,i_2)}
\varprojlim_{(k_1,k_2)} 
\varinjlim_{(k_0,k'_1)}
Hom(W_1(i'_1,k'_1),W_2(i_2,k_2)) \times
Hom(W_0(i_0,k_0),W_1(i_1,k_1)) 
$$
$$
=
\varprojlim_{(i_0,i'_1)}
\varinjlim_{\substack{(i_1,i_2) \\ i'_1 \leq_I i_1}}
\varprojlim_{(k_1,k_2)} 
\varinjlim_{\substack{(k_0,k'_1) \\ k'_1 \leq_K k_1}}
Hom(W_1(i'_1,k'_1),W_2(i_2,k_2)) \times
Hom(W_0(i_0,k_0),W_1(i_1,k_1)) 
$$
$$
\lra{\circ_{W_1}}
\varprojlim_{(i_0,i'_1)}
\varinjlim_{\substack{(i_1,i_2) \\ i'_1 \leq_I i_1}}
\varprojlim_{(k_1,k_2)} 
\varinjlim_{\substack{(k_0,k'_1) \\ k'_1 \leq_K k_1}}
Hom(W_0((i_0,k_0),(i_2,k_2)))
= Mor(W_0,W_2)
$$
where $Hom := Hom_{\mod{R}}$.
The {\it category $\Ind{R}$ of directed systems} is the full subcategory of $\sys{R}$
whose objects are directed systems.
Similarly the {\it category $\Pro{R}$
of inverse systems}
is the full subcategory whose objects are inverse systems.
\end{defn}

\begin{example} \label{example morphisms}
	A key example of a morphism
	$\phi \in Mor(W_0,W_1)$ of double systems
	$W_j : I_j \times K_j \lra{} \mod{R}$, $j=0,1$
	is a natural transformation $\phi$
	from $W_0 \circ (\text{id}_{I_0} \times F)$ to
	$W_1 \circ (G \times \text{id}_{K_1})$
	where $F : K_1 \lra{} K_0$
	and $G : I_0 \lra{} I_1$
	are functors.

	If $I_0 = I_1$ and $K_0 = K_1$
	and the natural transformation maps
	are double system maps,
	then such a morphism is equal to the identity map in $\sys{R}$. 
	We will call such a natural transformation a
	{\it standard endomorphism}.

	Suppose $\phi' \in Mor(W_1,W_2)$ is another morphism
	where $W_2 : I_2 \times K_2 \lra{} \mod{R}$
	is a double system and where it is also given by a natural transformation
	$\phi' : W_1 \circ (\text{id}_{I_1} \times F')$ to
	$W_2 \circ (G' \times \text{id}_{K_2})$
	where $F' : K_2 \lra{} K_1$
	and $G : I_1 \lra{} I_2$
	are functors.
	Then the composition $\phi' \circ \phi \in Mor(W_0,W_2)$ can be represented by the natural transformation
	$$(\phi'  \cdot (\text{id}_{G \times \text{id}_{K_2}})) \circ (\phi \cdot (\text{id}_{\text{id}_{I_0} \times F'}))  : W_0 \circ (\text{id}_{I_0} \times F') \lra{} W_2 \circ (G \times \text{id}_{K_2})$$
	where $\circ$ denotes vertical composition of natural transformations and
	$\cdot$ denotes horizontal composition
	and where $\text{id}_Q$ means the identity natural transformation from a functor $Q$ to itself.
	Here we have also used the identity
	$(G \times \textnormal{id}_{K_1}) \circ (\textnormal{id}_{I_0} \times F') = (\textnormal{id}_{I_1} \times F') \circ (G \times \textnormal{id}_{K_2})$.
\end{example}

\begin{defn} \label{definition inclusion subsystem}
A {\it cofinal subsystem} of a double system
$W : I \times K \lra{} \mod{R}$
is the restriction $\check{W} := W|_{\check{I} \times \check{K}}$
of $W$ to
the subcategory $\check{I} \times \check{K} \subset I \times K$
where $\check{I} \subset I$, $\check{K}^\op \subset K^\op$
are cofinal subsets.
The {\it inclusion morphism} is the morphism
$\iota_{\check{W},W} : \check{W} \lra{} W$
in $\sys{R}$ given by the natural transformation between
$\check{W} \circ (\text{id}_{\check{I}} \times F)$ to
$W \circ (G \times \text{id}_K)$
constructed using double system maps
where $F : K \lra{} \check{K}$
satisfies $k \leq_K F(k)$ for all $k \in K$
and where $G : \check{I} \lra{} I$ is the natural inclusion map.
Here $F$ is constructed using the axiom of choice.
\end{defn}

\begin{remark} \label{remark inclusion invariance}
The inclusion map does not depend on the choice of functor
$F$. This due to the fact that if we have a morphism $\iota'_{\check{W},W} : \check{W} \lra{} W$
constructed in the same way,
but using a different choice of functor $F$
then there are standard endomorphisms $E,E' : W \lra{} W$ so that
$E \circ \iota_{\check{W},W} = E' \circ \iota'_{\check{W},W}$.
Since standard endomorphisms represent the identity map
in $\sys{R}$ we get that
$\iota_{\check{W},W} = \iota'_{\check{W},W}$.
\end{remark}

\begin{lemma} \label{lemma isomorphism condition}
The inclusion map $\iota_{\check{W},W}$
from Definition \ref{definition inclusion subsystem}
is an isomorphism.
\end{lemma}
\begin{proof}[Proof of Lemma \ref{lemma isomorphism condition}.]
The inverse $\eta$ of $\iota_{\check{W},W}$ is given by the natural transformation
between
$W \circ (\text{id}_I \times F')$ to
$\check{W} \circ (G' \times \text{id}_K)$ built from double system maps
where
$F' : \check{K} \lra{} K$ is the natural inclusion map
and where
$G' : I \lra{} \check{I}$
satisfies $i \leq_I F(i)$ for all $i \in I$.
The compositions
$\eta \circ \iota_{\check{W},W}$,
$\iota_{\check{W},W} \circ \eta$ as described in Example
\ref{example morphisms} are standard endomorphisms and hence
$\iota_{\check{W},W}$ is an isomorphism.
\end{proof}

\begin{defn} \label{definition inverse direct limit is a functor}
We define $\varinjlim \varprojlim$
to be the functor
$$
\varinjlim \varprojlim : \sys{R} \lra{} \mod{R}, \quad
\varinjlim \varprojlim := \oplus_{m \in \Z} Hom_{\sys{R}}(\iota_R[-m],-)
$$
where $\iota_R$ is the trivial
double system.
We define
$\varinjlim : \Ind{R} \lra{} \mod{R}$
(resp. $\varprojlim : \Pro{R} \lra{} \mod{R}$)
to be the restriction of this functor to the subcategory
$\Ind{R}$ (resp. $\Pro{R}$).
\end{defn}

\begin{remark} \label{remark inverse direct limit is a functor}
The functor $\varinjlim \varprojlim$
sends a double system
$(W(i,k))_{(i,k) \in I \times K}$
to the limit
$\varinjlim_i \varprojlim_k W(i,k)$.
Similarly the functors $\varinjlim$ and $\varprojlim$
above
coincide with $\varinjlim$ and $\varprojlim$
from Definitions \ref{definition direct limit} and \ref{definition inverse limit}.
\end{remark}

\begin{defn} \label{definition of tensor product of double systems}
	Let $W_j : I_j \times K_j \lra{} \mod{R}$, $j=0,1$ be a double systems.
	We define
	$W_0 \times W_1$
	to be the double system
	$(W_0(i,k) \times W_1(\check{i},\check{k}))_{(i,\check{i},k,\check{k}) \in {(I_0 \times I_1) \times (K_0 \times K_1)}}$
	with the natural double system maps induced from the double system maps of $W_j$, $j=0,1$.
	We define
	$W_0 \otimes W_1$ 
	to be the double system
	$(W_0(i,k) \otimes_R W_1(\check{i},\check{k}))_{(i,\check{i},k,\check{k}) \in {(I_0 \times I_1) \times (K_0 \times K_1)}}$
	with the natural double system maps induced from the double system maps of $W_j$, $j=0,1$.
	
	Let $\check{W}_j : \check{I}_j \times \check{K}_j \lra{} \mod{R}$, $j=0,1$ be two additional double systems and let $T_{W_0,W_1,\check{W}_0,\check{W}_1}$ be the natural composition
	$$
	T_{W_0,W_1,\check{W}_0,\check{W}_1} : Mor_{\sys{R}}(W_0,\check{W}_0) \otimes_R Mor_{\sys{R}}(W_1,\check{W}_1) =
	$$
	$$
	(\varprojlim_{i_0} \varinjlim_{\check{i}_0} \varprojlim_{\check{k}_0} \varinjlim_{k_0} 
	Hom(W_0(i_0,k_0),\check{W}_0(\check{i}_0,\check{k}_0))) \otimes_R
	(\varprojlim_{i_1} \varinjlim_{\check{i}_1} \varprojlim_{\check{k}_1} \varinjlim_{k_1} Hom(W_1(i_1,k_1),\check{W}_1(\check{i}_1,\check{k}_1)))
$$
$$
\to
\varprojlim_{i_0}
\varprojlim_{i_1}
\varinjlim_{\check{i}_0}
\varinjlim_{\check{i}_1}
\varprojlim_{\check{k}_0}
\varprojlim_{\check{k}_1}
\varinjlim_{k_0} 
\varinjlim_{k_1} 
(Hom(W_0(i_0,k_0),\check{W}_0(\check{i}_0,\check{k}_0)))
\otimes_R
(Hom(W_1(i_1,k_1),\check{W}_1(\check{i}_1,\check{k}_1))) 
$$
$$
\to
\varprojlim_{(i_0,i_1)}
\varinjlim_{(\check{i}_0,\check{i}_1)}
\varprojlim_{(\check{k}_0,\check{k}_1)}
\varinjlim_{(k_0,k_1)} 
(Hom(W_0(i_0,k_0),\check{W}_0(\check{i}_0,\check{k}_0)))
\otimes_R
(Hom(W_1(i_1,k_1),\check{W}_1(\check{i}_1,\check{k}_1))) 
$$
$$
\to
\varprojlim_{(i_0,i_1)}
\varinjlim_{(\check{i}_0,\check{i}_1)}
\varprojlim_{(\check{k}_0,\check{k}_1)}
\varinjlim_{(k_0,k_1)} 
Hom\left(W_0(i_0,k_0) \otimes_R W_1(i_1,k_1),\check{W}_0(\check{i}_0,\check{k}_0) \otimes_R \check{W}_1(\check{i}_1,\check{k}_1)\right)
$$
$$
= Mor(W_0 \otimes W_1,\check{W}_0 \otimes \check{W}_1)
$$
where $Hom = Hom_{\mod{R}}$.
For any two morphisms
$\Phi_j : W_j \lra{} \check{W}_j$, $j=0,1$
in $\sys{R}$, we define
$\Phi_0 \otimes \Phi_1 := T_{W_0,W_1,\check{W}_0,\check{W}_1}(\Phi_0 \otimes_R \Phi_1) \in Mor_{\sys{R}}(W_0 \otimes W_1,\check{W}_0 \otimes \check{W}_1)$.

	Also, if $W : I \times K \lra{} \mod{R}$ is a double system,
	we define $\tau_W : W \otimes W \lra{} W \otimes W$ to be the morphism sending
	$x \otimes y$ to $(-1)^{pq} y \otimes x$
	for all $x \in (W(i,k))_p$, $y \in (W(i',k'))_q$,
	$(i,k)$, $(i',k')$ in $I \times K$ and $p,q \in \Z$.
\end{defn}

These operations make $\sys{R}$
into a symmetric monoidal category together with the identity object $\iota_R$ due to the fact that $\mod{R}$ is a symmetric monoidal category and the fact that natural transformations $\Phi$ between
double systems $W_j : I \times K \lra{} \mod{R}$, $j = 0,1$ where the morphisms of $\Phi$ are isomorphisms induce isomorphisms in $\sys{R}$.

\begin{defn}  \label{definition product on a double system}
	A {\it product} on $W : I \times K \lra{} \mod{R}$
	is a morphism
	$\mu : W \otimes W \lra{} W$ so that
	$\mu \circ (\id_W \otimes \mu) = \mu \circ (\mu \otimes \id_W)$ using the natural identification $(W \otimes W) \otimes W \cong W \otimes (W \otimes W)$.
	Such a product is {\it graded commutative}
	if $\mu \circ \tau = \mu$.
	The product $\mu$ is {\it unitary}
	if there exists a morphism
	$\iota : \iota_R \lra{} W$ where
	$\iota_R$ is the trivial double system
	satisfying
	$\mu \circ (\id_W \otimes \iota) = \mu \circ (\iota \circ \id_W) = \id_W$
	where we identify $W$ with $W \otimes \iota_R$
	and $\iota_R \otimes W$ in the natural way.
	A {\it morphism} of double systems $W_0$, $W_1$
	with products $\mu_j : W_j \otimes W_j \lra{} W_j$, $j=0,1$ is a morphism $\Phi : W_0 \lra{} W_1$
satisfying $\Phi \circ \mu_0 = \mu_1 \circ (\Phi \otimes \Phi)$.
	A (graded commutative) 
	(unital)  
	{\it product} on a directed or inverse system is a 
	(graded commutative) 
	(unital)
	product on the corresponding double system.
\end{defn}

\begin{remark} \label{remark products on direct and inverse limits}
	From now on, all products on double systems (resp. directed/inverse systems) will be
	unital
	 graded commutative  products. Hence from now on,
	we will just call them {\it products}.
	We will also assume all morphisms between such double systems with products preserve the unit.
	If $W$ is a double system with product $\mu : W \otimes W \lra{} W$, then we get a product
$$
(\varinjlim \varprojlim W) \otimes_R (\varinjlim  \varprojlim W) \lra{\oplus_{m,m'}  T_{\iota_R[-m],\iota_R[-m'
],W,W}}
(\varinjlim \varprojlim (W \otimes W)) \lra{\varinjlim  \varprojlim \mu} \varinjlim  \varprojlim W
$$
on
$\varinjlim \varprojlim W$ making it into a 
unital
 graded commutative $R$-algebra
where $T_{\iota_R[-m],\iota_R[-.m'
],W,W}$ is given in Definition \ref{definition of tensor product of double systems} and where we use the natural identification
$\iota_R \cong \iota_R \otimes \iota_R$.
Similarly a product on a directed/inverse system
gives us an induced product on its direct/inverse limit.	
\end{remark}

We will also need to show
that ``derived'' versions
of the functor
$\varinjlim \varprojlim$ exist.
We need some preliminary definitions and lemmas before we define such a functor.

\begin{defn} \label{definition extending functor}
Let $F : \Pro{R} \lra{} \mod{R}$
be a functor so that the corresponding maps
$F : \text{Mor}_{\Pro{R}}(P_0,P_1) \lra{} \text{Mor}_{\mod{R}}(F(P_0),F(P_1))$ are $R$-module maps for each pair of objects $P_0,P_1$ in $\Pro{(R)}$.
We define the functor
$\varinjlim F : \sys{R} \lra{} \mod{R}$
as follows.
If $W : I \times K \lra{} \mod{R}$
is a double system then
$\varinjlim F(W) := \varinjlim_i F(W|_i)$ where $W|_i$ is the inverse system
$(W_{i,k})_{k \in K}$.
Also if $W_j : I_j \times K_j \lra{} \mod{R}$, $j=0,1$ are double systems,
the corresponding functor on morphisms
is the natural composition:
$$\varinjlim F : Mor_{\sys{R}}(W_0,W_1)
= \varprojlim_{i_0} \varinjlim_{i_1} Mor_{\Pro{R}}(W_0|_{i_0},W_1|_{i_1})
 \lra{\varprojlim_{i_0} \varinjlim_{i_1} F}$$
$$
\varprojlim_{i_0} \varinjlim_{i_1} Mor_{\mod{R}}(F(W_0|_{i_0}), F(W_1|_{i_1}))
= Mor_{\Ind{R}}((F(W_0|_i))_{i \in I_0},(F(W_1|_i))_{i \in I_1})
$$
$$\lra{\varinjlim}
Mor_{\mod{R}}(\varinjlim_{i_0} F(W_0|_{i_0}), \varinjlim_{i_1} F(W_1|_{i_1}))
.$$
\end{defn}

\begin{defn} \label{definition catgeory of double systems with fixed I K}
Let $(K,\leq_K)$
be an inverse directed set.
Define
$\mod{R}^K$ to be the category
whose objects are inverse systems
$V : K \lra{} \mod{R}$
and whose morphisms
are natural transformations between such objects.
We define
$\alpha_K : \mod{R}^K \lra{} \sys{R}$
to be the functor sending objects $V$ to $V$ and sending morphisms
to the induced morphisms in $\sys{R}$.
\end{defn}

Since the category $\mod{R}^K$ has enough injectives
by \cite[Theorem 11.18]{shapetheorybook2013} we have the following definition below.
Technically \cite[Theorem 11.18]{shapetheorybook2013}
proves this for ungraded modules, but
the graded case follows immediately
since a graded module is a direct sum of ungraded ones.
This will also be true for other
theorems cited in \cite{shapetheorybook2013}.

\begin{defn} \label{definition derived lim sfunctor}
Let $(K,\leq_K)$
be an inverse directed set.
For each $k \in \N$, define
$\varprojlim^k|_K : \mod{R}^K \lra{} \mod{R}$
to be the $k$th right derived functor
of $\varprojlim \circ \alpha_K$.
\end{defn}

The following lemma follows immediately from
Theorem 15.5 and Remark 15.6 in \cite{shapetheorybook2013}.

\begin{lemma} \label{lemma derived extension to insys}
There is a natural functor $\varprojlim^k : \Pro{R} \lra{} \mod{R}$
satisfying $\varprojlim^k \circ \alpha_K = \varprojlim^k|_K$
for all inverse directed sets $(K,\leq_K)$.
\end{lemma}

For our purposes, it does not matter how the functor $\varprojlim^k$ is constructed.
We will only use the fact that it satisfies the property stated in the lemma above.

\begin{defn} \label{definition lim invlim functor}
For each $k \in \N$, define the {\it direct limit of $\varprojlim^k$} to be the functor
$\varinjlim \varprojlim^k : \sys{R} \lra{} \mod{R}$
where $\varinjlim$ is given as in Definition \ref{definition extending functor}
and $\varprojlim^k$ is constructed in Lemma \ref{lemma derived extension to insys}.
\end{defn}

Finally, we need a test telling us when $\varinjlim \varprojlim^1$ vanishes.
\begin{lemma} \label{lemma mittag leffler like test}
Suppose that $W : I \times K \lra{} \mod{R}$ is a double system
where $(I,\leq_I)$ (resp. $(K,\leq_K)$) is a directed (resp. inverse directed) set.
Suppose that there is a
cofinal family $\check{I} \subset I$
and a
countably infinite
cofinal family $\check{K} \subset K$
of $(K,\leq_K)^\op$
with the property that
for all $i \in \check{I}$, $k,k' \in \check{K}$ satisfying $k \leq_K k'$ we have that
$W((i,k) \lra{} (i,k'))$ is surjective.
Then $\varinjlim \varprojlim^1 W = 0$.
\end{lemma}
\begin{proof}[Proof of Lemma \ref{lemma mittag leffler like test}.]
Let $W|_i$ be the inverse system $(W(i,k))_{k \in K}$ for each $i \in I$
and let
$\check{W}|_i$ be the inverse system
$(W(i,k))_{k \in \check{K}}$.
Since $\check{K}$ is countably infinite, we can assume
after passing to a cofinal subset of $\check{K}$ that $(\check{K},\leq_K)$
is equal to the inverse directed set $(\N^\op,\leq^\op)$ by Lemma \ref{lemma isomorphism condition}.
By \cite[Proposition 3.5.7]{weibel1995introduction} combined with
\cite[Lemma 11.49]{shapetheorybook2013}
 we have that
$\varprojlim^1 \check{W}|_i = 0$ for each $i \in \check{I}$.
Hence by \cite[Theorem 14.9]{shapetheorybook2013} we have
that
$\varprojlim^1 W|_i = 0$ for each $i \in \check{I}$.
Hence $\varinjlim \varprojlim^1 W = 0$.
\end{proof}

\subsection{Novikov Rings.}

In this section we give a definition
of a Novikov ring
(which we will define using
inverse and direct limits).
Novikov rings are appropriate coefficient
rings for our Hamiltonian Floer cohomology groups.
Throughout this subsection
we will fix a (possibly empty)
contact cylinder $\check{C} = [1-\epsilon,1+\epsilon] \times C$
and we will let $D \subset M$
be its associated Liouville domain.

\begin{defn} \label{definition convex cone}

	Let $W$ be a finite dimensional real vector space.
	A {\it convex cone}
	is a subset $Q \subset W$
	so that for all $x,y \in Q$
	and all positive real numbers $\alpha,\beta>0$,
	we have that $\alpha x + \beta y \in Q$.
	Such a cone is called {\it salient}
	if, for each $x \in Q - 0$, $-x \notin Q$.
	Now suppose $(A,\cdot)$ is a finitely generated abelian group
	and let $Q$ be a cone in
	$(A \otimes_\Z \R)^*$.
	Define $\preceq_Q$
	to be the binary relation on $A$
	where $x \preceq_Q y$
	if and only if
	$f((y\cdot x^{-1}) \otimes 1) \geq 0$ for all $f \in Q$.
\end{defn}

\begin{remark}
If $Q$ is a closed salient cone, then $(A,\preceq_Q)$
and $(A^\op,\preceq_Q^\op)$ are directed sets.
\end{remark}

\begin{defn} \label{definition novikov ring associated to salient cone}
	Let $(A,\cdot)$ be a finitely generated abelian group and let $Q \subset (A \otimes_\Z \R)^*$ be a closed salient cone.
	For each $x \in A$,
	let $F^Q_x$
	be the free $\K$-module generated
	by elements of the set
	$S^Q_x := \{a \in A \ : \ x \preceq_Q a \}$.
	Let $(I, \leq_I)$ and be the (inverse) directed set
	$(A^\op, \preceq_Q^\op)$.
	Then $F^Q := (F^Q_{x_-,x_+})_{(x_-,x_+) \in I \times I}$ is a double system
	where $F^Q_{x_-,x_+} := F^Q_{x_-}/(F^Q_{x_-} \cap F^Q_{x_+})$ and where the double system maps are the natural compositions
	$$
	F^Q_{x_-,x_+}\lra{}
	F^Q_{x_-,x'_+} \lra{} F^Q_{x'_-,x'_+}
	$$
	for all $x_\pm, x'_\pm \in A$ satisfying  $x'_\pm \preceq_Q x_\pm$.
	Choose a function $M : A^2 \lra{} A$
	satisfying $M(x,y) \preceq_Q x$
	and $M(x,y) \preceq_Q y$ for all $(x,y) \in A^2$.
	For each $x^0_-,x^1_-,x^2_+ \in A$,
	let
	$$\mu : F^Q_{x^0_-,x^2_+ - x^1_-} \otimes_\K F^Q_{x^1_-,x^2_+-x^0_-} \lra{}
	F^Q_{x^0_- \cdot x^1_-,x^2_+}
	$$
	be the unique $\K$-linear map
	sending $[a_0] \otimes [a_1]$
	to $[a_0 \cdot a_1]$
	for all $a_j \in S^Q_{x^j_-}$, $j=0,1$.
	Then $\mu$ defines for us a product
	$\mu : F^Q \otimes F^Q \lra{} F^Q$ on the double system $F^Q$.
	We define the {\it $(A,Q)$-Novikov ring} to be the ring
	\begin{equation} \label{equation Novikov ring definition}
	\Lambda_\K^{A,Q} := \varinjlim_{x_- \in I} \varprojlim_{x_+ \in I} F^Q_{x_-,x_+}
	\end{equation}
	whose product is induced by $\mu$
	as in Remark \ref{remark products on direct and inverse limits}.
	We also have a subring
	\begin{equation} \label{equation positive Novikov ring}
	\Lambda_\K^{A,Q,+} :=
	\varprojlim_{x_+ \in I} F^Q_{0,x_+} \subset \Lambda_\K^{A,Q}
	\end{equation}
	called the {\it positive $(A,Q)$-Novikov ring}.
\end{defn}

\begin{remark} \label{remark description of Novikov ring}
We can think of $\Lambda^{A,Q}$
as the ring of formal power series
$$
\Lambda_\K^{A,Q} = \left\{\sum_{i \in \N} b_i t^{a_i} \ : \ b_i \in \K, \ a_i \in A \quad \forall  \ i \in \N, \ (a_i)_{i \in \N} \ \text{is cofinal in} \ (A,\preceq_Q)
\right\}.
$$
Intuitively, the terms in this series
must ``tend to infinity'' in the ``cone''
$\{x \in A \ : \ 0 \preceq_Q x \}$.
\end{remark}

\begin{defn} \label{definition convex cone associated to contact cylinder}
Let $\iota_D : H^2(M,D;\R) \lra{} H^2(M;\R)$ be the natural restriction map on cohomology and let $\pi_D : H^2(M,D;\R) \times \R \times \R \lra{} H^2(M,D;\R)$ be the natural projection map.
Let $Q \subset H^2(M,D;\R) \times \R \times \R$
be a cone so that
$\check{Q} := \iota_D(\pi_D(Q))$ is a closed and salient cone.
Using the canonical identification
$$(H_2(M;\Z) \otimes_\Z \R)^* = H^2(M;\R),$$
we can define the {\it $Q$-Novikov ring}
to be the $(H_2(M;\Z), \check{Q})$-Novikov ring
$$\Lambda_\K^Q := \Lambda_\K^{H_2(M;\Z),\check{Q}}$$
and the {\it positive $Q$-Novikov ring}
to be the positive $(H_2(M;\Z),\check{Q})$-Novikov ring
$$\Lambda_\K^{Q,+} := \Lambda_\K^{H_2(M;\Z),\check{Q},+}.$$
\end{defn}

The Novikov rings above
are designed to deal with
multiple action values encoded in
Definition \ref{definition action functionals}.
Note that we could have defined the above Novikov ring $\Lambda_\K^Q$ to be a
$(H_2(M,D;\Z) \times \Z \times \Z,Q)$-Novikov ring associated to $Q$ instead.
Such a definition would have
forced us to generalize the notion of a
capped $1$-periodic orbit so that the
``capping'' also has some
boundary components inside $D$ (one would have to modify the way action is calculated too and restrict the class of Hamiltonians further).
This would
have given us more information,
but would have added an extra layer of
unnecessary complication.
Therefore we have decided
to use this simpler definition.

\begin{example} \label{example Novikov ring associated to some forms}

Let $\widetilde{\omega}_\# = (\widetilde{\omega}_i)_{i \in S}$
be a finite collection of $\check{C}$-compatible $2$-forms with scaling constants
$(\lambda^\pm_i)_{i \in S}$ and let
$\omega_{\check{C}}$ be a $\check{C}$-compatible $2$-form
with scaling constants $0$ and $1$ and which is equal to $\omega$ outside $D \cup ([1,1+\epsilon/2] \times C)$.
Let $Q_{\widetilde{\omega}_\#} \subset H^2(M,D;\R) \times \R \times \R$ be the smallest convex cone containing 
$([\widetilde{\omega}_i - \lambda^-_i \omega + \lambda^-_i \omega_{\check{C}}] ,\lambda^-_i,\lambda^+_i) \in H^2(M,D;\R) \times \R \times \R$ for each $i \in S$.
The {\it $(\widetilde{\omega}_\#)$-Novikov ring} is the Novikov ring:
$$\Lambda_\K^{\widetilde{\omega}_\#} := \Lambda_\K^{Q_{\widetilde{\omega}_\#}}.$$
The two key examples for this paper are when
\begin{itemize}
\item $\widetilde{\omega}_\#$ has one element $\widetilde{\omega}$, giving us a Novikov ring $\Lambda_\K^{\widetilde{\omega}} := \Lambda_\K^{\widetilde{\omega}_\#}$,
which can be thought of as the set of power series
$$
\Lambda_\K^{\widetilde{\omega}} = \left\{\sum_{i \in \N} b_i t^{a_i} \ : \ b_i \in \K, \ a_i \in H_2(M;\Z), \quad \widetilde{\omega}(a_i) \to \infty
\right\}
$$
\item and when $\widetilde{\omega}_\#$ has two elements $\widetilde{\omega}_0$,$\widetilde{\omega}_1$, giving us a Novikov ring $\Lambda_\K^{\widetilde{\omega}_0,\widetilde{\omega}_1} := \Lambda_\K^{\widetilde{\omega}_\#}$ which can be thought of as the set of power series
$$
\Lambda_\K^{\widetilde{\omega}_0,\widetilde{\omega}_1} = \left\{\sum_{i \in \N} b_i t^{a_i} \ : \ b_i \in \K, \ a_i \in H_2(M;\Z), \quad \min(\widetilde{\omega}_0(a_i),\widetilde{\omega}_1(a_i)) \to \infty
\right\}.
$$
\end{itemize}
\end{example}

\subsection{Definition of Floer Cohomology using Alternative Filtrations.} \label{definition Floer alternative filtration}

In this section we will give a definition of Hamiltonian Floer cohomology using the action function in Definition \ref{definition action functionals}.
Throughout this subsection,
we will fix a (possibly empty) contact
cylinder $\check{C} = [1-\epsilon,1+\epsilon] \times C \subset M$ with associated Liouville domain $D$
and cylindrical coordinate $r_C$.
We also let $\iota_D : H^2(M,D;\R) \lra{} H^2(M;\R)$ be the natural restriction map on cohomology and
let
$\pi_D : H^2(M,D;\R) \times \R \times \R \lra{} H^2(M,D;\R)$ be the natural projection map.

In order to define Hamiltonian Floer cohomology with the right properties, we need to consider certain cones inside $H^2(M,D;\R) \times \R \times \R$.

\begin{defn} \label{defn chain complex}
Let $\omega_{\check{C}}$ be a $\check{C}$-compatible $2$-form
with scaling constants $0$ and $1$ and which is equal to $\omega$ outside $D \cup ([1,1+\epsilon/2] \times C)$.
A cone $Q \subset H^2(M,D;\R) \times \R \times \R$
is {\it $\check{C}$-compatible}
if
\begin{itemize}
\item both $\iota_D(\pi_D(Q))$ and $Q$ are closed and salient and
\item $Q \subset Q_{\check{C}} \cup \{0\}$ where $Q_{\check{C}}$ is given as in Definition \ref{definition action functionals}.
\end{itemize}  
A $\check{C}$-compatible cone $Q$
is called {\it thin}
if $\pi_D|_Q : Q \lra{} H^2(M,D;\R)$
is an injective map.
A $\check{C}$-compatible cone $Q$ is called {\it small}
if $Q \subset \R[\omega_{\check{C}}] \times \R \times \R$ and if the natural projection map $$\R[\omega_{\check{C}}] \times \R \times \R \lra{} \R[\omega_{\check{C}}] \times \R, \quad (q,\lambda_-,\lambda_+) \lra{} (q,\lambda_-)$$
restricted to $Q$ is injective.
A pair of $\check{C}$-compatible cones $(Q_-,Q_+)$ is called {\it wide}
if for each $q \in \pi_D(Q_-) \cup \pi_D(Q_+)$ there exists $\lambda_{\pm,-}^q,\lambda_{\pm,+}^q \in \R$ so that
$(q,\lambda_{\pm,-}^q,\lambda_{\pm,+}^q) \in Q_\pm$,
$$\lambda^q_{-,-} < \lambda^q_{+,-} \ \text{and} \quad \lambda^q_{-,+} = \lambda^q_{+,+}.$$
A {\it $\check{C}$-interval domain pair} is a pair $(Q_-,Q_+)$
of $\check{C}$-compatible cones
so that
\begin{enumerate}
		\item \label{item:action interval thin}
$Q_+$ is thin and not equal to the trivial cone $\{0\}$,
\item $Q_+ \subset Q_-$ and
\item  \label{item:action interval height property}
if $Q_-$ is not small,
then $(Q_-,Q_+)$ is wide.
\end{enumerate}

For any $\check{C}$-compatible
cone $Q$, we define
$\Sc(Q)$ to be the space
of continuous functions
$f : Q \lra{} \R$
satisfying $f(\sigma x) = \sigma f(x)$ for all $x \in Q$ and $\sigma \geq 0$ equipped with the $C^0_{\text{loc}}$-topology.
%
A {\it $\check{C}$-action interval}
is a pair $(a_-,a_+) \in \Sc(Q_-) \times \Sc(Q_+)$
where 
$(Q_-,Q_+)$
is a $\check{C}$-interval domain pair.
We say that $(a_-,a_+)$ is {\it small}
if $Q_-$ and $Q_+$ are small and {\it wide} if $(Q_-,Q_+)$ is wide.

For each subset $P \subset \Z$,
weakly $\check{C}$-compatible Hamiltonian
$H$
and $\check{C}$-action interval $(a_-,a_+) \in \Sc(Q_-) \times \Sc(Q_+)$,
define
$\Gamma^P_{\check{C},a_-,a_+}(H)$ to be the set of
capped $1$-periodic orbits
$\gamma$ of $H$
whose associated $1$-periodic orbit
is not contained in $[1+\epsilon/8,1+\epsilon/2] \times C$, whose index is in $P$ and
satisfying
\begin{equation} \label{equation capped periodic orbits in action interval}
a_- \leq \cA_{H,\check{C}}(\gamma)|_{Q_-}, \quad  a_+ \nleq \cA_{H,\check{C}}(\gamma)|_{Q_+}
\end{equation}
where $\cA_{H,\check{C}}$ is given in Definition \ref{definition action functionals}.

If $(a_-,a_+) \in \Sc(Q_-) \times \Sc(Q_+)$ is wide
then we define
the {\it Hamiltonian height of $(a_-,a_+)$} to be
$$
\text{height}(a_-,a_+) :=$$
\begin{equation} \label{equation height equation}
\sup \left\{ \frac{a_+(x,\lambda_{+,-},\lambda_{+,+}) - a_-(x,\lambda_{-,-},\lambda_{-,+})}
{\lambda_{+,-} - \lambda_{-,-}} \ : \begin{array}{c}
(x,\lambda_{\pm,-},\lambda_{\pm,+}) \in Q_\pm,  \\
\lambda_{-,-} < \lambda_{+,-}, \\
\lambda_{-,+} = \lambda_{+,+}
\end{array}
\right\}.
\end{equation}

For each $\check{C}$-action interval $(a_-,a_+)$ and each manifold $\Sigma$,
define
$$
\ccH^\Sigma(\check{C},a_-,a_+) :=$$
\begin{equation} \label{equation high enough hamiltonians}
\left\{\begin{array}{ll}
\ccH^\Sigma(\check{C}) & \text{if} \ (a_-,a_+) \ \text{is small} \\
\left\{
H \in \overline{\ccH}^\Sigma(\check{C}) \ : \ 
m_{H_\sigma} > \text{height}(a_-,a_+) \ \forall \ \sigma \in \Sigma
\right\} & \text{otherwise}
\end{array}
\right.
\end{equation}
where
$\ccH^\Sigma(\check{C})$ and
$\overline{\ccH}^\Sigma(\check{C})$
are given in Definition
\ref{definition wspaces of compatible objects} and $m_{H_\sigma}$ is the height of $H_\sigma$ as in Definition \ref{definition Hamiltonians compatible with contact cylinder}
for each $\sigma \in \Sigma$ where $H = (H_\sigma)_{\sigma \in \Sigma}$.
For each subset $P \subset \Z$,
define
$$\ccH^\reg(\check{C},a_-,a_+,P) \subset \ccH^\T(\check{C},a_-,a_+)$$
to be the subspace of time dependent Hamiltonians $H = (H_t)_{t \in \T}$
so that there exist neighborhoods
$N_-$, $N_+$ of $a_-$ and $a_+$
in $\Sc(Q_-)$ and $\Sc(Q_+)$ respectively so that for each
$a'_\pm \in N_\pm$, we
have
\begin{itemize} 
\item that every capped $1$-periodic orbit
in $\Gamma^\Z_{\check{C},a'_-,a'_+}(H)$ is non-degenerate, 
\item there are no $1$-periodic orbits contained in $[1+\epsilon/8,1+\epsilon/2] \times C$ and
\item  
$\Gamma^P_{\check{C},a'_-,a'_+}(H) = \Gamma^P_{\check{C},a_-,a_+}(H)$.
\end{itemize}
Define
$\ccH^\reg(\check{C},a_-,a_+) :=
\ccH^\reg(\check{C},a_-,a_+,\Z)
$.
\end{defn}

We have that
$\ccH^\reg(\check{C},a_-,a_+)$
 is an open dense subset of $\ccH^\T(\check{C},a_-,a_+)$ by Lemma \ref{lemma regular Hamiltonians} in Appendix A.
 The height condition
 in Equation
(\ref{equation high enough hamiltonians})
is essential for this density
property since
any
 Hamiltonian
 in $\ccH^\T(\check{C},a_-,a_+)$
 is constant outside
 $D \cup \check{C}$ if $(a_-,a_+)$ is wide
and therefore these constant orbits $\gamma$
must not satisfy
Equation (\ref{equation capped periodic orbits in action interval}) because they are degenerate.


\begin{defn} \label{definition Floer chain complex}
Let $(a_-,a_+) \in \Sc(Q_-) \times \Sc(Q_+)$ be a $\check{C}$-action interval
and let $H \in \ccH^\reg(\check{C},a_-,a_+,\{p\})$
for some $p \in \Z$.
Define
$CF^p_{\check{C},a_\pm,\infty}(H)$
to be the free $\K$-module generated by
capped $1$-periodic orbits $\gamma$
of $H$ of index $p$ and
satisfying $a_\pm \leq \cA_{H,\check{C}}(\gamma)|_{Q_\pm}$.
Define
$$CF^p_{\check{C},a_-,a_+}(H) := CF^p_{\check{C},a_-,\infty}(H)/(CF^p_{\check{C},a_-,\infty}(H) \cap CF^p_{\check{C},a_+,\infty}(H)).$$
\end{defn}

\begin{remark} \label{remark basis}
The $\K$-module $CF^p_{\check{C},a_-,a_+}(H)$
is naturally isomorphic to the free $\K$-module
generated by $\Gamma^p_{\check{C},a_-,a_+}(H) := \Gamma^{\{p\}}_{\check{C},a_-,a_+}(H)$.
From now on we will not distinguish between
describing such a group as a quotient
or a free module.
\end{remark}

\begin{defn} \label{definition connect sub of orbit}
Let $H$ be a time dependent Hamiltonian on $M$.
If $\gamma = (\widetilde{\gamma},\check{\gamma})$ is a capped $1$-periodic orbit of $H$
and $v \in H_2(M;\Z)$ is a homology class,
then we define
$\gamma \# v := (\widetilde{\gamma} \# v,\check{\gamma})$
where
$\widetilde{\gamma} \# v$
has the property that
$(\widetilde{\gamma} \# v) \star \widetilde{\gamma}$ is homologous to $v$
where
$\star$ is defined Equation (\ref{equation join of two cappings}).
\end{defn}

Above, we are `connect summing' $v$  to the capping surface $\widetilde{\gamma}$.

\begin{remark} \label{remark action module}
Let $(a_-,a_+) \in (Q_-,Q_+)$ be a $\check{C}$-action interval
and let $H \in \ccH^\reg(\check{C},a_-,a_+,\{p\})$
for some $p \in \Z$.
Define $\check{Q}_+ := \iota_D(\pi_D(Q_+))$.
The submonoid
\begin{equation} \label{equation dual submonoid of Q}
S^{Q_+}_0 := \{x \in H_2(M;\Z) \ : 0 \preceq_{\check{Q}_+} x\} \subset H_2(M;\Z)
\end{equation}
acts on $CF^p_{\check{C},a_-,a_+}(H)$ 
as follows:
An element $v \in S^{Q_+}_0$ represents
the unique $\K$-linear map
sending an element
$\gamma \in \Gamma^p_{\check{C},a_-,a_+}(H)$
to $[\gamma \# (-v)] \in CF^p_{\check{C},a_-,a_+}(H)$.
As a result,
the monoid ring $\K[S^{Q_+}_0]$
acts on $CF^p_{\check{C},a_-,a_+}(H)$.

Also, since set of $1$-periodic orbits of $H$ is a compact subset of the free loop space of $M$ with respect to the $C^\infty$ topology,
there is an element
$a \in S^{Q_+}_0$
so that for each $v \in S^{Q_+}_0$ satisfying
$a \preceq_{\check{Q}_+} v$ and each
$\alpha \in CF^p_{\check{C},a_-,a_+}(H)$,
we have $v \cdot \alpha = 0$.
This implies that
the $\K[S^{Q_+}_0]$-action extends
to an action of
the positive ${Q_+}$-Novikov ring $\Lambda_\K^{{Q_+},+}$ 
on $CF^p_{\check{C},a_-,a_+}(H)$
by Equation (\ref{equation positive Novikov ring}).
Hence from now on,
we will think of
$CF^p_{\check{C},a_-,a_+}(H)$
as a $\Lambda_\K^{{Q_+},+}$-module.
\end{remark}

The $\Lambda_\K^{{Q_+},+}$-module
$CF^p_{\check{C},a_-,a_+}(H)$ will be the module underlying our chain complex for our Hamiltonian Floer group.
We now need to explain what the differential on $CF^p_{\check{C},a_-,a_+}(H)$ is.

\begin{defn} \label{defn of differential}
Let 
$(a_-,a_+)$ be a $\check{C}$-action interval and let $P := \{p,p+1\}$ for some $p \in \Z$.
Let $H \in \ccH^\reg(\check{C},a_-,a_+,P)$
and let $J \in \ccJ^{\T,\reg}(H,\check{C})$
where $\ccJ^{\T,\reg}(H,\check{C})$
is given in Definition
\ref{definition space of Floer cylinders}.
We define
the {\it Floer differential}
$$\partial_{H,J}^{(p)} : CF^p_{\check{C},a_-,a_+}(H) \lra{}  CF^{p+1}_{\check{C},a_-,a_+}(H)$$
to be the unique $\K$-linear map
satisfying
$$\partial_{H,J}^{(p)}(\gamma_+) = \sum_{\gamma_- \in \Gamma^{p+1}_{\check{C},a_-,a_+}(H)} \# \overline{\ccM}(H,J,\gamma) \gamma_-, \quad \forall \ \gamma_+ \in \Gamma^p_{\check{C},a_-,a_+}(H)$$
where $\gamma$ is the pair
$(\gamma_-,\gamma_+)$ and where
$\# \overline{\ccM}(H,J,\gamma)$
is the number of elements in the $0$-dimensional manifold
$\overline{\ccM}(H,J,\gamma)$
from Definition \ref{definition space of Floer cylinders}
 counted with sign
according to their orientation.
We define $\partial_{H,J} := \partial^{(p)}_{H,J}$ if it is clear which $p$ we are using.
\end{defn}

The definition above uses the fact that
$\overline{\ccM}(H,J,\gamma)$ is a compact oriented $0$-dimensional manifold
for all pairs of capped $1$-periodic orbits
$\gamma = (\gamma_-,\gamma_+) \in \Gamma^{p+1}_{\check{C},a_-,a_+}(H) \times \Gamma^p_{\check{C},a_-,a_+}(H)$
which follows from
Propositions
\ref{proposition regularity for translation invariant floer cylinders}
and \ref{proposition compactness result}
and
\cite[Section 17]{ritter2013topological}.
Note that one can always find
$J$ as above since
$\ccJ^{\T,\reg}(H,\check{C})$
is a ubiquitous subset of
$\ccJ^\T(\check{C})$
(Definition \ref{definition wspaces of compatible objects})
by Proposition \ref{proposition regular subset for surface}.
By analyzing $1$-parameter families of solutions of the Floer equation for the cylinder
as in Definition \ref{definition parameterized moduli spaces},
one can show that $\partial_{H,J}^2 = 0$.
This is done by a gluing theorem \cite[Theorem 9.2.1]{audin2014morse} combined with the compactness result Proposition \ref{proposition compactness result} and the orientation conventions \cite[Section 17]{ritter2013topological}.
Here we have replaced the compactness result \cite[Theorem 9.1.7]{audin2014morse}
with Proposition \ref{proposition compactness result}.
%
%
Finally, note that
$\partial_{H,J}$ is a $\Lambda_\K^{Q_+,+}$-linear differential where $Q_+$ is the domain of $a_+$. Hence
we have the following definition.

\begin{defn} \label{definition hamiltonian floer cohomology}
Let 
$(a_-,a_+) \in \Sc(Q_-) \times \Sc(Q_+)$ be a $\check{C}$-action interval and let $P := \{p-1,p,p+1\}$ for some $p \in \Z$.
Let $H \in \ccH^\reg(\check{C},a_-,a_+,P)$
and let $J \in \ccJ^{\T,\reg}(H,\check{C})$.
We define the {\it Hamiltonian Floer cohomology group of $H$}
to be the $\Lambda_\K^{Q_+,+}$-module
$$HF^p_{\check{C},a_-,a_+}(H) := \ker(\partial_{H,J}^{(p)})/\text{Im}(\partial_{H,J}^{(p-1)}).$$
\end{defn}

This Hamiltonian Floer cohomology group does not depend on the choice $J$ by
continuation map methods, which will be explained in the subsequent section.
Here are the two main examples of Hamiltonian Floer cohomology groups that
should be kept in mind.

\begin{example}
If $\check{C}$ is the empty contact cylinder and
$Q_- = Q_+$ is the cone spanned by $([\omega],1,1)$
then we get the usual definition of Floer cohomology with the usual action functional
(see Example \ref{example usual action filtraiton}).
\end{example}

\begin{example}
Let $q_0,\cdots,q_k \in H^2(M,D;\R)$
be classes
representing $\check{C}$-compatible $2$-forms whose images in $H^2(M;\R)$ are linearly independent.
Suppose that $Q_+$
is the polyhedral cone spanned by
$(q_0,1,1)$, $\cdots$, $(q_k,1,1)$
and $Q_-$ is the polyhedral cone spanned by
$$(q_0,1,1),\cdots,(q_k,1,1),
(q_0,0,1),\cdots,(q_k,0,1).$$
Then for $(a_-,a_+) \in \Sc(Q_-) \times \Sc(Q_+)$ and $H \in \ccH^\reg(\check{C},a_-,a_+,\Z)$,
$HF^*_{\check{C},a_-,a_+}(H)$
is the Hamiltonian Floer cohomology group generated by capped $1$-periodic orbits
$\gamma$
satisfying
$\cA_{\check{C},H}(\gamma)(q_i,1,1) \geq a_-(q_i,1,1)$
and
$\cA_{\check{C},H}(\gamma)(q_i,0,1) \geq (q_i,0,1)$
for {\it all} $i \in \{0,\cdots,k\}$
and
$\cA_{\check{C},H}(\gamma)(q_i,1,1) \leq a_+(q_j,1,1)$
for {\it some}
$j \in \{0,\cdots,k\}$
where $\cA_{\check{C},H}$ is defined in
\ref{definition action functionals}.
\end{example}

\subsection{Continuation Maps}

Again, throughout this subsection,
we will fix a (possibly empty) contact
cylinder $\check{C} = [1-\epsilon,1+\epsilon] \times C \subset M$ with associated Liouville domain $D$ and cylindrical coordinate $r_C$.
We will also fix a $\check{C}$-action interval $(a_-,a_+)$.

\begin{defn} \label{definition partial order on Hamiltonians}
Let $H^- = (H^-_t)_{t \in \T},H^+ = (H^+_t)_{t \in \T} \in \ccH^\T(\check{C},a_-,a_+)$
have slopes $(\lambda_{H^\pm_t})_{t \in \T}$
and heights $(m_{H^\pm_t})_{t \in \T}$ at $\check{C}$ respectively.
We write $H^- <_{\check{C}} H^+$ if
\begin{enumerate}
\item $H^-_t < H^+_t$,
\item $\lambda_{H_t^-} < \lambda_{H_t^+}$, $m_{H_t^-} < m_{H_t^+}$
\item 
$m^+_t - m^-_t < (H^+_t - H^-_t)|_{M - (D \cup ([1,1+\epsilon/2] \times C)}$
\end{enumerate}
for all $t \in \T$ (See Figure \ref{fig:smallerhamiltonians}).
\end{defn}

\begin{center}
\begin{figure}[h]
	\begin{tikzpicture}
	
	\draw (-8.5,0.5) -- (-0.5,0.5);
	\draw [->](-8.5,0.5) -- (-8.5,4.5);
	\draw [](-6,0.6) -- (-6,0.4);
	\draw [](-3,0.6) -- (-3,0.4);
	\draw [](-4.5,0.6) -- (-4.5,0.4);
	\draw [dotted](-8.5,0.5) -- (-8.5,-2.5);
	\node at (-8.5,-3) {$r_C=0$};
	\node at (-4.5,0) {$r_C=1$};
	\node at (-6.5,0) {$r_C=1-\epsilon$};
	\node at (-2.5,0) {$r_C=1+\epsilon$};
	\draw [](-6,1) -- (-4,2);
	\draw [dashed](-6,1) -- (-8.5,-0.5);
	\draw [dashed](-8.5,0) -- (-6,2.5);
	\draw (-6,2.5) -- (-4,4.5);
	\draw (-4,4.5) .. controls (-3,5.5) and (-2,5) .. (-0.5,4.5);
	\draw (-6,2.5) .. controls (-7,1.5) and (-8,2) .. (-8.5,2.5);
	\draw (-6,1) .. controls (-6.5,0.5) and (-7.5,1.5) .. (-8.5,1);
	\draw (-4,2) .. controls (-3,2.5) and (-1.5,1.5) .. (-0.5,2.5);
\node at (-1,5.2) {$H_t^+$};
\node at (-1.1,2.7) {$H_t^-$};
	\draw [->](-9.1,5.2) -- (-5.6,3.3);
	\draw [->](-5.5,5.2) -- (-4.8,1.9);
	\node at (-7,5.5) {\underline{This} slope bigger than \underline{this} slope};
	
	\draw [<->](-8.8,0) -- (-8.8,-0.6);
	
	\draw [<->](-2.3,4.7) -- (-2.3,2.4);
	\draw [decoration={calligraphic brace,amplitude=7pt},decorate](-3,-0.5) -- (-6,-0.5);
	\node at (-4.5,-1) {$\check{C}$};
	\draw[decoration={calligraphic brace,amplitude=7pt},decorate] (-4.5,-1.5) -- (-8.5,-1.5);
	\node at (-6.5,-2) {$D$};
	\draw [<-](-9,-0.3) .. controls (-9.7,-1.2) and (-9.6,-2) .. (-7.2,-2.7);
	\draw [->](-2.8,-2.7) .. controls (0.1,1.4) and (1,3.5) .. (-2,3.9);
	\node at (-4.2,-3) {\underline{This} distance  smaller than \underline{this} distance};
	\node at (-10.1,-0.2) {$m^+_t - m^-_t$};
	\node at (-3.4,3.4) {$H^+_t - H^-_t$};
	\end{tikzpicture}
\caption{Picture of $H^-$ and $H^+$ satisfying $H^- <_{\check{C}} H^+$.} \label{fig:smallerhamiltonians}
\end{figure}
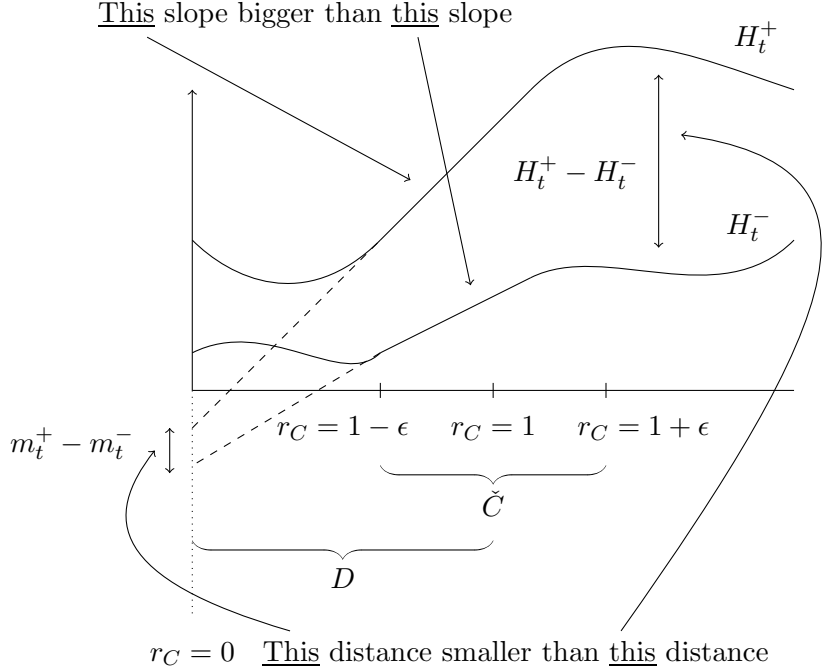
\end{center}

The following lemma below tells us that the relation $<_{\check{C}}$ has good properties.

\begin{lemma} \label{lemma partial order lemma}
Let $H \in \ccH^\T(\check{C},a_-,a_+)$.
Then there is a sequence
$(H_i)_{i \in \N}$ of elements in $\ccH^\T(\check{C},a_-,a_+)$ $C^\infty$ converging to $H$
so that
for each $K <_{\check{C}} H$, there exists
$i \in \N$ so that
$K <_{\check{C}} H_i <_{\check{C}} H$.
\end{lemma}
\proof
Suppose $H = (H_t)_{t \in \T}$ and let $\lambda_{H_t}$ and $m_{H_t}$ be the slope
and height of $H_t$ for each $t \in \T$.
Let $f : \R \lra{} \R$ be a smooth function satisfying
$$f|_{(-\infty,1]} = -2, \quad
f|_{[1+3\epsilon/4,\infty)} = -2-\epsilon, \quad
f' \leq 0$$
and
$f(x) = -1 - x$ for all $x \in [1 + \epsilon/8,1+\epsilon/2]$.
Define
$$K : M \lra{} \R, \quad
K(x) := \left\{
\begin{array}{ll}
-2 & \text{if} \ x \in D \\
f(r_C) & \text{if} \ x \in \check{C} \\
-2 - \epsilon & \text{otherwise.}
\end{array}
\right. 
$$
Then
$H_{i,t} := H_t + \frac{1}{i}K \in \ccH^\T(\check{C})$
has slope
$\lambda_{H_t} - \frac{1}{i}$
and height
$m_{H_{i,t}} := m_{H_t} - \frac{1}{i}$.
Also
$$m_{H_t} - m_{H_{i,t}} < \frac{2+\epsilon/2}{i} \leq (H_t - H_{i,t})|_{M - (D \cup ([1,1+\epsilon/2] \times C))}.$$
Therefore $H_i := (H_{i,t})_{t \in \T}$
has the properties we want for all $i$ large enough.
\qed

\begin{defn} \label{definition chain level continuation map}
Let $H^-,H^+ \in \ccH^\T(\check{C},a_-,a_+)$
satisfy $H^- <_{\check{C}} H^+$.
A smooth family of Hamiltonians
$H^{-+} = (H^{-+}_{s,t})_{(s,t) \in \R \times \T}$ {\it $(a_-,a_+)$-connects}
$H^-$ with $H^+$ if
\begin{itemize}
\item $H^{-+}_s := (H^{-+}_{s,t})_{t \in \T}$
is an element of $\ccH^\T(\check{C},a_-,a_+)$ for all $s \in \R$,
\item $H^{-+}_s = H^\pm$ for all $s \in \R$ satisfying $\mp s \gg 1$
\item and if $(\lambda_{H^{-+}_{s,t}})_{(s,t) \in \R \times \T}$,
$(m_{H^{-+}_{s,t}})_{(s,t) \in \R \times \T}$
are the corresponding slopes and heights of
$(H^{-+}_{s,t})_{(s,t) \in \R \times \T}$
along $\check{C}$ then
$$\frac{d}{ds}H^{-+}_{s,t}(x) \leq 0,
\ \forall \ x \in M, \quad
\frac{d}{ds}\lambda_{H^{-+}_{s,t}} \leq 0, \quad \frac{d}{ds}m_{H^{-+}_{s,t}} \leq 0,$$
$$\frac{d}{ds}H^{-+}_{s,t}(x)
\leq \frac{d}{ds}(m_{H^{-+}_{s,t}}), \ \forall \ x \in M- (D \cup ([1,1+\epsilon/2] \times C))
$$
for all $t \in \T$.
We will denote the space of such families of Hamiltonians by $$\ccH^{\R \times \T}(\check{C},a_-,a_+,H^-,H^+).$$
\end{itemize}
\end{defn}

Note that $\ccH^{\R \times \T}(\check{C},a_-,a_+,H^-,H^+)$
is contractible since it is a convex
subset of the space $\ccH^{\R \times \T}(\check{C},a_-,a_+)$ with at least one element equal to:
$$((1-\rho(s))H^+_t + \rho(s)H^-_t)_{(s,t) \in \R \times \T}$$
where $\rho : \R \lra{} \R$ is a smooth function satisfying $\rho' \geq 0$ and
where $\rho(s)$ is
equal to $0$ for $s \ll 0$ and $1$ for $s \gg 0$.

\begin{defn} \label{definition connecting pairs}
%
If $H^\pm \in \ccH^\T(\check{C},a_-,a_+)$
and $J^\pm \in \ccJ^\T(\check{C})$
(Definition \ref{definition wspaces of compatible objects})
then we say
$(H^{-+},J^{-+})$ {\it $(a_-,a_+)$-connects}
$(H^-,J^-)$ with $(H^+,J^+)$
if $H^{-+}$ $(a_-,a_+)$-connects $H^-$ with $H^+$
and if $J^{-+} \in \ccJ^{\R \times \T}((J^-,J^+),\check{C})$.
\end{defn}

\begin{defn} \label{definition connecting families of almost complex structures}
Let $H^{\pm} \in \ccH^\reg(\check{C},a_-,a_+,P)$
(Definition \ref{equation high enough hamiltonians}) where
$P = \{p-1,p,p+1\}$ for some $p \in \Z$ and let
$J^\pm \in
 \ccJ^{\T,\reg}(H^\pm,\check{C})$
where
$\ccJ^{\T,\reg}(H^\pm,\check{C})$
is given in Definition \ref{definition space of Floer cylinders}.
Let
$H^{-+} \in
\ccH^{\R \times \T}(\check{C},a_-,a_+,H^-,H^+)$
and
$J^{-+} \in 
\ccJ^{\R \times \T,\reg}(H^{-+},(J^-,J^+),\check{C})$ (Proposition \ref{proposition regular subset for surface}). 
The family $(H^{-+},J^{-+})$ then defines for us  a {\it continuation map}
$$\Phi^p_{H^{-+},J^{-+}} : HF^p_{\check{C},a_-,a_+}(H^-) \lra{}
HF^p_{\check{C},a_-,a_+}(H^+)$$
induced from the
{\it chain level continuation map}
\begin{equation} \label{equation continuation chain level map}
\widetilde{\Phi}^p_{H^{-+},J^{-+}} : CF^p_{\check{C},a_-,a_+}(H^-) \lra{}
CF^p_{\check{C},a_-,a_+}(H^+)
\end{equation}
which is the unique $\K$-linear map
satisfying
\begin{equation} \label{equation continuation chain level}
\widetilde{\Phi}^p_{H^{-+},J^{-+}}(\gamma_+)
:= \sum_{\gamma_- \in \Gamma^p_{\check{C},a_-,a_+}(H^+)} \# \ccM(H^{-+},J^{-+},\gamma) \gamma_-, \quad 
\forall \ \gamma_+ \in \Gamma^p_{\check{C},a_-,a_+}(H^-)
\end{equation}
where $\gamma = (\gamma_-,\gamma_+)$ and
$\# \ccM(H^{-+},J^{-+},\gamma)$
is the number of elements in the $0$-dimensional manifold 
$\ccM(H^{-+},J^{-+},\gamma)$
(Definition \ref{definition riemann surface admissible}) counted with sign according to orientation.
The set $\Gamma^p_{\check{C},a_-,a_+}(H^-)$
is defined in Definition \ref{defn chain complex}.
\end{defn}

The map $\Phi^p_{H^{-+},J^{-+}}$
is well defined since the
chain level map (\ref{equation continuation chain level map})
induced by Equation (\ref{equation continuation chain level})
commutes with the differentials
$\partial^{(p_\pm)}_{H^\pm,J^\pm}$, $p_- = p-1$, $p_+ = p$
by
the gluing theorem
\cite[Theorem 4.4.1]{schwarz1995cohomology}
combined with the compactness result
Proposition \ref{proposition compactness result}
and orientation conventions
\cite[Section 17]{ritter2013topological}.
The same gluing theorem also tells us that the composition of two continuation maps
is a continuation map.
Finally, if we have other elements
$K^{-+} \in
\ccH^{\R \times \T}(\check{C},a_-,a_+,H^-,H^+)
$
and
$Y^{-+} \in
\ccJ^{\R \times \T,\reg}(K^{-+},(J^-,J^+),\check{C})$
then $\Phi^p_{H^{-+},J^{-+}} = \Phi^p_{K^{-+},Y^{-+}}$.
Such an equivalence of maps is given by a chain homotopy
defined in a similar way
to Equation (\ref{equation continuation chain level map}),
except that we
count elements of
$\ccM(F,L,\gamma)$
where $\Sigma_\bullet$ is the smooth family of Riemann surfaces $(\R \times \T)_{\sigma \in [0,1]}$,
$$F = (F_{\sigma,s,t})_{(\sigma,s,t) \in [0,1] \times \R \times \T} \in \ccH^{\Sigma_\bullet}((H^-,H^+),\check{C})$$ satisfies
$F_{0,s,t} = H^{-+}_{s,t}$,
$F_{1,s,t} = K^{-+}_{s,t}$ and
$$(F_{\sigma,s,t})_{(s,t) \in \R \times \T} \in \ccH^{\R \times \T}(\check{C},a_-,a_+,H^-,H^+) \quad \forall \ \sigma \in [0,1]$$
and where
$$L \in
\ccJ^{\Sigma_\bullet,\reg}(F,(J^{-+},Y^{-+}),(J^-,J^+),\check{C})$$
(See Proposition \ref{proposition parameterized regular subset for surface}).
This map is a chain homotopy by \cite[Proposition 11.2.8]{audin2014morse}
where we replace the compactness result
\cite[Theorem 11.3.7]{audin2014morse}
by Proposition \ref{proposition compactness result}.
As a result of the facts above we have the following definition.
\begin{defn} \label{definition continuation map}
Let
$H^\pm \in \ccH^\reg(\check{C},a_-,a_+,P)$
(Definition \ref{defn chain complex}) satisfy $H^- <_{\check{C}} H^+$ where $P = \{p-1,p,p+1\}$ for some $p \in \Z$.
The {\it continuation map}
$$\Phi^p_{H^-,H^+} : HF^p_{\check{C},a_-,a_+}(H^-) \lra{}
HF^p_{\check{C},a_-,a_+}(H^+)$$
is defined to be $\Phi^p_{H^{-+},J^{-+}}$
for some choice of $(H^{-+},J^{-+})$
as in Definition \ref{definition chain level continuation map}.
\end{defn}


This is a $\Lambda_\K^{Q_+,+}$-module map.
We also have the following important lemmas giving us sufficient conditions ensuring that a continuation map is an isomorphism.

\begin{lemma} \label{lemma chain level continuation map filtration isomorphism}
Let $H \in \ccH^\reg(\check{C},a_-,a_+,P)$
where $P = \{p-1,p,p+1\}$
for some $p \in \Z$,
and let
$J \in \ccJ^\T(\check{C})$.
Then there is a convex neighborhood
$U_H \subset \ccH^\T(\check{C},a_-,a_+)$
of $H$
(Definition \ref{defn chain complex})
and a weakly contractible neighborhood
$V_J \subset \ccJ^\T(\check{C})$ of $J$
so that for all
$H^\pm \in U_H \cap \ccH^\reg(\check{C},a_-,a_+,P), \
J^\pm \in V_J \cap \ccJ^{\T,\reg}(H^\pm,\check{C}),$
\begin{equation} \label{equation H homotopy}
H^{-+}
= (H^{-+}_{s,t})_{(s,t) \in \R \times \T}
\in \ccH^{\R \times \T}(\check{C},a_-,a_+,H^-,H^+),
\end{equation}
\begin{equation} \label{equation J homotopy}
J^{-+} = (J^{-+}_{s,t})_{(s,t) \in \R \times \T} \in \ccJ^{\R \times \T,\reg}(H^{-+},(J^-,J^+),\check{C})
\end{equation}
(Proposition \ref{proposition regular subset for surface})
satisfying
$(H^{-+}_{s,t})_{t \in \T} \in U_H$
and
$(V^{-+}_{s,t})_{t \in \T} \in V_J$
for all $s \in \R$,
the degree $p$ chain level continuation map
$\widetilde{\Phi}^p_{H^{-+},J^{-+}}$
from Equation (\ref{equation continuation chain level map})
is an isomorphism.
\end{lemma}
\proof
The key idea of the proof of this lemma is to use a Gromov compactness result to show that
the low energy
solutions of the Floer equation
defining the chain level continuation map
induce an isomorphism.
For each $K \in \ccH^\reg(\check{C},a_-,a_+,P)$ and $t \in \T$,
let
$\text{ev}_t : \Gamma^P_{\check{C},a_-,a_+}(K) \lra{} M$
be the map sending a capped
$1$-periodic orbit $\gamma = (\widetilde{\gamma},\check{\gamma})$
to $\widetilde{\gamma}(\check{\gamma}(t))$
and define $\Gamma_K := \text{ev}_0(\Gamma^P_{\check{C},a_-,a_+}(K))$.
Define
$\Gamma := \Gamma_H$. 
Since each element of
$\Gamma^P_{\check{C},a_-,a_+}(H)$
is non-degenerate,
we have that $\Gamma$ is a finite subset of $M$.
Therefore we can find open subsets
$N'_\Gamma \subset N_\Gamma \subset M$
of $\Gamma$ so that
the inclusion maps
$\Gamma \hookrightarrow N'_{\Gamma}$, $N'_\Gamma \hookrightarrow N_\Gamma$
are homotopy equivalences and so that
the closure of $N'_{\Gamma}$ is contained in $N_\Gamma$.
Define
$\ccH^{\T,\Gamma} \subset \ccH^\reg(\check{C},a_-,a_+,P)$
to be the subspace of Hamiltonians
$K$ satisfying
$\phi^K_t(\Gamma_K) \subset \phi^H_t(N'_\Gamma)$ for all $t \in \T$
and so that the inclusion map
$\Gamma_K \hookrightarrow N'_\Gamma$ is a homotopy equivalence.

For each manifold $\Sigma$ and each subset $U_0 \subset \ccH^\T(\check{C})$
and
$V_0 \subset \ccJ^{\R \times \T}(\check{C})$,
define
$$\ccH^{\Sigma \times \T}(\check{C})|_{U_0} :=
\{(K_{\sigma,t})_{(\sigma,t) \in \Sigma \times \T} \in \ccH^{\Sigma \times \T}(\check{C}) \ : \ (K_{\sigma,t})_{t \in \T} \in U_0 \ \forall \ \sigma \in \Sigma \},$$
$$\ccJ^{\Sigma \times \R \times \T}(\check{C})|_{V_0} :=
\{(K_{\sigma,s,t})_{(\sigma,s,t) \in \Sigma  \times \R \times \T} \in \ccJ^{\Sigma \times \R \times \T}(\check{C}) \ : \ (K_{\sigma,s,t})_{(s,t) \in \R \times \T} \in V_0 \ \forall \ \sigma \in \Sigma \}$$
(See Definition \ref{definition wspaces of compatible objects}).
For each  $U_0 \subset \ccH^\T(\check{C})$,
define
$\ccJ^{\R \times \T,\Gamma}(U_0) \subset \ccJ^{\R \times \T}(\check{C})$
be the subspace of elements
$Y = (Y_{s,t})_{(s,t) \in \R \times \T}$
so that for each
\begin{itemize}
	\item $K^\pm \in \ccH^{\T,\Gamma} \cap U_0$,
	$\gamma^\pm \in \Gamma^P_{\check{C},a_-,a_+}(K^\pm)$,
	\item $K \in
	\ccH^{\R \times \T}(\check{C})|_{\ccH^{\T,\Gamma} \cap U_0} \cap
	\ccH^{\R \times \T}((K^-,K^+),\check{C})$ 
	and
	\item $u \in \ccM(K,Y,(\gamma_-,\gamma_+))$
	(Definition \ref{definition riemann surface admissible})
	satisfying $u_{s,t} \in \phi^H_t(N_\Gamma)$ for all $(s,t) \in \R \times \T$,
\end{itemize}
we have that
$u(s,t) \in \phi^H_t(N'_\Gamma)$ for all $(s,t) \in \R \times \T$.
Morally, the definition above is used to find those almost complex structures $Y$ for which we can select an isolated low energy region of
$\ccM(K,Y,(\gamma_-,\gamma_+))$.

Since
there are neighborhoods $N_-$, $N_+$ of $a_-$, $a_+$ so that
$\Gamma^P_{\check{C},a'_-,a'_+}(H) = \Gamma^P_{\check{C},a_-,a_+}(H)$ for all
$a'_\pm \in N_\pm$,
we have by a compactness argument
(such as the one in \cite[Lemma 2.3]{McLean:local})
that there exists a weakly contractible
open neighborhood
$U \subset \ccH^\T(\check{C},a_-,a_+)$
of $H$
satisfying
$U \subset \ccH^{\T,\Gamma}$.
Let
 $$
 \iota_{\textnormal{cpx}} : \ccJ^\T(\check{C}) \lhook\joinrel\lra{} \ccJ^{\T}(\check{C}), \quad \iota_{\textnormal{cpx}}((\check{J}_t)_{t \in \T}) := (\check{J}_t)_{(s,t) \in \R \times \T}
 $$
 be the natural inclusion map.
By a Gromov compactness argument
(such as the one in \cite[Lemma 2.3]{McLean:local})
we have, after shrinking $U$ further, that
there is a weakly contractible
open subset
$V \subset \ccJ^{\R \times \T}(\check{C})$ of $\iota_{\textnormal{cpx}}(J)$
so that
$V \subset \ccJ^{\R \times \T,\Gamma}(U_0)$.

As a result,
for each $K \in U$,
we can define the Hamiltonian Floer cohomology group
$HF^p(K|_{N_\Gamma})$
in the usual way except that
we only consider
capped $1$-periodic orbits
inside
$\cap_{t \in \T} \text{ev}_t^{-1}(\phi^H_t(N_\Gamma))$,
almost complex structures
$(J_t)_{t \in \T}$
satisfying
$\iota_{\textnormal{cpx}}((J_t)_{t \in \T}) \in V$,
and Floer trajectories
$u : \R \times \T \lra{} M$ satisfying
$u(s,t) \in \phi^H_t(N_\Gamma)$ for all $(s,t) \in \R \times \T$.
For topological reasons
we get that the differential vanishes and hence
$HF^p(K|_{N_\Gamma}) = CF^p(K|_{N_\Gamma})$ where
$CF^p(K|_{N_\Gamma})$ is the degree $p$ part of the chain complex defining $HF^p(K|_{N_\Gamma})$.

Also by only considering
families of Hamiltonians inside
$\ccH^{\R \times \T}(\check{C})|_U$,
$\ccH^{[0,1] \times \R \times \T}(\check{C})|_U$, families
of almost complex structures inside
$V$ and
$\ccJ^{[0,1] \times \R \times \T}(\check{C})|_V$
and Floer trajectories $u : \R \times \T \lra{} M$ satisfying
$u(s,t) \in \phi^H_t(N_\Gamma)$ for all $(s,t) \in \R \times \T$, one can show
that
the chain level continuation map
$\widetilde{\Phi}^p_{K^{-+},Y^{-+}} : CF^p(K|_{N_\Gamma}) \lra{} CF^p(\check{K}|_{N_\Gamma})$
is an isomorphism for each
$K,\check{K} \in U$,
$K^{-+} \in \ccH^{\R \times \T}(\check{C})|_U \cap 	\ccH^{\R \times \T}((K^-,K^+),\check{C})$
and $Y^{-+} \in \ccJ^{\R \times \T}(\check{C})|_V \cap \ccJ^{\R \times \T,\reg}(K,(Y^-,Y^+),\check{C})$
(see \cite[Chapter 11]{audin2014morse}).
The above arguments work because Gromov compactness \cite[Theorem 9.1.7]{audin2014morse}
still holds since
there are no Floer trajectories $u$
as above satisfying $u(s,t) \in N_\Gamma - N'_\Gamma$ for some $(s,t) \in \R \times \T$ (in other words we have selected appropriate clopen regions inside the moduli spaces of Floer cylinders needed to construct $\widetilde{\Phi}^p_{K^{-+},Y^{-+}}$
and show it is an isomorphism).

Hence if we choose $U_H \subset U$
to be a convex neighborhood of $H$
and $V_J \subset V
 \cap \ccJ^{\T}(\check{C})$
a weakly contractible neighborhood of $J$
then our lemma holds by an action filtration argument.
\qed

Very roughly, the following lemma
says that if one has a family of Hamiltonians whose Floer chain complexes
are all the same (after possibly pulling back by a diffeomorphism),
then the corresponding
continuation map is an isomorphism.

\begin{lemma} \label{lemma continuation map filtration isomorphism}
Let $P = [p_- - 1,p_+ +1]$ for some $p_\pm \in \Z$ and let $(a^m_-,a^m_+) \in \Sc(Q^m_-) \times \Sc(Q^m_+)$, $m \in I$ be a finite collection of $\check{C}$-action intervals.
Suppose there exists
$(H_{s,t})_{(s,t) \in \R \times \T} \in
\ccH^{[0,1] \times \T}(\check{C})$
so that
$H_{s,\bullet} := (H_{s,t})_{t \in \T} \in  \cap_{m \in I} \ccH^\reg(\check{C},a^m_-,a^m_+,P)$
(Definition \ref{equation high enough hamiltonians}) for all $s \in [0,1]$ and
$H_{s,\bullet} <_{\check{C}} H_{\check{s},\bullet}$ for each $s < \check{s}$.
Let $J^\pm \in  \ccJ^{\T,\reg}(H^\pm,\check{C})$ where
$H^- := H_{0,\bullet}$ and $H^+ := H_{1,\bullet}$.
Then there exists
\begin{equation} \label{equation path of H and J joining endpoints}
{H}^{-+}
\in \cap_{m \in I} \ccH^{\R \times \T}(\check{C},a^m_-,a^m_+,H^-,H^+), \quad
{J}^{-+} \in \ccJ^{\R \times \T,\reg}(H^{-+},(J^-,J^+),\check{C})
\end{equation}
so that the chain level continuation map
$$\widetilde{\Phi}^p_{{H}^{-+},{J}^{-+}} : CF^p_{\check{C},a^m_-,a^m_+}(H^-) \lra{} CF^p_{\check{C},a^m_-,a^m_+}(H^+)$$
is an isomorphism for each $p_- \leq p \leq p_+$ and each $m \in I$.
In particular the continuation
map $\Phi^p_{H^-,H^+} : HF^p_{\check{C},a^m_-,a^m_+}(H^-) \lra{} HF^p_{\check{C},a^m_-,a^m_+}(H^+)$ is an isomorphism for each such $p$ and $m$.
\end{lemma}
\proof
Here the key idea of the proof is to chop up the homotopy $(H_{s,t})_{(s,t) \in [0,1] \times \T}$ into many small homotopies and then apply Lemma \ref{lemma chain level continuation map filtration isomorphism} above.
Let $J \in \ccJ^\T(\check{C})$.
Let
$U^m_{H_{s,\bullet}} \subset \ccH^\T(\check{C},a^m_-,a^m_+)$
and
$V^m_{H_{s,\bullet}} \subset \ccJ^\T(\check{C},a^m_-,a^m_+)$
be open neighborhoods of $H_{s,\bullet}$
and $J$ respectively
so that the conclusion of
Lemma 
\ref{lemma chain level continuation map filtration isomorphism}
holds with $H$ and $(a_-,a_+)$ replaced
by $H_{s,\bullet}$ and $(a^m_-,a^m_+)$ respectively for each $m \in I$.
Let $U_{H_{s,\bullet}} \subset \cap_{m \in I} U^m_{H_{s,\bullet}}$
be a convex neighborhood of $H_{s,\bullet}$
and $V_{H_{s,\bullet}} \subset \cap_{m \in I} V^m_{H_{s,\bullet}}$
a weakly contractible open neighborhood of $J$ for each $s \in [0,1]$.
Since $[0,1]$ is compact,
we can choose
$s_0 = 0 < s_1 < s_2 \cdots < s_k = 1$
so that
$H_{s_{j+1},\bullet}  \in U_{H_{s_j,\bullet}}$
or
$H_{s_j,\bullet}  \in U_{H_{s_{j+1},\bullet}}$
for all $0 \leq j \leq k-1$.
Since
$H^- <_{\check{C}} H^+$
we have,
by applying
Lemma \ref{lemma partial order lemma}
inductively,
that there exists
$\check{H}_j \in U_{H_{s_j,\bullet}} \cap \cap_{i \in I } \ccH^\reg(\check{C},a^m_-,a^m_+,P)$,
$0 \leq j \leq k$
satisfying
$\check{H}_j <_{\check{C}} \check{H}_{j+1}$ and
$\check{H}_j, \check{H}_{j+1} \in U_{H_{s_j,\bullet}}$
or
$\check{H}_j, \check{H}_{j+1} \in U_{H_{s_{j+1},\bullet}}$
for all $0 \leq j \leq k-1$ and satisfying
$\check{H}_0 = H_{0,\bullet}$ and
$\check{H}_k = H_{1,\bullet}$.
Now choose
$\widetilde{J} \in \cap_{j=1}^k (V_{H_{s_j,\bullet}}
\cap \ccJ^{\T,\reg}(\check{H}_j,\check{C}))$
%
%
and choose
$$
H^{-+,j}
= (H^{-+,j}_{s,t})_{(s,t) \in \R \times \T}
\in \cap_{m \in I} \ccH^{\R \times \T}(\check{C},a^m_-,a^m_+,\check{H}_j,\check{H}_{j+1}),
$$
$$
J^{-+,j} = (J^{-+,j}_{s,t})_{(s,t) \in \R \times \T} \in \ccJ^{\R \times \T,\reg}(H^{-+},(\widetilde{J},\widetilde{J}),\check{C})
$$
satisfying
$(H^{-+,j}_{s,t})_{t \in \T} \in U_{H_{s_{j'},\bullet}}$
and
$(V^{-+,j}_{s,t})_{t \in \T} \in V_{H_{s_{j'},\bullet}}$
for all $s \in \R$,
for some $j' \in \{j,j+1\}$
(independent of $s$)
and for all
$0 \leq j \leq k-1$.
Then by Lemma
\ref{lemma chain level continuation map filtration isomorphism},
we have that
the degree $p$ chain level continuation map
$\widetilde{\Phi}^p_{H^{-+,j},J^{-+,j}}$
from Equation (\ref{equation continuation chain level map})
is an isomorphism for all
$0 \leq j \leq k-1$.
By a repeated
gluing argument
\cite[Theorem 11.1.16]{audin2014morse}
applied to
$(H^{-+,j},J^{-+,j})$ for all $0 \leq j \leq k-1$,
there exists
$(H^{-+},J^{-+})$
as in Equation (\ref{equation path of H and J joining endpoints})
so that so that the chain level continuation map
$$\widetilde{\Phi}^p_{H^{-+},J^{-+}} : CF^p_{\check{C},a^m_-,a^m_+}(H^-) \lra{} CF^p_{\check{C},a^m_-,a^m_+}(H^+)$$
is an isomorphism for each $p_- \leq p \leq p_+$ and $m \in I$.
\qed

\subsection{Action Maps.} \label{subsection Action maps and Cone isomorphisms}

Throughout this subsection,
we will fix a (possibly empty) contact
cylinder $\check{C} = [1-\epsilon,1+\epsilon] \times C \subset M$ with associated Liouville domain $D$.

\begin{defn} \label{definition compatible cones}
Let $(a^j_-,a^j_+) \in \Sc(Q^j_-) \times \Sc(Q^j_+)$ be a $\check{C}$-action interval for $j=0,1$.
We say that
{\it 
$(a^1_-,a^1_+)$
is smaller than
$(a^0_-,a^0_+)$} if
$Q^1_\pm \subset Q^0_\pm$
and if $a^1_\pm \leq a^0_\pm|_{Q^1_\pm}$.
\end{defn}

\begin{remark} \label{remark novikov rings}
If $(a^1_-,a^1_+)$ is smaller
than $(a^0_-,a^0_+)$ then we have induced morphisms of Novikov rings
$\Lambda_\K^{Q^0_+} \lra{} \Lambda_\K^{Q^1_+}$
and
$\Lambda_\K^{Q^0_+,+} \lra{} \Lambda_\K^{Q^1_+,+}$.
In particular,
any $\Lambda_\K^{Q^1_+}$ (resp. $\Lambda_\K^{Q^1_+,+}$)
module is naturally a
$\Lambda_\K^{Q^0_+}$ (resp. $\Lambda_\K^{Q^0_+,+}$)
module.
\end{remark}

\begin{defn} \label{definition action map}

Let $P = \{p-1,p,p+1\}$ for some $p \in \Z$.
Let
$(a^j_-,a^j_+)$ be a $\check{C}$-action interval
for $j=0,1$
so that
$(a^1_-,a^1_+)$
is smaller than $(a^0_-,a^0_+)$.
Let $H \in \cap_{j=0,1} \ccH^\reg(\check{C},a^j_-,a^j_+,P)$.
Then we have a natural chain map
\begin{equation} \label{equation chain version of action map}
CF^q_{\check{C},a^0_-,a^0_+}(H) \lra{\alpha}
CF^q_{\check{C},a^0_-,\infty}(H)/
(CF^q_{\check{C},a^0_-,\infty}(H) \cap CF^q_{\check{C},a^1_+,\infty}(H)) \lra{\beta}
CF^q_{\check{C},a^1_-,a^1_+}(H)
\end{equation}
for each $q \in P$
called a {\it chain level action map}
where $\alpha$ is the natural
quotient map and $\beta$ is
the natural inclusion map of $\K$-modules.
The induced map on homology
\begin{equation} \label{equation action map}
HF^p_{\check{C},a^0_-,a^0_+}(H) \lra{}
HF^p_{\check{C},a^1_-,a^1_+}(H)
\end{equation}
is called an {\it action map}.
\end{defn}

\begin{remark} \label{remark action map properties}
Action maps are morphisms of
$\Lambda_\K^{Q^0_+,+}$-modules
where $Q^0_+$ is the domain of $a^0_+$ (see Remark \ref{remark novikov rings}).
The composition of two action maps is an action map.
Action maps commute with continuation maps by Lemma \ref{lemma filtration}.
\end{remark}

We also have the following important lemma
giving us a sufficient condition for an action map to be an isomorphism.

\begin{lemma} \label{lemma action morphism is isomorphism}
Let $Q_j$, $(a^j_-,a^j_+)$, $j=0,1$, $P$, $p$, $H$ be as in Definition \ref{definition action map}.
Suppose that $\Gamma^P_{\check{C},a^0_-,a^0_+}(H)
= \Gamma^P_{\check{C},a^1_-,a^1_+}(H)$(Definition \ref{defn chain complex}).
Then the action map
$$
HF^p_{\check{C},a^0_-,a^0_+}(H) \lra{}
HF^p_{\check{C},a^1_-,a^1_+}(H)
$$
is an isomorphism.
\end{lemma}
\proof
This follows from the fact that the chain maps
$\alpha$ and $\beta$ from Equation
(\ref{equation chain version of action map})
are isomorphisms in degree $p$.
\qed

\subsection{Invariance Under Time Reparameterization.}

Throughout this subsection, $\check{C}$
is a contact cylinder whose associated Liouville domain is $D$.
We will also fix a $\check{C}$-interval domain $(Q_-,Q_+)$.

\begin{defn}
Let $F : \T \lra{} \T$
be a smooth non-decreasing map.
Let $H = (H_t)_{t \in \T}$ be a smooth Hamiltonian.
We define
$H^F = (H^F_t)_{t \in \T}$
by
$H^F_t := F'(t) H_{F(t)}$ for each $t \in \T$.
\end{defn}

\begin{prop} \label{proposition reparameterization isomorphism}
Let $F : \T \lra{} \T$ be a smooth non-decreasing map which is homotopic to the identity map.
Let $(a_-,a_+) \in \Sc(Q_-) \times \Sc(Q_+)$
be a $\check{C}$-action interval
and $P := \{p-1,p,p+1\}$ for some $p \in \Z$.
Then for each $H \in \ccH^\reg(\check{C},a_-,a_+,P)$
there is an isomorphism
\begin{equation} \label{equation reparameterization isomorphism}
HF^*_{\check{C},a_-,a_+}(H)
\cong HF^*_{\check{C},a_-,a_+}(H^F)
\end{equation}
which commutes with continuation maps and action maps.
\end{prop}

\begin{defn} \label{definition reparameterization isomorphism}
The isomorphism (\ref{equation reparameterization isomorphism})
will be called a {\it reparameterization isomorphism}.
\end{defn}

\begin{proof}[Proof of Proposition \ref{proposition reparameterization isomorphism}.]
The correspondence
sending a loop
$\gamma : \T \lra{} M$ to the loop $\gamma \circ F$
induces a bijection between $1$-periodic orbits of $H$
and $H^F$ respectively.
Since $F$ is isotopic through non-decreasing maps to the identity map,
we have that
this correspondence lifts in a natural way to capped $1$-periodic orbits.
Such a bijection induces an isomorphism of modules
$CF^p_{\check{C},a_-,a_+}(H)
\cong CF^p_{\check{C},a_-,a_+}(H^F)$ for each $p \in P$.
Since $F$ is homotopic to the identity map,
there is a unique function
$G : \T \lra{} \R$
satisfying
$G(0) = 0$
and $G'(t) = F'(t) - 1$.
Let $J = (J_t)_{t \in \T} \in \ccJ^\reg(H,\check{C})$ and define $J^F := (J_{F(t)})_{t \in \T}$.
Then there is a bijection between
$(H,J)$-Floer cylinders
$u : \R \times \T \lra{} M$
and
$(H^F,J^F)$-Floer cylinders
$$u^F : \R \times \T \lra{} M, \quad u^F(s,t) := u(s + G(t),F(t)).$$
Also
$J^F \in \ccJ^\reg(H^F,\check{C})$
since the corresponding
Fredholm operators 
linearizing the Floer equation
are canonically identified via a similar correspondence.
Putting everything together we get our isomorphism
(\ref{equation reparameterization isomorphism}).
By considering a similar correspondence for Floer trajectories defining continuation maps
as in Definition \ref{definition continuation map},
we see that such an isomorphism commutes with continuation maps.
Finally, since these isomorphisms preserve action
we get that they also respect action maps.
\end{proof}

\subsection{The pair of pants product}

Throughout this subsection
we will fix a (possibly empty) contact
cylinder $\check{C} = [1-\epsilon,1+\epsilon] \times C \subset M$ with associated Liouville domain $D$.

\begin{defn} \label{definition contact cylinders smaller than}
Let $(a^j_-,a^j_+) \in \Sc(Q^j_-,Q^j_+)$ be a $\check{C}$-action interval for $j=0,1,2$.
We say that
{\it $(a^2_-,a^2_+)$
is smaller than
$(a^j_-,a^j_+)_{j=0,1}$}
if $Q^2_\pm \subset Q^j_\pm$ for $j=0,1$ and
$$a^2_- \leq a^0_-|_{Q^2_-} + a^1_-|_{Q^2_-}, \quad
a^2_+ \leq \min(a^0_-|_{Q^2_+} + a^1_+|_{Q^2_+}, a^1_-|_{Q^2_+} + a^0_+|_{Q^2_+}).
$$
%
\end{defn}

Most of the time, it will be the case that
$$Q_0 = Q_1 = Q_2,
\ \
a_-^2 = a_-^0 + a_-^1 \ \ \text{and} \ \
a_+^2 = \min(a^0_- + a^1_+,a^0_+ + a^1_-).$$

\begin{defn} \label{definition catenation of hamitlonians}
For any time dependent Hamiltonian $H = (H_t)_{t \in \T}$ and any $k \in \N$,
define $(k)H := ((k)H_t)_{t \in \T}$
where $(k)H_t := kH_{kt}$ for all $t \in \T$.
\end{defn}

\begin{defn} \label{defnition pair of pants product}
Let $(a^j_-,a^j_+)$ be a $\check{C}$-action interval for $j=0,1,2$
so that $(a^2_-,a^2_+)$
is smaller than $(a^j_-,a^j_+)_{j=0,1}$.
Define $\kappa_0 = \kappa_1 := 1$
and $\kappa_2 := 2$.
Let
$H^j = (H^j_t)_{t \in \T}$, $j=0,1,2$ be Hamiltonians so that
$(\kappa_j) H^j \in \cap_{k=0}^2 \ccH^\reg(\check{C},a^k_-,a^k_+)$ for $j=0,1,2$.
Suppose that
for each
triple $(\gamma^j)_{j=0,1,2} \in \prod_{j=0,1,2} \Gamma^{p_j}_{\check{C},a^j_-,a^j_+}(\kappa_j H^j)$ of capped $1$-periodic orbits where $p_0,p_1,p_2 \in \Z$,
we have that the associated $1$-periodic
orbits of at least two of them have distinct images in $M$.

If $H^j <_{\check{C}} H^2$ for $j=0,1$,
we have a natural {\it pair of pants product map}
$$
\Phi_{H^0,H^1,H^2} : HF^{p_0}_{\check{C},a^0_-,a^0_+}(H^0) \otimes_{\K} HF^{p_1}_{\check{C},a^1_-,a^1_+}(H^1)
\lra{}
HF^{p_0+p_1}_{\check{C},a^2_-,a^2_+}((2)H^2)
$$
induced from a chain level map
\begin{equation} \label{equation chain level product}
\widetilde{\Phi}_{H,J} :
CF^{p_0}_{\check{C},a^0_-,a^0_+}(H^0) \otimes_{\K} CF^{p_1}_{\check{C},a^1_-,a^1_+}(H^1)
\lra{}
CF^{p_0+p_1}_{\check{C},a^2_-,a^2_+}((2)H^2)
\end{equation}
which is the unique $\K$-linear map satisfying
$$
\widetilde{\Phi}_{H,J}(\gamma_0 \otimes \gamma_1) =
\sum_{\gamma_2 \in \Gamma^{p_0+p_1}_{\check{C},a^2_-,a^2_+}(2H)} \# \ccM(H,J,\gamma) \gamma_2, \quad \ \gamma_j \in \Gamma^{p_j}_{\check{C},a_-,a_+}((\kappa_j) H_j), \ j=0,1
$$
for some fixed $(H,J)$ where
\begin{itemize}
\item
$H \in  \ccH^\Sigma((H^0,H^1,H^2),\check{C})$ 
 where $\Sigma$ is the pair of pants
as in Example \ref{example pair of pants},
\item
$J \in \cap_{k=0}^2 \ccJ^{\Sigma,\reg}(H,(J^0,J^1,J^2)) \cap \ccJ^\Sigma(\check{C})$
where
$J^j \in \ccJ^{\T,\reg}(\kappa_jH^j,\check{C})$ for all $j =0,1,2$
and
\item
$\gamma = (\gamma^j)_{j=0,1,2}$.
\end{itemize}

\end{defn}

\begin{remark} \label{remark pants product is graded commutative and associative and unital}
The condition
$H^j <_{\check{C}} H^2$ for $j=0,1$
ensures that the Hamiltonian $H$ exists
by considering the branched cover
$\varpi$ from Example \ref{example pair of pants}.
The pair of pants product is well defined
since the chain level map (\ref{equation chain level product})
commutes with the differentials
by a gluing argument
\cite[Theorem 4.4.1]{schwarz1995cohomology}
combined with the compactness result
Proposition \ref{proposition compactness result}
and orientation conventions
\cite[Section 17]{ritter2013topological}.
Since the space of families of Hamiltonians
$H$ as above is contractible
and since the space
$\ccJ^{\Sigma}((J^0,J^1,J^2),\check{C})$
is contractible,
we have
(by looking at the moduli spaces in Definition
\ref{definition parameterized moduli spaces}
and Proposition \ref{proposition parameterized regular subset for surface})
that
$\Phi_{H^0,H^1,H^2}$ does not depend on
$\Sigma$, $H$ or $J$ or the choice of
$\Sigma$-compatible $1$-form.
Such a fact follows from
\cite[Section 5.2]{schwarz1995cohomology}
where we only consider
families of Hamiltonians and almost complex structures from Definition
\ref{definition parameterized moduli spaces}
and where the compactness result
\cite[Proposition 5.2.3]{schwarz1995cohomology}
is replaced with Proposition
\ref{proposition compactness result}.

If $Q^0_+ = Q^1_+$
then $\Phi_{H^0,H^1,H^2}$
descends to a $\Lambda^{Q^0_+,+}$-bilinear map.
A gluing argument
\cite[Proposition 5.4.4]{schwarz1995cohomology} together with the use of the moduli spaces in Definition
\ref{definition parameterized moduli spaces}
and Proposition \ref{proposition parameterized regular subset for surface}
tells us
that this product commutes with continuation maps and is associative.
Since $\Phi_{H^0,H^1,H^2}$
only depends on the oriented diffeomorphism type of $\Sigma$, we get that it is commutative
when $a^0_\pm = a^1_\pm$
(see \cite[Section 5.5.1.3]{schwarz1995cohomology}).
The pair of pants product also commutes with action maps since this product respects
the filtrations on the Floer chain complexes
by Lemma \ref{lemma filtration}.
If the constant Hamiltonian
$\min(H^0)$ satisfies
$\min(H^0) <_{\check{C}} \check{H}^0$
for some $C^\infty$ small perturbation $\check{H}^0$ of $H$
and if
$a^0_-([\widetilde{\omega}],\lambda^-,\lambda^+) < \lambda^- \min(H^0)$ for all
$([\widetilde{\omega}],\lambda^-,\lambda^+) \in Q_0 - 0$
and $p_0 = 0$
then
this product has a {\it left unit}
${\bf 1} \in HF^{p_0}_{\check{C},a^0_-,a^0_+}(H^0)$
by \cite[Section 5.5.1.3]{schwarz1995cohomology}
together with the gluing and compactness results stated above.
By {\it left unit} we mean that for each
$H^3 \in\ccH^\reg(\check{C},a^2_-,a^2_+)$
satisfying $2H^2, H^2 <_{\check{C}} H^3$,
$\Phi^{p_1}_{2H^2,H^3}(\Phi_{H^0,H^1,H^2}({\bf 1} \otimes x))$
is the image of $x$ under the natural
composition
$$HF^{p_1}_{\check{C},a^1_-,a^1_+}(H^1) \lra{}
HF^{p_1}_{\check{C},a^1_-,a^1_+}(H^3)
\lra{}
HF^{p_0+p_1}_{\check{C},a^2_-,a^2_+}(H^3)$$
where the first map is a continuation map and the second map is an action map for each
$x \in HF^{p_1}_{\check{C},a^1_-,a^1_+}(H^1)$.
If $H^1$ satisfies similar conditions
then we have a {\it right unit}.
Left (resp. right) units get sent to left (resp. right) units under continuation and action maps.
\end{remark}


\section{Floer Cohomology for Lower Semi-Continuous Hamiltonians} \label{section floer cohomology for lower semi-continuous Hamitlonians}

In this section we introduce Hamiltonian Floer cohomology for certain lower semi-continuous Hamiltonians.
This will be the direct limit of Floer cohomology groups of smooth Hamiltonians strictly smaller than such a lower semi-continuous Hamiltonian.
These modules will basically satisfy all the same properties
as the Floer groups defined in Section
\ref{section Hamitlonian Floer cohomology and Filtrations}.
Such ideas have been
considered
in \cite{groman2015floer}.

\subsection{Main Definitions.}

Throughout this subsection we will fix a contact cylinder $\check{C} = [1-\epsilon,1+\epsilon] \times C$ with associated Liouville domain $D$ and cylindrical coordinate $r_C$.
We will also fix a $\check{C}$-action interval $(a_-,a_+)$.

\begin{defn} \label{definition lower semi-continuous hamiltonian}

Recall that a function $f : X \lra{} \R \cup \{\infty\}$ from a metric space $X$ is {\it lower semi-continuous} if for all $x_0 \in X$,
$\liminf_{x \to x_0} f(x) \geq f(x_0)$.
A {\it lower semi-continuous Hamiltonian} is a family of functions $(H_t)_{t \in \T}$
from $M$ to $\R \cup \{\infty\}$
so that the function
$$\widehat{H} : \T \times M \lra{} \R \cup \{\infty\}, \quad \widehat{H}(t,x) := H_t(x)$$
is lower semi-continuous.
Such a Hamiltonian is {\it continuous}
if $\widehat{H}$ is continuous.

A lower semi-continuous Hamiltonian $H = (H_t)_{t \in \T}$
is {\it weakly $\check{C}$-compatible}
if the restriction $H_t|_{[1+\epsilon/8,1+\epsilon/2] \times C}$ is either equal to $\infty$ or
$\lambda_{H_t} r_C + m_{H_t}$ for some constants $\lambda_{H_t}$, $m_{H_t}$ for each $t \in \T$.
We also require that the maps
$t \to \lambda_{H_t}$ and $t \to m_{H_t}$ from $\T$ to $\R$
are lower semi-continuous where $\lambda_{H_t} := 0$ and $m_{H_t} := \infty$ if $H_t|_{[1+\epsilon/8,1+\epsilon/2] \times C} = \infty$.
The {\it slope of $H_t$ along $\check{C}$}
is defined to be $\lambda_{H_t}$
and
the {\it height of $H_t$ along $\check{C}$}
is defined to be $m_{H_t}$ for all $t \in \T$.
Also the slope and height of $H_t$ is defined to be $0$ if $\check{C}$ is the empty contact cylinder.
We say that $H$ is {\it $\check{C}$-compatible}
if it is weakly $\check{C}$-compatible
and if
$H_t|_{M - (D \cup \check{C})}$ is constant
(possibly equal to $\infty$) for each $t \in \T$.
We define
$\ccH^{\T,\ls}(\check{C})$ (resp. $\overline{\ccH}^{\T,\ls}(\check{C})$)
to be the set of weakly $\check{C}$-compatible
(resp. $\check{C}$-compatible) lower semi-continuous Hamiltonians $H$.
Define
$$\ccH^{\T,\ls}(\check{C},a_-,a_+) := $$
\begin{equation} \label{equation Q lower semincontinuous hamiltonians}
\left\{\begin{array}{ll}
\ccH^{\T,\ls}(\check{C}) & \text{if} \ (a_-,a_+) \ \text{is small} \\
\{H \in \overline{\ccH}^{\T,\ls}(\check{C}) \ : \ m_{H_t} > \text{height}(a_-,a_+) \ \forall \ t \ \in \T \} & \text{otherwise}
\end{array}
\right.
\end{equation}
where $\text{height}(a_-,a_+)$
is defined as in Equation (\ref{equation height equation}) and where $(m_{H_t})_{t \in \T}$ are the heights of $H = (H_t)_{t \in \T}$.
For $H^-,H^+ \in \ccH^{\T,\ls}(\check{C})$,
we say $H^- \leq^\ls_{\check{C}} H^+$
if
\begin{enumerate}
\item $H^-_t \leq H^+_t$,
\item $\lambda_{H_t^-} \leq \lambda_{H_t^+}$, $m_{H_t^-} \leq m_{H_t^+}$,
\item 
$m_{H^+_t} - m_{H^-_t} \leq (H^+_t - H^-_t)|_{M - (D \cup ([1,1+\epsilon/2] \times C)}$ (where we define $\infty-\infty$ to be $0$)
\end{enumerate}
for all $t \in \T$
where $(\lambda_{H_t^\pm})_{t \in \T}$,
$(m_{H_t^\pm})_{t \in \T}$ are the slopes and heights of $H^\pm$.

For any
$P \subset \Z$
and any
$H \in \ccH^{\T,\ls}(\check{C},a_-,a_+)$,
define
$\ccH^\reg(<_{\check{C}} H,a_-,a_+,P)$
to be the subspace of smooth Hamiltonians
$\check{H} = (\check{H}_t)_{t \in \T} \in \ccH^\reg(\check{C},a_-,a_+,P)$
(Definition \ref{equation high enough hamiltonians})
satisfying
\begin{itemize}
\item $\check{H}_t < H_t$,
\item $\lambda_{\check{H}_t} < \lambda_{H_t}$,
$m_{\check{H}_t} < m_{H_t}$ and
\item $m_{H^+_t} - m_{H^-_t} < (H^+_t - H^-_t)|_{M - (D \cup ([1,1+\epsilon/2] \times C)}$ for all $t \in \T$.
\end{itemize}
Define
$\leq_{\check{C}}$ to be the relation on
$\ccH^\reg(<_{\check{C}} H,a_-,a_+,P)$
where
$H^- \leq_{\check{C}} H^+$
if $H^- = H^+$
or $H^- <_{\check{C}} H^+$.
Also define
$\ccH^\reg(<_{\check{C}} H,a_-,a_+) :=
\ccH^\reg(<_{\check{C}} H,a_-,a_+,\Z)$.
\end{defn}

\begin{lemma} \label{lemma leq check C is a directed system}
For any
$P \subset \Z$
and any
$H \in \ccH^{\T,\ls}(\check{C},a_-,a_+,P)$,
we have that
$(\ccH^\reg(<_{\check{C}} H,a_-,a_+,P), \leq_{\check{C}})$
is a non-empty directed set.
\end{lemma}
\proof
It is clear that $\leq_{\check{C}}$ is
 reflexive and transitive.
We now need to show that ever pair of elements has an upper bound.
Now suppose $H^j = (H^j_t)_{t \in \T} \in \ccH^\reg(<_{\check{C}} H,a_-,a_+,P)$ for $j=0,1$.
Let $(\lambda_{H_t})_{t \in \T}$,
and $(m_{H_t})_{t \in \T}$
(resp. $(\lambda_{H^j_t})_{t \in \T}$
and $(m_{H^j_t})_{t \in \T}$)
be the slopes and heights of
$H = (H_t)_{t \in \T}$
(resp. $H^j = (H^j_t)_{t \in \T}$ for $j=0,1$).
Since $\lambda_{H_t}$ and $m_{H_t}$
varies in a lower semi-continuous way
with respect to $t \in \T$,
we can find smooth families of constants
$(\lambda_t)_{t \in \T}$, $(m_t)_{t \in \T}$
so that
$\lambda_{H^j_t} < \lambda_t < \lambda_{H_t}$
and
$m_{H^j_t} < m_t < m_{H_t}$ for all $t \in \T$.
Also since $H$ is lower semi-continuous,
we can find a smooth Hamiltonian
$\check{H} = (\check{H}_t)_{t \in \T}$
so that
$H^j_t < \check{H}_t < H_t$ for all $t \in \T$
and also so that
$\check{H}_t|_{M - (\check{C} \cup D)}$ is constant for each $t \in \T$ if $(a_-,a_+)$ is not small.
We can also choose $\check{H}$ so that
$(H_t - \check{H}_t - (m_{H_t} - m_t))|_{M-(D \cup ([1,1+\epsilon/2] \times C))} > 0$ and
$(\check{H}_t - H^j_t - (m_t - m_{H^j_t}))|_{M-(D \cup ([1,1+\epsilon/2] \times C))} > 0$
for all $j=0,1$ and $t \in \T$.
Then by using a bump function on $\check{C}$
to interpolate between $\check{H}_t$
and $\lambda_t r_C + m_t$
in the region $((1,1+\epsilon/16) \cup (1+\epsilon/2,1+3\epsilon/4)) \times C$, we can assume
that $\check{H}_t|_{[1+\epsilon/8,1+\epsilon/2] \times C} = \lambda_t r_C + m_t$ for all $t \in \T$ as well.
Hence $\check{H} \in \ccH^\T(\check{C},a_-,a_+)$
and $H^j <_{\check{C}} \check{H} \leq^\ls_{\check{C}} H$.
By Lemma \ref{lemma partial order lemma}
combined with the fact that
$\ccH^\T(\check{C},a_-,a_+)$ is a ubiquitous
subset of $\ccH^\T(\check{C})$
by Lemma \ref{lemma regular Hamiltonians},
we can find
$H^2 \in \ccH^\reg(\check{C},a_-,a_+)$
satisfying
$H^j <_{\check{C}} H^2 <_{\check{C}} \check{H}$ for $j=0,1$.
Hence $H^2 \in \ccH^\reg(<_{\check{C}} H,a_-,a_+)$ and $H^j <_{\check{C}} H^2$ for $j=0,1$.
A similar construction also shows that
$\ccH^\reg(<_{\check{C}} H,a_-,a_+)$ is non-empty.
This completes our lemma.
\qed

\bigskip

\begin{defn} \label{definition of hamiltonian floer cohomology for lower semicontinuous hamiltonians}

Let
$H \in \ccH^{\T,\ls}(\check{C},a_-,a_+)$.
By Lemma \ref{lemma leq check C is a directed system},
we can define
$$HF^p_{\check{C},a_-,a_+}(H) := \varinjlim_{\check{H} \in \ccH^\reg(<_{\check{C}} H,a_-,a_+,P)} HF^*_{\check{C},a_-,a_+}(\check{H})$$
for any $P$ containing $\{p-1,p,p+1\}$
where the direct limit is taken with respect to
the directed system
$(\ccH^\reg(<_{\check{C}} H,a_-,a_+,P),\leq_{\check{C}})$ whose morphisms are continuation maps.
\end{defn}

If the lower semi-continuous Hamiltonian
$H$ from Definition \ref{definition of hamiltonian floer cohomology for lower semicontinuous hamiltonians}
is in fact smooth and an element
of $\ccH^\reg(Q,a_-,a_+,P)$
then it would be good to
check that
Definitions \ref{definition of hamiltonian floer cohomology for lower semicontinuous hamiltonians}
and
\ref{definition hamiltonian floer cohomology}
agree.
We will do this now.

%

\begin{lemma} \label{lemma both floer definitions agree}
Let $p \in \Z$ and let
$H \in \ccH^\reg(\check{C},a_-,a_+,P)$
where $P = \{p-1,p,p+1\}$.
Let $G_1$
be equal to
$HF^p_{\check{C},a_-,a_+}(H)$
from Definition \ref{definition hamiltonian floer cohomology}
and let
$G_2$
be equal to
$HF^p_{\check{C},a_-,a_+}(H)$
from Definition \ref{definition of hamiltonian floer cohomology for lower semicontinuous hamiltonians}.
Then the natural map
$G_2 \lra{} G_1$
induced by continuation maps is an isomorphism.
\end{lemma}

As a result, we have consistent notation and we do not need to distinguish between
Definition \ref{definition hamiltonian floer cohomology}
and
Definition \ref{definition of hamiltonian floer cohomology for lower semicontinuous hamiltonians}.

\begin{proof}[Proof of Lemma \ref{lemma both floer definitions agree}.]
For any
$\check{H} \in \ccH^\reg(<_{\check{C}} H,a_-,a_+,P)$,
we let
$HF^p_{\check{C},a_-,a_+}(\check{H})$
be as in Definition
\ref{definition hamiltonian floer cohomology}.
By Lemma \ref{lemma isomorphism condition},
it is sufficient for us to show that
there is a cofinal family $\Xi$
inside
$(\ccH^\reg(<_{\check{C}} H,a_-,a_+,P),\leq_{\check{C}})$
so that the natural map
$HF^p_{\check{C},a_-,a_+}(\check{H}) \lra{} G_1$ is an isomorphism for each element $\check{H}$ of this cofinal family.
Let $U_H \subset \ccH^\T(\check{C},a_-,a_+)$ be the neighborhood of $H$
satisfying the properties of Lemma
\ref{lemma chain level continuation map filtration isomorphism}.	
Define $\Xi := U_H \cap \ccH^\reg(<_{\check{C}} H,a_-,a_+,P)$.
Then $\Xi$ is a cofinal family by Lemma
\ref{lemma partial order lemma}
and the natural map
$HF^*_{\check{C},a_-,a_+}(\check{H}) \lra{} G_1$ is an isomorphism
for all $\check{H} \in \Xi$ by Lemma 
\ref{lemma chain level continuation map filtration isomorphism}. This completes the lemma.
\end{proof}

\subsection{Continuation Maps}

Throughout this subsection, fix a contact cylinder $\check{C}$
together with a $\check{C}$-action interval $(a_-,a_+)$.

\begin{defn} \label{definition continuation map for lower semicontinuous Hamiltonians}

Let $p \in \Z$ and let $P \subset \Z$
satisfy $\{p-1,p,p+1\} \subset P$.
Let
$H^-,H^+ \in \ccH^{\ls}(\check{C},a_-,a_+)$
satisfy
$H^- \leq^\ls_{\check{C}} H^+$.
Then we have a natural inclusion of directed systems
$$\ccH^\reg(<_{\check{C}} H^-, a_-,a_+,P) \subset \ccH^\reg(<_{\check{C}} H^+, a_-,a_+,P)$$
and this induces a morphism
$$\Phi^p_{H^-,H^+} : HF^p_{\check{C},a_-,a_+}(H^-) \lra{}
HF^p_{\check{C},a_-,a_+}(H^+)$$
called a {\it continuation map}.
\end{defn}

If $H^-,H^+$ from Definition \ref{definition continuation map for lower semicontinuous Hamiltonians} above
are elements of
$\ccH^{\reg}(\check{C},a_-,a_+,P)$
then the continuation map from this definition is equal to the continuation map
from Definition \ref{definition continuation map}
by Lemma \ref{lemma both floer definitions agree}.
The composition of two continuation maps is a continuation map.

The following Lemma is a generalization of a slightly weaker version of Lemma \ref{lemma continuation map filtration isomorphism}
to all smooth Hamiltonians (not just ones in $\ccH^\reg(\check{C},a_-,a_+)$).

\begin{lemma} \label{lemma continuation map isomorphism action gap}
Let $H = (H_{s,t})_{(s,t) \in [0,1] \times \T}$
be a smooth family of autonomous smooth Hamiltonians
and define $H_{s,\bullet} := (H_{s,t})_{t \in \T}$ for all $s \in [0,1]$.
Fix $p \in \Z$.
Suppose 
\begin{enumerate}
	\item $H_{s,\bullet} \in \ccH^{\T,\ls}(\check{C},a_-,a_+)$ for each $s \in [0,1]$ and
	$H_{s_1,\bullet} \leq^\ls_{\check{C}} H_{s_2,\bullet}$
	for all $s_1 \leq s_2$,
	\item \label{item:neighborhoods}
	 there are neighborhoods
	$N_-$, $N_+$ of $a_-$, $a_+$ in $\Sc(Q_-)$, $\Sc(Q_+)$ respectively so that
	$$\Gamma^P_{\check{C},a_-,a_+}(H_{s,\bullet}) = \Gamma^P_{\check{C},a'_-,a'_+}(H_{s,\bullet})$$
	(Definition \ref{defn chain complex})
	for all $a'_\pm \in N_\pm$, $s \in [0,1]$
	where
	$P = [p-1-n,p+1+n]$
	\item \label{item:assumption3}
	 and that there are no $1$-periodic orbits of $H_{s,\bullet}$
	contained in $[1+\epsilon/8,1+\epsilon/2] \times C$ for all $s \in [0,1]$.
\end{enumerate}
Then the continuation map
$$\Phi^p_{H_{0,\bullet},H_{1,\bullet}} : HF^p_{\check{C},a_-,a_+}(H_{0,\bullet}) \lra{}
HF^p_{\check{C},a_-,a_+}(H_{1,\bullet})$$
in degree $p$ is an isomorphism.
\end{lemma}
\proof
This lemma will be proven by showing that for each $s \in [0,1]$ there is a constant $\epsilon_s>0$ so that the continuation map
\begin{equation} \label{equation isomorphism  nondegenerate continuation}
HF^p_{\check{C},a_-,a_+}(H_{s_0})
\lra{}
HF^p_{\check{C},a_-,a_+}(H_{s_1})
\end{equation}
is an isomorphism for all $s_0,s_1 \in [0,1] \cap (s -\epsilon_s,s+\epsilon_s)$
satisfying $s_0 \leq s_1$.
So from now on, fix $s \in [0,1]$.

By (\ref{item:assumption3})
there are constants
$\delta_s, \zeta > 0$ so that for each
$s' \in [s-\delta_s,s+\delta_s]$
that every $1$-periodic orbit 
of
$H_{s',\bullet}$
has image not intersecting
$N := [1+\epsilon/8-\zeta,1+\epsilon/2+\zeta] \times C$.
By Lemma \ref{lemma partial order lemma},
there exists a cofinal family
$(K_{i,s'})_{i \in \N}$
in
$\ccH^\reg(<_{\check{C}} H_{s,\bullet}, a_-,a_+,P)$
which $C^\infty$ converges to $H_{s',\bullet}$ for each $s' \in [0,1]$.

Let
$F : M \lra{} \R$
be a $\check{C}$-compatible
autonomous Hamiltonian
so that
$F$ is locally
constant outside
$N$ and so that
$0 <_{\check{C}} F$.
Since $F$ is locally constant outside $N$,
there is a constant
$0<\eta_s<\delta_s$ so that
for each
$s' \in (s-\eta_s,s+\eta_s)$
and
$\tau \in [-\eta_s,\eta_s]$
there is a constant $N_{s'} >0$
so that
$K_{i,s'} + \tau F \in \ccH^{\T,\reg}(\check{C},a_-,a_+,P)$
for all $i \geq N_{s'}$
by assumptions (\ref{item:neighborhoods})
and (\ref{item:assumption3}).
Hence by Lemma
\ref{lemma continuation map filtration isomorphism},
the continuation maps
$HF^p_{\check{C},a_-,a_+}(K_{i,s'}+\tau^- F) \lra{}
HF^p_{\check{C},a_-,a_+}(K_{i,s'}+\tau^+ F)$
are isomorphisms for all
$s' \in (s-\eta_s,s+\eta_s)$,
$i \geq N_{s'}$ and
$\tau^-,\tau^+ \in [-\eta_s,\eta_s]$
satisfying $\tau^- < \tau^+$
and hence
the continuation maps
$HF^p_{\check{C},a_-,a_+}(H_{s',\bullet}+\tau^- F) \lra{}
HF^p_{\check{C},a_-,a_+}(H_{s',\bullet}+\tau^+ F)$
are isomorphisms for all
$s' \in (s-\eta_s,s+\eta_s)$
and
$\tau^-,\tau^+ \in [-\eta_s,\eta_s]$
satisfying $\tau^-\leq \tau^+$.

Choose $0<\epsilon_s<\eta_s$
small enough so that
$$H_{s_1,\bullet} -\eta_sF
\leq^\ls_{\check{C}}
H_{s_0,\bullet}
\leq^\ls_{\check{C}}
H_{s_1,\bullet}
\leq^\ls_{\check{C}}
H_{s_0,\bullet} + \eta_sF$$
for all $s_0,s_1 \in (s-\epsilon_s,s+\epsilon_s)$
satisfying $s_0 \leq s_1$.
By the discussion above
combined with the fact that
the composition of any two
continuation maps is a continuation map,
we get that
the composition of any two maps in
$$
HF^p_{\check{C},a_-,a_+}(H_{s_1,\bullet}-\eta_s F)
\to
HF^p_{\check{C},a_-,a_+}(H_{s_0,\bullet})
\lra{\alpha}
HF^p_{\check{C},a_-,a_+}(H_{s_1,\bullet})
\to
HF^p_{\check{C},a_-,a_+}(H_{s_0,\bullet}+\eta_s F)
$$
is an isomorphism.
Hence $\alpha$
is an isomorphism
for all $s_0,s_1 \in (s-\epsilon_s,s+\epsilon_s)$
satisfying $s_0 \leq s_1$.
\qed

\subsection{Action Maps.}

We can also define action maps for lower semi-continuous Hamiltonians in a similar way
to Subsection \ref{subsection Action maps and Cone isomorphisms}.
Throughout this section we will fix a contact cylinder $\check{C} = [1-\epsilon,1+\epsilon] \times C \subset M$ with associated Liouville domain $D$.

\begin{defn} \label{definition action maps for lower semi-continuous hamiltonians}
Let $P = \{p-1,p,p+1\}$ for some $p \in \Z$.
Let $(a^j_-,a^j_+)$ be a $\check{C}$-action interval for $j=0,1$
so that
$(a^1_-,a^1_+)$
is smaller than $(a^0_-,a^0_+)$.
Let
$H \in \cap_{j=0,1} \ccH^{\T,\ls}(\check{C},a^j_-,a^j_+)$ (Definition \ref{definition lower semi-continuous hamiltonian})
and let
$\Xi := \cap_{j=0,1} \ccH^\reg(<_{\check{C}} H,a^j_-,a^j_+)$
be the directed set with relation $\leq_{\check{C}}$.
Since
\begin{equation} \label{equation compute from directed subsystem}
HF^p_{\check{C},a^j_-,a^j_+}(H) = \varinjlim_{\check{H} \in \Xi} HF^p_{\check{C},a^j_-,a^j_+}(\check{H}), \quad j=0,1
\end{equation}
and since action maps commute with continuation maps,
we get that the action maps
$HF^*_{\check{C},a^0_-,a^0_+}(\check{H})
\lra{}
HF^*_{\check{C},a^1_-,a^1_+}(\check{H})$
for each $\check{H} \in \Xi$
from Definition \ref{definition action map}
induce a map
$$HF^p_{\check{C},a^0_-,a^0_+}(H)
\lra{} 
HF^p_{\check{C},a^1_-,a^1_+}(H)$$
which we also call an {\it action map}.
\end{defn}

Again action maps commute with continuation maps.
If the Hamiltonian
$H$ is an element of
$\cap_{j=0,1} \ccH^\reg(\check{C}, a^j_-,a^j_+,\{p-1,p,p+1\})$,
then the action map
from Definition \ref{definition action map} is equal to
the action map from Definition
\ref{definition action maps for lower semi-continuous hamiltonians}
by Lemma \ref{lemma both floer definitions agree}.
We also have the following analogue of Lemma \ref{lemma action morphism is isomorphism}.

\begin{lemma} \label{lemma action morphism for lower semicontinuous isomorphism}
Let $P = [p-1-n,p+1+n]$ for some $p \in \Z$.
Let $(a^j_-,a^j_+)$ be a $\check{C}$-action interval for $j=0,1$
so that
$(a^1_-,a^1_+)$
is smaller than $(a^0_-,a^0_+)$.
Let
$H \in \cap_{j=0,1} \ccH^{\T,\ls}(\check{C},a^j_-,a^j_+)$ be a smooth Hamiltonian. 
Suppose that $\Gamma^P_{\check{C},a^0_-,a^0_+}(H)
= \Gamma^P_{\check{C},a^1_-,a^1_+}(H)$.
Then the action map
$$
HF^p_{\check{C},a^0_-,a^0_+}(H) \lra{}
HF^p_{\check{C},a^1_-,a^1_+}(H)
$$
is an isomorphism.
\end{lemma}
\proof
By Lemma \ref{lemma partial order lemma}
there exists a cofinal family
$(H^i)_{i \in \N}$ in $\cap_{j=0,1} \ccH^\reg(<_{\check{C}} H, a^j_-,a^j_+)$
so that $H^i$ $C^\infty$ converges to $H$.
Then for all $i$ sufficiently large,
we have by \ref{item:shortpath}
from Definition \ref{definition conley zehnder index}
that
$\Gamma^{\{p-1,p,p+1\}}_{\check{C},a^0_-,a^0_+}(H^i)
= \Gamma^{\{p-1,p,p+1\}}_{\check{C},a^1_-,a^1_+}(H^i)$.
Hence by Lemma \ref{lemma action morphism is isomorphism}, the action morphism
$$
HF^p_{\check{C},a^0_-,a^0_+}(H^i) \lra{}
HF^p_{\check{C},a^1_-,a^1_+}(H^i)
$$
is an isomorphism for all $i$ sufficiently large.
Therefore since action maps commute with continuation maps and since
Equation (\ref{equation compute from directed subsystem}) holds,
we get our result.
\qed

\subsection{Invariance Under Time Reparameterization.}

Throughout this subsection, $\check{C}$
is a contact cylinder whose associated Liouville domain is $D$.
We will also fix a $\check{C}$-interval domain $(Q_-,Q_+)$.

\begin{defn} \label{definition ls reparameterized Hamiltonian}
Let $F : \T \lra{} \T$
be a smooth non-decreasing map.
Let $H = (H_t)_{t \in \T}$ be a lower semi-continuous Hamiltonian.
We define
$H^F = (H^F_t)_{t \in \T}$
by
$H^F_t := F'(t) H_{F(t)}$.
\end{defn}

\begin{prop} \label{proposition ls reparameterization isomorphism}
Let $F : \T \lra{} \T$ be a smooth non-decreasing map which is homotopic to the identity map.
Let $(a_-,a_+) \in \Sc(Q_-) \times \Sc(Q_+)$
be a $\check{C}$-action interval.
Then for each $H \in \ccH^{\T,\ls}(\check{C},a_-,a_+)$
there is an isomorphism
of $\Lambda^{Q_+,+}$-modules
\begin{equation} \label{equation ls reparameterization isomorphism}
HF^*_{\check{C},a_-,a_+}(H)
\cong HF^*_{\check{C},a_-,a_+}(H^F)
\end{equation}
which commutes with continuation maps and action maps.
\end{prop}

\begin{defn} \label{definition ls reparameterization isomorphism}
The isomorphism (\ref{equation ls reparameterization isomorphism})
will be called a {\it reparameterization isomorphism}.
\end{defn}

\begin{proof}[Proof of Proposition \ref{proposition ls reparameterization isomorphism}.]
Since
$$HF^p_{\check{C},a_-,a_+}(H^F) = \varinjlim_{\check{H} \in \ccH^\reg(<_{\check{C}} H,a_-,a_+,\Z)} HF^*_{\check{C},a_-,a_+}(\check{H}^F)$$
by Proposition \ref{proposition reparameterization isomorphism}
and since the reparameterization isomorphisms
$$HF^*_{\check{C},a_-,a_+}(\check{H})
\cong HF^*_{\check{C},a_-,a_+}(\check{H}^F)$$
from Definition \ref{definition reparameterization isomorphism}
commute with continuation maps,
we get our isomorphism
(\ref{equation ls reparameterization isomorphism}).
\end{proof}

\subsection{Pair of Pants Product}

Throughout this section we will fix a contact cylinder $\check{C} = [1-\epsilon,1+\epsilon] \times C \subset M$ with associated Liouville domain $D$.

\begin{defn} \label{definition pair of pants product for lower semicontinuous hamiltonians}

Let $(a^j_-,a^j_+)$ be a $\check{C}$-action interval for $j=0,1,2$
so that $(a^2_-,a^2_+)$
is smaller than $(a^j_-,a^j_+)_{j=0,1}$ as in Definition \ref{definition contact cylinders smaller than}.
Define $\kappa_0 = \kappa_1 := 1$
and $\kappa_2 := 2$.
Let
$H^j = (H^j_t)_{t \in \T}$ be a lower semi-continuous Hamiltonian so that
$(\kappa_j) H^j \in \ccH^{\T,\ls}(\check{C},a^j_-,a^j_+)$ for $j=0,1,2$
where $(\kappa_j) H^j$ is defined as in Definition \ref{definition catenation of hamitlonians}
and suppose $H^j \leq^\ls_{\check{C}} H^2$ for $j=0,1$.

Choose a cofinal sequence
$((\kappa_j) \check{H}^{i,j})_{i \in \N}$
in $\cap_{k=0}^2 \ccH^\reg(<_{\check{C}},a^k_-,a^k_+,(\kappa_j) H^j)$
(Definition \ref{definition lower semi-continuous hamiltonian}) 
for each $j=0,1,2$
so that $\check{H}^{i,j} <_{\check{C}} \check{H}^{2,j}$ for
$j=0,1$ and so that
for each triple of capped $1$-periodic orbits
$(\gamma^j)_{j=0,1,2} \in \prod_{j=0,1,2} \Gamma^{p_j}_{\check{C},a^j_-,a^j_+}((\kappa_j) H^j)$  where $p_0,p_1,p_2 \in \Z$,
we have that the associated $1$-periodic
orbits of at least two of them have distinct images in $M$.
Define
$Q_{i,j} := HF^*_{\check{C},a^j_-,a^j_+}((\kappa_j) \check{H}^{i,j})$
for all $i \in \N$ and $j = 0,1,2$.
Then the {\it pair of pants product}
is the natural composition
$$\Phi_{H^0,H^1,H^2} : HF^*_{\check{C},a^0_-,a^0_+}(H^0) \otimes_{\K} HF^*_{\check{C},a^1_-,a^1_+}(H^1) =$$
$$(\varinjlim_{i \in \N} Q_{i,0})
\otimes_{\K} (\varinjlim_{i \in \N} Q_{i,1})
\lra{} \varinjlim_{i \in \N} (Q_{i,0} \otimes_{\K} Q_{i,1})
\lra{\alpha} \varinjlim_{i \in \N} Q_{i,2}
= HF^*_{\check{C},a^2_-,a^2_+}(2H^2)$$
where the morphism $\alpha$ is the direct limit of the natural pair of pants product map
$Q_{i,0} \otimes_{\K} Q_{i,1} \lra{} Q_{i,2}$
as in Definition \ref{defnition pair of pants product}.
\end{defn}

\begin{remark} \label{remark lower semicontinuous pants product}
If $(\kappa_j) H^j \in \ccH^\reg(\check{C},a^j_-,a^j_+)$ for $j=0,1,2$ then
the pair of pants product map above is identical to the one defined in Definition \ref{defnition pair of pants product}
by Lemma \ref{lemma both floer definitions agree}.
Also action maps and continuation maps
commute with pair of pants product maps.
Such a product is associative and graded commutative by Remark
\ref{remark pants product is graded commutative and associative and unital}.
If the domains of $a^0_+$ and $a^1_+$
agree then this product is
$\Lambda^{Q_+,+}$-bilinear where
$Q_+$ is the domain of $a^0_+$.
If the constant Hamiltonian
$\min(H^0)$ satisfies
$\min(H^0) \leq^\ls_{\check{C}} H^0$,
if
$a^0_-([\widetilde{\omega}],\lambda^-,\lambda^+) < \lambda^- \min(H^0)$ for all
$([\widetilde{\omega}],\lambda^-,\lambda^+) \in Q_0 - 0$ and if
$p_0 = 0$
then it has a {\it left unit}
${\bf 1} \in HF^{p_0}_{\check{C},a^0_-,a^0_+}(H^0)$
by 
Remark \ref{remark pants product is graded commutative and associative and unital}
combined with the fact that continuation maps
send left units to left units
 and the fact that
$HF^*_{\check{C},a^0_-,a^0_+}(K)$
has a left unit for any sufficiently $C^2$ small perturbation of $\min(H^0)$.
By {\it left unit} we mean
that for each $H^3 \in \ccH^{\T,\ls}(\check{C},a^2_-,a^2_+)$
satisfying
$2H^2,H^2 \leq^\ls_{\check{C}} H^3$,
$\Phi^{p_1}_{2H^2,H^3}(\Phi_{H^0,H^1,H^2}({\bf 1} \otimes x))$
is equal to the image of $x$ under the composition
$$HF^{p_1}_{\check{C},a^1_-,a^1_+}(H^1) \lra{}
HF^{p_1}_{\check{C},a^1_-,a^1_+}(H^3)
\lra{}
HF^{p_0+p_1}_{\check{C},a^2_-,a^2_+}(H^3)$$
of action maps and continuation maps for each
$x \in HF^{p_1}_{\check{C},a^1_-,a^1_+}(H^1)$.
This follows from Remark
\ref{remark pants product is graded commutative and associative and unital}.
We have a {\it right unit} under similar conditions. The only difference is that the indexes
`$0$' and `$1$' are swapped.
\end{remark}

\section{Definition of Symplectic Cohomology} \label{section definition of symplectic cohomology}

Throughout this section we will fix a contact cylinder $\check{C} = [1-\epsilon,1+\epsilon] \times C \subset M$ with associated Liouville domain $D$.


\begin{defn} \label{definition symplectic cohomology}
Let $K \subset \T \times M$ be a closed set with the property that $K \subset \T \times D$ if $D$ is non-empty.
We define
$\ccH^{\T,\ls}(\check{C},K,\leq 0)$
to be the subset of $\ccH^{\T,\ls}(\check{C})$ (Definition \ref{definition lower semi-continuous hamiltonian})
consisting of lower
semi-continuous Hamiltonians
$H = (H_t)_{t \in \T}$ 
with the property that
$H_t(x) \leq 0$ if $(t,x) \in K$
and $H_t(x) = \infty$ otherwise.
We define
$\ccH^{\T,\ls}(\check{C},\leq 0)$
to be the union 
of all such subsets $\ccH^{\T,\ls}(\check{C},K,\leq 0)$.
For each $H \in \ccH^{\T,\ls}(\check{C},\leq 0)$
and each $\check{C}$-interval domain
$(Q_-,Q_+)$ as in Definition \ref{defn chain complex}, define
$\SH^*_{\check{C},Q_-,Q_+}(H)$
to be the double system of $\Z$-graded $\Lambda_\K^{Q_+,+}$-modules
(as in Definition \ref{definition double system}):
$$
\SH^*_{\check{C},Q_-,Q_+}(H) = (HF^*_{\check{C},a_-,a_+}(H))_{(a_-,a_+) \in \Sc(Q_-) \times \Sc(Q_+)}
$$
where 
\begin{itemize}
\item the ordering on $\Sc(Q_\pm)$ is given by $\geq$,
\item the double system morphisms are action maps as in Definition \ref{definition action maps for lower semi-continuous hamiltonians}.
\end{itemize}
We define {\it symplectic cohomology of $H$}
to be
$SH^*_{\check{C},Q_-,Q_+}(H) := \varinjlim \varprojlim \SH^*_{\check{C},Q_-,Q_+}(H)$
where $\varinjlim \varprojlim$
is given in Definition \ref{definition inverse direct limit is a functor}.
For any closed set $K \subset \T \times M$ which is contained in $\T \times D$ if $D$ is non-empty,
define
$\SH^*_{\check{C},Q_-,Q_+}(K \subset \T \times M)$ (resp. $SH^*_{\check{C},Q_-,Q_+}(K \subset \T \times M)$) to be
$\SH^*_{\check{C},Q_-,Q_+}(H_K)$ (resp. $SH^*_{\check{C},Q_-,Q_+}(H_K)$)
where 
$$
H_K = (H_{K,t})_{t \in \T}, \quad
H_{K,t} : M \lra{} \R \cup \{\infty\} \quad H_{K,t}(x) := \left\{
\begin{array}{ll}
0 & \text{if} \ (t,x) \in K \\
\infty & \text{otherwise.}
\end{array}
\right.$$
The algebra $SH^*_{\check{C},Q_-,Q_+}(K \subset \T \times M)$
is called {\it symplectic cohomology of $K$ in $M$}.

If $K \subset M$ is a closed subset, which is contained in $D$ if $D$ is non-empty, then we define
$\SH^*_{\check{C},Q_-,Q_+}(K \subset M) := \SH^*_{\check{C},Q_-,Q_+}(\T \times K \subset \T \times M)$ 
and
$SH^*_{\check{C},Q_-,Q_+}(K \subset M) := SH^*_{\check{C},Q_-,Q_+}(\T \times K \subset \T \times M)$.
\end{defn}

\begin{remark} \label{remark on definition of symplectic cohomology}
The definition of symplectic cohomology makes sense because $\ccH^{\T,\ls}(\check{C}, \leq 0) \subset \ccH^{\T,\ls}(\check{C},a_-,a_+)$ for all $\check{C}$ action intervals
$(a_-,a_+) \in \Sc(Q_-) \times \Sc(Q_+)$.

If the Hamiltonian $H$ is autonomous then
$\SH^*_{\check{C},Q_-,Q_+}(H)$
is a double system of $\Z$-graded $\Lambda_\K^{Q_+,+}$-modules
with product induced by
the pair of pants product maps
\begin{equation} \label{equation pair of pants product for symplectic cohomology}
HF^*_{\check{C},a^0_-,a^0_+}(H) \otimes_{\Lambda_\K^{Q_+,+}} HF^*_{\check{C},a^1_-,a^1_+}(H)
\lra{} HF^*_{\check{C},a^2_-,a^2_+}(H)
\end{equation}
where
$a^2_- = a^0_- + a^1_-$, $a^2_+ = \min(a^0_- + a^1_+, a^1_- + a^0_+)$.
This product is well defined since $H$ is autonomous and $H \leq^\ls_{\check{C}} \frac{1}{2}H$ because the slope of $H$ along $\check{C}$ is $0$ and because $H$ is either non-positive or $\infty$.
Also
$SH^*_{\check{C},Q_-,Q_+}(H)$ is a graded $\Lambda_\K^{Q_+,+}$-algebra by Remark \ref{remark products on direct and inverse limits}.
Such an algebra is a unital graded commutative algebra by Remark \ref{remark lower semicontinuous pants product}.
\end{remark}

\begin{defn} \label{definition transfer map}
Let $(Q_-,Q_+)$ be a $\check{C}$-interval domain pair.
Let $H^\pm \in \ccH^{\T,\ls}(\check{C},\leq 0)$ satisfy $H^- \leq H^+$.
Then $H_- \leq_{\check{C}} H_+$ and hence the natural continuation maps
$$HF^*_{\check{C},a_-,a_+}(H_-) \lra{} HF^*_{\check{C},a_-,a_+}(H_+)$$
for all $\check{C}$-action intervals $(a_-,a_+)$
give us a morphism of double systems
$$\SH^*_{\check{C},Q_-,Q_+}(H_-) \lra{} \SH^*_{\check{C},Q_-,Q_+}(H_+)$$
called a {\it transfer morphism}.
This also induces a morphism
of algebras
$SH^*_{\check{C},Q_-,Q_+}(H_-) \lra{} SH^*_{\check{C},Q_-,Q_+}(H_+)$ called a {\it transfer map}.
In particular, if $K_+ \subset K_- \subset D$ are closed subsets,
we have a {\it transfer map}
$SH^*_{\check{C},Q_-,Q_+}(K_- \subset M) \lra{} SH^*_{\check{C},Q_-,Q_+}(K_+ \subset M)$.
\end{defn}

If $H_\pm$ is autonomous then
this transfer morphism or map respects the natural product structures on the corresponding double systems or modules
and they also send units to units.

\begin{defn} \label{action maps for symplectic cohomology}
Let $(Q^j_-,Q^j_+)$ be a
$\check{C}$-interval domain pair
for $j=0,1$
so that $Q^1_\pm \subset Q^0_\pm$.
Suppose
$H \in \ccH^{\T,\ls}(\check{C},\leq 0)$.
Then the action maps
$$HF^*_{\check{C},a_-,a_+}(H) \lra{} HF^*_{\check{C},a_-|_{Q^1_-},a_+|_{Q^1_+}}(H), \quad (a_-,a_+) \in \Sc(Q^0_-) \times \Sc(Q^0_+)$$
give us a morphism of double systems
$\SH^*_{\check{C},Q^0_-,Q^0_+}(H) \lra{} \SH^*_{\check{C},Q^1_-.Q^1_+}(H)$ called an {\it action morphism}.
\end{defn}

Again this morphism respects the product structure if $H$ is autonomous.

\begin{defn} \label{definition reparameterization isomorphism for SH}
Let $(Q_-,Q_+)$ be a $\check{C}$-interval domain pair and let $H \in \ccH^{\T,\ls}(\check{C},\leq 0)$.
Let $F : \T \lra{} \T$
be a smooth non-decreasing map.
Then the reparameterization isomorphisms
$$
HF^*_{\check{C},a_-,a_+}(H)
\cong HF^*_{\check{C},a_-,a_+}(H^F)
$$
from Definition \ref{definition ls reparameterization isomorphism}
for each
$\check{C}$-action interval $(a_-,a_+) \in \Sc(Q_-) \times \Sc(Q_+)$
gives us an isomorphism of double systems
$$\SH^*_{\check{C},Q_-,Q_+}(H) \cong \SH^*_{\check{C},Q_-,Q_+}(H^F)$$
called a
{\it reparameterization isomorphism}.
\end{defn}

\begin{lemma} \label{lemma extend to Novikov field action} 
Let
$(Q_-,Q_+)$ be a $\check{C}$-interval domain pair and let $H \in \ccH^{\T,\ls}(\check{C},\leq 0)$.
Then the $\Lambda_\K^{Q_+,+}$-module structure on
$SH^*_{\check{C},Q_-,Q_+}(H)$ extends uniquely to a
$\Lambda_\K^{Q_+}$-module structure on $SH^*_{\check{C},Q_-,Q_+}(H)$ making it into a
$\Lambda_\K^{Q_+}$-algebra.
\end{lemma}
\proof
Let
$\pi_D : H^2(M,D;\R) \times \R \times \R \lra{} H^2(M,D;\R)$ be the natural projection map
and let $\iota_D : H^2(M,D;\R) \lra{} H^2(M;\R)$ be the natural restriction map.
Let
$$Q := \iota_D(\pi_D(Q_+)) \subset H^2(M;\R) = (H_2(M;\Z) \otimes_\Z \R)^*$$
and let $\preceq_Q$ be the induced ordering
on $H_2(M;\Z)$ as in Definition
\ref{definition convex cone}.
Define $Q^\vee := \{x \in H_2(M;\Z) \ : \ 0 \preceq_{Q} x \} \subset H_2(M;\Z)$.
By Definition \ref{definition convex cone associated to contact cylinder},
we have that
$Q^\vee$ is naturally a submonoid of the
multiplicative group
$(\Lambda_\K^{Q_+,+})^\times$.
All such elements are invertible in $\Lambda_\K^{Q_+}$.
Also each element
$w \in \Lambda_\K^{Q_+}$ has the property that there exists some element
$s_w \in Q^\vee$
satisfying $s_w w \in \Lambda_\K^{Q_+,+}$.
Therefore by \cite[Tag 07JY]{stacks-project}
it is sufficient for us to show that
each $s \in Q^\vee$
acts as an automorphism
on $SH^*_{\check{C},Q_-,Q_+}(H)$.
%


We now fix such an element $s$.
Since $s \in H_2(M;\Z)$, it defines a linear function on $H^2(M;\R)$ and hence by pulling back via $\iota_D$ and $\pi_D$ and restricting to $Q_+$, a function $L_s \in \Sc(Q_+)$.
The map sending a capped loop
$\gamma$ to $\gamma \# (-s)$
induces an isomorphism
$HF^*_{\check{C},a_- - L_s,a_+ - L_s}(H)
\lra{\cong} HF^*_{\check{C},a_-,a_+}(H)$.
Since
$SH^*_{\check{C},Q_-,Q_+}(H) =
\varinjlim_{a_-} \varprojlim_{a_+} HF^*_{\check{C},a_- - L_s,a_+ - L_s}(H)$,
we get that such a map
induces an automorphism of
$SH^*_{\check{C},Q_-,Q_+}(H)$.
This automorphism coincides with the
natural
module
action of $s \in \Lambda_\K^{Q_+,+}$.
Hence the $\Lambda_\K^{Q_+,+}$-action extends to a $\Lambda_\K^{Q_+}$-action.
%
%
%
\qed

\begin{remark} \label{remark continuation maps are module morphisms}
Since continuation maps and action maps between symplectic cohomology algebras of autonomous Hamiltonians are induced by
inclusions of double systems,
they are naturally $\Lambda_\K^{Q_+,+}$-algebra homomorphisms and hence $\Lambda_\K^{Q_+}$-algebra homomorphisms by the lemma above.
\end{remark}

\section{Properties of Symplectic Cohomology} \label{section properties of symplectic cohomology}

\subsection{Changing Contact Cylinders}

\begin{prop} \label{proposition changing contact cylinder}
Let $\check{C}_j$ be a contact cylinder
with associated Liouville domain
$D_j$ for $j=0,1$
so that $D_1 \subset D_0$.
Suppose also  that
the natural restriction  map
$\iota := H^2(M,D_0;\R) \lra{} H^2(M,D_1;\R)$
is an isomorphism and define
\begin{equation} \label{equation product restriction map cone}
\widetilde{\iota} : H^2(M,D_0;\R) \times \R^2 \lra{} H^2(M,D_1;\R) \times \R^2, \quad \widetilde{\iota} := \iota \times \id_{\R^2}.
\end{equation}
Let $(Q_-,Q_+)$ be a wide $\check{C}_0$-interval domain and let $H \in \ccH^{\T,\ls}(\check{C}_1,\leq 0)$ be autonomous.
Then there is an isomorphism of double systems
with product
\begin{equation} \label{equation isomorphism of double systems changing contact cylinder}
\SH^*_{\check{C}_0,Q_-,Q_+}(H)
\lra{\cong} \SH^*_{\check{C}_1,\widetilde{\iota}(Q_-),\widetilde{\iota}(Q_+)}(H).
\end{equation}
\end{prop}

Before we prove this proposition,
we need a preliminary lemma.

\begin{lemma} \label{lemma changing contact cylinder disjiont}
Proposition \ref{proposition changing contact cylinder}
is true when $\check{C}_0$ and $\check{C}_1$ are disjoint
\end{lemma}
\begin{proof}[Proof of lemma \ref{lemma changing contact cylinder disjiont}.]
Let 
$(a^j_-)_{j \in \N}$,
$(a^j_+)_{j \in \N}$ be a cofinal family
of $(\Sc(Q_-^0), \geq)$
and $(\Sc(Q_+^0), \leq)$ respectively.
After passing to a subsequence,
we can assume that the function
$i \to \text{height}(a_-^i,a_+^i)$
is increasing.

Since $H$ has slope $0$ along both contact cylinders and $\check{C}_0 \cap \check{C}_1 = \emptyset$
we can find a cofinal family of Hamiltonians
$H_i \in \cap_{k=0,1} \cap_{j \in \N} \ccH^{\T,\reg}(<_{\check{C}_k} H,a_-^j,a_+^j)$, $i \in \N$
so that 
$H_i|_{M-D_1}$ is constant and
$H_i|_{M - D_1} > \text{height}(a_-^i,a_+^i)$ for all $i \in \N$.
Then the Floer chain complexes
computing
$HF^*_{\check{C}_1,a_-^j,a_+^j}(H_i)$
and
$HF^*_{\check{C}_0,a_-^j,a_+^j}(H_i)$
are identical for each $i,j \in \N$ satisfying $i \geq j$ by Corollary \ref{corollary action of capped loop disjoint from contact cylinder}. The continuation maps, action maps and pair of pants product maps coincide under such an isomorphism.
This gives us our isomorphism (\ref{equation isomorphism of double systems changing contact cylinder}) by Lemma \ref{lemma isomorphism condition}.
\end{proof}

\begin{proof}[Proof of Proposition \ref{proposition changing contact cylinder}.]

Let
$$\check{C}_j = [1-\epsilon_j,1+\epsilon_j] \times C_j, \quad j=0,1$$
be our contact cylinders.
By a Moser argument (\cite[Exercise 3.36]{McduffSalamon:sympbook}), we can extend them
to contact cylinders
$$\check{C}'_j = [1-\epsilon_j-\delta,1+\epsilon_j+\delta] \times C_j, \quad j=0,1$$
for some $0<\delta < \min_{j=0,1} (\epsilon_j/2)$.
We can also assume that $\delta$ is small enough so that
$[1-\delta,1+\delta] \times C_1$ is disjoint from $[1+\epsilon_0 + \delta/2, 1+ \epsilon_0 + \delta] \times C_0$.
We now apply Lemma \ref{lemma changing contact cylinder disjiont}
four times, with four different pairs of contact cylinders in the following order:
\begin{enumerate}
		\item 
		$\check{C}_0$,
		$[1 + \epsilon_0 + \delta/2,1+\epsilon_0 + \delta] \times C_0$,
	\item 
	$[1 + \epsilon_0 + \delta/2,1+\epsilon_0 + \delta] \times C_0$,
	$[1 - \delta,1+ \delta] \times C_1$,
	\item
	$[1 - \delta,1+ \delta] \times C_1$,
	$[1 + \epsilon_1 + \delta/2,1+\epsilon_1 + \delta] \times C_1$ and
	\item $[1 + \epsilon_1 + \delta/2,1+\epsilon_1 + \delta] \times C_1$,
	$\check{C}_1$.
\end{enumerate}
This completes the proposition.
\end{proof}

We also need a proposition telling us what to do when we forget the contact cylinder.

\begin{defn} \label{definition Novikov ring for quantum cohomology}
	Let $\check{C} = \emptyset$ be the empty contact cylinder.
	The {\it standard $\omega$-cone} is
	the cone
	$Q_\omega \subset H^2(M,\emptyset;\R) \times \R \times \R$ defined to
	be the $1$-dimensional cone spanned by
	$([\omega],1,1)$.
	The {\it standard Novikov ring} $\Lambda_\K^\omega$
	is defined to be $\Lambda_\K^{Q_\omega}$.
	The {\it standard positive Novikov ring} $\Lambda_\K^{\omega,+}$
	is defined to be $\Lambda_\K^{Q_\omega,+}$.
\end{defn}

Note that the standard Novikov ring is equal to the Novikov ring (\ref{equation novikov ring omegaX}) in the introduction
with $\omega_X$ replaced by $\omega$.

\begin{prop} \label{proposition forgetting contact cylinder}
Let $\check{C} = [1-\epsilon,1+\epsilon] \times C$
be a contact cylinder with associated Liouville domain $D$.
Let $Q_{\omega_{\check{C}}} \subset H^2(M,D;\R) \times \R^2$ be the cone spanned by $([\omega_{\check{C}}],1,1)$ where
$\omega_{\check{C}}$ is a $\check{C}$-compatible $2$-form
with scaling constants $0$ and $1$ and which is equal to $\omega$ outside $D \cup ([1,1+\epsilon/2] \times C)$.
Then for any autonomous $H \in \ccH^{\T,\ls}(\check{C},\leq 0)$
(Definition \ref{definition symplectic cohomology}) there is an isomorphism of double systems
with product
\begin{equation} \label{equation isomorphism of double systems removing contact cylinder}
\SH^*_{\check{C},Q_{\omega_{\check{C}}},Q_{\omega_{\check{C}}}}(H)
\lra{\cong} \SH^*_{\emptyset,Q_\omega,Q_\omega}(H).
\end{equation}
\end{prop}
\proof
Let $\Pi : Q_{\omega_{\check{C}}} \lra{} Q_\omega$ be the restriction of the natural map
$H^2(M,D;\R) \times \R^2 \lra{} H^2(M;\R) \times \R^2$ induced by cohomological restriction.
Choose any cofinal family of Hamiltonians
$H_i \in \ccH^{\T,\reg}(<_{\check{C}} H,a_-,a_+)$.
Then the isomorphism
(\ref{equation isomorphism of double systems removing contact cylinder}) follows from the natural isomorphisms
$HF^*_{\check{C},a_- \circ \Pi,a_+ \circ \Pi}(H) \cong HF^*_{\emptyset,a_-,a_+}(H)$
coming from the fact that the corresponding Floer chain complexes are identical
for each $(a_-,a_+) \in \Sc(Q_\omega) \times \Sc(Q_\omega)$.
\qed

\subsection{Partial Independence of the Hamiltonian.}

Throughout this subsection, $\check{C}$
is a contact cylinder whose associated Liouville domain is $D$.
We will also fix a $\check{C}$-interval domain $(Q_-,Q_+)$.

\begin{prop} \label{proposition continuation map isomorphism}
Let $K \subset \T \times M$ be a closed subset which is contained in $D$ if $D$ is non-empty.
Let $H^\pm \in \ccH^{\T,\ls}(\check{C},K,\leq 0)$
satisfy $H^-_t \leq H^+_t$.
Then the transfer morphism
\begin{equation} \label{equation transfer isomorphism}
\SH^*_{\check{C},Q_-,Q_+}(H^-) \lra{} \SH^*_{\check{C},Q_-,Q_+}(H^+)
\end{equation}
 is an isomorphism in the category
$\sys{\Lambda_\K^{Q_+,+}}$.
\end{prop}

Since $\varinjlim \varprojlim$
is  a functor by Definition \ref{definition inverse direct limit is a functor},
Proposition \ref{proposition continuation map isomorphism} implies that
the transfer map
$SH^*_{\check{C},Q_-,Q_+}(\check{H}) \lra{} SH^*_{\check{C},Q_-,Q_+}(H)$ is an isomorphism.
Before we prove Proposition \ref{proposition continuation map isomorphism}, we need some preliminary lemmas and the following definition.

\begin{defn} \label{definition linear function on cone}
Let $\nu \in \R$ be a constant.
Define
$$L_\nu : H^2(M,D;\R) \times \R \times \R \lra{} \R, \quad L_\nu(q,\lambda^-,\lambda^+) := \lambda^-\nu.$$
\end{defn}

\begin{lemma} \label{lemma adding constant}
	Let 
	$(a_-,a_+) \in \Sc(Q_-) \times \Sc(Q_+)$ be a $\check{C}$-action interval, $\nu \in (0,\infty)$ and
	let
	$H \in \ccH^{\T,\ls}(\check{C},a_- + L_\nu,a_+ + L_\nu)$.
	Then we have a commutative diagram:
	\begin{center}
		\begin{tikzpicture}
		\node at (-1.7,0.5) {$HF^*_{(\check{C},a_-,a_+)}(H)$};
		\node at (-5.5,0.5) {$HF^*_{\check{C},a_-,a_+}(H-\nu)$};
		\node at (-5.5,-0.5) {$HF^*_{\check{C},a_-+L_\nu,a_+ + L_\nu}(H)$};
		\draw [->] (-5.4,0.2) node [] {} -- (-5.4,-0.3) node [] {};
		\node [] at (-5.1416,-0.0908) {$\cong$};
		\draw [->] (-3.7527,0.5) node [] {} -- (-3.2,0.5) node [] {};
		\node [] at (-3.4387,0.7323) {$\alpha$};
		
		\draw [->] (-3.6352,-0.3973) node [->] {} -- (-2.3,0.1) node [->] {};
		\node [] at (-2.9,-0.4) {$\beta$};
		\node at (-5.7276,-0.0908) {$\gamma$};
		\end{tikzpicture}
	\end{center}
	
	where $\alpha$ is a continuation map
	and $\beta$ is an action map.
	The map $\gamma$ commutes with continuation maps and action maps.
\end{lemma}

\begin{defn} \label{definition translation isomorphism}
	We will call the map $\gamma$
	a {\it translation isomorphism}.
\end{defn}

\begin{proof}[Proof of Lemma \ref{lemma adding constant}.]
	We will prove this lemma in two steps.
	In the first step, we will prove it in the case when $H \in \ccH^\reg(\check{C},a_-+L_\nu,a_++L_\nu)$
	(Definition \ref{defn chain complex})
	and the second step will deal with the general case.
	
	{\it Step 1}. Suppose $H \in  \ccH^\reg(\check{C},a_-+L_\nu,a_++L_\nu)$.
	For $N \in \N$ large enough,
	we have that
	$H - c \in \ccH^\reg(\check{C},a_-+L_\nu,a_++L_\nu)$ for all $c \in [0,\frac{2\nu}{N}]$.
	Therefore $H - (N-k+s)\nu/N \in \ccH^\reg(\check{C},a_- + k L_\nu/N,a_+ + kL_\nu/N)$ for all $k \in \Z$ and all
	$s \in [0,1]$.
	Define
	$$B_{k,l} :=
	HF^*_{\check{C},a_- + k L_\nu/N,a_+ + kL_\nu/N}(H - (N-l) \nu/N)$$
	for all
	$l,k \in \{0,\cdots,N\}$.
	Then 
	by Lemmas \ref{lemma continuation map filtration isomorphism} and \ref{lemma action morphism is isomorphism}
	we have a commutative diagram consisting of continuation maps and action maps (with the exception of $\gamma$ which is given as a definition):
	
	\begin{center}
		\begin{tikzpicture}
		
		\node at (-2,2) {$B_{0,0}$};
		\node at (-2,1.2) {$B_{1,0}$};
		\node at (-2,0.4) {$B_{1,1}$};
		\node at (-2,-0.4) {$B_{2,1}$};
		\node at (-2,-1.2) {$B_{2,2}$}; 
		\node at (-2,-2) {$B_{3,2}$};
		\node at (-2,-3.2) {$B_{N-1,N-1}$};
		\node at (-2,-4) {$B_{N,N-1}$};
		\node at (-2,-4.8) {$B_{N,N}$};
		\node at (1.8,2) {$B_{0,N}$};
		
		\node at (-0.2,1.2) {$B_{1,N}$};
		\node at (-0.2,0.4) {$B_{1,N}$};
		\node at (-0.2,-0.4) {$B_{2,N}$};
		\node at (-0.2,-1.2) {$B_{2,N}$};
		\node at (-0.2,-2) {$B_{3,N}$};
		\node at (-0.2,-3.2) {$B_{N-1,N}$};
		\node at (-0.2,-4) {$B_{N,N}$};
		\draw [->](-1.4,2) -- (1.2,2);
		\draw [->](-1.4,1.2) -- (-0.8,1.2);
		\draw [->](-1.4,0.4) -- (-0.8,0.4);
		\draw [->](-1.4,-0.4) -- (-0.8,-0.4);
		\draw [->](-1.4,-1.2) -- (-0.8,-1.2);
		\draw [->](-1.4,-2) -- (-0.8,-2);
		\draw [->](-1.2,-3.2) -- (-0.9,-3.2);
		\draw [->](-1.2,-4) -- (-0.8,-4);
		\draw [->](-1.2,-4.8) .. controls (-0.6,-4.8) and (4.2,-6.2) .. (2.4,1.6);
		\draw [->](0.4,-4) .. controls (1.3,-4) and (2.6,0) .. (2.2,1.6);
		\draw [->](0.5,-3.2) .. controls (0.8,-3.2) and (2.4,-0.6) .. (2,1.6);
		\draw [->](0.2,-2) .. controls (0.6,-1.8) and (1.8,-1) .. (1.8,1.6);
		\draw [->](0.2,-1.2) .. controls (0.4,-1) and (1.4,-0.4) .. (1.6,1.6);
		\draw [->](0.2,-0.4) .. controls (0.6,-0.2) and (1,1) .. (1.4,1.6);
		\draw [->](0.2,0.4) .. controls (0.4,0.6) and (0.8,1.4) .. (1.2,1.6);
		\draw [->](0.2,1.2) .. controls (0.4,1.2) and (0.8,1.6) .. (1.2,1.8);
		\draw [->](-1.9,1.4) -- (-1.9,1.7);
		\draw [->](-1.9,0.9) -- (-1.9,0.6);
		\draw [->](-1.9,-0.2) -- (-1.9,0.1);
		\draw [->](-1.9,-0.7) -- (-1.9,-0.9);
		\draw [->](-1.9,-1.4) -- (-1.9,-1.7);
		\draw [->](-1.9,-2.3) -- (-1.9,-2.5);
		
		\draw [->](-1.9,-3.8) -- (-1.9,-3.4);
		\draw [->](-0.2,-3.8) -- (-0.2,-3.4);
		\draw [->](-1.9,-4.3) -- (-1.9,-4.6);
		\draw [->](-0.2,-0.2) -- (-0.2,0.2);
		\draw [->](-0.2,-1.8) -- (-0.2,-1.4);
		\draw [double equal sign distance](-0.2,-2.4) -- (-0.2,-2.2);
		\draw [double equal sign distance](-0.2,-0.9) -- (-0.2,-0.6);
		\draw [double equal sign distance](-0.2,0.6) -- (-0.2,0.9);
		\node at (-0.3,2.3) {$\alpha$};
		\node at (2.2,-4.3) {$\beta$};
		\draw [->](-2.4,1.9) .. controls (-2.7,1.2) and (-3.3,-4.8) .. (-2.5,-4.8);
		\node at (-3.1,-1.4) {$\gamma$};
		\node at (-1.6,1.6) {$\cong$};
		\node at (-1.6,0.8) {$\cong$};
		\node at (-1.6,-0.1) {$\cong$};
		\node at (-1.6,-0.8) {$\cong$};
		\node at (-1.6,-1.6) {$\cong$};
		\node at (-1.6,-2.3) {$\cong$};
		
		\node at (-1.6,-3.6) {$\cong$};
		\node at (-1.6,-4.4) {$\cong$};
		\node at (-3.2,-2.2) {$\cong$};
		\node at (-0.9,-2.5) {$\vdots$};
		
		\node at (-1.9,-2.7) {$\vdots$};
		\node at (-0.2,-2.7) {$\vdots$};
		\end{tikzpicture}
	\end{center}
	
	The map $\gamma$ commutes with continuation maps and action maps since the diagram above is functorial with respect to continuation and action maps.
	
	\smallskip
	
	{\it Step 2}:
	Now suppose $H \in \ccH^{\T,\ls}(\check{C},a_- + L_\nu,a_+ + L_\nu)$.
	Then
	since
	$$HF^*_{\check{C},a_-,a_+}(H-\nu) = \varinjlim_{\check{H} \in \ccH^\reg(\check{C},a_-+L_\nu,a_++L_\nu)}
	HF^*_{\check{C},a_-,a_+}(\check{H}-\nu),$$
	$$HF^*_{\check{C},a_- + L_\nu,a_+ + L_\nu}(H) = \varinjlim_{\check{H} \in \ccH^\reg(\check{C},a_-+L_\nu,a_++L_\nu)}
	HF^*_{\check{C},a_- + L_\nu,a_+ + L_\nu}(\check{H})$$
	and continuation maps commute with the translation isomorphisms constructed in Step 1,
	we get our result.
\end{proof}

\bigskip

\begin{lemma} \label{lemma proposition for constant hamiltonians}
Let $H \in \ccH^{\T,\ls}(\check{C},\leq 0)$ and let $\nu>0$ be a constant.
	Then the transfer morphism
	$\SH^*_{\check{C},Q_-,Q_+}(H - \nu) \lra{} \SH^*_{\check{C},Q_-,Q_+}(H)$ is an isomorphism in $\sys{\Lambda_\K^{Q_+,+}}$.
\end{lemma}
\proof
For each action interval
$(a_-,a_+) \in \Sc(Q_-) \times \Sc(Q_+)$,
we have a translation isomorphism
$$\gamma_{a_-,a_+} : HF^*_{\check{C},a_-,a_+}(H-\nu)
\lra{} HF^*_{\check{C},a_- + L_\nu,a_++ L_\nu}(H).$$
Since such isomorphisms commute with
continuation maps, we get an isomorphism
$$\gamma : \SH^*_{\check{C},Q_-,Q_+}(H-\nu) \lra{}
\SH^*_{\check{C},Q_-,Q_+}(H).$$
Also for each action interval
$(a_-,a_+) \in \Sc(Q_-) \times \Sc(Q_+)$,
we have an action morphism
$$\beta_{a_-,a_+} : HF^*_{\check{C},a_- + L_\nu,a_++ L_\nu}(H)
\lra{} HF^*_{\check{C},a_-,a_+}(H)$$
and this induces an isomorphism
$$\beta : \SH^*_{\check{C},Q_-,Q_+}(H)
\lra{} \SH^*_{\check{C},Q_-,Q_+}(H)$$
(such an isomorphism is induced by a standard endomorphism as in Example \ref{example morphisms}).
Also by Lemma \ref{lemma adding constant},
$$\beta_{a_-,a_+} \circ \gamma_{a_-,a_+} :
HF^*_{\check{C},a_-,a_+}(H-\nu) \lra{}
HF^*_{\check{C},a_-,a_+}(H)
$$
is a continuation map for all $a_-,a_+ \in \Sc(Q)$
giving us a transfer morphism
$$\SH^*_{\check{C},Q_-,Q_+}(H-\nu) \lra{} \SH^*_{\check{C},Q_-,Q_+}(H).$$
Our lemma now follows from the fact that $\beta \circ \gamma$ is a composition of isomorphisms in $\sys{\Lambda_\K^{Q,+}}$.
\qed

\begin{proof}[Proof of Proposition
	\ref{proposition continuation map isomorphism}.]
	
	Choose a constant $\nu > 0$
	so that $H^+ - \nu \leq H^-$.
	Then since $H^\pm \in \ccH^{\T,\ls}(\check{C},\leq 0)$ we get
	$H^+ - \nu \leq_{\check{C}} H^-$.
	We then have transfer morphisms
	$$\SH^*_{\check{C},Q_-,Q_+}(H^+ - 2\nu)
	\lra{}
	\SH^*_{\check{C},Q_-,Q_+}(H^- - \nu)
	\lra{\alpha}
	\SH^*_{\check{C},Q_-,Q_+}(H^+ - \nu)
	\lra{}
	\SH^*_{\check{C},Q_-,Q_+}(H^-).$$
	By Lemma \ref{lemma proposition for constant hamiltonians},
	the composition of any two such morphisms is an isomorphism and hence
	$\alpha$ is an isomorphism.
	Also by Lemma \ref{lemma proposition for constant hamiltonians},
	the natural continuation morphisms
	$$\beta_\pm : \SH^*_{\check{C},Q_-,Q_+}(H^\pm - \nu) \lra{} \SH^*_{\check{C},Q_-,Q_+}(H^\pm)$$ are isomorphisms.
	Our Proposition now follows from the fact that the continuation map (\ref{equation transfer isomorphism})
	is equal to the following composition of isomorphisms:	$\beta_+ \circ \alpha \circ (\beta_-)^{-1}$.

\end{proof}

\subsection{Relation with Quantum Cup Product}

\begin{theorem} \label{theorem isomorphic to quantum cohomology}
	Let $Q_\omega$, $\Lambda^\omega_\K$ be the standard $\omega$-cone
	and standard Novikov ring respectively
	as in Definition \ref{definition Novikov ring for quantum cohomology}.
	Suppose that $\K$ is a field.
	Then there is an isomorphism of $\Lambda_\K^\omega$-algebras
$$
SH^*_{\emptyset,Q_\omega,Q_\omega}(M \subset M) \cong QH^*(M,\Lambda_\K^{\omega})
$$
	where $QH^*$ is quantum cohomology.
	Also,
	$\varinjlim \varprojlim^1 \SH^*_{\emptyset,Q_\omega,Q_\omega}(M \subset M) = 0$
	where $\varinjlim \varprojlim^1$ is given in Definition \ref{definition lim invlim functor}.
\end{theorem}


\proof
We identify the cone $Q_\omega$
with $[0,\infty)$ via the identification
$\lambda([\omega],1,1) \lra{} \lambda$ for all $\lambda \geq 0$.
Under this identification,
$\Sc(Q_\omega)$ becomes the space
of linear functions
$a_c : [0,\infty) \lra{} \R, \quad a_c(\lambda) := c\lambda$,
$c \in \R$.

We will first show
$\varinjlim \varprojlim^1 \SH^*_{\emptyset,Q_\omega,Q_\omega}(M \subset M) = 0$ by using Lemma
\ref{lemma mittag leffler like test}.
Let $H : M \lra{} \R$ be a $C^2$ small Morse function which is negative.
Since $[\omega]$ lifts
to an integral cohomology class in
$H^2(M;\Z)$,
we have for all $c_\pm \in \frac{1}{2}+\Z$
that the continuation maps
\begin{equation} \label{equation continuation map isomorphism for small Hamiltonians}
HF^*_{\emptyset,a_{c_-},a_{c_+}}(\lambda_1 H) \lra{} HF^*_{\emptyset,a_{c_-},a_{c_+}}(\lambda_2 H)
\end{equation}
are isomorphisms for all $0 < \lambda_2 \leq \lambda_1 \leq \frac{1}{4}$
by Lemma \ref{lemma continuation map filtration isomorphism}.
Hence the double system
$W := (HF^*_{\emptyset,a_{c_-},a_{c_+}}(\lambda H))_{c_\pm \in \frac{1}{2} + \Z}$
for some small $\lambda > 0$
is isomorphic to the double system
$(\varinjlim_{\check{\lambda} \to 0}HF^*_{\emptyset,a_{c_-},a_{c_+}}(\check{\lambda} H))_{c_\pm \in \frac{1}{2} + \Z}$ which, in turn is isomorphic to
$\SH^*_{\emptyset,Q_\omega,Q_\omega}(M \subset M) = 0$ by Lemma \ref{lemma isomorphism condition}.
Therefore
by Lemma
\ref{lemma mittag leffler like test}
it is sufficient
for us to show that the action map
\begin{equation} \label{equation action map for quantum cohomology}
HF^*_{\emptyset,a_{c_-},a_{c_+}}(\lambda H)
\lra{}
HF^*_{\emptyset,a_{c_-},a_{c_+-1}}(\lambda H)
\end{equation}
is surjective
for each $c_\pm \in \frac{1}{2} + \Z$.
Since $\lambda H$ is $C^2$ small,
we have by \cite[Theorem 10.1.1]{audin2014morse}
that the only Floer trajectories connecting
$1$-periodic orbits of $\lambda H$
are Morse flowlines (I.e. $t$ independent Floer trajectories).
This implies that the natural map
(\ref{equation action map for quantum cohomology})
is surjective
for each $c_\pm \in \frac{1}{2} + \Z$
which in turn implies that
$\varinjlim \varprojlim^1 \SH^*_{\emptyset,Q_\omega,Q_\omega}(M \subset M) = 0$
by Lemma \ref{lemma mittag leffler like test}.

We will now prove the first part of the lemma.
Choose $J \in \cap_{c_\pm \in \frac{1}{2} + \Z} \ccJ^{\T,\reg}(\lambda H,a_{c_-},a_{c_+})$.
Define
$$HF^*(\lambda H) := H_*\left(\varinjlim_{c_-} \varprojlim_{c_+} CF^*_{\emptyset,a_{c_-},a_{c_+}}(\lambda H)\right).$$
Since $\lim \lim^1 W =0$, we have a natural isomorphism 
$HF^*(\lambda H) = 
\varinjlim_{a_-}
\varprojlim_{a_+} HF^*_{\emptyset,a_-,a_+}(\lambda H)$
by
\cite[Theorem 3.5.8]{weibel1995introduction}
and Lemma \ref{lemma isomorphism condition}.
Since the morphism (\ref{equation continuation map isomorphism for small Hamiltonians}) is an isomorphism
for all $c_\pm \in \frac{1}{2}\Z$ and $0 < \lambda_2 < \lambda_1 \leq \frac{1}{4}$,
we get that the natural map
\begin{equation} \label{equation isomorphism with smyplectic cohomology}
HF^*(\lambda H) = \varinjlim_{a_-} \varprojlim_{a_+} HF^*_{\emptyset,a_-,a_+}(\lambda H)
\lra{} \varinjlim_{a_-} \varprojlim_{a_+} \varinjlim_{\check{\lambda} \to 0} HF^*_{\emptyset,a_-,a_+}(\check{\lambda} H) = SH^*_{\emptyset,Q_\omega,Q_\omega}(M \subset M)
\end{equation}
is an isomorphism for all sufficiently small $\lambda > 0$.
By 
\cite{piunikhin1996symplectic}
and \cite{oh2011floer}
we have isomorphisms
\begin{equation} \label{equation isomorphism with cohomology}
HF^*(\lambda H) \cong H^*(M;\Lambda^\omega_\K)
\end{equation}
for all $\lambda > 0$ small which commute with continuation maps
$HF^*(\lambda H) \lra{} HF^*(\lambda' H)$ for $0 < \lambda < \lambda'$ small.
%
These isomorphisms also commute with the pair of pants product as follows.
If $\lambda > 0$ is sufficiently small,
we have a commutative diagram of $\Lambda^{\omega}_\K$-modules

\begin{center}
\begin{tikzpicture}

\node at (0,0) {$(HF^*(\lambda H)) \otimes (HF^*(\lambda H))$	};

\node at (7,0) {$HF^*(2\lambda H)$};

\node at (0,-1.5) {$QH^*(M;\Lambda^{\omega}_\K) \otimes QH^*(M;\Lambda^{\omega}_\K)$};

\node at (7,-1.5) {$QH^*(M;\Lambda^{\omega}_\K)$};

\draw [->](2.45,0) -- (5.75,0);

\draw [->](2.55,-1.5) -- (5.75,-1.5);

\draw [->](-1.25,-0.5) -- (-1.25,-1);

\draw [->](1,-0.5) -- (1,-1);

\draw [->](6.75,-0.5) -- (6.75,-1);

\node at (-1.5,-0.75) {$\cong$};

\node at (0.75,-0.75) {$\cong$};

\node at (6.5,-0.75) {$\cong$};

\node at (4.25,0.25) {$\alpha$};

\node at (4.25,-1.25) {$\beta$};

\end{tikzpicture}

\end{center}
where $\alpha$
is the pair of pants product and $\beta$
is the quantum cup product and where the vertical isomorphisms are induced by (\ref{equation isomorphism with cohomology}) (the papers \cite{piunikhin1996symplectic}
and \cite{oh2011floer} only prove such isomorphisms over $\Q$, however these proofs are identical if one just formally replaces the coefficient field `$\Q$' with `$\K$' in their papers since $(M,\omega)$ is semi-positive).
Therefore by Equation (\ref{equation isomorphism with smyplectic cohomology})
we have an isomorphism of
 $\Lambda^\omega_{\K}$-algebras
$$SH^*_{\emptyset,Q_\omega,Q_\omega}(M \subset M) = 
\varinjlim_{\lambda \to 0}
HF^*(\lambda H) \cong
QH^*(M;\Lambda^{\omega}_\K)$$
since the isomorphism
(\ref{equation isomorphism with smyplectic cohomology}) commutes with the pair of pants product $\alpha$.


\qed

\subsection{Stably Displaceable Complements}

\begin{defn} \label{definition stably displaceable}
	A subset $B$ of a symplectic manifold is {\it Hamiltonian displaceable}
	if there is a Hamiltonian symplectomorphism $\phi$ satisfying $\phi(B) \cap B = \emptyset$.
	A subset $A \subset M$ is {\it stably displaceable} if $A \times \T \subset M \times T^* \T$ is Hamiltonian displaceable
	inside the product symplectic manifold
	$M \times T^* \T$ where $T^*\T = \R \times \T$ has the standard symplectic form $d\sigma \wedge d\tau$ where $\sigma : \R \times \T \lra{} \R$, $\tau : \R \times \T \lra{} \T$ are the natural projection maps.
\end{defn}

Throughout this subsection,
we fix the coordinates $\sigma,\tau$ above.
We also let $Q_\omega$ be the standard Novikov cone associated to the empty contact cylinder
as in Definition \ref{definition Novikov ring for quantum cohomology}.
The aim of this subsection is to prove the following proposition.

\begin{theorem}  \label{theorem stably displaceable complement}
Let $K \subset M$ be a closed set
so that $\overline{M-K}$ is stably displaceable.
Then the transfer morphism
\begin{equation} \label{equation transfer isomorphism displaceable}
\SH^*_{\emptyset,Q_\omega,Q_\omega}(M \subset M) \lra{} \SH^*_{\emptyset,Q_\omega,Q_\omega}(K \subset M)
\end{equation}
is an isomorphism in $\sys{\Lambda_\K^{Q_\omega,+}}$.
\end{theorem}

The proof of this proposition
relies heavily on an idea due to Ginzburg in \cite{GinzburgConleyConjecture}.
Before we prove this proposition,
we need some preliminary definitions and lemmas.

\begin{defn} \label{definition hamiltoniansonproduct}

A lower semi-continuous Hamiltonian $H = (H_t)_{t \in \T}$
on $M \times T^*\T$ is
{\it admissible}
if there is a compact subset $K_H \subset M \times T^* \T$ and a lower semi-continuous Hamiltonian 
$K = (K_t)_{t \in \T}$ on $M$
so that
$H_t(x,(\sigma,\tau)) = K_t(x) - \frac{1}{2}|\sigma|$
for all $(x,(\sigma,\tau)) \in M \times T^* \T - K_H$.

Let $J_{T^*\T}$ be the almost complex structure on $\R \times \T$
satisfying $J_{T^* \T}(\frac{\partial}{\partial \sigma}) = \frac{\partial}{\partial \tau}$.
A smooth family of almost complex structures
$J = (J_t)_{t \in \T}$
on $M \times T^* \T$
is {\it admissible}
if they are $\omega + d\sigma \wedge d\tau$-tame
and
if they are equal to $J_M \oplus J_{T^* \T}$
outside a compact subset of $M \times T^* \T$
where $J_M$ is an $\omega$-tame almost complex structure on $M$.

For any admissible lower semi-continuous Hamiltonian
$H$ on $M \times T^* \T$ and any
$a_-,a_+ \in \Sc(Q_\omega)$,
we can define
$$HF^*_{\emptyset,a_-,a_+}(H)$$
as in Definition \ref{definition of hamiltonian floer cohomology for lower semicontinuous hamiltonians}
where we restrict ourselves to admissible Hamiltonians and almost complex structures. We can also define continuation maps and action maps in the same way.
\end{defn}

\begin{remark}
Such a definition, along with the theorems, propositions and lemmas used to construct such a definition, are identical except that $M$ is replaced by $M \times T^*\T$
and all Hamiltonians and almost complex structures involved are admissible.
The only additional ingredient needed is that one needs a maximum principle
to prove compactness to ensure that all Floer trajectories stay inside a fixed compact subset of $M \times T^* \T$
(see, for example \cite[Lemma 1.5]{Oancea:survey}).
Also note that since the natural projection map
$M \times T^* \T \lra{} M$ induces an isomorphism
$H_2(M \times T^*\T;\Z) \lra{\cong} H_2(M;\Z)$, we have that $\Lambda_\K^{Q_\omega}$ and
$\Lambda_\K^{Q_\omega,+}$
are the correct Novikov rings to use
(and not some larger Novikov rings).
\end{remark}

From now on we identify $H_2(M \times T^*\T;\Z)$ with $H_2(M;\Z)$
as in the remark above.

\begin{defn}
A closed subset $K \subset \T \times M \times T^* \T$
is {\it admissible} if there is a compact set $\kappa \subset \T \times M \times T^* \T$
and a closed subset $K_M \subset M$ so that
$K \cup \kappa = (\T \times K_M \times T^*\T) \cup \kappa$.
In other words, this closed subset is a product $\T \times K_M \times T^*\T$
near infinity (see Figure \ref{fig:admissibleclosedsubset}).
\end{defn}

\begin{center}
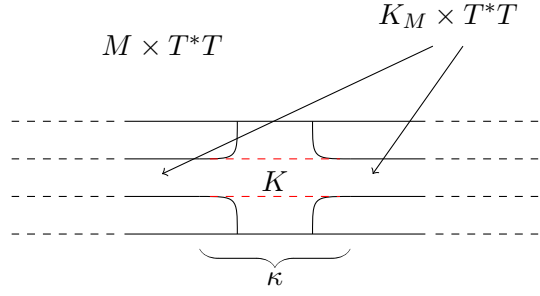
\begin{figure}[h]
\usetikzlibrary{decorations.pathreplacing}
\begin{tikzpicture}

\draw (-3,1) node (v1) {} -- (1,1) node (v3) {};
\draw (-3,-0.5) node (v2) {} -- (1,-0.5) node (v4) {};
\draw [dashed](v1) -- (-4.5,1);
\draw [dashed](v2) -- (-4.5,-0.5);
\draw [dashed](v3) -- (2.5,1);
\draw [dashed](v4) -- (2.5,-0.5);

\draw [](-1.5,1) .. controls (-1.5,0.5) and (-1.5,0.5) .. (-2,0.5);
\draw (-0.5,1) .. controls (-0.5,0.5) and (-0.5,0.5) .. (0,0.5);
\draw (-2,0.5) node (v9) {} -- (-3,0.5) node (v5) {};
\draw (-2,0) -- (-3,0) node (v6) {};
\draw (0,0.5) node (v10) {} -- (1,0.5) node (v7) {};
\draw (0,0) -- (1,0) node (v8) {};
\draw (-1.5,-0.5) .. controls (-1.5,0) and (-1.5,0) .. (-2,0) node (v11) {};
\draw (0,0) node (v12) {} .. controls (-0.5,0) and (-0.5,0) .. (-0.5,-0.5);
\draw [dashed](v5) -- (-4.5,0.5);
\draw [dashed](v6) -- (-4.5,0);
\draw [dashed](v7) -- (2.5,0.5);
\draw [dashed](v8) -- (2.5,0);
\node at (-2.5,2) {$M \times T^*T$};
\draw [->](1.5,2) -- (0.3,0.3);
\draw [->](1.1,2) -- (-2.5,0.3);
\node at (1.3,2.4) {$K_M \times T^*T$};
\draw [dashed,red](v9) -- (v10);
\draw [dashed,red](v11) -- (v12);
\draw [decoration={calligraphic brace,amplitude=7pt},decorate] (0,-0.7) -- (-2,-0.7);
\node at (-1,-1.1) {$\kappa$};
\node at (-1,0.2) {$K$};
\end{tikzpicture}
\caption{Admissible closed subset.} \label{fig:admissibleclosedsubset}
\end{figure}
\end{center}

\begin{defn} \label{defn appropriate lower semi continuous hamiltonians}
For any admissible closed subset $K \subset \T \times M \times T^* \T$,
we define $\ccH^{\T,\ls}(K,\leq 0)$ to be the set of admissible lower semi-continuous Hamiltonians $H = (H_t)_{t \in \T}$
on $M \times T^* \T$
satisfying $H_t(x,(\sigma,\tau)) \leq 0$ for $(t,x,(\sigma,\tau)) \in K$
and $H_t(x,(\sigma,\tau)) = \infty$
for $(t,x,(\sigma,\tau)) \notin K$.
%
A closed subset $K' \subset M \times T^* \T$ is {\it admissible} if $\T \times K'$ is admissible.
For such a $K'$, define $\ccH^{\T,\ls}(K',\leq 0) := \ccH^{\T,\ls}(\T \times K',\leq 0)$.

For any admissible closed subset $K \subset \T \times M \times T^* \T$ and any $H \in \ccH^{\T,\ls}(K,\leq 0)$
we define $\SH^*_{\emptyset,Q_\omega,Q_\omega}(H)$
to be the double system
$$(HF^*_{\emptyset,a_-,a_+}(H))_{a_-,a_+ \in \Sc(Q_\omega)}.$$

If $K_-,K_+ \subset \T \times M \times T^* \T$ are admissible closed subsets satisfying $K_+ \subset K_-$ and
$H^\pm \in \ccH^{\T,\ls}(K_\pm,\leq 0)$
satisfies $H^- \leq H^+$,
then the natural morphism of double systems
$$\SH^*_{\emptyset,Q_\omega,Q_\omega}(H^-) \lra{} \SH^*_{\emptyset,Q_\omega,Q_\omega}(H^+)$$
induced by continuation maps
is called a {\it transfer morphism}.

\end{defn}

We will need the following lemma
whose proof is identical to the proof of
Lemma \ref{lemma continuation map isomorphism action gap} in the case when $\check{C} = \emptyset$ and $Q = Q_\omega$,
except that all Hamiltonians and almost complex structures are admissible on $M \times T^* \T$.

\begin{lemma} \label{lemma continuation map isomorphism action gap on M times TT}
Let $H = (H_{s,t})_{(s,t) \in [0,1] \times \T}$
be a smooth family of autonomous admissible Hamiltonians on $M \times T^*\T$
and define $H_{s,\bullet} := (H_{s,t})_{t \in \T}$ for all $s \in [0,1]$.
Fix $p \in \Z$ and $a_\pm \in \Sc(Q_\omega)$.
Suppose 
\begin{itemize}
	\item 
	$H_{s_1,\bullet} \leq H_{s_2,\bullet}$
	for all $s_1 \leq s_2$ and
	\item that there are neighborhoods
	$N_-$, $N_+$ of $a_-$, $a_+$ in $\Sc(Q_\omega)$ respectively so that
	$$\Gamma^P_{\emptyset,a_-,a_+}(H_{s,\bullet}) = \Gamma^P_{\emptyset,a'_-,a'_+}(H_{s,\bullet})$$
	(Definition \ref{defn chain complex})
	for all $a'_\pm \in N_\pm$
	where
	$P = [p-1-n,p+1+n]$
	for all $s \in \R$.
\end{itemize}
Then the continuation map
$$\Phi^p_{H_{0,\bullet},H_{1,\bullet}} : HF^p_{\check{C},a_-,a_+}(H_{0,\bullet}) \lra{}
HF^p_{\check{C},a_-,a_+}(H_{1,\bullet})$$
in degree $p$ is an isomorphism.
\end{lemma}

We also have the following proposition whose proof is identical to the proof of Proposition \ref{proposition continuation map isomorphism},
except that all Hamiltonians, almost complex structures and closed subsets are admissible on $M \times T^* \T$.

\begin{prop} \label{proposition continuation map isomorphism for M times TT}

Let $K \subset \T \times M \times T^* \T$ be an admissible closed subset.
Let $H^-,H^+ \in \ccH^{\T,\ls}(K,\leq 0)$ satisfy $H^- \leq H^+$.
Then the transfer morphism
$$\SH^*_{\emptyset,Q_\omega,Q_\omega}(H^-) \lra{} \SH^*_{\emptyset,Q_\omega,Q_\omega}(H^+)$$
is an isomorphism in $\sys{\Lambda_\K^{Q_\omega,+}}$.
\end{prop}

\begin{lemma} \label{lemma product with s1}
Let $K_M \subset M$ be closed and define $K := K_M \times T^*\T \subset M \times T^*\T$ and let
$H \in \ccH^{\T,\ls}(K,\leq 0)$.
Then there is an isomorphism of double systems
\begin{equation} \label{equation symplectic cohomology of product}
\Psi_K : \SH^*_{\emptyset,Q_\omega,Q_\omega}(H) \lra{}
\SH^*_{\emptyset,Q_\omega,Q_\omega}(K_M \subset M) \otimes_\K H^*(\T;\K)
\end{equation}
where $H^*(\T;\K)$ is thought of as a double system $(H^*(\T;\K))_{(i,k) \in \{\star\} \times \{\star\}}$
where $\{\star\}$ is the single element (inverse) directed set.
Such a morphism commutes with transfer maps. In other words, for any admissible closed subsets $K^+ = K_M^+ \times T^*\T \subset K^- = K_M^- \times T^*\T \subset M \times T^*\T$ and any $H^\pm \in \ccH^{\T,\ls}(K^\pm,\leq 0)$ satisfying $H^- \leq H^+$, we have the following commutative diagram in $\sys{\Lambda_\K^{Q_\omega,+}}$:
\begin{center}
\begin{tikzpicture}

\node at (-0.9,-0.5) {$\SH^*_{\emptyset,Q_\omega,Q_\omega}(H^-)$};

\node at (4.9,-0.5) {$\SH^*_{\emptyset,Q_\omega,Q_\omega}(K_M^- \subset M) \otimes_\K H^*(\T;\K)$};

\node at (-0.9,-1.5) {$\SH^*_{\emptyset,Q_\omega,Q_\omega}(H^+)$};

\node at (4.9,-1.5) {$\SH^*_{\emptyset,Q_\omega,Q_\omega}(K_M^+ \subset M) \otimes_\K H^*(\T;\K)$};

\draw [->](-0.6,-0.8) -- (-0.6,-1.3);

\draw [->](4.6,-0.8) -- (4.6,-1.2);

\draw [->](0.5,-0.5) -- (1.8,-0.5);

\draw [->](0.5,-1.5) -- (1.8,-1.5);

\node at (1.15,-0.2) {$\Psi_{K^-}$};

\node at (1.15,-1.2) {$\Psi_{K^+}$};

%

\end{tikzpicture}
\end{center}
where the vertical morphisms are induced by transfer morphisms.
\end{lemma}
\begin{proof}[Proof of Lemma \ref{lemma product with s1}.]

Let $f : \R \lra{} \R$ be a smooth function so that
\begin{itemize}
\item $f(0)=0$, $f''(0) < 0$,
\item $f'(x) > 0$ for $x < 0$ and $f'(x) < 0$ for $x  > 0$ and
\item $f(x) = -\frac{1}{2}|x|$ for $|x| \geq 1$ (See Figure \ref{fig:graphoff}).
\end{itemize}
\begin{center}
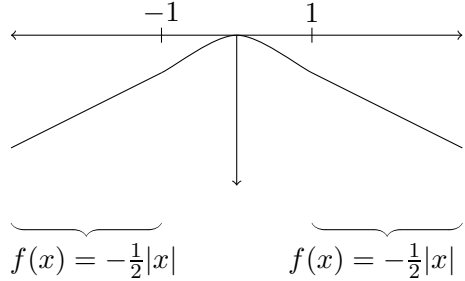
\begin{figure}[h]
\begin{tikzpicture}

\draw [<->](-4.5,0.5) -- (1.5,0.5);
\draw [<-](-1.5,-1.5) -- (-1.5,0.5);
\draw[decoration={calligraphic brace,amplitude=5pt},decorate]  (-2.5,-2) -- (-4.5,-2);
\draw[decoration={calligraphic brace,amplitude=5pt},decorate] (1.5,-2) -- (-0.5,-2);
\draw (-2.5,0.4) -- (-2.5,0.6);
\draw (-0.5,0.4) -- (-0.5,0.6);

\node at (-2.5,0.8) {$-1$};
\node at (-0.5,0.8) {$1$};

\node at (-3.4,-2.5) {$f(x)=-\frac{1}{2}|x|$};
\node at (0.3,-2.5) {$f(x)=-\frac{1}{2}|x|$};
\draw (-0.5,0) -- (1.5,-1);
\draw (-2.5,0) -- (-4.5,-1);

\draw (-2.5,0) .. controls (-2.3,0.1) and (-1.8,0.5) .. (-1.5,0.5);
\draw (-0.5,0) .. controls (-0.7,0.1) and (-1.2,0.5) .. (-1.5,0.5);
\end{tikzpicture}
\caption{Graph of $f$.} \label{fig:graphoff}
\end{figure}
\end{center}

Define $$H_{K_M} : M \lra{} \R, \quad H_{K_M}(x) = \left\{
\begin{array}{ll}
0 & \text{if} \ x \in K_M \\
\infty & \text{otherwise}
\end{array}
\right.$$
and let $\pi_M : M \times T^*\T \lra{} M$ be the natural projection map.
Then by combining \cite[Proposition 2.2]{CieliebakFloerHoferWysocki:SymhomIIApplications}
with a K\"{u}nneth formula argument,
we get an isomorphism
\begin{equation} \label{equation kunneth isomorphismn for specific hamiltonians}
HF^*_{\emptyset,a_-,a_+}(\pi_M^*(H_{K_M}) + f(\sigma)) \lra{\cong} HF^*_{\emptyset,a_-,a_+}(H_{K_M}) \otimes_\K H^*(\T;\K)
\end{equation}
for all $a_-,a_+ \in \Sc(Q_\omega,Q_\omega)$ since the only capped $1$-periodic orbits of $f(\sigma) : T^* \T \lra{} \R$ are the constant orbits at its maximum.
Such an isomorphism commutes with action maps and continuation maps (possibly tensored with $H^*(\T;\K)$).
Hence we get our isomorphism 
(\ref{equation symplectic cohomology of product}) when
$H = \pi_M^*(H_{K_M}) + f(\sigma)$.
For general $H$ this isomorphism exists by Proposition \ref{proposition continuation map isomorphism for M times TT}.
Again because the isomorphism
(\ref{equation kunneth isomorphismn for specific hamiltonians})
commutes with continuation maps,
we get that the isomorphism
(\ref{equation symplectic cohomology of product})
commutes with transfer morphisms.
\end{proof}

\begin{lemma} \label{lemma cross to product transfer isomorphism}

Let $K_M \subset M$ be a closed subset and let $\nu > 0$.
Let $H^\pm \in \ccH^{\T,\ls}(K^\pm,\leq 0)$ where
$$K_- := (K_M \times T^* \T) \cup (M \times \{ \sigma \geq \nu\}), \quad
K_+ := K_M \times T^*\T$$
(See Figure \ref{fig:kminusplus})
and where $H^- \leq H^+$.
Then the transfer morphism
\begin{equation} \label{equation transfer cross to product}
\SH^*_{\check{C},Q_\omega,Q_\omega}(H^-)
\lra{}
\SH^*_{\check{C},Q_\omega,Q_\omega}(H^+)
\end{equation}
is an isomorphism in $\sys{\Lambda_\K^{Q_\omega,+}}$.
\end{lemma}

\begin{center}
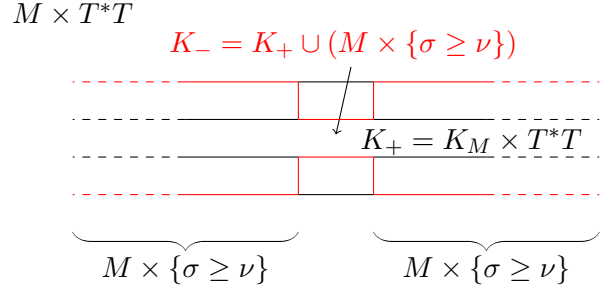
\begin{figure}[h]
	\usetikzlibrary{decorations.pathreplacing}
\begin{tikzpicture}

\draw[red] (-3,1) node (v1) {} -- (-1.5,1) node (v3) {};
\draw[red] (-3,-0.5) node (v2) {} -- (-1.5,-0.5) node (v4) {};
\draw[red] [dashed](-3,1) -- (-4.5,1);
\draw[red] [dashed](-3,-0.5) -- (-4.5,-0.5);
\draw[red] [dashed](1,1) -- (2.5,1);
\draw[red] [dashed](1,-0.5) -- (2.5,-0.5);

\node at (-4.5,1.9) {$M \times T^*T$};
\draw [->](-0.8,1.2) -- (-1,0.3);
\node at (0.8,0.2) {$K_+ = K_M \times T^*T$};

\draw (-3,0.5) node (v7) {} -- (-1.5,0.5) node (v5) {};
\draw (-3,0) node (v8) {} -- (-1.5,0) node (v6) {};
\draw (-0.5,0.5) -- (1,0.5);
\draw (1,0) -- (-0.5,0);
\draw [dashed](1,0.5) -- (2.5,0.5);
\draw[dashed](1,0)-- (2.5,0);
\draw [dashed](-3,0.5) -- (-4.5,0.5);
\draw [dashed](-3,0) -- (-4.5,0);

\draw[red] (-1.5,1) -- (-1.5,0.5);
\draw[red] (-0.5,1) -- (-0.5,0.5);
\draw[decoration={calligraphic brace,amplitude=7pt},decorate] (-1.5,-1) -- (-4.5,-1);
\draw[decoration={calligraphic brace,amplitude=7pt},decorate] (2.5,-1) -- (-0.5,-1);
\node at (-3,-1.5) {$M \times \{\sigma \geq \nu\}$};
\node at (1,-1.5) {$M \times \{\sigma \geq \nu\}$};
\node at (-0.9,1.5) {\color{red} $K_- = K_+ \cup (M \times \{\sigma \geq \nu\})$};
\draw[red] (-1.5,-0.5) -- (-1.5,0);
\draw[red] (-0.5,0) -- (-0.5,-0.5);
\draw (-0.5,1) -- (-1.5,1);
\draw[red] (-0.5,1) -- (1,1);

\draw[red] (-0.5,0.5) -- (-1.5,0.5);
\draw[red] (-1.5,0) -- (-0.5,0);
\draw[red] (-0.5,-0.5) -- (1,-0.5);
\draw (-0.5,-0.5) -- (-1.5,-0.5);
\end{tikzpicture}
\caption{The subsets $K_-$ and $K_+$.} \label{fig:kminusplus}
\end{figure}
\end{center}

\begin{proof}[Proof of Lemma \ref{lemma cross to product transfer isomorphism}.]
Let $\pi : M \times T^*\T \lra{} M$ be the natural projection map.
Let
$(K^k)_{k \in \N}$ be a sequence of smooth functions on $M$ so that
\begin{enumerate}
	\item $K^k \leq K^{k+1}$
	for all $k \in \N$,
	\item $K^k|_{M-K_M} >0$ and $K^k|_{K_M} = 0$ and
	\item for each $x \in M-K_M$, $K^k(x)$ tends to infinity as $k$ tends to infinity.
\end{enumerate}
Let $\rho : \R \lra{} \R$ be a smooth function so that
\begin{enumerate}
\item $\rho|_{[-\nu/2,\nu/2]} = 1$,
\item $\rho|_{(-\nu,\nu)} > 0$, $\rho|_{(-\infty,-\nu] \cup [\nu,\infty)} = 0$ and
\item $\rho'|_{(-\nu,-\nu/2)} > 0$,
$\rho'|_{(\nu/2,\nu)} < 0$ (See Figure \ref{fig:graphofrho}).
\end{enumerate}
\begin{center}
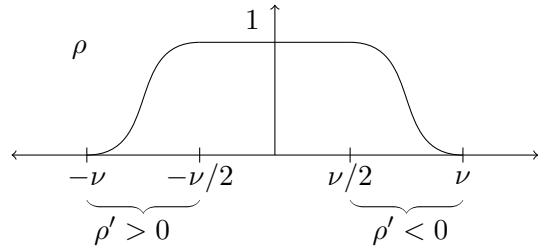
\begin{figure}[h]
	\begin{tikzpicture}
	
	\draw [<->](-5,-1.5) -- (2,-1.5);
	\draw [->](-1.5,-1.5) -- (-1.5,0.5);
	\draw (-2.5,0) -- (-0.5,0);
	\draw (-2.5,0) .. controls (-3.5,0) and (-3,-1.5) .. (-4,-1.5);
	
	\draw (-0.5,0) .. controls (0.5,0) and (0,-1.5) .. (1,-1.5);
	\draw[decoration={calligraphic brace,amplitude=5pt},decorate]  (-2.5,-2.1) -- (-4,-2.1);
	\draw[decoration={calligraphic brace,amplitude=5pt},decorate] (1,-2.1) -- (-0.5,-2.1);
	\node at (-4.1,-0.1) {$\rho$};
	\draw (-2.5,-1.4) -- (-2.5,-1.6);
	\draw (-4,-1.4) -- (-4,-1.6);
	\draw (-0.5,-1.4) -- (-0.5,-1.6);
	\draw (1,-1.4) -- (1,-1.6);
	\node at (-4,-1.8) {$-\nu$};
	\node at (-2.5,-1.8) {$-\nu/2$};
	\node at (-0.5,-1.8) {$\nu/2$};
	\node at (1,-1.8) {$\nu$};
	\node at (-1.8,0.3) {$1$};
	\node at (-3.4,-2.5) {$\rho' >0$};
	\node at (0.3,-2.5) {$\rho' < 0$};
	\end{tikzpicture}
	\caption{Graph of $\rho$.} \label{fig:graphofrho}
	\end{figure}
\end{center}

Let $f : \R \lra{} \R$ be a smooth function so that
\begin{enumerate}
\item $f(0) = 0$,
\item $f'|_{(-\infty,0)} > 0$, $f'|_{(0,\infty)} < 0$ and
\item $f(x) = -\frac{1}{2}|x|$ for $|x| \geq \nu$ (See Figure \ref{fig:graphoff2}).
\end{enumerate}
\begin{center}
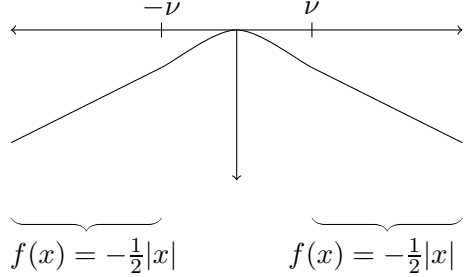
\begin{figure}[h]
\begin{tikzpicture}

\draw [<->](-4.5,0.5) -- (1.5,0.5);
\draw [<-](-1.5,-1.5) -- (-1.5,0.5);
\draw[decoration={calligraphic brace,amplitude=5pt},decorate]  (-2.5,-2) -- (-4.5,-2);
\draw[decoration={calligraphic brace,amplitude=5pt},decorate] (1.5,-2) -- (-0.5,-2);
\draw (-2.5,0.4) -- (-2.5,0.6);
\draw (-0.5,0.4) -- (-0.5,0.6);

\node at (-2.5,0.8) {$-\nu$};
\node at (-0.5,0.8) {$\nu$};

\node at (-3.4,-2.5) {$f(x)=-\frac{1}{2}|x|$};
\node at (0.3,-2.5) {$f(x)=-\frac{1}{2}|x|$};
\draw (-0.5,0) -- (1.5,-1);
\draw (-2.5,0) -- (-4.5,-1);

\draw (-2.5,0) .. controls (-2.3,0.1) and (-1.8,0.5) .. (-1.5,0.5);
\draw (-0.5,0) .. controls (-0.7,0.1) and (-1.2,0.5) .. (-1.5,0.5);
\end{tikzpicture}
\caption{Graph of $f$.} \label{fig:graphoff2}
\end{figure}
\end{center}

Define
$$H^{k,s} : M \times T^*\T \lra{} \R, \quad
H^{k,s}(x,(\sigma,\tau)) := 
(s + (1-s)\rho(\sigma))K^k(x) + f(\sigma) - \frac{1}{k}$$
for each $k \in \N$ and $s \in [0,1]$.
Define
$H^{k,-} := H^{k,0}$ and $H^{k,+} := H^{k,1}$.
Also define
$$\check{H}^\pm : M \times T^* \T \lra{} \R, \quad \check{H}^{\pm}(x,(\sigma,\tau)) := \left\{
\begin{array}{ll}
\infty & \text{if} \ (x,(\sigma,\tau)) \notin K_\pm \\
f(\sigma) & \text{otherwise.}
\end{array}
\right\}.$$
By Lemma \ref{lemma action spectrum},
we can choose a sequence of elements
$(a_i)_{i \in \Z}$ in $\Sc(Q_\omega)$
so that
\begin{itemize}
\item $a_i([\omega],1,1)$ is not in the action spectrum of
$H^{k,\pm}$
 for all $k,j \in \N$ and
\item $a_i$ tends to infinity as $i$ tends to infinity and $a_i$ tends to $-\infty$ as $i$ tends to $-\infty$.
\end{itemize}
Then
$$HF^*_{(a_i,a_j)}(\check{H}^\pm) = \varinjlim_{k \in \N}(HF^*_{(a_i,a_j)}(H^{k,\pm}))$$
for each $i,j \in \Z$ and such an isomorphism is induced by continuation maps.
Hence
by Lemma \ref{lemma isomorphism condition},
we have an inclusion isomorphism of double systems
\begin{equation} \label{equation isomorphism of double systems for K plus minus}
(HF^*_{(a_i,a_j)}(\check{H}^\pm))_{(i,j) \in \Z \times \Z} \lra{\cong} \SH^*_{\emptyset,Q_\omega,Q_\omega}(\check{H}^{k,\pm})
\end{equation}
Since all the null homologous $1$-periodic
orbits of $H^{k,s}$ are contained
in $\{\sigma = 0\}$ because $\sigma \frac{\partial H^{k,s}}{\partial \sigma} < 0$ and $\frac{\partial H^{k,s}}{\partial \tau} = 0$ away from $\{\sigma=0\}$ for each $s \in [0,1]$,
we have by Lemma \ref{lemma continuation map isomorphism action gap on M times TT}
that the continuation map
$$HF^*_{(a_i,a_j)}(H^{k,-}) \lra{} HF^*_{(a_i,a_j)}(H^{k,+})$$
is an isomorphism for each $i,j \in \Z$ and $k \in \N$.
Combining this with the fact that the inclusion map (\ref{equation isomorphism of double systems for K plus minus}) is an isomorphism,
we have that the continuation map
$$\SH^*_{\emptyset,Q_\omega,Q_\omega}(\check{H}^-) \lra{} \SH^*_{\emptyset,Q_\omega,Q_\omega}(\check{H}^+)$$
is an isomorphism.
Hence by Proposition
\ref{proposition continuation map isomorphism for M times TT},
the map (\ref{equation transfer cross to product})
is an isomorphism.
\end{proof}

\begin{lemma} \label{lemma transfer map reparameterization}
Let $K^\pm \subset M \times T^* \T$ be admissible closed subsets so that $K^+ \subset K^-$.
Let $H^\pm \in \ccH^{\T,\ls}(K^\pm,\leq 0)$ (Definition \ref{defn appropriate lower semi continuous hamiltonians})
satisfy $H^- \leq H^+$
and let $\check{H}^\pm \in \ccH^{\T,\ls}((\T \times K^\pm) \cup ([0,1/2] \times M \times T^* \T),\leq 0)$
satisfy $\check{H}^- \leq \check{H}^+$.
Then the transfer map
$$\SH^*_{\check{C},Q_\omega,Q_\omega}(H^-)
\lra{}
\SH^*_{\check{C},Q_\omega,Q_\omega}(H^+)$$
is an isomorphism if and only if
the transfer map
$$\SH^*_{\check{C},Q_\omega,Q_\omega}(\check{H}^-)
\lra{}
		\SH^*_{\check{C},Q_\omega,Q_\omega}(\check{H}^+)$$
is an isomorphism.

\end{lemma}
\proof
By Proposition \ref{proposition continuation map isomorphism for M times TT},
it is sufficient for us to prove
this for specific $H^\pm$ and $\check{H}^\pm$.
Let $F : \T \lra{} \T$
be a non-decreasing smooth function homotopic to the identity
so that $F'(t) > 0$ for $t \in (1/2,1)$
and $F'(t)=0$ for $t \in [0,1/2]$.
Choose $H^\pm$ and $\check{H}^\pm$
so that $(H^\pm)^F = \check{H}^\pm$
where $(H^\pm)^F$ is given in Definition
\ref{definition ls reparameterized Hamiltonian}.
Then our lemma follows from the fact that we have reparameterization isomorphisms
$$\SH^*_{\check{C},Q_\omega,Q_\omega}(H^\pm)
\lra{}
\SH^*_{\check{C},Q_\omega,Q_\omega}(\check{H}^\pm)$$
which commute with continuation maps.
\qed

%
%

\begin{proof}[Proof of Theorem \ref{theorem stably displaceable complement}]

The key idea of the proof is
to use the displacing Hamiltonian $H^-$
to construct a family of Hamiltonians
$(H^- + H_{s,\bullet})_{s \in [0,\infty)})$ (see below)
with identical orbits realizing
the transfer isomorphism (\ref{equation transfer isomorphism displaceable}).
One then uses Lemma \ref{lemma continuation map isomorphism action gap on M times TT}.

For each $\nu > 0$, define
$$\widehat{K}_\nu := (K \times T^* \T) \cup (M \times \{ \sigma \geq \nu\})$$
and
$$\widetilde{K}_\nu := (\T \times \widehat{K}_\nu) \cup ([0,1/2] \times M \times T^* \T).$$
By Lemmas \ref{lemma product with s1},
\ref{lemma cross to product transfer isomorphism} and \ref{lemma transfer map reparameterization},
it is sufficient for us to show that
\begin{equation} \label{equation transfer isomorphism from M to K}
\SH^*_{\check{C},Q_\omega,Q_\omega}(H^-)
\lra{} \SH^*_{\check{C},Q_\omega,Q_\omega}(H^+)
\end{equation}
is an isomorphism for some $\nu > 0$
and some appropriate lower-semi-continuous Hamiltonians $H^- \in \ccH^{\T,\ls}(M \times T^* \T,\leq 0)$, $H^+ \in \ccH^{\T,\ls}(\widetilde{K}_\nu,\leq 0)$ satisfying $H^- \leq H^+$.

Since $(\overline{M-K}) \times (\{0\} \times \T)$ is Hamiltonian displaceable inside $M \times T^*\T$,
there is a smooth admissible Hamiltonian
$\check{H} = (\check{H}_t)_{t \in \T}$ on $M \times T^*\T$
and a constant $\nu > 0$
so that $\phi^{\check{H}}_1(Q) \cap Q = \emptyset$
where
$$Q := M \times T^* \T - \widehat{K}_\nu = (M -K) \times \{\sigma < \nu\}.$$
By subtracting a constant from $\check{H}$,
we can assume that $\check{H} < 0$.
Now let $F,G : \T \lra{} \T$ be a smooth non-decreasing functions
homotopic to the identity map
satisfying $F(0)=0$, $F'|_{[0,1/2]} = 0$,
$F'|_{(1/2,1)}>0$ and
$G(0) = 0$,
$G'|_{[1/2,1]} = 0$
and $G'|_{(0,1/2)} >0$.
Define $H^- := \check{H}^G$.
We define the lower semi-continuous
Hamiltonian $H^+ := (H^+_t)_{t \in \T}$ by
$$H^+_t(x) := \left\{
\begin{array}{ll}
H^-_t(x) & \text{if} \ (t,x) \in \widetilde{K}_\nu \\
\infty & \text{otherwise.} 
\end{array}
\right.
$$

Now choose a smooth family of autonomous Hamiltonians
$H = (H_{s,t})_{(s,t) \in [0,\infty) \times \T}$ on $M \times T^*\T$
satisfying
\begin{itemize}
\item $H_{s,t} \geq 0$, $\frac{\partial H_{s,t}}{\partial s} \geq 0$,
\item $H_{s,t}(x) = 0$ if and only if $(t,x) \in \widetilde{K}_\nu$ or $s = 0$,
\item $dH_{s,t}|_x = 0$ for all $(t,x) \in \widetilde{K}_\nu$.
\item $H_{s,t}(x) \to \infty$ as $s \to \infty$ for all $(t,x) \in \T \times M \times T^* \T - \widetilde{K}_\nu$.
\end{itemize}
Define $H_{s,\bullet} := (H_{s,t})_{t \in \T}$ for all $s \in [0,\infty)$.
By Lemma \ref{lemma action spectrum},
there is a sequence $(a_i)_{i \in \Z}$
in $\Sc(Q_\omega)$
so that
\begin{itemize}
\item $a_i([\omega],1,1) \to \pm \infty$ as $i \to \pm \infty$ and
\item $a_i([\omega],1,1)$ is not in the action spectrum of $H^-$ for each $i \in \Z$.
\end{itemize}
Note that the capped $1$-periodic orbits of $H^-$
and $H_{s,\bullet} + H^-$ are identical for all $s \in [0,\infty)$ since the support of $H_{s,t}$ is contained inside $(1/2,1) \times Q$ for all $(s,t) \in [0,\infty) \times \T$ (making $H_{s,\bullet}$ and $H^-$ Poisson commute) and since there are no $1$-periodic orbits $\gamma : \T \lra{} M \times T^* \T$ of $H^-$ or $H_{s,\bullet} + H^-$
satisfying $\gamma(0) \in Q$ for each $s \in [0,\infty)$.
Also, the corresponding actions of these capped $1$-periodic orbits are the same as well since $H^-_t|_{\gamma(t)} = (H_{s,\bullet} + H^-)_t|_{\gamma(t)}$
for all $t \in \T$ and all $1$-periodic orbits $\gamma$ of $H^-$.
Therefore by Lemma \ref{lemma continuation map isomorphism action gap on M times TT},
the natural continuation map
$$HF^*_{(a_i,a_j)}(H^-) \lra{} HF^*_{(a_i,a_j)}(H_{s,\bullet} + H^-)$$
is an isomorphism for all $i,j \in \Z$
and $s \in [0,\infty)$.
Hence the induced map of double systems
$$(HF^*_{(a_i,a_j)}(H^-))_{(i,j) \in \Z} \lra{} (HF^*_{(a_i,a_j)}(H_{s,\bullet} + H^-))_{(i,j) \in \Z}$$
is an isomorphism.
By Lemma \ref{lemma isomorphism condition},
this implies that
$$\SH^*_{\check{C},Q_\omega,Q_\omega}(H^-) \lra{} \SH^*_{\check{C},Q_\omega,Q_\omega}(H^+)$$
is an isomorphism.
This completes our proposition.
\end{proof}

\subsection{Symplectic Cohomology and Alternative Filtrations.}
\label{section alternative filtratrions}

In this subsection we will show
in Proposition \ref{proposition alternative filtrations defining symplectic cohomology}
below
that symplectic cohomology
defined with respect to
a particular
wide action interval domain
(Definition \ref{defn chain complex})
is isomorphic to one
defined over an action interval domain
that is not wide
under certain conditions.

\begin{defn} \label{definition index bounded contact cylinder prelim}
Let $(C,\alpha_C)$ be a $2n-1$-manifold 
with contact form.
Recall that the {\it Reeb vector field} 
of $\alpha_C$ is the unique vector field
$R_{\alpha_C}$ satisfying $i_{R_{\alpha_C}} d\alpha_C = 0$
and $\alpha_C(i_{R_{\alpha_C}})=1$.
A {\it periodic Reeb orbit of length $\lambda>0$}
is a map $\gamma : \R / \lambda \Z \lra{} C$ satisfying $\dot{\gamma} = R_{\alpha_C}$.
The {\it Reeb flow} of $\alpha_C$ is the flow
$(\phi^\alpha_t : C \lra{} C)_{t \in \R}$
of $R_{\alpha_C}$.
We will define $\text{length}(\gamma) := \lambda$.

Now let
$\check{C} = [1-\epsilon,1+\epsilon] \times C \subset M$
be a contact cylinder inside $M$ and let $\iota_C : C \lra{} M$ be the natural inclusion map sending $x$ to $(1,x) \in \check{C} \subset M$.
Let $r_C$ be the cylindrical coordinate of $\check{C}$ and let $\alpha_C$ be the associated contact form.
By abuse of notation, we will define
$R_{\alpha_C}$ to be the unique vector field on $\check{C}$ which projects to $R_{\alpha_C}$ in $C$
and $0$ in $[1-\epsilon,1+\epsilon]$.
We let $\frac{\partial}{\partial r_C}$
be the gradient of $r_C$ with respect to any product metric on $\check{C}$ where the factor $[1-\epsilon,1+\epsilon]$ has the standard Euclidean metric.
Let $\gamma$ be a periodic Reeb orbit of $\alpha_C$ of length $\lambda$ so that $\iota_C \circ \gamma$ is null homologous in $M$.
Let $\widehat{\gamma} = (\widetilde{\gamma},\check{\gamma})$
be a capped loop
so that
$\gamma(\lambda t) = \widetilde{\gamma}(\check{\gamma}(t))$ for each $t \in [0,1]$.
Then since we have a splitting
$TM|_{\check{C}} = \ker(\alpha_C) \oplus \text{Span}(\frac{\partial}{\partial r_C},R_{\alpha_C})$
of symplectic vector bundles
with associated symplectic forms $d\alpha_C|_{\ker(\alpha_C)}$ and $dr_C \wedge \alpha_C|_{\text{Span}(\frac{\partial}{\partial r_C},R_{\alpha_C})}$
and since $\text{Span}(\frac{\partial}{\partial r_C},R_{\alpha_C})$ has a natural choice of symplectic trivialization,
any symplectic trivialization of $\widetilde{\gamma}^* TM$
gives an induced symplectic bundle
trivialization
$$\tau : \gamma^* \ker(\alpha_C) \lra{} (\R / \lambda \Z) \times \C^{n-1}.$$
Let $P : (\R/\lambda \Z) \times \C^{n-1} \lra{} \C^{n-1}$ be the natural projection map.
The {\it Conley-Zehnder index} $CZ(\gamma)$
of $\gamma$ is defined
to be the Conley-Zehnder index
of the path of symplectic matrices
$$P \circ \tau \circ \phi^\alpha_t \circ (P \circ \tau|_0)^{-1}, \quad t \in [0,\lambda].$$
This does not depend on the choice of
trivialization
$\tau$ by \ref{item:homotopyinvariance}
or on the choice of $\widehat{\gamma}$ since $c_1(M) = 0$.
The {\it index} of a Reeb orbit $\gamma$
is defined to be $|\gamma| := n-CZ(\gamma)$.

Let $\Gamma_{\alpha_C}$ be the set of periodic Reeb orbits $\gamma$ of $\alpha_C$ so that $\iota_C \circ \gamma$ is null homologous inside $M$.
The {\it index $[-m,m]$ period spectrum}
of $\check{C}$ is the set
$$
\{ \
\text{length}(\gamma) \ : \ \gamma \in \Gamma_{\alpha_C}, \ -m \leq |\gamma| \leq m \
\}
\subset \R.
$$
\end{defn}

\begin{defn} \label{definition index bounded contact cylinder}
The contact cylinder $\check{C}$
as above
is {\it index bounded}
if for every $m > 0$,
there exists $\mu_m>0$ so that the index $[-m,m]$ period spectrum of $\check{C}$ is contained in the interval $(0,\mu_m)$.
\end{defn}

\begin{prop} \label{proposition alternative filtrations defining symplectic cohomology}
Suppose that $\check{C}$ is an index bounded contact cylinder with associated Liouville domain $D$.
Let $\omega_{\check{C}}$ be a $\check{C}$-compatible $2$-form
with scaling constants $0$ and $1$ and which is equal to $\omega$ outside $D \cup ([1,1+\epsilon/2] \times C)$.
Also let $Q_+ = Q_{\omega_{\check{C}}} \subset H^2(M,D;\R) \times \R \times \R$ be the cone spanned by $([\omega_{\check{C}}],1,1)$
and
$Q_-$ the cone spanned by
$([\omega_{\check{C}}],1,1)$
and
$([\omega_{\check{C}}],0,1)$.
Then the action map
\begin{equation} \label{transfer isomorphism alternative filtration}
\SH^*_{\check{C},Q_-,Q_+}(D \subset M) \lra{} \SH^*_{\check{C},Q_+,Q_+}(D \subset M)
\end{equation}
is an isomorphism in $\sys{\Lambda_\K^{Q_+,+}}$.
\end{prop}

Before we prove Proposition \ref{proposition alternative filtrations defining symplectic cohomology},
we need some preliminary lemmas.
The first lemma relates the indices of Reeb orbits with
the indices of certain Hamiltonian orbits.

\begin{lemma} \label{lemma index calculation}
Let $\check{C} = [1-\epsilon,1+\epsilon] \times C$ be an index bounded contact cylinder with cylindrical coordinate $r_C$ and associated contact form $\alpha_C$
and let $\pi : \check{C} \lra{} C$ be the natural projection map.
Let $f : [1-\epsilon,1+\epsilon] \lra{} \R$ be a smooth function, $R_{\lambda,[-m,m]}$ the set of Reeb orbits in $\Gamma_{\alpha_C}$ of length $\lambda$ and index in $[-m,m]$
and $O_{\lambda,[-m,m]}$ the set of $1$-periodic orbits
of $f(r_C)$ contained in $\{r_C = (f')^{-1}(\lambda)\}$
of index in $[-m,m]$ which are  null homologous in $M$.
Then the map
$$O_{\lambda,[-m,m]} \lra{} R_{\lambda,[-m-\frac{1}{2},m + \frac{1}{2}]}$$
sending $\gamma : \T \lra{} \check{C}$
to $\pi \circ \gamma \circ b_\lambda$
is well defined where
$$b_\lambda : [0,\lambda] \lra{} [0,1], \quad b_\lambda(t) := (1/\lambda) t, \quad \forall \ t \in [0,\lambda].$$
\end{lemma}
\proof
Since $X_{f(r_C)} = f'(r_C) R_{\alpha_C}$
where $R_{\alpha_C}$ is the natural lift of the Reeb vector field of $\alpha_C$ to $\check{C}$,
we see that $\gamma \in O_{\lambda,[m,m]}$
gets pushed forward
to a Reeb orbit $\gamma \circ \pi \circ b_\lambda$ of length $\lambda$.
Therefore, all we need to do is compute the index of $\gamma \circ \pi \circ b_\lambda$.

Let $\xi := \ker(\pi^* \alpha_C) \cap \ker(dr_C)$ be a codimension $2$ distribution inside $\check{C}$ and
define $\xi^\perp := \text{Span}(\frac{\partial}{\partial r_C},R_{\alpha_C})$.
Then these are symplectic subbundles of $\gamma^* TM$ which are symplectically orthogonal to each other.
We have a natural symplectic trivialization
$T : \gamma^* \xi^\perp \lra{} \T \times \C$
sending $\frac{\partial}{\partial r}$ to $\frac{\partial}{\partial x}$
and $R$ to $\frac{\partial}{\partial y}$
where $z = x + iy$ is the natural complex coordinate on $\C$.

Let $\widehat{\gamma} := (\widetilde{\gamma},\check{\gamma})$
be a capped loop whose associated loop is $\gamma$ and where
the domain of $\widetilde{\gamma}$
is an oriented surface $\Sigma$.
Let $\iota : \gamma^* TM \lra{} \widetilde{\gamma}^* TM$ be the natural bundle inclusion map covering $\check{\gamma}$.
Let $$\tau : \widetilde{\gamma}^* TM \lra{} \Sigma \times \C^n, \quad
\check{\tau} : \gamma^* \xi \lra{} \T \times \C^{n-1}$$
be symplectic bundle trivializations so that
$\tau \circ
\iota
 = \check{\tau} \oplus T$ after identifying $\partial \Sigma$ with $\T$ via $\check{\gamma}$.
Let $P_{\C^n} : \Sigma \times \C^n \lra{} \C^n$ be the natural projection map.
Let $\phi_t : \check{C} \lra{} \check{C}$ be the time $t$ Hamiltonian flow of $f(r_C)$.
We have that
$$P_{\C^n} \circ \tau \circ D\phi_t \circ (P_{\C^n} \circ \tau|_{\check{\gamma}(0)})^{-1} : \C^n \lra{} \C^n$$
is equal to a block diagonal matrix
$$
\left(
\begin{array}{cc}
A_t & 0 \\
0 & B_t
\end{array}
\right)
$$
with respect to the splitting
$\C^n \cong \C^{n-1} \oplus \C$ for all $t \in [0,1]$.
By Definition \ref{definition index bounded contact cylinder prelim} and \ref{item:homotopyinvariance}
we have
$CZ((A_t)_{t \in \T}) = CZ(\pi \circ \gamma \circ b_\lambda)$.
Also
$B_t =
\left(
\begin{array}{cc}
1 & f''(r)t \\
0 & 1
\end{array}
\right)
$ for all $t \in [0,1]$.
Hence
$$CZ(\gamma) \stackrel{\text{\ref{item:czadditive}}}{=} CZ((A_t)_{t \in \T}) + CZ((B_t)_{t \in \T}) \stackrel{\text{\ref{item:sheartransformation}}}{=} CZ((A_t)) + \eta$$ where $\eta = 0$ or $\pm \frac{1}{2}$ and so
$$
|CZ(\gamma) - CZ(\pi \circ \gamma \circ b_\lambda)| \leq \frac{1}{2}
$$
and this completes our lemma.

\qed

The following lemma constructs for us a smooth family of Hamiltonians (which will compute symplectic cohomology later on)
with the property that all the $1$-periodic orbits in a fixed index range and whose action is not too big are identical.

\begin{lemma} \label{lemma nice cofinal family of Hamiltonians for index bounded contact cylinder}

Let $\check{C}, D$ be as in Proposition \ref{proposition alternative filtrations defining symplectic cohomology}.
Let $m > 0$ be a constant.
Then there is a smooth family of
autonomous $\check{C}$-compatible Hamiltonians
$H = (H_s)_{s \in [0,\infty)}$
on $M$ so that
\begin{enumerate}
\item \label{item non decreasing property of hamiltonian}
$H_s$ is locally constant outside $[1-\epsilon/2,1+\epsilon/8] \times C$
and a non-decreasing function of the radial coordinate of $\check{C}$ inside $\check{C}$
which has positive derivative near $\partial D$,
\item $\frac{dH_s(x)}{ds} \geq 0$ for all $x \in M$ and $H_s(x) \to \infty$ as $s \to \infty$ for all $x \in M - D$,
\item \label{item fixed on D}
$H_s|_{D} = H_0|_{D} < 0$ for all $s \in [0,\infty)$ and
\item \label{item index for large H}
for each $\check{C}$-action interval $(a_-,a_+) \in \Sc(Q_-) \times \Sc(Q_+)$ where $(Q_-,Q_+)$ is wide,
there exists $S \geq 0$
(depending continuously on $(a_-,a_+)$)
so that for each $s \geq S$
and each capped $1$-periodic orbit
$(\widetilde{\gamma},\check{\gamma}) \in \Gamma^{[-m,m]}_{\check{C},a_-,a_+}(H_s)$,
we have
$\textnormal{Image}(\widetilde{\gamma} \circ \check{\gamma}) \subset D$
where $\Gamma^{[-m,m]}_{\check{C},a_-,a_+}(H_s)$
is given in Definition \ref{defn chain complex}.
\end{enumerate}
\end{lemma}

\begin{remark} \label{remark same action in D0}
Note that every
capped $1$-periodic orbit of $H_s$ whose associated $1$-periodic orbit has image in $D$ is also a capped $1$-periodic orbit of $H_0$ with the same $\check{C}$ action.
\end{remark}

\begin{remark} \label{remark hamiltonian existence extension}
We can in fact strengthen Lemma \ref{lemma nice cofinal family of Hamiltonians for index bounded contact cylinder}
slightly.
We can show that there is a constant $\delta>0$ (independent of $s \in [0,\infty]$)
so that
properties
(\ref{item non decreasing property of hamiltonian})-
(\ref{item index for large H})
from 
Lemma \ref{lemma nice cofinal family of Hamiltonians for index bounded contact cylinder}
hold with $H_s$ replaced by
$e^\tau H_s$ for all $-\delta \leq \tau \leq \delta$.
The proof of this stronger lemma is exactly the same as
the proof of Lemma \ref{lemma nice cofinal family of Hamiltonians for index bounded contact cylinder}
below, except that the
action estimates (\ref{equation infimum action})
and
(\ref{equation action lower estimate}) below
have to be scaled appropriately.
This strengthening will be needed in Lemma \ref{lemma linear case} below.
\end{remark}

\begin{proof}[Proof of Lemma \ref{lemma nice cofinal family of Hamiltonians for index bounded contact cylinder}.]
Let $\check{C} = [1-\epsilon,1+\epsilon] \times C$ be our contact cylinder
with cylindrical coordinate $r_C$ and contact form $\alpha_C$.
Since $\check{C}$ is index bounded,
there exists a constant $\Xi > 0$
so that every Reeb orbit of length
greater than $\Xi$ has index in $\Z - [-m-2,m+2]$.

Let $f_s : [1-\epsilon,1+\epsilon] \lra{} \R$, $s \in [0,\infty)$
be a smooth family of functions
satisfying the following properties (see picture below):
\begin{enumerate}[label=(\alph*)]
\item $\frac{d f_s(x)}{d s} \geq 0$ for all $x \in M$ and $f_s(x)$ tends to infinity 
as $s \to \infty$ for each $x > 1$,
\item $f_s(x) = f_0(x)<0$ for all $x \leq 1$
and
there is constant $C_s \geq 0$
so that $f_s(x) = f_0(x) + C_s$
for all $x \in [1+\epsilon/16,\infty)$ and $s \in [0,\infty)$,
\item $f_s|_{(-\infty,1-\epsilon/2] \cup [1+\epsilon/8,\infty)}$ is
locally constant and
\item $f'_s \geq 0$ and $f'_s(x) > \Xi$ for all
$x \in [1,1+\epsilon/16]$ for all $s \geq 0$ (See Figure \ref{fig:graphoffs}).
\end{enumerate}

\begin{center}
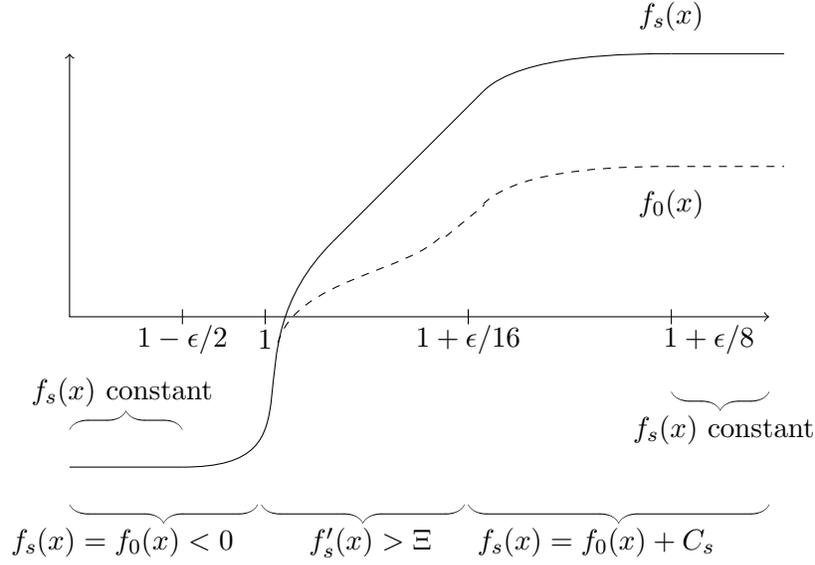
\begin{figure}[h]
\begin{tikzpicture}

\draw [<->](-2,1.5) -- (-2,-2) -- (7.3,-2);
\draw (0.6,-1.9) -- (0.6,-2.1);
\node at (0.6,-2.3) {$1$};
\draw (3.3,-1.9) -- (3.3,-2.1);
\node at (3.3,-2.3) {$1+\epsilon/16$};
\draw (6,-1.9) -- (6,-2.1);
\node at (6.5,-2.3) {$1+\epsilon/8$};
\draw (-0.5,-1.9) -- (-0.5,-2.1);
\node at (-0.5,-2.3) {$1-\epsilon/2$};

\draw (-2,-4) -- (-0.5,-4);

\draw (-0.5,-4) .. controls (1.5,-4) and (0,-2.5) .. (1.5,-1) {};

\draw (1.5,-1) -- (3.5,1);

\draw (3.5,1) .. controls (4,1.5) and (5.5,1.5) .. (6,1.5);

\draw (6,1.5) -- (7.5,1.5);

\draw [dashed](3.5,-0.5) .. controls (4,0) and (5.5,0) .. (6,0);

\draw [dashed](6,0) -- (7.5,0);

\node at (6,-0.5) {$f_0(x)$};

\node at (6,2) {$f_s(x)$};

\draw[dashed] (3,-0.9331) -- (3.5155,-0.5114);

\draw [decoration={calligraphic brace,amplitude=7pt},decorate](3.3 - 0.05,-4.5) -- (0.5+0.05,-4.5);

\node at (2,-5) {$f'_s(x)>\Xi$};

\draw [dashed](0.7704,-2.3496) .. controls (1,-1.5) and (2.5377,-1.4025) .. (2.9632,-0.9037);

\draw [decoration={calligraphic brace,amplitude=7pt},decorate](7.3,-4.5) -- (3.3,-4.5);

\node at (5,-5) {$f_s(x)=f_0(x)+C_s$};

\draw [decoration={calligraphic brace,amplitude=7pt},decorate](0.5,-4.5) -- (-2,-4.5);

\node at (-1.3,-5) {$f_s(x)=f_0(x)<0$};

\draw[decoration={calligraphic brace,amplitude=7pt},decorate] (-2,-3.5) -- (-0.5,-3.5);
\draw[decoration={calligraphic brace,amplitude=7pt},decorate] (7.3,-3) -- (6,-3);
\node at (6.7,-3.5) {$f_s(x)$ constant};
\node at (-1.3,-3) {$f_s(x)$ constant};
\end{tikzpicture}
\caption{Graph of $f_s$ for each $s \in [0,\infty)$.} \label{fig:graphoffs}
\end{figure}
\end{center}

Define
$$H_s : M \lra{} \R, \quad H_s(x) := \left\{
\begin{array}{ll}
f_0(1-\epsilon/2) & \text{if} \ x \in D - \check{C} \\
f_s(r_C) & \text{if} \ x \in \check{C} \\
f_s(1+\epsilon/8) & \text{otherwise}
\end{array}
\right.
$$
for each $s \geq 0$.
It is clear that $(H_s)_{s \in [0,\infty)}$ satisfies
properties (\ref{item non decreasing property of hamiltonian})-(\ref{item fixed on D}).
Therefore we only need to show that it satisfies (\ref{item index for large H}).
Let $\pi_D: H^2(M,D;\R) \times \R^2 \lra{} H^2(M,D;\R)$ be the natural projection map.
Let $(a_-,a_+) \in \Sc(Q_-,Q_+)$ be a $\check{C}$-action interval where $(Q_-,Q_+)$ is wide.
Let $h := \text{height}(a_-,a_+)$ as in Equation (\ref{equation height equation}).

Choose a $1$-form
$\theta \in \Omega^1(M)$
so that $d\theta = \omega - \omega_{\check{C}}$ where $\omega_{\check{C}}$ is a $\check{C}$-compatible $2$-form
with scaling constants $0$ and $1$ and which is equal to $\omega$ outside $D \cup ([1,1+\epsilon/2] \times C)$ and so that $\theta = 0$ outside $D \cup ([1,1+\epsilon/2] \times C)$.
Since $C_s$ tends to infinity as $s \to \infty$, we
can choose $S \geq 0$
so that
\begin{equation} \label{equation infimum action}
\min(f_0) + C_S - \sup_{\breve{\gamma}} \int_\T \breve{\gamma}^* \theta > h
\end{equation}
where the infimum is taken over
all $1$-periodic orbits
$\breve{\gamma} : \T \lra{} M$
of $H_0$.
Now if $s \geq S$ and
$\gamma = (\widetilde{\gamma},\check{\gamma}) \in \Gamma^{[-m,m]}_{\check{C},a_-,a_+}(H_s)$
is a capped $1$-periodic orbit,
then by Lemma \ref{lemma index calculation},
we have that the image of the $1$-periodic orbit $\widetilde{\gamma} \circ \check{\gamma}$
does not intersect $[1,1+\epsilon/16] \times C$
since $f'_s(x) > \Xi$
for all $x \in [1,1+\epsilon/16]$.
Suppose, for a contradiction, the image of $\widetilde{\gamma} \circ \check{\gamma}$
does not intersect $D \cup ([1,1+\epsilon/16] \times C)$
then by Equation
(\ref{equation contact cylinder action}),
$$
|\cA_{H_s,\check{C}}(\gamma)(q,\lambda^q_{+,-},\lambda^q_{+,+}) -
\cA_{H_s,\check{C}}(\gamma)(q,\lambda^q_{-,-},\lambda^q_{-,+})| \geq 
$$
\begin{equation} \label{equation action lower estimate}
(\lambda_{+,-}^q - \lambda_{-,-}^q) \left(\min(f_0) + C_s - \int_0^1 \check{\gamma}^* \theta \right)
\end{equation}
for every $(q,\lambda^q_{\pm,-},\lambda^q_{\pm,+}) \in Q_\pm$ satisfying 
$
\lambda^q_{-,-} < \lambda^q_{+,-}$
and
$\lambda^q_{+,+} = \lambda^q_{-,+}$.
But this contradicts
Equation (\ref{equation infimum action}) and the fact that
$\gamma \in \Gamma^{[-m,m]}_{\check{C},a_-,a_+}(H_s)$.
Hence $\widetilde{\gamma} \circ \check{\gamma}$
must have image contained in $D$.
\end{proof}

\begin{proof}[Proof of Proposition \ref{proposition alternative filtrations defining symplectic cohomology}]

In this proof we will use the family of
Hamiltonians from Lemma \ref{lemma nice cofinal family of Hamiltonians for index bounded contact cylinder}
to compute symplectic cohomology and show that our action map
(\ref{transfer isomorphism alternative filtration}) is an isomorphism.
The key idea of the proof is to show that the $1$-periodic orbits
of these Hamiltonians
outside $D$ are not needed to compute
these symplectic cohomology groups
mainly by using property (\ref{item index for large H}) from Lemma \ref{transfer isomorphism alternative filtration}.

Fix $p \in \Z$.
Let $(H_s)_{s \in [0,\infty)}$ be as in Lemma \ref{lemma nice cofinal family of Hamiltonians for index bounded contact cylinder} with
$m= |p| + n + 1$.
%
%
Since the $1$-periodic orbits of $H_s$
form a compact subset of $C^\infty(\T,M)$,
there is a non-decreasing family of positive constants
$(c_s)_{s \in [0,\infty)}$ so that
for each $s \geq 0$ and each
capped $1$-periodic orbit $\gamma$ of $H_s$,
\begin{equation} \label{equation action difference small}
\left|
\cA_{H_s,\check{C}}(\gamma)([\omega_{\check{C}}],1,1) -
\cA_{H_s,\check{C}}(\gamma)([\omega_{\check{C}}],0,1)
\right| <  c_s.
\end{equation}
Let
$a_s : Q_- \lra{} \R$
be the unique linear map
satisfying
$a_s([\omega_{\check{C}}],1,1) = 0$
and $a_s([\omega_{\check{C}}],0,1) = -c_s$.
Then for each $\check{C}$-action interval
$(a_-,a_+) \in \Sc(Q_-) \times \Sc(Q_+)$,
the action morphism
$$HF^p_{\check{C},a_-+a_s,a_+}(H_s)
\lra{} HF^p_{\check{C},a_-|_{Q_+},a_+|_{Q_+}}(H_s)$$
is an isomorphism for each $s \geq 0$
by Equation (\ref{equation action difference small}) combined with Lemma \ref{lemma action morphism for lower semicontinuous isomorphism}.
Also if
\begin{equation} \label{equation slanted action interval}
a_-([\omega_{\check{C}}],0,1) < a_-([\omega_{\check{C}}],1,1) - c_0,
\end{equation}
then there is a function
$v : \Sc(Q_-) \times \Sc(Q_+) \lra{} [0,\infty)$
so that
the action morphism
$$HF^p_{\check{C},a_-,a_+}(H_{\check{s}})
\lra{} HF^p_{\check{C},a_-+a_s,a_+}(H_{\check{s}})$$
is an isomorphism for each
$\check{s} \geq v(a_-,a_+)$
by Lemma \ref{lemma nice cofinal family of Hamiltonians for index bounded contact cylinder} part (\ref{item index for large H}) combined with
Equation (\ref{equation action difference small}) (with $s=0$) and Lemma \ref{lemma action morphism for lower semicontinuous isomorphism}.
Therefore we have a commutative diagram:

\begin{center}
\begin{tikzpicture}

\node at (-4,0.5) {$HF^p_{\check{C},a_-,a_+}(H_s)$};
\node at (-4,2) {$HF^p_{\check{C},a_-,a_+}(H_{\check{s}})$};
\node at (4,0.5) {$HF^p_{\check{C},a_-|_{Q_+},a_+|_{Q_+}}(H_s)$};
\node at (4,2) {$HF^p_{\check{C},a_-|_{Q_+},a_+|_{Q_+}}(H_{\check{s}})$};
\node at (-0.3,0.5) {$HF^p_{\check{C},a_-+a_s,a_+}(H_s)$};
\node at (-0.3,2) {$HF^p_{\check{C},a_-+a_s,a_+}(H_{\check{s}})$};

\draw[->] (-2.5,2.1) -- (-2.1,2.1);
\draw[->] (-2.5,0.6) -- (-2.1,0.6);

\draw[->] (1.4,0.6) -- (2,0.6);
\draw[->] (1.4,2.1) -- (2,2.1);

\draw[->] (-4.1,0.9) -- (-4.1,1.6);
\draw[->] (-0.6,0.9) -- (-0.6,1.6);

\draw[->] (3.6,0.9) -- (3.6,1.6);

\node at (1.7,0.9) {$\cong$};

\node at (-2.3,2.4) {$\cong$};

\end{tikzpicture}
\end{center}
consisting of action maps (horizontal arrows)
and continuation maps (vertical arrows)
for each $(a_-,a_+) \in \Sc(Q_-) \times \Sc(Q_+)$ satisfying (\ref{equation slanted action interval}) and each $s \geq 0$ and $\check{s} \geq s + v(a_-,a_+)$.
Therefore we have a commutative diagram:
\begin{center}
\begin{tikzpicture}
\node at (-0.1,0.5) {$HF^p_{\check{C},a_-,a_+}(H_s)$};
\node at (-0.1,2) {$HF^p_{\check{C},a_-,a_+}(H_{\check{s}})$};
\node at (4,0.5) {$HF^p_{\check{C},a_-|_{Q_+}}(H_s)$};
\node at (3.9,2) {$HF^p_{\check{C},a_-|_{Q_+},a_+|_{Q_+}}(H_{\check{s}})$};
\draw[->] (1.4,0.6) -- (1.9,0.6);
\draw[->] (1.4,2.1) -- (1.8,2.1);
\draw[->] (-0.3,0.8) -- (-0.3,1.6);
\draw[->] (3.6,0.9) -- (3.6,1.6);
\draw [->](2.4,1) -- (0.9,1.7);
\end{tikzpicture}
\end{center}
for each such $a_-,a_+,s, \check{s}$.
This induces an isomorphism of directed systems
$$\left(\varinjlim_s HF^p_{\check{C},a_-,a_+}(H_s)\right)_{(a_-,a_+) \in \Sc(Q_-) \times \Sc(Q_+)} \cong \left(\varinjlim_s  HF^p_{\check{C},a'_-,a'_+}(H_s)\right)_{(a'_-,a'_+) \in \Sc(Q_+)^2}.$$
Therefore by Proposition \ref{proposition continuation map isomorphism},
the action map $\SH^*_{\check{C},Q_-,Q_+}(D \subset M) \lra{} \SH^*_{\check{C},Q_+,Q_+}(D \subset M)$ is an isomorphism in
$\sys{\Lambda_\K^{Q_+,+}}$.
\end{proof}

\subsection{Transfer Isomorphisms between Index bounded Liouville Domains.} \label{subsection Transfer Isomorphisms between Index bounded Liouville Domains}


In this subsection we give
some sufficient conditions for
a transfer map to be an isomorphism.

\begin{defn} \label{definition skeleton of contact cylinder}
Let $\check{C}$ be a contact cylinder
with associated Liouville domain
$D$.
A {\it Liouville form associated to
$\check{C}$}
is a $1$-form $\theta \in \Omega^1(D)$
satisfying $d\theta = \omega|_D$
and $\theta|_{D \cap \check{C}} = r_C \alpha_C$
where $r_C$ and $\alpha_C$
is the radial coordinate and contact form associated to $\check{C}$.
A {\it Liouville vector field associated to $\check{C}$}
is the unique vector field $X_\theta$
on $D$
satisfying $i_{X_\theta} \omega = \theta$ for some Liouville form $\theta$ associated to $\check{C}$.
The {\it skeleton} of a Liouville form  $\theta$ associated to $\check{C}$
is the set of points $x \in D$
where the time $t$ flow of $x$
along $X_\theta$ exists for all time $t \in \R$.
A {\it skeleton of $\check{C}$}
is the skeleton $S$ of $\theta$ for some Liouville for $\theta$ associated to $\check{C}$.
We call $\theta$ a {\it Liouville form associated to $S$}.
\end{defn}

\begin{remark} \label{remark skeleton nonunique}
Note that a skeleton $S$ of a contact cylinder is not unique.
For instance one can obtain other skeletons of $\check{C}$
by pushing forward $S$
via a Hamiltonian diffeomorphism
compactly supported in the interior of $D$.
Also note that $D$ is a disjoint union of its skeleton and a tubular neighborhood $(-\infty,0] \times \partial D$
given by flowing $\partial D$ backwards along $X_\theta$.
\end{remark}

\begin{defn}
	Let $\check{C} = [1-\epsilon,1+\epsilon] \times C$ be a contact cylinder with associated Liouville domain $D$.
	Define $\omega_{\check{C}}$ to be a $\check{C}$-compatible $2$-form
	with scaling constants $0$ and $1$ and which is equal to $\omega$ outside $D \cup ([1,1+\epsilon/2] \times C)$.
	We define $Q_+^{\check{C}} \subset H^2(M,D;\R) \times \R \times \R$ to be the cone spanned by $([\omega_{\check{C}}],1,1)$
	and
	$Q_-^{\check{C}}$ the cone spanned by
	$([\omega_{\check{C}}],1,1)$
	and
	$([\omega_{\check{C}}],0,1)$.
\end{defn}

\begin{prop} \label{proposition transfer isomorphism between index bounded Liouville domains}
Let $\check{C}_i$ be a contact cylinder with associated Liouville domain
$D_i$, Liouville form $\theta_i$ and skeleton $S_i$ of $\theta_i$
for $i=0,1,2,3,4$
so that
\begin{enumerate}
	\item $D_i \subset D_j$ and $\theta_j|_{D_i} - \theta_i$ is exact for $i \geq j$,
	\item $\theta_0|_{D_2} = \theta_2$ and $\theta_1|_{D_4} = \theta_4$,
	\item $S_i$ is contained in the interior of $D_4$ for all $i$,
	\item and $D_0$ and $D_3$ are index bounded (See Figure \ref{fig:liouviledomains}).
\end{enumerate}
Then the transfer map
\begin{equation} \label{equation transfer map of index bounded}
\SH^*_{\check{C}_0,Q_-,Q_+}(D_0 \subset M) \lra{} \SH^*_{\check{C}_0,Q_-,Q_+}(D_3 \subset M)
\end{equation}
is an isomorphism in $\sys{\Lambda^{Q_+}_\K}$ 
where
$Q_\pm := Q_\pm^{\check{C}_0}$.
\end{prop}

\begin{center}
\begin{figure}[h]
	\begin{tikzpicture}
	
	\draw  (-0.5,-1.5) node (v1) {} ellipse (5.5 and 2);
	\draw  (v1) node (v2) {} ellipse (4.5 and 1.5);
	\draw  (v2) node (v3) {} ellipse (3.5 and 1);
	\draw  (v3) node (v4) {} ellipse (2.5 and 0.5);
	\draw  (v4) ellipse (6.5 and 2.5);
	\node at (6.3,-1.4) {$D_0$};
	\node at (5.3,-1.4) {$D_1$};
	\node at (4.3,-1.4) {$D_2$};
	\node at (3.3,-1.4) {$D_3$};
	\node at (2.3,-1.4) {$D_4$};
	\node at (-0.6,-3.2) {$\theta_0|_{D_2}=\theta_2$};
	\node at (-0.6,-1.8) {$\theta_1|_{D_4}=\theta_4$};
	\draw[red] (0.8,-1.6) -- (-0.8,-1.2);
	\draw[red] (-0.4,-1.2) -- (-1.4,-1.4);
	\node at (1,-1.4) {\color{red}$S_1=S_4$};
	\draw[blue] (-1.4,-1.2) -- (-1.2,-1.6);
	\node at (-2.1,-1.5) {\color{blue} $S_0=S_2$};
	\end{tikzpicture}
	\caption{Schematic picture of the Liouville domains $D_0,\cdots,D_4$.} \label{fig:liouviledomains}
	\end{figure}
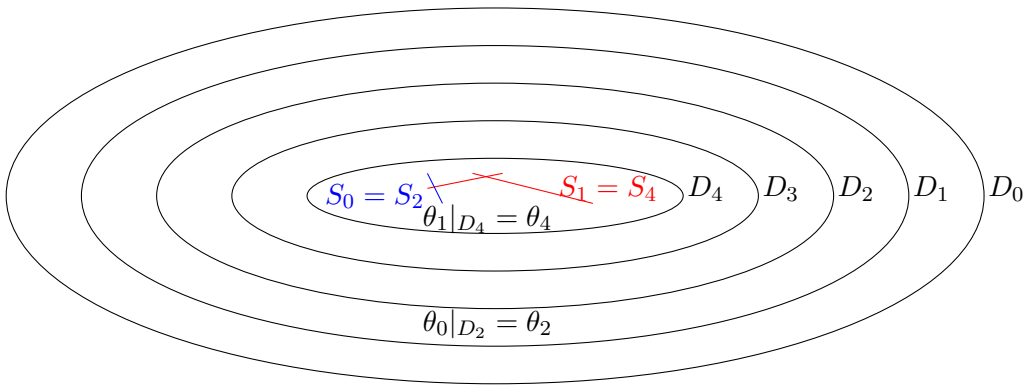
\end{center}

Before we prove this proposition,
we need a few preliminary lemmas.

\begin{lemma} \label{lemma linear case}
Let $\widehat{C}$ be a contact cylinder
with associated Liouville domain $\widehat{D}$
and let $\check{C} := [1-\epsilon,1+\epsilon] \times C$
be an index bounded
contact cylinder
with associated Liouville domain $D$ 
 so that $D \cup \check{C} \subset \widehat{D}$.
Then the transfer map
\begin{equation} \label{equation transfer map of cylindrical index bounded}
\SH^*_{\widehat{C},Q_-^{\widehat{C}},Q_+^{\widehat{C}}}(D \cup ([1,t''] \times C) \subset M) \lra{} \SH^*_{\widehat{C},Q_-^{\widehat{C}},Q_+^{\widehat{C}}}(D \cup ([1,t'] \times C) \subset M)
\end{equation}
is an isomorphism for all $1 \leq t' \leq t'' < 1+\epsilon$.
\end{lemma}
\proof
Define $D_t := D \cup ([1,t] \times C)$
for each
$t \in [1,1+\epsilon)$.
It is sufficient for us to show
that for each $t \in [1,1+\epsilon)$,
there exists a constant
$0 < \delta_t < 1$
so that the transfer map
\begin{equation} \label{equation transfer map of close index bounded}
SH^*_{\widehat{C},Q_-,Q_+}(D_{t''} \subset M) \lra{} SH^*_{\widehat{C},Q_-,Q_+}(D_{t'} \subset M)
\end{equation}
is an isomorphism for each $t',t'' \in [1,1+\epsilon)$ satisfying
$t-\delta_t \leq t' \leq t'' \leq t+\delta_t$.

Fix $t \in [1,1+\epsilon)$.
Let $\check{C}_t \subset \check{C}$
be a contact cylinder with associated Liouville domain $D_t$ and whose cylindrical coordinate is $r_C/t$
where $r_C$ is the cylindrical coordinate associated to $\check{C}$.
Let $(H_s)_{s \in [0,\infty)}$ be as in Lemma \ref{lemma nice cofinal family of Hamiltonians for index bounded contact cylinder} with
$m= |p| + n + 1$ for some $p \in \Z$ and with $\check{C}$ replaced by $\check{C}_t$.
Let $A \subset \R$
be the set of action values
$\cA_{\widehat{C},H_0}(\gamma)([\omega_{\widehat{C}}],1,1) \in \R$
and
$\cA_{\widehat{C},H_0}(\gamma)([\omega_{\widehat{C}}],0,1) \in \Z$
where $\gamma$ runs over all
capped $1$-periodic orbits of $H_0$
whose associated $1$-periodic orbit is contained in $D_0$.
Since $A$
is closed and of measure $0$
by Lemma \ref{lemma action spectrum},
there exists $a \in \R$
and $\delta' > 0$
so that
$(a + k - \delta', a+k +  \delta') \cap A = \emptyset$
for all $k \in \Z$.
Hence there exists a cofinal family
$(a^j_-)_{j \in \N}$
in $(\Sc(Q_-^{\widehat{C}}),\geq)$, a cofinal family
$(a^j_+)_{j \in \N}$
in $(\Sc(Q_+^{{\widehat{C}}}),\leq)$
and neighborhoods
$N_\pm \subset \Sc(Q_\pm^{{\widehat{C}}})$
of $0$
so that
for each $j \in \N$ and
each capped $1$-periodic orbit $\gamma$ of $H_0$ whose associated $1$-periodic orbit is contained in $D_0$,
we have that $a^j_\pm(\gamma) - \cA_{\widehat{C},H_0}(\gamma)|_{Q^{\widehat{C}}_\pm} \notin N_\pm$.
Therefore by Lemma \ref{lemma nice cofinal family of Hamiltonians for index bounded contact cylinder}
combined with Remarks
\ref{remark same action in D0} and
\ref{remark hamiltonian existence extension} there is an increasing sequence of constants $(s_j)_{j \in \N}$
and a constant $\delta_t \in (0,1)$ so that
\begin{equation} \label{equation same action spectra}
\Gamma^{[-m,m]}_{\widehat{C},a^j_-,a^j_+}(e^\tau H_s) =
 \Gamma^{[-m,m]}_{\widehat{C},a^j_- + n_-,a^j_+ + n_+}(e^\tau H_s),
\end{equation}
(See Definition \ref{defn chain complex})
for all $j \in \N, \ s > s_j, \ n_\pm \in N_\pm
$ and $-\delta_t \leq \tau \leq \delta_t$
after possibly shrinking $N_\pm$.

Let $\phi_t : M \lra{} M$
be the time $t$ flow of a vector field on $M$
equal to $r\frac{\partial}{\partial r}$
inside $\check{C}$ 
for each $t \in \R$.
By Equation (\ref{equation same action spectra}) and Lemma
\ref{lemma continuation map isomorphism action gap} combined with the fact that
$H_s|_D = H_0|_D$
for each $s \in [0,\infty)$
we have that
the natural continuation map
$$HF^p_{\widehat{C},a^j_-,a^j_+}((\phi_{t''})_* H_s) \lra{}
HF^p_{\widehat{C},a^j_-,a^j_+}((\phi_{t'})_* H_s)$$
is an isomorphism for each $t',t'' \in [1,1+\epsilon)$ satisfying
$-\delta_t \leq t' < t'' \leq \delta_t$, $j \in \N$ and $s> s_j$.
This implies that the continuation map
$$HF^p_{\widehat{C},a^j_-,a^j_+}(H^{D_{t''}}) \lra{}
HF^p_{\widehat{C},a^j_-,a^j_+}(H^{D_{t'}})$$
is an isomorphism for each $t',t'' \in [1,1+\epsilon)$ satisfying
$t-\delta_t \leq t' < t'' \leq t+ \delta_t$ and $j \in \N$
where
$H^{D_\tau} \in \ccH^{\T,\ls}(\check{C},\leq 0)$
(Definition \ref{definition symplectic cohomology})
is equal to $(\phi_\tau)_*(H_0)$
inside $D_\tau$ and $\infty$ otherwise for each $\tau \in [-\delta_t,\delta_t]$.
Hence by Lemma
\ref{lemma isomorphism condition} and Proposition \ref{proposition continuation map isomorphism},
the map
(\ref{equation transfer map of close index bounded})
is an isomorphism.
\qed

\smallskip

We wish to reduce the proof of Proposition \ref{proposition transfer isomorphism between index bounded Liouville domains}
to Lemma \ref{lemma linear case} above.
In order to do this, we need to
`rearrange' our contact cylinders slightly.
This is the purpose of the following two lemmas.

\begin{lemma} \label{lemma skeleton Hamiltonian isotopy}
Let $(\check{C}_i)_{i=0,1}$ be
contact cylinders with associated Liouville domains
$(D_i)_{i=0,1}$
and skeletons $(S_i)_{i=0,1}$.
Let $(\theta_i)_{i=0,1}$ be Liouville forms associated to $(S_i)_{i=0,1}$.
Suppose that $D_1 \subset D_0$,
$S_0$ is contained in the interior of $D_1$, $S_1 \subset S_0$ and that $\theta_1 - \theta_0|_{D_1}$ is exact.
Suppose also that the time $t$ flow of $X_{\theta_0}$ sends $S_1$ to $S_1$ for all $t \in \R$.
Then for each neighborhood
$N$ of $S_1$,
there is a Hamiltonian isotopy
$\phi_t : M \lra{} M$, $t \in [0,1]$
compactly supported in
the interior of $D_1$
satisfying $\phi_0 = \id$,
$\phi_t(S_1) = S_1$
for all $t \in [0,1]$
and
$\phi_1(S_0) \subset N$.
\end{lemma}
\proof
Let $\phi^i_t : D_i \lra{} D_i$
be the time $t$ flow of
the Liouville vector field
$-X_{\theta_i}$ for each $i =0,1$.
Since $S_0$ is contained in the interior of $D_1$,
there exists $T>0$ so that
$\phi^0_t(D_0) \subset D_1$
for all $t \geq T$.
We can also assume
$\phi^1_t(D_1) \subset N$ for all $t \geq T$.
Define the embedding
$$\iota_t : D_0 \lra{} D_1, \quad \iota_t := \phi^1_t \circ \phi^0_{2T-t}$$
for each $t \in [0,T]$.
Then
$\frac{d}{dt}\iota_t$ is a Hamiltonian vector field on $\iota_t(D_0)$
for all $t \in [0,T]$,
$\iota_0(S_0) = S_0$,
$\iota_t(S_1) = S_1$ for all $t \in [0,T]$
and $\iota_T(S_0) \subset N$.
Extending this time dependent
Hamiltonian vector field to a Hamiltonian vector field on $M$
compactly supported in $D_1$
completes our lemma.
\qed

\begin{lemma} \label{lemma making skeleton larger}
Let $\check{C}_i$ be a contact cylinder with associated Liouville domain $D_i$
and associated
Liouville form $\theta_i$ for $i=0,1,2$.
Suppose
\begin{itemize}
	\item $D_2 \subset D_1 \subset D_0$,
	\item the skeleton of $D_i$ is contained in the interior of
	$D_2$ for each $i$,
	\item $\theta_0|_{D_2} = \theta_2$
	and $\theta_0|_{D_1} - \theta_1$ is exact.
\end{itemize}
Then there is a $1$-form $\widehat{\theta}_0$ associated to
$\check{C}_0$ so that
$\widehat{\theta}_0 - \theta_0$ is exact,
$\widehat{\theta}_0|_{D_0 - D_2} = \theta_0|_{D_0 - D_2}$,
the skeleton $\widehat{S}_0$
of $\widehat{\theta}_0$
contains the skeleton of $\theta_1$
and $\widehat{S}_0$ is contained in the interior of $D_2$.
Also we can assume that the time $t$ flow along $X_{\widehat{\theta}_0}$
of $S_1$ is $S_1$ for all $t \in \R$.
\end{lemma}
\proof
Let $S_i$ be the skeleton of $\theta_i$
for each $i$.
Then $S_0 = S_2$.
Choose a neighborhood
$N$ of $S_1$ whose closure is contained in the interior of $D_2$.
Let $f : D_0 \lra{} \R$ be a smooth
function with compact support
in the interior of $D_2$
so that
$\theta_0|_N - \theta_1|_N = df|_N$.
Define
$\widehat{\theta}_0 :=
\theta_0 - df$.
This has the properties we want.
\qed

\begin{proof}[Proof of Proposition \ref{proposition transfer isomorphism between index bounded Liouville domains}]

Let $(D^o_i)_{i=0,1,2,3,4}$
be the interiors of $(D_i)_{i=0,1,2,3,4}$
respectively.
By applying Lemma \ref{lemma making skeleton larger}
twice we can find a Liouville form
$\widehat{\theta}_i$
associated to $\check{C}_i$
with associated skeleton
$\widehat{S}_i$
so that
$\widehat{\theta}_i - \theta_i$
is exact
for $i=0,1$
and so that
$S_3 \subset \widehat{S}_1 \subset \widehat{S}_0$
and $\widehat{S}_0 \subset D_2^o$
and $\widehat{S}_1 \subset D_4^o$.
Also we can assume that the time $t$ flow
along $X_{\theta_1}$ of $S_3$ is contained in $S_3$
and the time $t$ flow along $X_{\widehat{\theta}_0}$
of $\widehat{S}_1$
is contained in $\widehat{S}_1$.

By Lemma \ref{lemma skeleton Hamiltonian isotopy},
we can push forward
$\widehat{\theta}_0$ by an appropriate
Hamiltonian isotopy $\chi_t : M \lra{} M$, $t \in [0,1]$ compactly supported
in the interior of $D_1$
satisfying $\chi_t(\widehat{S}_1) = \widehat{S}_1$ for all $t \in [0,1]$
so that
$\chi_1(\widehat{S}_0) \subset D_4^o$.
Hence after replacing $\widehat{\theta}_0$ by its pushforward by $\chi_1$, we can still assume that $S_3 \subset \widehat{S}_1 \subset \widehat{S}_0$.

Let
$\phi_t : D_0 \lra{} D_0$
be the time $t$ flow
of $-X_{\widehat{\theta}_0}$
and $\psi_t : D_3 \lra{} D_3$
the time $t$ flow of
$-X_{\theta_3}$.
Since $S_3 \subset \widehat{S}_0 \subset D^o_4$,
there are constants
$T_0, T_3 > 0$
so that
$D'_0 := \phi_{T_0}(D_0) \subset D_3$
and $D'_3 := \psi_{T_3}(D_3) \subset D'_0$.
Since
$\cup_{t \in [0,T]} \phi_t(\partial D_0)$ is the image of a contact cylinder in $M$ for all $T > 0$,
we can apply Lemma \ref{lemma linear case}
to show that the transfer map
$$\SH^*_{\check{C}_0,Q_-^{\check{C}_0},Q_+^{\check{C}_0}}(D_0 \subset M) \lra{} \SH^*_{\check{C}_0,Q_-^{\check{C}_0},Q_+^{\check{C}_0}}(D'_0 \subset M)$$
is an isomorphism.
Similarly
$$\SH^*_{\check{C}_0,Q_-^{\check{C}_0},Q_+^{\check{C}_0}}(D_3 \subset M) \lra{} \SH^*_{\check{C}_0,Q_-^{\check{C}_0},Q_+^{\check{C}_0}}(D'_3 \subset M)$$
is an isomorphism.
Since $D_3' \subset D_0' \subset D_3 \subset D_0$ and since the composition of two transfer maps is a transfer map,
this implies that the transfer map
(\ref{equation transfer map of index bounded}) is an isomorphism.

\end{proof}

\subsection{A Chain Complex For Symplectic Cohomology} \label{subsection chain complex for sh}

In this section we construct a double system of chain complexes from Liouville domains associated index bounded contact cylinders.
In the next section
(see the proof of Theorem \ref{theorem changing novikov ring}),
we'll show that these double systems of chain complexes compute symplectic cohomology in some cases.
Throughout this subsection we fix
an {\it index bounded} (Definition \ref{definition index bounded contact cylinder}) contact cylinder
$\check{C} = [1-\epsilon,1+\epsilon] \times C$ with associated Liouville domain $D$.

\begin{lemma} \label{lemma chain complex}

Let $(Q^m_-,Q^m_+)_{m \in I}$ be a finite collection of
$\check{C}$-interval domain pairs which are wide (Definition \ref{defn chain complex}). For each $m \in I$ let
$(a^{j,m}_-)_{j \in \N}$,
$(a^{j,m}_+)_{j \in \N}$ be a cofinal family
of $(\Sc(Q^m_-), \geq)$
and $(\Sc(Q^m_+), \leq)$ respectively so that
$a^{j+1,m}_- \leq a^{j,m}_-$ and $a^{j,m}_+ \leq a^{j+1,m}_+$ for all $j \in \N$.
Then there is an element
$H_D \in \ccH^{\T,\ls}(\check{C},D,\leq 0)$
(Definition \ref{definition symplectic cohomology}),
a smooth family of autonomous Hamiltonians
$(H_{s,t})_{(s,t) \in [1,\infty) \times \T}$, a constant $s_p \geq 1$ and a subset
$N_p \subset D \cup ([1,1+\epsilon/8] \times C)$ which is open in $M$ for each $p \in \N$
so that
\begin{enumerate}
	\item \label{item cofinal property}
	$\{H_{s,\bullet} \ : \ s \geq s_p \}$
	is a cofinal subset of
	$\ccH^\reg(<_{\check{C}} H_D,a^{p,m}_-,a^{p,m}_+,[-p,p])$(Definition \ref{definition lower semi-continuous hamiltonian}) 
	for each $m \in I$
	where $H_{s,\bullet} := (H_{s,t})_{t \in \T}$ and $H_{s,\bullet} <_{\check{C}} H_{\check{s},\bullet}$ for each $s < \check{s}$,
	\item all $1$-periodic orbits of $H_{s,\bullet}$
	of index in $[-p,p]$ have image contained in $N_p$ for each $s \geq s_p$ and
	\item \label{item equal up to a shift in D}
	$H_{s,\bullet}|_{N_p} + 1/s = H_{\check{s},\bullet}|_{N_p} + 1/\check{s}$
	for each $s,\check{s} \geq s_p$ and each $p \in \N$.
\end{enumerate}
\end{lemma}
\proof
Let
$g : (-\infty,1] \lra{} (-\infty,-1]$
be a continuous function so that
\begin{itemize}
\item $g|_{(-\infty,1)}$ is smooth and $g'(x) \geq 0$ for $x < 1$,
\item $g|_{(-\infty,1-\epsilon/16)} = -2$
\item $g(x) = -1-\sqrt{1-x}$ for $x \geq 1-\epsilon/32$.
\end{itemize}
Define
$$\widetilde{g} : [1+\epsilon/16,\infty) \lra{} (1,\infty), \quad \widetilde{g}(x) := 1 - g(2+\epsilon/16 - x).$$

Let $(a_s)_{s \in [1,\infty)}$ and
$(b_s)_{s \in [1,\infty)}$
be smooth families of constants so that
$\frac{d}{ds}(a_s) > 0$,
$\frac{d}{ds}(b_s) < 0$,
$a_s \in (1-\epsilon/32,1)$,
$b_s \in (1+\epsilon/16,1+3\epsilon/32)$
for all $s \geq 1$,
$a_s \to 1$, $b_s \to 1+\epsilon/16$ and
$g'(a_i)$
is not equal to the length of any Reeb orbit of $\alpha_C$ for each $i \in \N_{\geq 1}$.
Such constants exist since the set of lengths of Reeb orbits has measure $0$ in $\R$ by \cite[Proposition 3.2]{popov1993length}.
%
Let $f_s : \R \lra{} \R$, $s \in [1,\infty)$
be a smooth family of smooth functions satisfying:
\begin{itemize}
\item $f_s' \geq 0$, $\frac{d}{ds}f_s(x) > 0$,
\item $f_s(x) = g(x)-\frac{1}{s}$ for $x \leq a_s$,
$f_s(x) = \widetilde{g}(x) + s$
for $x \geq b_s$ and
\item $f_s|_{(1,\infty)}$ pointwise tends to infinity and the function
 $f'_s|_{[a_s,b_s]}$ uniformly tends to infinity as $s$ tends to infinity (See Figure \ref{fig:graphoffs2}).
\end{itemize}

Let $h : [1-\epsilon,1+\epsilon] \lra{} \R$ satisfy
\begin{itemize}
	\item $h' \leq 0$,
	\item $h|_{[1-\epsilon,1+\frac{3\epsilon}{32}]} = 0$,
	\item $h(x) = -\frac{1}{4}(1+\frac{x}{1+\epsilon})$ for all $x$ inside $[1+\epsilon/8,1+\epsilon/2]$,
	\item $h|_{[1+3\epsilon/4,1+\epsilon]} = -1$ (See Figure \ref{fig:graphofh}).
\end{itemize}

\begin{center}
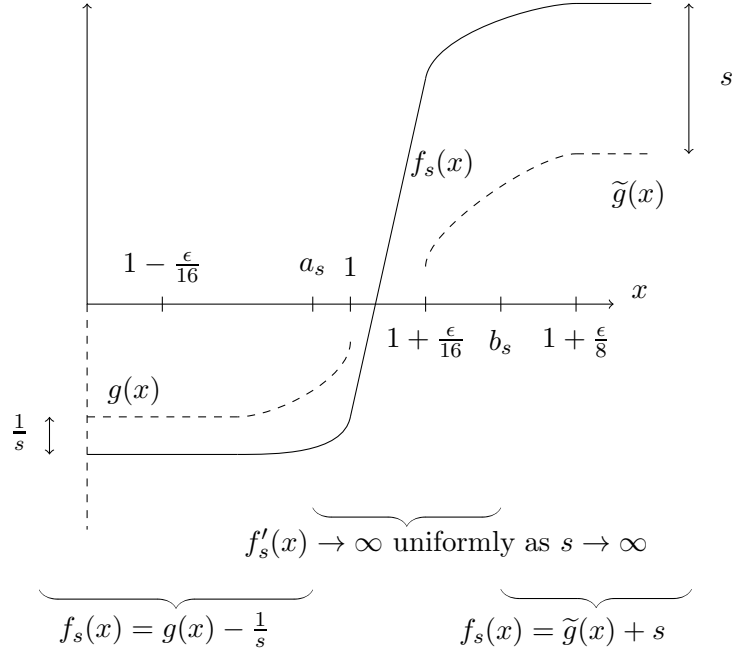
\begin{figure}[h]
\begin{tikzpicture}

\draw [<->](-3,-1) -- (-3,-5) -- (4,-5);

\node at (0.5,-4.5) {$1$};

\node at (0,-4.5) {$a_s$};

\node at (-2,-4.5) {$1-\frac{\epsilon}{16}$};

\node at (2.5,-5.5) {$b_s$};

\node at (3.5,-5.5) {$1+\frac{\epsilon}{8}$};

\draw [dashed](0.5,-5.5) .. controls (0.5,-6) and (-0.5,-6.5) .. (-1,-6.5);

\draw [dashed](-3,-6.5) -- (-1,-6.5);

\node at (1.5,-5.5) {$1+\frac{\epsilon}{16}$};

\draw [dashed](1.5,-4.5) .. controls (1.5,-4) and (3,-3) .. (3.5,-3);

\draw [dashed](3.5,-3) -- (4.5,-3);

\draw (-3,-7) -- (-1,-7);

\draw (-1,-7) .. controls (-0.5,-7) and (0.4,-7) .. (0.5,-6.5);

\draw (4.5,-1) -- (3.5,-1);

\draw (3.5,-1) .. controls (3,-1) and (1.6,-1.4) .. (1.5,-2);

\draw (0.5,-6.5) -- (1.5,-2);

\draw [<->](-3.5,-6.5) -- (-3.5,-7);

\draw [<->](5,-1) -- (5,-3);

\node at (5.5,-2) {$s$};

\node at (-3.9,-6.7) {$\frac{1}{s}$};

\node at (1.7217,-3.135) {$f_s(x)$};

\node at (-2.3704,-6.1581) {$g(x)$};

\node at (4.3484,-4.83) {$x$};

\node at (4.3327,-3.5488) {$\widetilde{g}(x)$};

\draw [decoration={calligraphic brace,amplitude=7pt},decorate] (5.0359,-8.7939) -- (2.5,-8.7939);

\node at (3.3171,-9.3564) {$f_s(x) = \widetilde{g}(x) + s$};

\draw [decoration={calligraphic brace,amplitude=7pt},decorate](0,-8.8018) -- (-3.636,-8.8018);

\node at (-1.9798,-9.3301) {$f_s(x) = g(x)-\frac{1}{s}$};

\draw [decoration={calligraphic brace,amplitude=7pt},decorate]  (2.5,-7.7104) -- (0,-7.7104);

\node at (1.7702,-8.1791) {$f_s'(x) \to \infty$ uniformly as $s \to \infty$};

\draw (-2,-4.9) -- (-2,-5.1);

\draw (0,-4.9) -- (0,-5.1);

\draw (0.5,-4.9) -- (0.5,-5.1);

\draw (1.5,-4.9) -- (1.5,-5.1);

\draw (2.5,-4.9) -- (2.5,-5.1);

\draw (3.5,-4.9) -- (3.5,-5.1);

\draw [dashed](-3,-5) -- (-3,-8);

\end{tikzpicture}
\caption{Graph of $f_s$ for each $s \in [1,\infty)$.} \label{fig:graphoffs2}
\end{figure}
\end{center}

\begin{center}
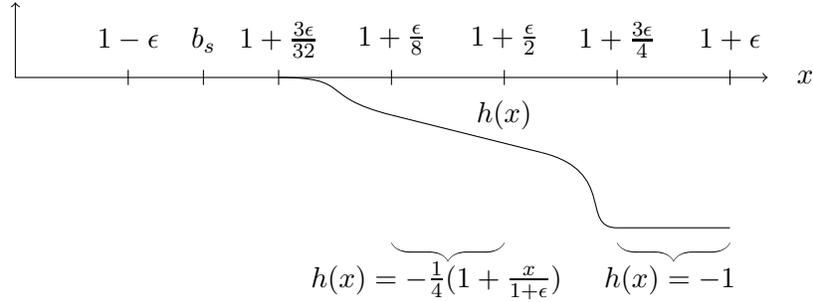
\begin{figure}[h]
\begin{tikzpicture}

\draw [<->](-2.5,-2.5) -- (-2.5,-3.5) -- (7.5,-3.5);

\node at (8,-3.5) {$x$};

\node at (-1,-3) {$1-\epsilon$};

\node at (1,-3) {$1+\frac{3\epsilon}{32}$};

\node at (2.5,-3) {$1+\frac{\epsilon}{8}$};

\node at (4,-3) {$1+\frac{\epsilon}{2}$};

\node at (5.5,-3) {$1+\frac{3\epsilon}{4}$};

\node at (7,-3) {$1+\epsilon$};

\draw (2.5,-4) -- (4.5,-4.5);

\draw (5.5,-5.5) -- (7,-5.5);

\draw (1,-3.5) .. controls (2,-3.5) and (1.5,-3.75) .. (2.5,-4);

\draw (4.5,-4.5) .. controls (5.5,-4.75) and (5,-5.5) .. (5.5,-5.5);

\node at (4,-4) {$h(x)$};

\draw[decoration={calligraphic brace,amplitude=7pt},decorate] (7,-5.7) -- (5.5,-5.7);

\node at (6.2,-6.2) {$h(x)=-1$};

\draw[decoration={calligraphic brace,amplitude=7pt},decorate] (4,-5.7) -- (2.5,-5.7);

\node at (3.1,-6.2) {$h(x) = -\frac{1}{4}(1+\frac{x}{1+\epsilon})$};

\node at (0.,-3) {$b_s$};

\draw (-1,-3.4) -- (-1,-3.6);

\draw (0,-3.4) -- (0,-3.6);

\draw (1,-3.4) -- (1,-3.6);

\draw (2.5,-3.4) -- (2.5,-3.6);

\draw (4,-3.4) -- (4,-3.6);

\draw (5.5,-3.4) -- (5.5,-3.6);

\draw (7,-3.4) -- (7,-3.6);

\end{tikzpicture}
\caption{Graph of $h$.} \label{fig:graphofh}
\end{figure}
\end{center}

Define
$$K_s : M \lra{} \R, \quad K_s := \left\{
\begin{array}{ll}
-2-1/s & \text{inside} \ D - \check{C} \\
f_s(r_C) + \frac{1}{s} h(r_C) & \text{inside} \ \check{C} \\
3+s - \frac{1}{s} & \text{otherwise}
\end{array}
\right\}, \quad s \in [1,\infty).$$
Then $K_s <_{\check{C}} K_{\check{s}}$ for all $s < \check{s}$.
Let $\rho : \R \lra{} [0,1]$ be a smooth function equal to $0$ inside $(-\infty,0]$
and $1$ inside $[1,\infty)$.
Define $B_i := (a_i,a_{i+1}) \times C$ for each integer $i \geq 1$
and $B_0 := D - ([a_1,1] \times C)$.

Since $\check{C}$ is index bounded,
let $\Xi_p$ be larger than the length of the longest Reeb orbit of $\alpha_C$ of index in $[-p,p]$
for each $p \in \N$ so that
$(\Xi_p)_{p \in \N}$ is increasing.
Let $h_p := \max_{m \in I} \textnormal{height}(a^{p,m}_-,a^{p,m}_+)$
where
$\textnormal{height}$ is given in Definition \ref{defn chain complex}.
For each $p \in \N$ choose
$\check{s}_p \geq 1$
so that
$f'_s|_{[a_s,b_s]} > \Xi_{p+n+1}$
for each $s \geq \check{s}_p$
and so that
$(\check{s}_p)_{p \in \N}$ is increasing.
Define
$$s_p := h_p + 2+ 
(1+\epsilon) \Xi_p
+ \check{s}_p.$$
Then by Corollary
\ref{corollary action of capped loop disjoint from contact cylinder}
and Lemma \ref{lemma index calculation}
we have that
for each $p \in \Z$
the $1$-periodic orbit
associated to each
element of
$\Gamma^{[-p-n-1,p+n+1]}_{\check{C},a^{p,m}_-,a^{p,m}_+}(H_s)$
(Definition \ref{defn chain complex})
has image contained in
$N_p := D - ([a_{s_p},1] \times \check{C})$ for each $m \in I$ and $s \geq s_p$.
Also $H_i$ has no $1$-periodic
orbits in $\{r_C = a_i\}$ for each $i \in \N$ since $g'(a_i)$ is not the length of any Reeb orbit of $\alpha_C$ for each $i$.
Therefore
by repeatedly applying Lemma \ref{lemma generic Hamiltonian constant elsewhere},
we can find smooth Hamiltonians
$W_k = (W_{k,t})_{t \in \T}$,
$k \in \N_{\geq 0}$
which are
$C^\infty$ small (in particular, $C^\infty$ tending to $0$
as $i$ tends to infinity) and non-negative and
where $W_i$ has support inside $B_i$ for each $i$
so that
$K_i + \sum_{k=0}^{i-1} W_k \in \cap_{j = 1}^q \ccH^\reg(<_{\check{C}},a^{j,m}_-,a^{j,m}_+,[-j,j])$ for each $q, i \in \N$ satisfying $i \geq s_q$.
Define
$H_{s,\bullet} :=
K_s + \sum_{k=0}^{\lfloor s \rfloor-1} \rho(s - k) W_k$ for each $s \geq 1$
where $\lfloor s \rfloor$ is the largest integer $\leq s$.
Then
$H_{s,\bullet} <_{\check{C}} H_{\check{s},\bullet}$
for all $s < \check{s}$.
Since $W_k$ is $C^\infty$ small for each $k$,
we can assume that
the associated $1$-periodic orbit of every element of $\Gamma^{[-p,p]}_{\check{C},a^{p,m}_-,a^{p,m}_+}(H_{s,\bullet})$
has image in
$N_p$ for each $m \in I$ (since the same is true for $H_s$).
Similarly, we can assume that
$H_{s,\bullet}|_D < 0$ for all $s \geq 1$ since $W_k$ is small for each $k \geq 0$.
%
All such orbits are non-degenerate
by construction
and $H_{s,\bullet}|_{N_p} + \frac{1}{s} = H_{\check{s},\bullet}|_{N_p} + \frac{1}{\check{s}}$
for each $s,\check{s} \geq s_p$.
Hence property (\ref{item cofinal property})-(\ref{item equal up to a shift in D}) holds, after possibly making $s_p$ larger so that $\frac{1}{s_p}$ is sufficiently small.
This completes the lemma.
%
%
\qed

\begin{defn} \label{definition chain complex for symplectic cohomology}
A {\it double system of chain complexes}
is a double system (Definition \ref{definition double system}) $W : I \times J \lra{} \mod{R}$ of $\Z$-graded $R$ modules where each module
$W(i,j)$ is a chain complex over $R$, where the differential has degree $1$ and each morphism
$W((i,j) \lra{} (i',j'))$ is a chain map.
We define $\varinjlim_i \varprojlim_j W(i,j)$
in the same way as Definitions \ref{definition direct limit}
and \ref{definition inverse limit}, except that it is now a chain complex with a differential.
Let $(Q^m_-,Q^m_+)_{m \in I}$, $(a^{j,m}_\pm)_{j \in \N}$,
$H_D$, $(H_{s,t})_{(s,t) \in [1,\infty) \times \T}$,
$(H_{s,\bullet})_{s \in [1,\infty)}$, $(s_p)_{p \in \N}$, $(N_p)_{p \in \N}$ be as in Lemma
\ref{lemma chain complex}.
Choose
$J \in \cap_{i \in \N} \ccJ^\reg(H_{i,\bullet},\check{C})$.
By Lemma \ref{lemma continuation map filtration isomorphism}
there exists
\begin{equation} \label{equation continuation map data}
H^{i,-+}
\in \cap_{m \in I} \ccH^{\R \times \T}(\check{C},a^{j,m}_-,a^{j',m}_+,H_{i,\bullet},H_{i+1,\bullet}), \quad
J^{i,-+} \in \ccJ^{\R \times \T,\reg}(H^{-+},(J,J),\check{C})
\end{equation}
so that
the chain level continuation map:
\begin{equation} \label{equation continuation for symplectic homology isomorphism}
\Phi^q_{i,m,j,j'} := \widetilde{\Phi}^q_{\check{H}^{i,-+},\check{J}^{i,-+}} : CF^q_{\check{C},a^{j,m}_-,a^{j',m}_+}(H_{i,\bullet}) \lra{} CF^q_{\check{C},a^{j,m}_-,a^{j',m}_+}(H_{i+1,\bullet})
\end{equation}
is an isomorphism for each $i,q,j,j' \in \N$ satisfying $j,j' \leq q$,
$i \geq s_p$, $-p \leq q \leq p$, $p \in \N$ and $m \in I$.
Finally define
$$\Phi^q_{i \to i',m,j,j'} :
CF^q_{\check{C},a^{j,m}_-,a^{j',m}_+}(H_{i,\bullet}) \lra{} CF^q_{\check{C},a^{j,m}_-,a^{j',m}_+}(H_{i',\bullet}), $$
$$\Phi^q_{i \to i',m,j,j'} := \Phi^q_{i'-1,m,j,j'} \circ \Phi^q_{i'-2,m,j,j'} \circ \cdots \circ \Phi^q_{i+1,m,j,j'} \circ  \Phi^q_{i,m,j,j'}$$
for all $i,q,j,j',i' \in \N$ satisfying
 $j,j' \leq q$, $i' \geq i$,
$i > s_p$, $-p \leq q \leq p$, $p \in \N$ and $m \in I$ where such a map is the identity map for $i = i'$.
These maps give us a directed system
$(CF^q_{\check{C},a^{j,m}_-,a^{j',m}_+}(H_{i,\bullet}))_{i \geq s_p}$
for each $q,j,j' \in \N$ satisfying $j,j' \leq q$,
$-p \leq q \leq p$, $p \in \N$ and $m \in I$.
Define
$$W^q_{j,j',m} := \varinjlim_{i \geq s_p} CF^q_{\check{C},a^{j,m}_-,a^{j',m}_+}(H_{i,\bullet})$$
for each $q,j,j',m$ as above.
A {\it compatible collection of double systems of chain complexes for $\SH^*_{\check{C},Q^m_-,Q^m_+}(D \subset M)$, $m \in I$} is defined to be
the double systems of chain complexes
$(W^*_{j,j',m})_{j,j' \in \N}$, $m \in I$
where the double system maps are
chain level action maps and where $\N$ has the ordering $\geq$.
If $I$ has just one element $m$ then such a double system is called a
{\it double system of chain complexes for
$\SH^*_{\check{C},Q^m_-,Q^m_+}(D \subset M)$}.
\end{defn}

\begin{remark} \label{remark double system isomoprhic to symplectic cohomology}
Suppose that $I = \{0,1\}$ and that
$(Q^1_-,Q^1_+)$ is smaller than
$(Q^0_-,Q^0_+)$ (Definition \ref{definition compatible cones}).
Then $\SH^*_{\check{C},Q^m_-,Q^m_+}(D \subset M)$
is isomorphic to the double system
$(H_*(W^*_{j,j',m}))_{j,j' \in \N}$
for each $m \in I$ 
by Lemma \ref{lemma isomorphism condition}
and the action map
$$\SH^*_{\check{C},Q^0_-,Q^0_+}(D \subset M)
\lra{} \SH^*_{\check{C},Q^1_-,Q^1_+}(D \subset M)$$
is equal to the natural map
$$(H_*(W^*_{j,j',0}))_{j,j' \in \N} \lra{} (H_*(W^*_{j,j',1}))_{j,j' \in \N}$$
induced by the corresponding chain level action maps under this isomorphism of double systems.
\end{remark}

\subsection{Changing Novikov Rings.} \label{subsection changing Novikov rings}

\begin{theorem} \label{theorem changing novikov ring}
Let $\check{C}$ be an
index bounded
contact cylinder with associated Liouville domain $D$ and
let $(Q^j_-,Q^j_+)$ be a $\check{C}$-interval domain for $j=0,1$
so that
$(Q^1_-,Q^1_+)$
is smaller than
$(Q^0_-,Q^0_+)$
(Definition \ref{definition compatible cones})
and so that $(Q^j_-,Q^j_+)$ is wide
for $j=0,1$
(Definition \ref{defn chain complex}). 
Suppose that $\Lambda_\K^{Q^1}$
is a flat $\Lambda_\K^{Q^0}$-module
and that
$\varinjlim \varprojlim^1 (\SH^*_{\check{C},Q^1_-,Q^1_+}(D \subset M)) = 0$.
Then the map
$$SH^*_{\check{C},Q^0_-,Q^0_+}(D \subset M) \otimes_{\Lambda_\K^{Q^0_+}} \Lambda_\K^{Q^1_+}
\lra{} SH^*_{\check{C},Q^1_-,Q^1_+}(D \subset M)$$
induced by the corresponding action map
is an isomorphism.
\end{theorem}
\begin{proof}[Proof of Theorem \ref{theorem changing novikov ring}.]

Let $(W^*_{j,j',m})_{j,j' \in \N}$, $m = 0,1$
be a compatible collection of double systems of chain complexes for
$\SH^*_{\check{C},Q^m_-,Q^m_+}(D \subset M)$, $m = 0,1$
as in Definition \ref{definition chain complex for symplectic cohomology}.
By Remark
\ref{remark double system isomoprhic to symplectic cohomology}
it is sufficient for us to show that the natural map
\begin{equation} \label{equation isomorphism of specific chain complex homologies}
(\varinjlim_j \varprojlim_{j'} H_*(W^*_{j,j',0})) \otimes_{\Lambda_\K^{Q^0_+}} \Lambda_\K^{Q^1_+} \lra{}
\varinjlim_j \varprojlim_{j'} H_*(W^*_{j,j',1})
\end{equation}
induced by chain level action maps
is an isomorphism.

By Remark
\ref{remark double system isomoprhic to symplectic cohomology} combined with the fact that
$\varinjlim \varprojlim^1 \SH^p_{\check{C},Q^1_-,Q^1_+}(D \subset M) = 0$,
we get
$\varinjlim \varprojlim^1 (H_*(W^*_{j,j',1}))_{j,j' \in \N} = 0$.
Hence by \cite[Theorem 3.5.8]{weibel1995introduction}
combined with the fact that
direct limits preserve short exact sequences and commute with homology
and that  $\Lambda_\K^{Q^1_+}$ is a flat
$\Lambda_\K^{Q^0_+}$-module,
we get a commutative diagram:
\begin{center}
\begin{tikzpicture}

\node at (-5,0.5) {$(\varinjlim_j \varprojlim_{j'}^1 H_{p-1}(W^*_{j,j',0})) \otimes_{\Lambda_\K^{Q^0_+}} \Lambda_\K^{Q^1_+}$};
\node at (-5,-0.5) {$0$};
\node at (-5,2) {$H_p(\varinjlim_j \varprojlim_{j'} W^*_{j,j',0}) \otimes_{\Lambda_\K^{Q^0_+}} \Lambda_\K^{Q^1_+}$};
\node at (-5,4) {$0$};
\node at (-5,3) {$\varinjlim_j \varprojlim_{j'} H_p(W^*_{j,j',0}) \otimes_{\Lambda_\K^{Q^0_+}} \Lambda_\K^{Q^1_+}$};
\node at (-1.2,1.2) {$H_p((\varinjlim_j \varprojlim_{j'} W^*_{j,j',0}) \otimes_{\Lambda_\K^{Q^0_+}} \Lambda_\K^{Q^1_+})$};
\node at (2.5,0.5) {$0$};
\node at (2.7,2) {$H_p(\varinjlim_j \varprojlim_{j'} W^*_{j,j',1})$};
\node at (2.8,3) {$\varinjlim_j \varprojlim_{j'} H_p(W^*_{j,j',1})$};
\node at (2.5,4.1) {$0$};
\node at (2.5,-0.5) {$0$};
\draw [<-](-4.1,1.3) -- (-4.6,1.6);
\draw [->](-2.5-2.5,3.4) -- (-2.5-2.5,3.7);
\draw [->](-2.5-2.5,2.3) -- (-2.5-2.5,2.7);
\draw [->](-2.5-2.5,1.3-0.5) -- (-2.5-2.5,1.7);
\draw [->](-2.5-2.5,0.3-0.5) -- (-2.5-2.5,0.7-0.5);
\draw [->](-2.3,3) -- (1,3);
\draw [->](0.25-2.5,2-0) -- (0.6,2);
\draw [->](0.5-2.5,1-0.5) -- (2.1,1-0.5);
\draw [->](2.5,1.2-0.5) -- (2.5,1.6);
\draw [->](2.5,2.3) -- (2.5,2.6);
\draw [->](2.5,3.3) -- (2.5,3.8);
\draw [->](2.5,0.3-0.5) -- (2.5,0.7-0.5);
\node at (-4.5,1.3) {$\alpha$};
\node at (2,1.2) {$\beta$};
\node at (-0.6,3.4) {$\delta$};
\draw[->] (1.6,1.3) -- (2.1,1.6);
\node at (-0.6,2.3) {$\eta$};
\end{tikzpicture}
\end{center}
where the
vertical morphisms form short exact sequences
and the
remaining maps are induced by chain level action maps for each $p \in \Z$.
Since $\Lambda_\K^{Q^1_+}$ is a flat
$\Lambda_\K^{Q^0_+}$-module,
we have that $\alpha$ is an isomorphism.
Also $\beta$ is an isomorphism since $\varinjlim_j \varprojlim_{j'} W^p_{j,j',k}$ is a free finitely generated $\Lambda_\K^{Q^k_+}$-module for $k=0,1$
and $\beta$ sends the generators of one module to the other.
%
This implies that the map 
(\ref{equation isomorphism of specific chain complex homologies})
is an isomorphism.
\end{proof}


\section{Symplectic Geometry of Projective Varieties and Singular Ample Divisors.} \label{symplectic geometry of projective varieties}

\subsection{Constructing Appropriate K\"{a}hler Forms.}

In order to show that birational K\"{a}hler manifolds have the same small quantum groups,
we need to modify their K\"{a}hler forms
so that they
are identical on some large compact subset of a common open affine subset.
This will enable us
to show that various
Hamiltonian Floer
algebras are the
same on both Calabi-Yau
manifolds.
This technical subsection
is devoted to manipulating
certain symplectic forms on smooth affine varieties in order to achieve this goal.
It is also needed to ensure that this large compact subset contains a `large' index bounded contact cylinder (see Section \ref{subsection constructing contact cylinder}).

\begin{defn} \label{defn gradient with respect to symplectic vector field}
Let $\omega_X$ be a K\"{a}hler form on a complex manifold $(X,J_X)$.
Then for any open subset $U \subset X$ and any smooth function $f : U \lra{} \R$,
we define $\nabla_{\omega_X} f$
to be the unique vector field on $U$
satisfying
$\omega_X(\nabla_{\omega_X}f,J_X(-)) = df(-)$
(I.e. the gradient of $f$ with respect to the metric $\omega_X(-,J_X(-))$).
\end{defn}

\begin{defn} \label{definition algebraic plurisubharmonic function}
Let $A$ be a smooth affine variety.
Let $J_A : TA \lra{} TA$ be the complex
structure on $A$.
A smooth function $\rho : A \lra{} \R$ is
{\it exhausting} if it is proper and bounded
from below.
Define $d^c \rho := d\rho \circ J_A$.
We say that $\rho$ is {\it plurisubharmonic}
if $-dd^c \rho$ is a K\"{a}hler form.
For each plurisubharmonic function $\rho$,
we define $\omega_\rho := -dd^c \rho$
and for each function $f : A \lra{} \R$, we define
$\nabla_\rho f := \nabla_{\omega_\rho} f$.

A smooth function
$\rho : A \lra{} \R$ is an
{\it algebraic plurisubharmonic function}
if there exists a smooth projective variety
$X$ compactifying $A$, a holomorphic line bundle $L$ over $X$
with a Hermitian metric $|\cdot|$
and a holomorphic section $s$ of $L$
so that $\rho = -\log(|s|)$,
$s^{-1}(0) = X - A$
and $\rho$ is plurisubharmonic.

A smooth function $\rho : A \lra{} \R$
is a
{\it partially algebraic plurisubharmonic function}
if there is an algebraic plurisubharmonic
function $\rho_\infty : A \lra{} \R$
and a compact subset $K \subset A$
so that $\rho|_{A - K} = \rho_\infty|_{A-K}$ and if $\rho$ is plurisubharmonic.
\end{defn}

Algebraic plurisubharmonic functions always
exist since every affine variety can be compactified to a smooth projective variety $X$ by \cite{hironaka:resolution}.
One then can choose an ample line bundle $L$ together with a section $s$ satisfying $s^{-1}(0) = X - A$.
Any ample line bundle admits a positive Hermitian metric $|\cdot|$
(e.g. a metric induced from the Fubini Study metric) which implies that $-\log(|s|)$ is plurisubharmonic.
Having said that,
in general
we do not require that
the metric $|\cdot|$ be positive
outside $A$.
Also note that all algebraic plurisubharmonic
functions are exhausting.

The following technical lemma is needed to prove Corollaries 
\ref{lemma making forms coincide at infinity}
and \ref{corollary flow time} below.
This lemma is basically about controlling the size of the derivatives of partially algebraic plurisubharmonic functions near infinity.

\begin{lemma} \label{lemma gradient large enough at infinity}
Let $\rho_0,\rho_1$ be two partially algebraic plurisubharmonic functions on
a smooth affine variety $A$.
Define $\rho_t := \rho_0 + t \rho_1$ for each $t \in [0,1]$.
Let $k$ be a positive integer.
Then there is a
compact subset $K \subset A$
so that
\begin{equation} \label{equation gradient is very big}
d\rho_t(\nabla_{\rho_t} \rho_t) > \rho_t^k \quad \forall \ t \in [0,1]
\end{equation}
outside $K$.
Also there is a vector field $V$ on $A$
so that
\begin{equation} \label{equation vector field positive}
d\rho_0(V) > 0, \quad d\rho_1(V) > 0
\end{equation}
outside $K$.
\end{lemma}
\proof
The key idea here is to reduce
Equations (\ref{equation gradient is very big}) and (\ref{equation vector field positive}) to a local estimate near each point at infinity.
A similar thing has been done previously in
the proof of
\cite[Lemma 4.3]{Seidel:biasedview}.

First of all since we only require Equations (\ref{equation gradient is very big}) and (\ref{equation vector field positive}) to hold outside a compact set, we can assume that
$\rho_j$ is an algebraic plurisubharmonic
function for all $j = 0,1$.
Therefore, by definition,
for each $j \in \{0,1\}$
there is
\begin{itemize}
\item a smooth projective variety $X_j$ compactifying $A$,
\item  a holomorphic line bundle $\check{L}_j$ over $X_j$,
\item a Hermitian metric $\|\cdot\|_j$ on $\check{L}_j$ and
\item a holomorphic section $\check{s}_j$ of $\check{L}_j$
\end{itemize}
so that
$\rho_j = -\log(\|\check{s}_j\|_j)$
and
$\check{s}_j^{-1}(0) = X_j - A$.
By the Hironaka resolution of singularities
theorem \cite{hironaka:resolution},
there is a smooth projective variety $X$ compactifying $A$ and morphisms
$\pi_j : X \lra{} X_j$, $j=0,1$
satisfying $\pi_j(a) = a$ for each $a \in A$ and each $j =0,1$.
We can also assume that $X - A$
is a normal crossings variety (I.e. it is locally a transverse intersection of complex hypersurfaces).
Let $L_j = \pi_j^* \check{L}_j$,
$|\cdot|_j = \pi_j^*\|\cdot\|_j$
and
$s_j = \pi_j^* \check{s}_j$
be the corresponding pullbacks of our line bundle, metric and section to $X$ for each $j=0,1$.
Then
\begin{equation} \label{equation formula for rhoj}
\rho_j = -\log(|s_j|_j) \ \text{and} \ s_j^{-1}(0) = X - A \quad \forall \ j = 0,1.
\end{equation}
In order to prove Equation (\ref{equation gradient is very big})
it is sufficient for us to show that for each $x \in X - A$,
Equation (\ref{equation gradient is very big})
holds on a small neighborhood of $x$ since $X-A$ is compact.
Similarly, in order to prove
Equation (\ref{equation vector field positive}),
it is sufficient to show that there is a vector field $V_x$ defined in a neighborhood of $x$
satisfying $d\rho_j(V_x) > 0$, $j=0,1$.
This is because we can construct our desired vector field $V$ by patching together finitely many such vector fields $V_{x_1},\cdots,V_{x_k}$ using partitions of unity.

Therefore fix $x \in X - A$.
Choose a holomorphic coordinates $z_1,\cdots,z_n$
on a small chart $U_x \subset X$ centered at $x$
so that $(X - A) \cap U_x = \left\{\prod_{j = 1}^l z_j = 0\right\}$ for some $1 \leq l \leq n$.
After shrinking $U_x$, we can
choose trivializations of the line bundles $L_0|_{U_x}$ and $L_1|_{U_x}$.
We will also assume that the coordinates $z_1,\cdots,z_n$ and trivializations above extend to a neighborhood of the closure of $U_x$ in order to ensure that $C^1$ bounds hold.
For each $j=0,1$,
there are smooth functions $\eta_j : U_x \lra{} \R$ so that
$|\cdot|_j = e^{-\eta_j} |\cdot|$
with respect to the trivializations of $L_0|_U$ and $L_1|_U$ above.
Also for each $j=0,1$,
there are positive integers
$a^j_k$, $k=1,\cdots,l$
and a  holomorphic function $h_j : U_x \lra{} \C$ whose norm is bounded below by a positive constant
so that
$s_j = h_j\prod_{k=1}^l z_k^{a^j_k}$
with respect to the trivializations above.
Hence
$$
(|s_j|_j) |_{U_x} = e^{-\eta_j}|h_j| \prod_{k=1}^l |z_k|^{a^j_k} \quad \forall \ j = 0,1.
$$
Therefore
\begin{equation} \label{equation rhoj}
\rho_j|_{U_x \cap A} = \eta_j - \log(|h_j|) - \sum_{k=1}^l a^j_k \log(|z_k|) \quad \forall \ j = 0,1.
\end{equation}
Let $J_X$ be the complex structure on $X$.
Let $g(-,-)$ be
the standard Euclidean metric  on $U_x$ and let $\|\cdot\|_g$ be the induced norm on the cotangent bundle $T^*U_x$.
Let $\|\cdot\|_{\rho_0}$
be the induced norm on $T^* A$ coming from the metric
$\omega_{\rho_0}(-,J_X(-))$ on $A$.
The metric
$\omega_{\rho_0}(-,J_X(-))$
smoothly extends to a $(0,2)$-tensor $g_{\rho_0}$ on $X$ (which may be degenerate along points of $X-A$).
Since $g_{\rho_0}(Y,Y) \geq 0$ for all $Y \in TX$,
there is a constant $c>0$ so that
$g(Y,Y) \geq c g_{\rho_0}(Y,Y)$
for each $Y \in TU_x$.
Hence
$c\|\cdot\|_g|_A \leq  \|\cdot\|_{\rho_0}$.
Therefore, by Equation (\ref{equation rhoj}),
$$d\rho_t(\nabla_{\rho_t} \rho_t)|_{U_x \cap A}
= \|d\rho_t\|_{\rho_t}^2
\geq c^2\|d\rho_t\|_g^2$$
$$
\geq  \sum_{k=1}^l \frac{c^2}{|z_k|^2} \left\| (a^0_k+ta^1_k) d(|z_k|) \right\|_g^2
-$$
$$
\sum_{k=1}^l \frac{2c^2}{|z_k|}
\left(
\left\| (a^0_k+ta^1_k) d(|z_k|) \right\|_g \|d(\eta_0 - \log(|h_0|) + t d(\eta_1 - \log(|h_1|))\|_g^2\right), \quad \forall \ t \in [0,1].$$
Hence
Equation (\ref{equation gradient is very big}) holds near $x$ since the function $\frac{1}{y^2}$ grows much faster than $-\log(y)$ and $\frac{1}{y}$ as $y \to 0_+$.

Let $x_k$ be the real part of $z_k$
and $y_k$ the imaginary part for each $k=1,\cdots,n$.
Then $x_1,y_1,\cdots,x_n,y_n$ are real coordinates on $U_x$.
Now define the vector field $V_x := -\sum_{k=1}^n (x_k \frac{\partial}{\partial x_k} + y_k \frac{\partial}{\partial y_k})$ on $U_x$.
Then by Equation (\ref{equation rhoj}),
\begin{equation} \label{equation positive vector field stuff}
d\rho_j(V_x) = d(\eta_j - \log(|h_j|))(V_x)
+ \sum_{k=1}^l a^j_k, \quad \forall \ j=0,1
\end{equation}
which is positive in a neighborhood $\check{U}_x$ of $x$ since the $\|\cdot\|_g$ norm of $d(\eta_j - \log(|h_j|))$ is bounded and the $g$-norm of $V_x$ tends to $0$ as we approach $x$.
Choose points $x^1,\cdots,x^k$ in $X-A$ so that $\check{U}_{x^1},\cdots,\check{U}_{x^k}$ cover $X -A$.
Choose a partition of unity $f : A \lra{} [0,1]$, $f_j : \check{U}_j \lra{} [0,1]$, $j=1,\cdots,k$,  subordinate to the cover
$A, \check{U}_{x^1},\cdots,\check{U}_{x^k}$ of $X$.
Define $V := \sum_{j=1}^k f_j V_{x_j}$.
Then since Equation (\ref{equation positive vector field stuff}) holds inside $\check{U}_x$ for each $x \in X-A$, we get that $V$ satisfies Equation (\ref{equation vector field positive})
outside a compact subset of $X$.
%
%
%
%
%
%
%
%
%
%
%
\qed

\smallskip

%
%

The following corollary of Lemma \ref{lemma gradient large enough at infinity} will be used to construct appropriate K\"{a}hler forms on birational Calabi-Yau manifolds so that they coincide on some large compact set of some common affine variety.
This will be used in the proof of Theorem \ref{theorem main theorem} in Section \ref{section proof of main theorem} below.
Before we state this corollary,
we need a preliminary definition.
\begin{defn} \label{definition skeleton of plurisubharmonic function}
Let $\rho : A \lra{} \R$ be a plurisubharmonic function on an affine variety $A$.
Let $\phi_t : A \lra{} A$
be the time $t$ flow of $-\nabla_\rho \rho$ for each $t \geq 0$.
The {\it skeleton}
of $\rho$ is the subset
$$\cap_{t > 0} \phi_t(A) \subset A.$$
\end{defn}

Note that the complement of the skeleton of $\rho$ is diffeomorphic to a product $\R \times Y$
where $\{r\} \times Y$ is a level set of $\rho$ for each $r \in \R$.
As a result, one can think of $A$ as a `cylindrical end' $\R \times Y$ with the skeleton `glued' to one side (See Figure \ref{fig:skeleton}).
\begin{center}
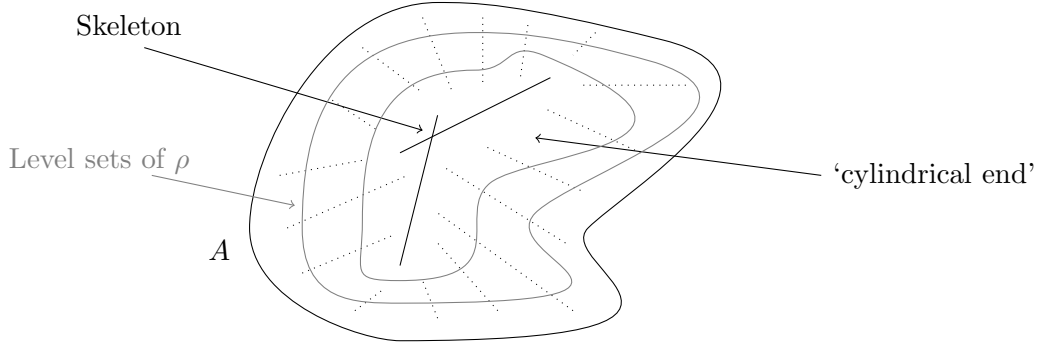
\begin{figure}[h]
\begin{tikzpicture}

\draw (0,2.5) .. controls (-1.5,2.5) and (-2.5,0.5) .. (-2.5,-0.5) .. controls (-2.5,-1.5) and (-1,-2) .. (-0.5,-2) .. controls (4.5,-2) and (1.5,-1) .. (2,-0.5) .. controls (2.5,0) and (5,1.5) .. (3,2) .. controls (1,2.5) and (0.5,2.5) .. (0,2.5);
\draw (1.5,1.5) -- (-0.5,0.5);
\draw (0,1) -- (-0.5,-1);
\draw [dotted](0.6,2.3) -- (0.6,1.4);
\draw [dotted](-1,1.9) -- (-0.2,1.3);
\draw [dotted](-0.2,2.3) -- (0.2,1.3);
\draw [dotted](-1.4,1.2) -- (-0.8,0.8);
\draw [dotted](-2.1,0.2) -- (-1,0.4);
\draw [dotted](-2,-0.5) -- (-0.4,0.2);
\draw [dotted](-1.8,-1.1) -- (-0.6,-0.6);
\draw [dotted](-1.1,-1.6) -- (-0.7,-1.3);
\draw [dotted](0,-1.7) -- (-0.2,-1.2);
\draw [dotted](0.8,-1.6) -- (0,-0.7);
\draw [dotted](1.8,-1.6) -- (0,-0.3);
\draw [dotted](1.7,-0.7) -- (0.1,0.3);
\draw [dotted](1.9,0) -- (0.6,0.6);
\draw [dotted](2.7,0.5) -- (1.4,1.1);
\draw [dotted](3.3,1.4) -- (1.9,1.4);
\draw [dotted](2.1,2.1) -- (1.8,1.8);
\draw [dotted](1.2,2.2) -- (1.1,1.5);
\draw [->](-3.9,1.9) -- (-0.2,0.8);
\node at (-4.1,2.2) {Skeleton};
\node at (-2.9,-0.8) {$A$};
\draw [->](5.1,0.2) -- (1.3,0.7);
\node at (6.6,0.2) {`cylindrical end'};
\draw [gray](0.5,2.1) .. controls (-1.5,2.1) and (-1.8,0.8) .. (-1.8,-0.5) .. controls (-1.8,-1.5) and (-1.1,-1.5) .. (-0.5,-1.5) .. controls (3.7,-1.5) and (0.7,-0.9) .. (1.3,-0.3) .. controls (1.8,0.2) and (4.7,1.1) .. (2.9,1.7) .. controls (1.5,2.1) and (0.9,2.1) .. (0.5,2.1);
\draw [gray](0.5,1.6) .. controls (-0.7,1.6) and (-1,0.8) .. (-1,-0.4) .. controls (-1,-1.1) and (-1.2,-1.2) .. (-0.5,-1.2) .. controls (1.1,-1.2) and (0.2,-0.3) .. (0.7,0.1) .. controls (1.2,0.5) and (4.3,0.6) .. (1.4,1.8) .. controls (0.9,2) and (1,1.6) .. (0.5,1.6);
\draw [gray,->](-3.8,0.2) -- (-1.9,-0.2);
\node at (-4.5,0.4) {\color{gray} Level sets of $\rho$};
\end{tikzpicture}
\caption{Skeleton of $A$.} \label{fig:skeleton}
\end{figure}
\end{center}

\begin{corollary} \label{lemma making forms coincide at infinity}
Let $\rho_0, \rho_1$ be two partially algebraic plurisubharmonic functions
on a smooth affine variety $A$.
Then for any compact subset $K \subset A$
there is a third partially algebraic plurisubharmonic function
$\rho : A \lra{} \R$, a compact set $Q$ containing $K$ and constants
$\kappa_1, \kappa_2 \in \N$
so that
\begin{itemize}
\item $\rho|_K = \rho_1|_K$,
\item $\rho$ is equal to $\kappa_1 (\rho_0 - \log{\kappa_2})$ outside $Q$ and
\item the skeleton of $\rho$ is equal to the skeleton of $\rho_1$.
\end{itemize}
\end{corollary}
\proof
By Lemma \ref{lemma gradient large enough at infinity}
there is a vector field $V$ on $A$
and a compact subset $Q' \subset A$
so that $d\rho_0(V) > 0$ and $d\rho_1(V) >0$ outside $Q'$.
By enlarging $Q'$, we can assume that it contains $K$.
Let $\alpha := \max (\rho_0|_{Q'})+\log(2)$.
Choose a compact subset $Q \subset A$
whose interior contains $Q'$
and so that $\rho_0$ is greater than
$\max(\alpha,2)$ outside a compact subset of the interior of $Q$.
Choose $\kappa_2 \in \N$ so that
$\max (\rho_0|_{Q'}) < \log(\kappa_2) < \max(\alpha,2)$.
Then
$$\rho_0|_{Q'} - \log(\kappa_2) < 0, \quad \rho_0 |_{\overline{A - Q}} - \log(\kappa_2) > 0.$$
Hence we can choose an integer $\kappa_1 \gg 1$
so that
$f := \kappa_1 (\rho_0 - \log(\kappa_2))$ satisfies
$f|_{Q'} < \rho_1|_{Q'}$
and $f > \rho_1$ outside a compact subset of the interior of $Q$.
Then by 
\cite[Propositions 3.20]{CieliebakEliashberg:symplecticgeomofsteinmflds},
we can smooth the function $\max(\rho_1,f)$
to a plurisubharmonic function
$\rho : A \lra{} \R$ so that
$\rho|_{Q'} = \rho_1|_{Q'}$,
$\rho|_{A - Q} = f|_{A-Q}$
and
$d\rho(V) >0$
outside a compact subset of the interior of $Q'$. Then $\rho$ has the required properties.
\qed

\bigskip

The following technical corollary of Lemma
\ref{lemma gradient large enough at infinity}
will be used in the proof of
Proposition \ref{proposition constructino of index bounded contact cylinder on appropriate affine variety}
to construct certain index bounded contact cylinders.

\begin{corollary} \label{corollary flow time}
Let
$\rho_0$, $\rho_1$ be two partially algebraic plurisubharmonic functions on
 a smooth affine variety $A$
 and let $K \subset A$ be a compact subset.
Then there are
constants $0 < \delta \ll T \ll 1$,
a compact set $Q \subset A$ containing $K$
and an exhausting plurisubharmonic function
$\rho$ on $A$
satisfying
\begin{enumerate}
\item $\rho|_Q = \rho_0|_Q$,
\item $\rho$ is equal to $\rho_0 + \delta \rho_1$ outside a large compact set and
\item for all $x \in A$,
the time $T$ flow of $x$
along $-\nabla_\rho \rho$ is contained in $Q$ and 
is disjoint from $K$ if, in addition, $x \notin Q$.
\end{enumerate}
\end{corollary}
\proof
Define $\rho_t := \rho_0 + t \rho_1$ for each $t \in [0,1]$.
Let $\phi_{t,\tau} : A \lra{} A$
be the time $\tau$
flow of $-\nabla_{\rho_t} \rho_t$
for all $\tau \geq 0$ and $t \in [0,1]$.
For each $\tau \geq 0$,
define
$$\widehat{A}_\tau := \cup_{t \in [0,1]} \phi_{t,\tau}(A),
\quad
\check{A}_\tau := \cap_{t \in [0,1]} \phi_{t,\tau}(A).
$$
By Lemma \ref{lemma gradient large enough at infinity},
\begin{equation} \label{equation large enough gradient}
d\rho_t(\nabla_{\rho_t} \rho_t) > \rho_t^2, \quad \forall t \in [0,1]
\end{equation}
outside a compact subset of $A$.
This implies that
$\widehat{A}_\tau$ and $\check{A}_\tau$ are relatively compact subsets for each $\tau > 0$ since $(\rho_t)_{t \in [0,1]}$ are exhausting functions.
We also have the following
properties:
\begin{enumerate}[label=(\alph*)]
\item \label{item closure contained in interior} $\overline{\widehat{A}_{\tau_0}} \subset \widehat{A}_{\tau_1}$
and
$\overline{\check{A}_{\tau_0}} \subset \check{A}_{\tau_1}$
for each $0 \leq \tau_1 < \tau_2$.
\item \label{item flowing sets} $\phi_{t,\tau_0}(\check{A}_{\tau_1}) \subset \widehat{A}_{\tau_0 + \tau_1}$
and
$\phi_{t,\tau_0}(\widehat{A}_{\tau_1}) \supset \check{A}_{\tau_0 + \tau_1}$
for each $\tau_0,\tau_1 \geq 0$.
\end{enumerate}
Choose constants $1 \gg \tau_0 \gg \tau_1 \gg \tau_2>0$
so that
$K \subset \check{A}_{4\tau_0}$,
$\overline{\widehat{A}_{\tau_0}} \subset \check{A}_{2\tau_1}$
and
$\overline{\widehat{A}_{\tau_1}}
\subset \check{A}_{2\tau_2}$.
Define 
$T := 2 \tau_0$
and
$Q := \overline{\widehat{A}_{\tau_0}}$.
Let
$\beta : A \lra{} \R$
be a smooth function equal to
$0$ along $Q$
and $1$ outside $\check{A}_{2\tau_1}$.
Define
$\check{\rho}_t := (1-\beta) \rho_0 + \beta \rho_t$ for all $t \in [0,1]$.
Since being plurisubharmonic is a $C^2$ open condition,
there is a constant $\eta > 0$
so that $\check{\rho}_t$ is plurisubharmonic for all $t \in [0,\eta]$.
Let $\psi_{t,\tau}$ be the time $\tau$ flow of $- \nabla_{\check{\rho}_t} \check{\rho}_t$
for each $t \in [0,\eta]$.
Since 
$\rho_t$
converges to $\rho_0$ in the $C^\infty_{\textnormal{loc}}$ topology as $t \to 0$ and by \ref{item closure contained in interior}, \ref{item flowing sets} above combined with the fact that $\widehat{A}_\tau$, $\check{A}_\tau$ are relatively compact for all $\tau>0$,
there exists $\delta \in (0,\eta]$ small enough so that
\begin{equation} \label{eqn large compact subsets flowing correctly}
\psi_{\delta,\tau}(\check{A}_{\tau_2}) \subset Q, \quad
\psi_{\delta,\tau}(\check{A}_{\tau_2} - Q) \cap K = \emptyset, \quad \forall \ \tau \in \{T - \tau_1, T\}.
\end{equation}
Since $\beta^{-1}([0,1)) \subset \check{A}_{2\tau_1}$  and by \ref{item closure contained in interior}, \ref{item flowing sets}, we have
$$\psi_{\delta,\tau_1}(A - \widehat{A}_{\tau_1}) = \phi_{\delta,\tau_1}(A - \widehat{A}_{\tau_1}) \subset \widehat{A}_{\tau_1} - \check{A}_{2\tau_1}$$
and hence by Equation (\ref{eqn large compact subsets flowing correctly}),
$$\psi_{\delta,T}(A) \subset Q, \quad
\psi_{\delta,T}(A - Q) \cap K = \emptyset.
$$
Hence $\rho := \check{\rho}_{\delta}$ satisfies the properties we want.
\qed

\subsection{Constructing Index Bounded Contact Cylinders.} \label{subsection constructing contact cylinder}

The main aim of this section is to
prove the proposition below
which
 constructs appropriate
index bounded contact cylinders
in Calabi-Yau manifolds.
Recall that a
{\it contact cylinder}
is a codimension $0$ symplectic embedding of a subset
$\check{C} = [1-\epsilon,1+\epsilon] \times C$
of a symplectization of a contact manifold $C$  which bounds a Liouville domain $D$ (see Definition \ref{definition contact cylinder} for more precise details).
Such a contact cylinder
is {\it index bounded}
if 
for each $m > 0$,
there is a constant $\mu_m>0$ so that
each Reeb orbit in $C$
of index in $[-m,m]$
has length $<\mu_m$
(see Definition \ref{definition index bounded contact cylinder}).

\begin{prop} \label{proposition constructino of index bounded contact cylinder on appropriate affine variety}
Let $X$ be a smooth projective variety satisfying
$c_1(X) = 0$
and let $A \subset X$ be an affine open subset.
Let $\rho : A \lra{} \R$ be a partially algebraic plurisubharmonic function as in Definition \ref{definition algebraic plurisubharmonic function} so that $-dd^c \rho$ extends to a K\"{a}hler form $\omega_X$ on $X$.
Then for any compact subset $K \subset A$,
there exists an index bounded contact cylinder $\check{C}$ of $(X,\omega_X)$
inside $A - K$
whose associated Liouville domain $D$
satisfies $K \subset D \subset A$.
Also $\check{C}$ contains a contact cylinder $\check{C}_0 \subset \check{C}$
whose associated Liouville domain
domain $D_0$ contains $D$ and where $-d^c \rho|_{D_0}$ is a Liouville form associated to $\check{C}_0$.
\end{prop}

The idea of the proof is to blow up $X$ so that the complement of $A$ is a smooth normal crossing divisor.
One then finds a nice symplectic neighborhood of this divisor (after deformation)
and constructs an index bounded contact cylinder in this neighborhood.
As a result we need some technical definitions and lemmas
about the symplectic geometry of normal crossing divisors.

Before that, we need a lemma
about Conley-Zehnder indices of matrices.
\begin{lemma} \label{lemma bounded family of symplectic matrices}
Let $\R^{2n}$ be
the standard symplectic vector space and let $L \subset \R^{2n}$
be a fixed linear Lagrangian subspace.
Let $A := (A_t)_{t \in [0,1]}$
be a smooth family of symplectic matrices on $\R^{2n}$
so that $A_t(x) = x$ for all $x \in L$.
Then $CZ(A) \in [-2n,2n]$.
\end{lemma}
\proof
Let $x_1,x_2,\cdots,x_n,y_1,\cdots,y_n$
be a basis of $\R^{2n}$
so that
the standard linear symplectic
form on it is
$\Omega_\std = \sum_{i=1}^n x_i^* \wedge y_i^*$
where $x_1^*,\cdots,x_n^*,y_1^*,\cdots,y_n^*$
is the corresponding dual basis.
We can also assume that $L = \{y_1^*=y_2^*=\cdots=y_n^*=0\}$.
Since $A_t(x_i)=x_i$ for each $i=1,\cdots,n$, 
there exists a smooth family of $n \times n$ symmetric matrices
$(B_t)_{t \in [0,1]}$
so that
$$
A_t = \left(
\begin{array}{ll}
\id & -B_t \\
0 & \id
\end{array}
\right), \quad \forall \ t \in [0,1]
$$
with respect to the basis above where $\id$ is the $n \times n$ identity matrix.
Since the space of symmetric $n \times n$ matrices is convex, we have that the above path of matrices is homotopic relative to its endpoints to the catenation of the paths $A^0 = (A^0_t)_{t \in [0,1]}$
and $A^1 = (A^1_t)_{t \in [0,1]}$ where:
\begin{equation} \label{eqn shear matrices}
A^0_t := \left(
\begin{array}{ll}
\id & -(1-t)B_0 \\
0 & \id
\end{array}
\right), \quad
A^1_t := \left(
\begin{array}{ll}
\id & -tB_1 \\
0 & \id
\end{array}
\right),
\quad t \in [0,1].
\end{equation}
Therefore by \ref{item:catenationczproperty} and \ref{item:homotopyinvariance},
\begin{equation} \label{equation sum of indices}
CZ(A)
= CZ(A^0)
+ CZ(A^1).
\end{equation}
Also by
\ref{item:catenationczproperty},
$CZ(A^0)
= -CZ((A^0_{1-t})_{t \in [0,1]})$,
and hence by \ref{item:sheartransformation} and
Equation
(\ref{eqn shear matrices}),
$$CZ(A^0)
= -\frac{1}{2} \textnormal{Sign}(B_0),
\quad
CZ(A^1)
= \frac{1}{2}\textnormal{Sign}(B_1).
$$
Combining this with Equation
(\ref{equation sum of indices})
gives us
$$CZ((A_t)_{t \in [0,1]})
= \frac{1}{2}\left(
\textnormal{Sign}(B_1)-\textnormal{Sign}(B_0)
\right) \in [-2n,2n].$$
%
%
%
%
\qed

\begin{defn} \label{defn sc divisor}
\cite[Definition 2.1]{McLeanTehraniZinger:normalcrossings}.
Let $(\Delta_i)_{i \in S}$ be a finite
collection of
transversally intersecting
closed codimension $2$ symplectic submanifolds
of a compact symplectic manifold
$(W,\omega_W)$ so that
$\Delta_I := \cap_{i \in I} \Delta_i$ is a symplectic submanifold for all $I \subset S$ (our convention is that if $I = \emptyset$ then $\Delta_I = W$).
The {\it symplectic orientation}
on $\Delta_I$ is the orientation on $\Delta_I$
induced by the symplectic form.
Let $N\Delta_I$ be the normal bundle of $\Delta_I$ for each $I \subset S$.
The {\it intersection orientation}
on $\Delta_I$ is the orientation
on $\Delta_I$ coming from the symplectic orientation on its normal bundle
induced by the splitting
$N\Delta_I = \oplus_{i \in I} N\Delta_i|_{\Delta_I}$
and the symplectic orientation on $M$.

A {\it symplectic crossings divisor}
or {\it SC divisor}
inside a symplectic manifold $(W,\omega_W)$
is a finite
collection
$(\Delta_i)_{ i \in S}$
of transversally intersecting closed submanifolds of $W$ as above
so that $\Delta_I$ is symplectic and so that the
symplectic orientation and the intersection orientation
of $\Delta_I$ agree for all
$I \subset S$.
\end{defn}

One of the main examples
of an SC divisor to keep in mind
is a union of transversally intersecting complex hypersurfaces
in a K\"{a}hler manifold.
The definition above
should be thought of a
symplectic version
of such a union of complex
hypersurfaces.
This definition is more flexible
and will enable us to control what
the symplectic structure
looks like near $\cup_{i \in S} \Delta_i$
after deforming $(\Delta_i)_{i \in S}$.
For instance,
it would be nice for these $\Delta_i$'s to be symplectically orthogonal to each other after deformation.
The condition ensuring that the symplectic and intersection orientation agree is crucial for such a deformation to exist.
An example where this does not happen is if
$W = T^* \R^2$ and
if $\Delta_1$ is the graph of $xdy$
and $\Delta_2$ the graph of $y dx$
where $x,y$ are the standard coordinates on $\R^2$.
These are two symplectic hypersurfaces which intersect negatively at the origin.
There is no way of isotoping these linear symplectic subspaces so that
they intersect orthogonally with respect to the symplectic form.
A more sophisticated example where
this orientation condition fails is contained in \cite[Example 2.7]{McLeanTehraniZinger:normalcrossings}.

\begin{defn} \label{definition nice neighborhood of sc divisor}
Let $Q \subset W$ be a submanifold of a
manifold $W$.
A {\it tubular neighborhood}
of $Q$ is a smooth fibration
$\pi_Q : U_Q \lra{} Q$
so that
\begin{enumerate}
	\item $U_Q \subset W$ is an open
	subset containing $Q$,
	\item there is a metric $g$
	on $W$
	so that
$\pi_Q = \pi_{DQ} \circ \exp^{-1}$
where
$$DQ := \{v  \in TW|_x \ : \ x \in Q, \ g(v,v) < 1, \ g(v,w) = 0 \ \forall \ w \in TQ|_x \}$$
is the unit disk normal bundle,
$\pi_{DQ} : DQ \lra{} Q$ is the natural projection map
and
$\exp : DQ \lra{} W$ is the exponential map with respect to $g$ so that
\item $\exp$ is an embedding and $\exp(DQ) = U_Q$.
\end{enumerate}

Recall that an {\it Ehresmann connection}
on a smooth fiber bundle $\pi : E \lra{} B$
is a subbundle $H \subset TE$
so that $D\pi|_{H|_x} : H|_x \lra{} TB|_{\pi(x)}$ is an isomorphism for each
$x \in E$.
Such a connection is {\it complete}
if for every smooth embedding
$p : [0,1] \lra{} B$
and every $x \in \pi^{-1}(p(0))$,
there is a unique lift
$\widetilde{p} : [0,1] \lra{} E$
of $p$ tangent to $H$.
In other words,
points in $E$ do not parallel transport to infinity in finite time.

If $\pi : E \lra{} B$
is a smooth fibration
and $\Omega$ a closed
$2$-form on $E$ making the fibers of $E$ symplectic
then the {\it associated symplectic connection}
is defined to be the Ehresmann connection consisting of vectors $\Omega$-orthogonal to the fibers. I.e.
$$H_\Omega := \{Q \in TE|_x \ : \ x \in E, \ \Omega(Q,A) = 0 \ \forall \ A \in \ker(D\pi)|_x\}.$$

Let $I$ be a finite set.	
A {\it symplectic
$U(1)^I$ neighborhood}
of a symplectic submanifold $Q \subset W$
of a symplectic manifold $(W,\omega_W)$
is a tubular neighborhood
$\pi_Q : U_Q \lra{} Q$ of $Q$
where
\begin{enumerate}
	\item the fibers are symplectic submanifolds symplectomorphic to $\prod_{i \in I}\D_i(\epsilon)$ where $\D_i(\epsilon) \subset \C$ is the open symplectic disk of radius $\epsilon$ labeled by $i \in I$,
	\item the fiber bundle $\pi_Q$
	has structure group $U(1)^I := \prod_{i \in I} U(1)$ given by rotating such disks in the natural way and
	\item
	the associated symplectic connection is complete and
	the parallel transport maps
	induced by the symplectic connection
	respect the above structure group (in other words, parallel transport maps
	between fibers in a $U(1)^I$ trivialization are elements of $U(1)^I$).
\end{enumerate}
For $I' \subset I$,
let
$U_Q^{I'} \subset U_Q$
be the subset of points fixed under
the $U(1)^{I'} \subset U(1)^I$ action.
For each $I' \subset I$,
the {\it $U(1)^{I'}$-bundle
	associated to $\pi_Q$}
is the fibration
$\pi^{I'}_Q : U_Q \lra{} U_Q^{I'}$
whose restriction to each fiber in a $U(1)^I$ trivialization
is the natural projection map
$$\prod_{i \in I} \D_i(\epsilon)
\lra{} \prod_{i \in I'} \D_i(\epsilon)$$
and with induced structure group
$U(1)^{I'}$.
This is naturally a symplectic $U(1)^{I'}$ neighborhood of $U_Q^{I'}$
inside $U_Q$.

A {\it standard tubular neighborhood}
of an SC divisor $(\Delta_i)_{i \in S}$
consists of a symplectic $U(1)^I$
neighborhood
$\pi_I : U_I \lra{} \Delta_I$ for each $I \subset S$
so that
\begin{enumerate}
	\item $U_I \cap U_{I'} = U_{I \cup I'}$ for all $I , I' \subset S$,
	\item $\pi_{I'}(U_I) = U_I \cap \Delta_{I'}$ for all $I' \subset I \subset S$ and
	\item the $U(1)^{I'}$-bundle
	associated to
	$\pi_I$ is equal to
	$\pi_{I'}|_{U_I} : U_I \lra{} U_I \cap \Delta_{I'}$
	as fiber bundles with structure groups
	$U_Q^{I'}$
	for all $I' \subset I \subset S$.
\end{enumerate}

The {\it radius} of this standard tubular neighborhood is $\epsilon$.
The {\it radial coordinate $r_i : U_i \lra{} \R$ corresponding to $\Delta_i$}
is the map whose restriction to each fiber $\D_i(\epsilon)$ of a $U(1)^{\{i\}}$-trivialization
of $\pi_i$ is the standard radial coordinate on this disk.
\end{defn}

\begin{defn} \label{definition wrapping number discrepancy}
Let $(W,\omega_W)$ be a closed
symplectic manifold of dimension $2n$
and let $(\Delta_i)_{i \in S}$ be a symplectic
SC divisor in $W$. Define $W^o := W - \cup_{j \in S} \Delta_j$.
Let $\theta \in \Omega^1(W^o)$ satisfy $d\theta = \omega_W|_{W^o}$.
Let $N$ be a neighborhood of $\cup_{j \in S} \Delta_j$ which deformation retracts on to $\cup_{j \in S} \Delta_j$
and let $\beta : N \lra{} [0,1]$ be a compactly supported smooth function equal to
$1$ near $\cup_{j \in S} \Delta_j$.
Let $\omega_c$ be the compactly supported closed $2$-form
on $N$ equal to $\omega$ near $\cup_{j \in S} \Delta_j$
and $d(\beta \theta)$ inside $N \cap W^o$.
Then since $N$ deformation retracts
onto $\cup_{j \in S} \Delta_j$,
the Lefschetz dual of $\omega_c$
is equal to $-\sum_j w_j [\Delta_j] \in H_{2n-2}(N;\R)$ for unique real numbers $(w_j)_{j \in S}$.
The {\it wrapping number of $\theta$ around $\Delta_j$} is defined to be $w_j$ for each $j \in S$.
We call $(\Delta_j)_{j \in S}$ a {\it negatively wrapped divisor} if there exists a $1$-form
$\theta$ on $W^o$ as above so that
the wrapping number of $\theta$ around $W_j$ is negative for each $j \in S$.

\end{defn}

The following lemma gives us an important example of a negatively wrapped divisor.
\begin{lemma} \label{lemma negatively wrapped divisor}
Let $X$ be a complex projective variety
and let $A \subset X$ be a codimension $0$ affine subvariety so that $X - A$ is a union of transversally intersecting complex hypersurfaces
$D_1,\cdots,D_l$.
Let $\rho : A \lra{} \R$ be an exhausting
plurisubharmonic function on $A$
as in Definition \ref{definition algebraic plurisubharmonic function}
so that $-dd^c \rho$ extends to
a K\"{a}hler form $\omega_X$ on $X$.
Then $(D_i)_{i=1}^l$ is a negatively wrapped divisor on $(X,\omega_X)$.
\end{lemma}
\proof
Let $\theta = -d^c \rho$.
We will show that the wrapping number
of $\theta$ around $D_i$
is negative for each $i \in \{1,\cdots,l\}$.
Fix such a $D_i$.
We will use a characterization of wrapping number
in terms of embedded disks
(See \cite[Lemma 5.5]{McLean:isolated}).
Let $\D \subset \C$ be the closed unit disk
and let $\iota : \D \hookrightarrow X$
be a holomorphic embedding
so that
\begin{itemize}
\item $\iota^{-1}(D_i) = \{0\}$,
\item $\iota(\D)$ intersects $D_i$ transversally and
\item $\iota^{-1}(D_j) = \emptyset$ for each $j \neq i$.
\end{itemize}
Since $\omega_X$ is a K\"{a}hler form
and since $\D$ is contractible,
there exists a smooth function
$f : \D \lra{} \R$
so that $\iota^*\omega_X = -dd^c f$.
Since $-dd^c (\iota^*\rho -f) = 0$,
we get that $-d^c(\iota^*\rho - f)$
represents a De Rham cohomology class
in $H^1(\D - \{0\};\R)$.
By \cite[Lemma 5.5]{McLean:isolated},
this cohomology class determines the wrapping number $w_i$ of $\theta$ around $D_i$.
Since De Rham cohomology classes
in $H^1(\D - \{0\};\R)$
are determined by integration around any loop wrapping once positively around the origin,
we get (\cite[Lemma 5.5]{McLean:isolated}):
\begin{equation} \label{equation wrapping number as integral}
w_i = \frac{1}{2\pi}\int_{\partial \D} -d^c(\iota^*\rho - f).
\end{equation}

Let $(r,\vartheta)$ be the standard polar coordinates on $\D$ and let $\partial_r := \frac{\partial}{\partial r}$ be the unit radial vector field on $\D - \{0\}$.
Define
$$\kappa : \D - \{0\} \lra{} \R, \quad \kappa := d(\iota^*\rho - f)\left(\partial_r\right).$$
Since $\rho$ is an exhausting function,
we have that $\rho(z) - f(z)$ tends to infinity as $|z|$ tends to $0$.
Hence there exists $\eta \in (0,1)$
so that
$$\int_0^{2\pi} \kappa(\eta e^{i\vartheta}) d\vartheta < 0.$$
Therefore
$$\int_{\{r = \eta\}} -d^c(\iota^*\rho -f)
= \eta\int_0^{2\pi} \kappa(\eta e^{i\vartheta}) d\vartheta < 0.$$
Hence the integral (\ref{equation wrapping number as integral})
is negative
which implies that the wrapping number
of $\theta$ around $D_i$ is negative.
\qed

\begin{defn} \label{defn anti canonical bundle}
Let $W$ be a manifold of dimension $2n$ with an almost complex structure $J$.
The {\it anti-canonical bundle}
of $(W,J)$ is the complex line bundle
$\kappa^*_W := \wedge^n_\C(TW,J)$.
Let $(\Delta_j)_{j \in S}$
be a finite collection of transversally intersecting codimension $2$ submanifolds.
Define $W^o := W - \cup_{j \in S} \Delta_j$.
Let $\tau : \kappa^*_W|_{W^o} \lra{} W^o \times \C$ be a trivialization of the anti-canonical bundle of $(W^o,J)$.
Let $N \subset W$ be a neighborhood of $\cup_{j \in S} \Delta_j$ which deformation retracts on to $\cup_{j \in S} \Delta_j$ and let $s$ be a smooth
section of $\kappa^*_W|_N$
transverse to $0$
so that $s(x) := \tau^{-1}(x,1)$
for all $x$
outside a compact subset of $N$.
Then $[s^{-1}(0)] = -\sum_j a_j [\Delta_j] \in H_{2n-2}(N;\Z)$
for unique $a_j \in \Z, \ j \in S$.
We define
the {\it $\tau$-discrepancy of $\Delta_j$}
to be $a_j$ for each $j \in S$.
If it is clear that we are using the trivialization $\tau$ (up to isotopy),
we will call the $\tau$-discrepancy of $\Delta_j$ the {\it discrepancy of $\Delta_j$}
for each $j \in S$.
\end{defn}

The following lemma gives us
an example of an SC divisor
where the discrepancy of its components are non-negative.
Before we state the lemma,
we remind the reader of some important notions in algebraic geometry.

\begin{defn} \label{definition divisors and line bundle etc}
Let $Y$ be a compact complex manifold of complex dimension $n$.
We define the {\it canonical bundle} $\kappa_Y$ of $Y$ to be the dual of the anti-canonical bundle of $Y$.
Let $\pi : \check{Y} \lra{} Y$ be a holomorphic map. 
We define the
{\it relative canonical bundle} of $\pi$
to be
\begin{equation} \label{equation relative canonical bundle}
\kappa_{\widetilde{Y}/Y} := \kappa_{\widetilde{Y}} \otimes (\pi^* \kappa_Y)^*.
\end{equation}
For any subvariety $F \subset \check{Y}$,
we say that $F$ is {\it contracted by $\pi$} if the dimension of the subvariety $\pi(F) \subset Y$ is less than the dimension of $F$.

Recall that a {\it divisor}
in $Y$
is a formal $\Z$-linear combination 
$D = \sum_{i=1}^l a_i D_i$
of codimension $1$ subvarieties  $D_1,\cdots,D_l$ in $Y$.
Such a divisor is {\it effective}
if $a_i \geq 0$ for each $i=1,\cdots,l$.
The {\it support} $\textnormal{supp}(D)$
of $D$ is the subset $\bigcup_{i \in \{1,\cdots, l\}, a_i \neq 0} D_i$.
The associated {\it homology class
$[D]$} is the sum
$\sum_{i=1}^l a_i [D_i] \in H_{2n-2}(Y;\Z)$
where $[D_i]$ is the fundamental class of $D_i$ (See \cite[Chapter 0, Section 4]{GriffithsHarris:algeraicgeometry}).
We will define
$[D]^* \in H^2(Y;\Z)$
to be the Poincar\'{e} dual of $[D]$.
Two divisors are {\it numerically equivalent}
if they represent the same homology class.
The {\it divisor line bundle correspondence}
gives us a 1-1 correspondence
between divisors $E$
and pairs $(L,s)$
where $L$ is a line bundle and $s$ a holomorphic section of $L$
(see \cite[Chapter 1, Section 1]{GriffithsHarris:algeraicgeometry}).
We define $(s)$ to be the {\it divisor associated to $(L,s)$} under the divisor line bundle correspondence.
\end{defn}

Under the divisor line bundle correspondence, we have that $[(s)]^* = c_1(L)$ (\cite[Chapter 1, Section 1]{GriffithsHarris:algeraicgeometry})).
Therefore if $s$ and $\check{s}$
are holomorphic sections of two line bundles $L$ and $\check{L}$
then $(s)$ is numerically equivalent to $(\check{s})$
if and only if $c_1(L) = c_1(\check{L})$.

\begin{lemma} \label{lemma non-negative discrepancy example}
Let $\pi : \widetilde{X} \lra{} X$
be a morphism of smooth projective varieties over $\C$.
Let $A \subset X$, $\widetilde{A} \subset \widetilde{X}$ be Zariski dense 
affine subvarieties
so that
\begin{equation} \label{equation isomorphism outside divisor}
\pi|_{\widetilde{A}} : \widetilde{A} \lra{} A
\end{equation}
is an isomorphism.
Suppose $\widetilde{X} - \widetilde{A}$ is a union of transversally intersecting complex hypersurfaces $D_1,\cdots,D_l$.
Let $\tau : \kappa^*_X \lra{} X \times \C$ be a trivialization of the anti-canonical bundle
of $X$
and let
$\widetilde{\tau} := \tau \circ (\pi|_{\widetilde{A}})$
be the induced trivialization of the anti-canonical bundle of $\widetilde{A}$.
Then the $\widetilde{\tau}$-discrepancy
of $D_i$ is non-negative for each $i=1,\cdots,l$.
\end{lemma}
\proof
Let $\kappa_{\widetilde{X}}$ and $\kappa_X$ be the canonical bundles of $\widetilde{X}$ and $X$ respectively
and let $\kappa_{\widetilde{X}/X}$
be the relative canonical bundle of $\pi$.
The {\it Jacobian} of $\pi$
gives us a holomorphic section $s$
of $\kappa_{\widetilde{X}/X}$
in the following way:
By Equation
(\ref{equation relative canonical bundle}),
$\kappa_{\widetilde{X}/X}$ is naturally isomorphic to
$\Hom(\kappa^*_{\widetilde{X}},\pi^* \kappa^*_X)$
where $\kappa^*_{\widetilde{X}}$
and $\kappa^*_X$ are the anti-canonical bundles of $\widetilde{X}$ and $X$ respectively (See Definition \ref{defn anti canonical bundle}).
Under this identification,
the section $s$ is given
by the map sending
$v_1 \wedge \cdots \wedge v_1 \in \kappa^*_{\widetilde{X}}$
to $D\pi(v_1) \wedge \cdots \wedge D\pi(v_n) \in \pi^* \kappa^*_X$.
We have
$(s) = \sum_{i=1}^l b_i F_i$
is an effective divisor
with the property that $F_i$ is contracted by $\pi$ for each $i=1,\cdots,l$.
Since (\ref{equation isomorphism outside divisor})
is an isomorphism,
we get that $F_i \subset \widetilde{X} - \widetilde{A}$ for each $i=1,\cdots,k$.
Hence for each $i =1,\cdots,k$,
there exists $j_i \in \{1,\cdots,l\}$ so that
$F_i = D_{j_i}$.
Hence
$(s) = \sum_{i=1}^l a_i D_i$
for some non-negative integers $a_1,\cdots,a_l$.

Since $\tau$ is a trivialization of the anti-canonical bundle of $X$,
we get an induced trivialization
$\check{\tau} : \kappa_X \lra{} X \times \C$
of the canonical bundle (this is the unique trivialization so that $\check{\tau} \otimes \tau$ is the natural trivialization of $\kappa_X \otimes \kappa^*_X$).
Let $\sigma : X \lra{} \kappa_X$
be the unique smooth section
satisfying $\check{\tau}(\sigma(x)) = (x,1)$
for each $x \in X$.
By (\ref{equation relative canonical bundle}), we have $\kappa_{\widetilde{X}} = \kappa_{\widetilde{X}/X} \otimes \pi^* \kappa_X$
and
hence $s \otimes \pi^* \sigma$
is a smooth section of $\kappa_{\widetilde{X}}$.
Let $s'$ be a $C^\infty$ small perturbation of $s \otimes \pi^* \sigma$
which is transverse to $0$ and which is equal to $s \otimes \pi^* \sigma$ outside a small neighborhood of $\cup_{i =1}^l D_i$.
Then since $\sigma$ is nowhere zero,
we get that $[(s')^{-1}(0)]$
is homologous to $\sum_{i=1}^l a_i [D_i]$.
Since $k^*_{\widetilde{X}}$
is dual to $k_X$,
this implies that the $\tau$-discrepancy of $D_i$ is $a_i \geq 0$ for each $i=1,\cdots,l$.
\qed

\begin{defn} \label{definition index bounded contact cylinder II}

Let $(W,\omega_W)$ be a closed
symplectic manifold of dimension $2n$
and let $(\Delta_i)_{i \in S}$ be a symplectic
SC divisor in $W$. Define $W^o := W - \cup_{j \in S} \Delta_j$.
Let $J$ be an $\omega_W$-tame almost complex structure
and let $\tau$ be a trivialization of the anti-canonical bundle of $(W^o,J)$.

A {\it contact cylinder in}
$\check{C} \subset W^o$
is a contact cylinder as in Definition
\ref{definition contact cylinder}
with the symplectic manifold $(M,\omega)$
replaced by $(W^o,\omega_W)$.
We say that $\check{C}$ is {\it index bounded}
if it is index bounded in the sense of
Definition
\ref{definition index bounded contact cylinder}
with $(M,\omega)$ replaced
by $(W^o,\omega_W)$
and where the symplectic trivialization used along each Reeb orbit is induced by $\tau$ and where we consider all Reeb orbits (not just null homologous ones).
\end{defn}

	\begin{prop} \label{proposition constructing appropriate contact cylinder near normal crossing divisor}
Let $(\Delta_i)_{i \in S}$ be a negatively wrapped symplectic SC divisor
in a closed
symplectic manifold $(W,\omega_W)$
and let $U$ be an open neighborhood
of $\cup_{j \in S} \Delta_j$ in $W$.
Let $W^o, J, \tau$
be as in
Definition \ref{definition index bounded contact cylinder II} above.
Let $\rho : W - \cup_{i \in S} \Delta_i \lra{} \R$ be an exhausting smooth function.
Suppose that the $\tau$-discrepancy of
$\Delta_j$ is non-negative for each $j \in S$.
Then there is an index bounded contact cylinder $\check{C}$ contained in $U \cap W^o$
whose associated Liouville domain
contains $W -U$ and so that
$\check{C}$ contains a regular level set of $\rho$.
\end{prop}

\begin{proof}[Proof of Proposition \ref{proposition constructing appropriate contact cylinder near normal crossing divisor}.]
Let $w_j$, $a_j$ be the wrapping number and $\tau$-discrepancy of $\Delta_j$
for each $j \in S$.
By \cite{McLean:affinegrowth}[Lemma 5.3 and Lemma 5.14] (or by \cite[Theorem 2.12]{McLeanTehraniZinger:normalcrossings})
we can assume, after smoothly deforming $(\Delta_i)_{i \in S}$ through a family of
symplectic SC divisors (which doesn't change the symplectomorphism type of
the complement $(W^o,\omega_W)$ by
\cite[Lemma 5.15]{McLean:affinegrowth}),
that $(\Delta_i)_{i \in I}$ admits a standard tubular neighborhood
as in Definition \ref{defn sc divisor}.
Therefore we let $\pi_I : U_I \lra{} \Delta_I$, $I \subset S$ and $r_i : U_i \lra{} \R$, $i \in S$
be as in Definition \ref{defn sc divisor}.
We can also assume after shrinking
the radius
$\epsilon$ that this standard
tubular neighborhood is contained in
$U$.
Also the constants $a_j, w_j$ do not change under isotopy for all $j \in S$,
so we can still assume that
$w_j < 0$ and
$a_j \geq 0$ for all $j \in S$.
Let $\dot{\Delta}_I := \Delta_I - \cup_{j \in S - I} \Delta_j$,
$\dot{U}_I := U_I \cap W^o$
and 
$$\dot{\pi}_I : \dot{U}_I \lra{} \dot{\Delta}_I, \quad \dot{\pi}_I(x) := \pi_I(x)$$
for all $I \subset S$.
The map $\dot{\pi}_I$ is a symplectic fibration
with fibers symplectomorphic to
$\prod_{i \in I} \dot{\D}_i(\epsilon)$
where $\dot{\D}_i(\epsilon) := \D_i(\epsilon) - 0$
and whose structure group
is given by the natural action of $U(1)^I$.

By the proof of \cite[Lemma 5.18]{McLean:isolated},
we can find a smooth function
$g : W \lra{} \R$
with the property that $\theta + dg$
restricted to any fiber
$\prod_{i \in I} \dot{\D}_i(\epsilon)$
of $\dot{\pi}_I$
is $\sum_{i \in I} (\frac{1}{2} r_i^2 + \frac{1}{2\pi} w_i) d\vartheta_i$
after shrinking $\epsilon$
where  $(r_i,\vartheta_i)$ are the natural polar coordinates on $\D_i(\epsilon)$.
Choose $t' > 0$ so that
\begin{equation} \label{equation t is small enough}
e^{t'} \in \left(1,\frac{(\frac{1}{2}\epsilon)^2 + \frac{1}{\pi}w_i}{\epsilon^2 + \frac{1}{\pi}w_i}\right), \quad \forall \ i \in I.
\end{equation}
Choose $\check{\epsilon} \in (0,\frac{1}{2}\epsilon)$
small enough so that
\begin{equation} \label{equation tubular neighrohood containing regular level set}
\cup_{i \in I} \left\{r_i^2 < (\check{\epsilon}^2 + \frac{1}{\pi} w_i) e^{-t'} - \frac{1}{\pi} w_i \right\} - \cup_{i \in I} \{r_i^2 \leq \check{\epsilon}^2\}
\end{equation}
contains a regular level set of $\rho$.
Let $f : [0,\epsilon^2) \lra{} \R$
be a smooth function so that
$f(x) = 1-x$ for all $x \leq \frac{1}{2} \check{\epsilon}^2$,
$f|_{[\check{\epsilon}^2,\epsilon^2]} = 0$, $f \geq 0$, $f' \leq 0$, $f'|_{(0,\check{\epsilon}^2)}<0$ and $f'' \geq 0$ (See Figure \ref{fig:graphoff3}).

\begin{center}
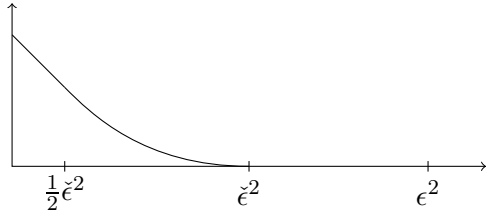
\begin{figure}[h]
\begin{tikzpicture}[scale=0.7]

\draw [<->](-1.5,-0.4) -- (-1.5,-3.5) -- (7.5,-3.5);

\draw (3,-3.5) .. controls (1,-3.5) and (0,-2.5) .. (-0.5,-2);

\draw (-1.5,-1) -- (-0.5,-2);

\draw (6.4,-3.4) -- (6.4,-3.6);

\draw (3,-3.4) -- (3,-3.6);

\draw (-0.5,-3.4) -- (-0.5,-3.6);

\draw (4.2,0.9);

\node at (6.4,-4) {$\epsilon^2$};

\node at (3,-4) {$\check{\epsilon}^2$};

\node at (-0.5,-4) {$\frac{1}{2}\check{\epsilon}^2$};


\end{tikzpicture}
\caption{Graph of $f$.} \label{fig:graphoff3}
\end{figure}
\end{center}

Define $U_I^o := U_I - \bigcup_{j \in S-I} U_j$
and
$\dot{U}_I^o := U_I^o \cap \dot{U}_I$.
Let
$H : W \lra{} \R$
be the unique function which satisfies
$H(x) = \sum_{i \in I} f(r_i^2)$
for all $x \in U^o_I$ and all $I \subset S.$

Let $\frac{\partial}{\partial r_i}$
and $\frac{\partial}{\partial \vartheta_i}$
be the unique vector fields on
$\dot{U}_i$
which are tangent to the fibers of
$\dot{\pi}_I$ and equal to
$\frac{\partial}{\partial r_i}$
and $\frac{\partial}{\partial \vartheta_i}$
inside the fibers $\dot{\D}_i(\epsilon)$
of $\dot{\pi}_i$ respectively
where $(r_i,\vartheta_i)$ are standard polar coordinates on $\D_i(\epsilon)$.
Let $Q_I \subset TU_I$ be the natural symplectic connection for $\pi_I$
as in Definition \ref{definition nice neighborhood of sc divisor}.
Since the $d\theta$-dual $X_{\theta + dg}$ of $\theta + dg$
inside $\dot{U}_I^o$ is equal to
$$\sum_{i \in I}\left(\frac{1}{2} r_i + \frac{w_i}{2\pi r_i}\right) \frac{\partial}{\partial r_i} +  E_I =
\sum_{i \in I}\left(r_i^2 + \frac{w_i}{\pi}\right) \frac{\partial}{\partial (r_i^2)} + E_I$$
where $E_I$ is a vector field tangent to $Q_I$ for all $I \subset S$,
we get that $dH(X_{\theta + dg}) > 0$ inside $H^{-1}((0,\infty)) \cap \cup_{I \subset S} \dot{U}_I$.
Hence $C_\delta := H^{-1}(\delta)$ is a compact submanifold of $U$ with contact form $\alpha_\delta := (\theta+dg)|_{H^{-1}(\delta)}$
for all sufficiently small $\delta>0$.
For each $x \in W^o$,
let $\psi_t(x)$ be the time $t$ flow of $x$ along $-X_{\theta+dg}$ for all $t$ (when defined).
Since the region (\ref{equation tubular neighrohood containing regular level set})
contains a regular level set of $\rho$,
we have by Equation (\ref{equation t is small enough}) that the contact cylinder
with image
$$\check{C}_\delta := \bigcup_{t \in [0,t']} \psi_t(C_\delta)$$
contains a regular level set of $\rho$
for all sufficiently small $\delta > 0$.
Equation (\ref{equation t is small enough}) also ensures that this contact cylinder is contained in $U$.
To finish our lemma we will now show that the contact cylinder $\check{C}_\delta$ is index bounded for all $\delta >0$ small enough by
computing the Conley-Zehnder indices of the Reeb orbits of $\alpha_\delta$ for $\delta>0$ small enough.

Let $R_\delta$ be the Reeb vector field of $\alpha_\delta$.
Inside $\dot{U}_I^o$,
we have that
$X_H = 2 \sum_{i \in I} f'(r_i^2) \frac{\partial}{\partial \vartheta i}$.
Define
$$b : \dot{U}^o_I \lra{} \R, \quad
b := \sum_{i \in I} f'(r_i^2)(r_i^2 + \frac{w_i}{\pi}).
$$
Then
\begin{equation} \label{equation for reeb vector field}
R_\delta = \frac{1}{(\theta+dg)(X_H)} X_H =
\frac{2}{b}\sum_{i \in I}  f'(r_i^2) \frac{\partial}{\partial \vartheta i}.
\end{equation}
In particular, a Reeb orbit which starts inside
$\dot{U}^o_I \cap H^{-1}(\delta)$, stays inside $\dot{U}^o_I \cap H^{-1}(\delta)$
and all such Reeb orbits
are contained in fibers of $\dot{\pi}_I|_{H^{-1}(\delta)}$.
We will show that the Conley-Zehnder
index of every Reeb orbit of $\alpha_\delta$
is a bounded above by a linear function
of its length
whose slope is negative
when $\delta > 0$ is sufficiently small.
This will be sufficient for us to show that $\check{C}_\delta$ is index bounded for all $\delta>0$ sufficiently small.

Now let $\zeta : \R / \lambda \Z \lra{} H^{-1}(\delta)$
be a Reeb orbit of $\alpha_\delta$ of length
$\lambda$ .
Now $\zeta$ is contained inside
$\dot{\pi}_I^{-1}(q)$ for some $q \in \Delta_I - \cup_{j \in S-I}U_j$ and $I \subset S$
and so there 
exists a smooth map
$\check{\zeta} : \D \lra{} \pi_I^{-1}(q)$
from the closed unit disk $\D \subset \C$ so that
$\check{\zeta}(e^{2\pi i t}) = \zeta(\lambda t)$ for all $t \in [0,1]$.
Let $d_i \in \Z$ be the intersection number of
$\check{\zeta}$ with $\Delta_i \cap \pi_I^{-1}(q)$ inside $\pi^{-1}(q)$
for each $i \in I$.
Define
$CZ(\zeta)$ to be the Conley-Zehnder index of $\zeta$ inside $W^o$
and let $CZ(\check{\zeta})$
to be the Conley-Zehnder index of $\zeta$
inside a small neighborhood of $\pi_I^{-1}(q)$ (I.e. we think of a portion of our contact cylinder containing this orbit as a contact cylinder in a neighborhood of $\pi_I^{-1}(q)$).
By \ref{item:cznormalization}, \ref{item:czadditive} and \ref{item:catenationczproperty}
we have
$$CZ(\zeta) = CZ(\check{\zeta}) + 2\sum_{i \in I} d_i a_i.$$
By Equation (\ref{equation for reeb vector field}) combined with the fact that $f'(f^{-1}(x)) <0$ for all small $x > 0$,
we see that $d_i < 0$ for all $i \in I$
and that the
length of the Reeb orbit $\zeta$
is bounded below by a positive constant times $-\sum_{i \in I} d_i$.
Therefore it is sufficient for us to show
that $CZ(\zeta)$ is less than
or equal to some fixed linear function
of $\sum_{i \in I} d_i$ of positive slope.

Let $T := T(\pi_I^{-1}(q)) \subset TW$ be the tangent space of the fiber containing $\zeta$
and let $T^\perp \subset TW$
be the set of vectors which are $\omega_W$ orthogonal to $T$.
Let $x_1,y_1,\cdots,x_{|I|},y_{|I|}$
be symplectic coordinates of $\pi_I^{-1}(q)$
coming from a $U(1)^{|I|}$ trivialization of this fiber and let $J$ be the natural complex structure coming from this trivialization.
These coordinates induce a symplectic trivialization
$\tau_T : \zeta^* T \lra{} \R/\lambda\Z \times \C^{|I|}$
of $\zeta^*T$.
Define $K_\delta := \ker(\alpha_\delta) \cap T \subset T$
and let $K^\perp_\delta \subset T|_{\dot{\pi}_I^{-1}(q)}$
be the symplectic vector subspace of $T$ consisting of vectors which are $\omega_W|_T$-orthogonal to $K_\delta$.
Since $R_\delta$ is contained in $K_\delta^\perp$ and $K_\delta^\perp$ is a two dimensional vector bundle,
there is a unique, up to homotopy,
trivialization
$\tau_{R_\delta} : K_\delta^\perp \lra{} \dot{\pi}_I^{-1}(q) \times \C$ of
$K_\delta^\perp$ which maps $R_\delta$ to the constant section whose value is $i \in \C$.
Hence
there is a symplectic trivialization
$\check{\tau} : \zeta^* K_\delta \lra{} \R/\lambda \Z \times \C^{|I|-1}$
so that
$\check{\tau} \oplus \tau_{R_\delta}$
gives us a trivialization of
$\zeta^* T = \zeta^* (K_\delta \oplus K_\delta^\perp)$
homotopic to $\tau_T$.

Let
$\tau_{T^\perp} : \zeta^* T^\perp \lra{} \R/\lambda \Z \times \C^{n-|I|}$ be a trivialization of $\zeta^* T^\perp$
which is a restriction of a trivialization of $T^\perp$ (such a trivialization is unique up to homotopy since $\pi_I^{-1}(q)$ is contractible).
Then
\begin{equation} \label{equation trivialization of kernal in fiber}
\check{\tau} \oplus \tau_{T^\perp} : \zeta^* (K_\delta \oplus T^\perp) = \zeta^* \ker(\alpha_\delta) \lra{} \R/\lambda \Z \times \C^{n-1}
\end{equation}
is a trivialization of $\zeta^* \ker(\alpha_\delta)$.
This is the trivialization we need in order
to compute $CZ(\check{\zeta})$
since the trivialization $\tau_{R_\delta} \oplus \check{\tau} \oplus \tau_{T^\perp}$
of $\zeta^* TW$ extends over the disk $\check{\zeta}^* TW$ after identifying the boundary of $\partial \D$ with $\R/\lambda \Z$ in the natural way as explained earlier.

If $\phi^\delta_t : H^{-1}(\delta) \lra{} H^{-1}(\delta)$ is the time $t$ flow
of $R_\delta$
then its linearization $D\phi_t : \ker(\alpha_\delta)|_{\zeta(0)} \lra{} \ker(\alpha_\delta)|_{\zeta(t)}$
gives us a family of symplectic matrices
$(A_t)_{t \in [0,\lambda]}$ with respect to the trivialization (\ref{equation trivialization of kernal in fiber}).
Equation (\ref{equation for reeb vector field})
tells us that
$D\phi_t^
\delta$ respects the splitting
$\zeta^* \ker(\alpha_\delta) = \zeta^* K_\delta \oplus \zeta^* T^\perp$
and hence
$A_t = B_t \oplus C_t$ for some family of matrices
$(B_t)_{t \in [0,\lambda]}$ and $(C_t)_{t \in [0,\lambda]}$
with respect to the trivializations
of $\zeta^*K_\delta$ and
$\zeta^* T^\perp$.
Also Equation (\ref{equation for reeb vector field}) tells us that $C_t = \id$ for all $t \in [0,\lambda]$ after homotoping $\tau_{T^\perp}$ appropriately.
By \ref{item:constantrankpath},
we have that $CZ((C_t)_{t \in [0,\lambda]})=0$.
Hence by \ref{item:czadditive}, all we need to do is compute
the Conley-Zehnder index of $(B_t)_{t \in [0,\lambda]}$.

In order to compute $CZ((B_t)_{t \in [0,\lambda]})$ we will compute
the Conley-Zehnder index with respect to
an alternative trivialization
of $K_\delta$
and relate it to the trivialization $\check{\tau}$ above.
Choose a total ordering on $I$.
This means we have a natural identification
$I = \{1,\cdots,l\}$.
Define
\begin{equation} \label{equation explicit basis for s delta perp}
\check{R}_i := \left(2r_l f'(r_l^2)
\right)\frac{\partial}{\partial r_i} - \left(2r_i f'(r_i^2)\right)\frac{\partial}{\partial r_l}
\end{equation}
$$
\Theta_i := \left(\frac{1}{2}r_l^2 + \frac{1}{2\pi}w_l
\right)\frac{\partial}{\partial \vartheta_i} - \left(\frac{1}{2}r_i^2 + \frac{1}{2\pi}w_i\right)\frac{\partial}{\partial \vartheta_l}
$$
for all $i \in \{1,\cdots,l-1\}$.
Let $L$ be the span of
$\Theta_1,\cdots,\Theta_{l-1}$
and $L^\perp$ the span of
$\check{R}_1,\cdots,\check{R}_{l-1}$.
Since $K_\delta = \ker(\alpha_\delta) = \ker(\theta + dg) \cap \ker(dH)$,
we get that $\check{R}_1,\Theta_1,\cdots,\check{R}_{l-1},\Theta_{l-1}$
is a basis for $K_\delta$.
The problem with this basis is that it is not a symplectic basis.
However since $L$ and $L^\perp$
are Lagrangian, there is a new basis
$R_1,\cdots,R_{l-1}$
of $L^\perp$
so that 
$R_1,\Theta_1,\cdots,R_{l-1},\Theta_{l-1}$
is a symplectic basis of $K_\delta$.
Since the Poisson bracket of
$R_\delta$  with each of
$\Theta_1,\cdots,\Theta_{l-1}$ is zero
by Equation (\ref{equation for reeb vector field}),
we have that the flow
of $R_\delta$
sends $\Theta_i$ to $\Theta_i$ for each $i = 1,\cdots,l-1$.
The linearization $D\phi^\delta_t|_{K_\delta} : K_\delta|_{\zeta(0)} \lra{} K_\delta|_{\zeta(t)}$
of the Reeb flow $\phi^\delta_t$
of $R_\delta$
is a family of symplectic matrices
$(W_t)_{t \in [0,\lambda]}$ with respect
to the symplectic basis $R_1,\Theta_1,\cdots,R_{l-1},\Theta_{l-1}$.
Since $W_t(x) = x$ on $L$ with respect to this basis for all $t \in [0,\lambda]$,
we have by Lemma \ref{lemma bounded family of symplectic matrices}
that $CZ((W_t)_{t \in [0,\lambda]}) \in [-2l-2,2l+2] \subset [-2n+2,2n-2]$.

Let $\widehat{\tau} : \zeta^* K_\delta \lra{} \R/\lambda \Z \times \C^{l-1}$
be the trivialization of
$\zeta^* K_\delta$
induced by the symplectic basis
$R_1,\Theta_1,\cdots,R_{l-1},\Theta_{l-1}$.
Then $\widehat{\tau} \oplus \tau_{R_\delta}$
is isotopic to the trivialization induced by the symplectic basis
$$\frac{1}{r_1}\frac{\partial}{\partial r_1},\frac{\partial}{\partial \theta_1},\cdots,\frac{1}{r_l}\frac{\partial}{\partial r_l},\frac{\partial}{\partial \theta_l}.$$
Hence the bundle automorphism
$\check{\tau} \circ \widehat{\tau}^{-1}$ gives us a family of symplectic matrices
parameterized by $[0,\lambda]$
whose Conley-Zehnder index is
$2\sum_{i \in I} d_i$.
Hence by \ref{item:cznormalization}, \ref{item:czadditive} and \ref{item:catenationczproperty},
$$CZ((B_t)_{t \in [0,\lambda]}) =
CZ((W_t)_{t \in [0,\lambda]}) +
2\sum_{i \in I} d_i.$$
Hence by \ref{item:czadditive},
we get that
$$CZ((A_t)_{t \in [0,\lambda]}) =
CZ((B_t)_{t \in [0,\lambda]})
+ CZ((C_t)_{t \in [0,\lambda]}
= 2\sum_{i \in I} d_i + CZ((W_t)_{t \in [0,\lambda]}).$$
Hence
$$CZ(\zeta) = 2\sum_{i \in I} (a_i+1)d_i + CZ((W_t)_{t \in [0,\lambda]})$$
which is bounded above by a linear function of length of $\zeta$ with negative slope since we have $CZ((W_t)_{t \in [0,1]}) \in [-2n+2,2n-2]$.
Therefore $\check{C} := \check{C}_\delta$
is index bounded.
The associated Liouville domain
of $\check{C}$ is contained in $W^o$ and contains $K$ since its complement is contained in $\cup_{i \in S} U_i$.
\end{proof}

\begin{proof}[Proof of Proposition \ref{proposition constructino of index bounded contact cylinder on appropriate affine variety}.]

By \cite{hironaka:resolution},
we can blow up $X$ along smooth subvarieties
above $X - A$ giving us a variety
$\widetilde{X}$ so that the preimage of $X - A$ is a smooth normal crossing variety.
Let $Bl : \widetilde{X} \lra{} X$ be the blowdown map.
Choose an algebraic plurisubharmonic function
$\widetilde{\rho} : A \lra{} \R$
so that $-dd^c (Bl^*(\widetilde{\rho}))$
extends to a K\"{a}hler form
$\omega_{\widetilde{X}}$
on $\widetilde{X}$
(I.e. $\widetilde{\rho} = -dd^c \log(|s|)$
for some appropriate section $s$ of an ample line bundle on $\widetilde{X}$ with an appropriate Hermitian metric).
By Corollary \ref{corollary flow time}
there are constants $0 < \delta \ll T \ll 1$ and a compact set $Q \subset A$ containing $K$
and a plurisubharmonic function $\widehat{\rho}$ on $A$
%
satisfying
\begin{enumerate}
	\item $\widehat{\rho}|_Q = \rho|_Q$,
	\item $\widehat{\rho} = \rho + \delta \widetilde{\rho}$ outside a large compact set and
	\item for all $x \in A$, the time $T$ flow of $x$ along $-\nabla_{\widehat{\rho}} \widehat{\rho}$ is contained in $Q$ and is disjoint from $K$ if, in addition, $x \notin Q$.
\end{enumerate}

Let $\phi_T : A \lra{} A$ be the time $T$ flow of $-\nabla_{\widehat{\rho}} \widehat{\rho}$.
By Lemmas \ref{lemma negatively wrapped divisor} and
\ref{lemma non-negative discrepancy example} combined with
Proposition \ref{proposition constructing appropriate contact cylinder near normal crossing divisor},
we can find an index bounded contact cylinder
$\check{C}_1 \subset Bl^{-1}(A)$ with respect to the symplectic form $Bl^*\omega_X + \delta\omega_{\widetilde{X}}$ which is disjoint from $Q$ and whose associated Liouville domain contains $Q$
and which contains a regular level set $\widehat{\rho}^{-1}(C)$ for some $C \in \R$.
Then
$\check{C} := \phi_T(Bl(\check{C}_1)) \subset A$
is an index bounded contact cylinder with respect to $\omega_X$ since $\omega_X|_Q = \omega_{\widetilde{X}}|_Q$
and since both of these symplectic forms are K\"{a}hler.
Also the Liouville domain associated to $\check{C}$ contains $K$ and $\check{C}$ contains $C_0 := \phi_T(Bl(\widehat{\rho}^{-1}(C)))$ which is a hypersurface transverse to $\nabla_{\rho} \rho$ bounding a region containing $D$.
By a Moser argument (\cite[Exercise 3.36]{McduffSalamon:sympbook}), 
we can enlarge $\check{C}$ slightly so that
it contains a contact cylinder
$\check{C}_0 := [1-\epsilon_0,1+\epsilon_0] \times C_0$
where $\{0\} \times C_0 = C_0$.
The associated Liouville domain $D_0$ contains $D$ and $-d^c \rho|_{D_0}$
is a Liouville form associated to $\check{C}_0$.
\end{proof}

\subsection{Divisors are Stably Displaceable}

We recall from Section \ref{section notation}
that $(M,\omega)$ is a compact symplectic manifold.
However, the results of this section also work when $M$ is non-compact.
We also do not require that $(M,\omega)$ satisfy any other conditions, such as the Chern class condition $c_1(M,\omega) = 0$.

\begin{defn} \label{definition partily stratified symplectic subset}
A
{\it partly stratified symplectic subset}
$S$
of $(M,\omega)$ is a subset equal to a disjoint union of 
subsets $S_1,\cdots,S_l$ of $M$
so that for each $j \in \{1,\cdots,l\}$,
\begin{itemize}
\item $\cup_{i \leq j} S_i$
is a compact subset of $M$ and
\item $S_j$ is a proper codimension $\geq 2$
symplectic submanifold of $M - \cup_{i < j} S_i$ without boundary.
\end{itemize}
We call the subsets $S_1,\cdots,S_l$
the {\it strata of $S$}.
\end{defn}

\begin{example} 
\label{example subvariety as partly stratified symplectic subset}
\cite[Theorem 19.2]{whitney1965tangents}.
If
$M$ is a K\"{a}hler manifold
then any compact codimension $\geq 1$
subvariety of $M$
is a partly stratified symplectic subset.
\end{example}

\begin{prop} \label{proposition stably displaceable partly stratified symplectic subset}
Any partly stratified symplectic subset is stably displaceable as in Definition \ref{definition stably displaceable}.
\end{prop}

We have the following immediate corollary:

\begin{corollary} \label{corollary divisor stably displaceable}
Any compact positive codimension subvariety of a 
K\"{a}hler manifold is stably displaceable.
\end{corollary}

The rest of this subsection is devoted to the proof of Proposition \ref{proposition stably displaceable partly stratified symplectic subset}.
Throughout this subsection we will let
$(\sigma,\tau)$
be the natural Darboux coordinates on
$T^* \T = \R \times \T$
where $\sigma$ is the projection map to $\R$ and $\tau$ is the projection map to $\T$.
We will also let $(\check{M},\check{\omega})$ be the symplectic manifold $(M \times T^*\T, \omega + d\sigma \wedge d\tau)$.

Let us first give a sketch of the proof of Proposition \ref{proposition stably displaceable partly stratified symplectic subset}.
It would be nice if one could displace $S \times \T$ using the symplectic vector field $V := \frac{\partial}{\partial \sigma}$.
However this is not a Hamiltonian vector field.
What we wish to do is to subtract another time dependent symplectic vector field $V_t$, $t \in \R$
from $V$ so that
$V - V_t$ is Hamiltonian
and so that the time $T$ flow of $S \times \T$ along $V - V_t$ is
a bounded distance from the time $T$ flow of $S \times \T$ along $V$ for some large $T \in \R$.
To do this, we first `curl up' $S \times T^* \T$
into a small neighborhood of $S \times \T$.
In other words, we find an appropriate symplectic embedding $\iota$
of a neighborhood $N$ of $S \times T^* \T$ into
a relatively compact subset of $M \times \T^* \T$.
This is done via an explicit embedding technique from \cite[Lemma 12.1.2]{eliashbergmishchevhprinciple}.
We then define $V_t$ to be an appropriate extension of $(\phi^V_t)_* (\iota_* V)$, $t \in \R$ to $\check{M}$ where $\phi^V_t$, $t \in \R$ is the flow of $V$ (See Figure \ref{fig:vectorfieldscurledup}).
Then $V - V_t$ is the Hamiltonian vector field displacing $S \times \T$ in $M \times T^* \T$. This completes the sketch of the proof.

\begin{center}
\begin{figure}[h]
\usetikzlibrary{arrows}
\begin{tikzpicture}
[scale=1.2]

\draw (-6.5,-0.5) -- (2.2,-0.5);
\draw [->](-2,0.8) -- (-2.7,-0.4);
\node at (-2,1) {$S \times T^* \T$};

\draw (-3.7,-0.5) .. controls (-2.7,-0.5) and (-2.7,-1) .. (-2.7,-1.3) .. controls (-2.7,-1.6) and (-3,-1.9) .. (-3.4,-1.9) .. controls (-3.8,-1.9) and (-4,-1.6) .. (-4,-1.3) .. controls (-4,-1) and (-3.7,-0.8) .. (-3.4,-0.8) .. controls (-3.1,-0.8) and (-3,-1) .. (-3,-1.3) .. controls (-3,-1.5) and (-3.1,-1.7) .. (-3.4,-1.7) .. controls (-3.7,-1.7) and (-3.8,-1.5) .. (-3.8,-1.3) .. controls (-3.8,-1.1) and (-3.5,-1.1) .. (-3.4,-1.1) .. controls (-3.2,-1.1) and (-3.2,-1.3) .. (-3.2,-1.4) .. controls (-3.2,-1.5) and (-3.3,-1.6) .. (-3.4,-1.6) .. controls (-3.5,-1.6) and (-3.6,-1.5) .. (-3.6,-1.3);
\draw (-4.4,-0.5) .. controls (-5.2,-0.5) and (-5.3,-1) .. (-5.3,-1.3) .. controls (-5.3,-1.6) and (-5,-1.9) .. (-4.6,-1.9) .. controls (-4.2,-1.9) and (-4.1,-1.6) .. (-4.1,-1.3) .. controls (-4.1,-1) and (-4.3,-0.8) .. (-4.6,-0.8) .. controls (-4.9,-0.8) and (-5,-1) .. (-5,-1.3) .. controls (-5,-1.6) and (-4.8,-1.7) .. (-4.6,-1.7) .. controls (-4.4,-1.7) and (-4.3,-1.5) .. (-4.3,-1.3) .. controls (-4.3,-1.1) and (-4.5,-1.1) .. (-4.6,-1.1) .. controls (-4.8,-1.1) and (-4.8,-1.2) .. (-4.8,-1.3) .. controls (-4.8,-1.5) and (-4.5,-1.5) .. (-4.5,-1.3);

\draw[shift={(4,0)}] (-3.7,-0.5) .. controls (-2.7,-0.5) and (-2.7,-1) .. (-2.7,-1.3) .. controls (-2.7,-1.6) and (-3,-1.9) .. (-3.4,-1.9) .. controls (-3.8,-1.9) and (-4,-1.6) .. (-4,-1.3) .. controls (-4,-1) and (-3.7,-0.8) .. (-3.4,-0.8) .. controls (-3.1,-0.8) and (-3,-1) .. (-3,-1.3) .. controls (-3,-1.5) and (-3.1,-1.7) .. (-3.4,-1.7) .. controls (-3.7,-1.7) and (-3.8,-1.5) .. (-3.8,-1.3) .. controls (-3.8,-1.1) and (-3.5,-1.1) .. (-3.4,-1.1) .. controls (-3.2,-1.1) and (-3.2,-1.3) .. (-3.2,-1.4) .. controls (-3.2,-1.5) and (-3.3,-1.6) .. (-3.4,-1.6) .. controls (-3.5,-1.6) and (-3.6,-1.5) .. (-3.6,-1.3);
\draw[shift={(4,0)}] (-4.4,-0.5) .. controls (-5.2,-0.5) and (-5.3,-1) .. (-5.3,-1.3) .. controls (-5.3,-1.6) and (-5,-1.9) .. (-4.6,-1.9) .. controls (-4.2,-1.9) and (-4.1,-1.6) .. (-4.1,-1.3) .. controls (-4.1,-1) and (-4.3,-0.8) .. (-4.6,-0.8) .. controls (-4.9,-0.8) and (-5,-1) .. (-5,-1.3) .. controls (-5,-1.6) and (-4.8,-1.7) .. (-4.6,-1.7) .. controls (-4.4,-1.7) and (-4.3,-1.5) .. (-4.3,-1.3) .. controls (-4.3,-1.1) and (-4.5,-1.1) .. (-4.6,-1.1) .. controls (-4.8,-1.1) and (-4.8,-1.2) .. (-4.8,-1.3) .. controls (-4.8,-1.5) and (-4.5,-1.5) .. (-4.5,-1.3);

\draw [->](-2.6,-2.3) -- (-3.8,-2.1);
\node at (-2.5,-2.5) {$\iota(S \times T^*\T)$};
\draw [->](1.5,-2.2) -- (0.6,-2);
\node at (1.7,-2.5) {$\phi^V_t(\iota(S \times T^* \T))$};

\draw [-triangle 90](-5.4,-0.5) -- (-5.1,-0.5);
\draw [-triangle 90](-2.5,-0.5) -- (-2.2,-0.5);
\draw [-triangle 90](1.5,-0.5) -- (1.8,-0.5);
\draw [-triangle 90](-2.7,-1.1) -- (-2.7,-1.3);
\draw [-triangle 90](-5.3,-1.4) -- (-5.3,-1.3);
\draw [-triangle 90,shift={(4,0)}](-2.7,-1.1) -- (-2.7,-1.3);
\draw [-triangle 90,shift={(4,0)}](-5.3,-1.4) -- (-5.3,-1.3);

\draw [->](0.3,0.2) -- (1.5,-0.3);
\draw [->](-0.2,0.2) -- (-2.1,-0.3);
\node at (0.1,0.3) {$V$};
\draw [->](-2.2,-1.6) -- (-2.5,-1.3);
\node at (-2,-1.8) {$V_0$};
\draw [->](-5.9,-1.7) -- (-5.6,-1.5);
\node at (-6,-1.9) {$V_0$};
\draw [->](2,-1.3) -- (1.5,-1.2);
\node at (2.3,-1.4) {$V_t$};
\draw [->](-1.8,-1.1) -- (-1.5,-1.3);
\node at (-2,-1) {$V_t$};
\draw [-triangle 90](-5.1,0.1) -- (-4.4,0.1);
\draw [-triangle 90](1.1,0.5) -- (1.6,0.5);
\draw [->](0.3,0.3) -- (0.9,0.5);
\draw [black,fill] (-4,-0.5) ellipse (0.05 and 0.05);
\draw [->](-3.8,-0.1) -- (-4,-0.4);
\node at (-3.7,0.1) {$S \times \T$};

\end{tikzpicture}
\caption{Picture of $V$ and $V_t$, $t \in \R$.} \label{fig:vectorfieldscurledup}
\end{figure}
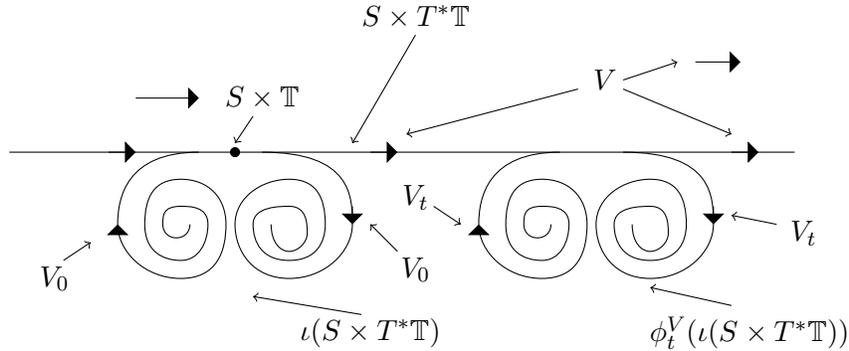
\end{center}

If $S$ has only one stratum (I.e. $S$ is a symplectic submanifold)
then the methods of 
\cite{laudenbach1994hamiltonian} and
\cite[Section 4.3]{gurel2008totally}
can be used to displace $S \times \T$ instead.
However it is hard to see how these methods could be used when $S$ has singularities.
This is due to the fact that near certain singular points $p$ of $S$ (such as normal crossing points), there is no vector field that can infinitesimally displace a neighborhood of $p$ in $S$ from $S$.
Also there are examples which seem difficult to infinitesimally displace.
For instance a $C^\infty$ generic Hamiltonian on $\C \P^2 \times T^* \T$ cannot displace $(A \cup B) \times \T$ by an infinitesimally small amount where $A,B \subset \C\P^2$ are distinct complex lines.
%
%
%

Before we prove Proposition \ref{proposition stably displaceable partly stratified symplectic subset},
we need some preliminary definitions and lemmas.
We will first provide a criterion for Hamiltonian displacement in terms of `curled up' symplectic embeddings.

\begin{defn} \label{definition curled up embedding}
For each $\nu >0$, define
$$\check{M}_{<\nu} := \{(x,(\sigma,\tau)) \in \check{M} \ : \ |\sigma| < \nu \}.$$
Let $Q \subset M$ be a subset.
A {\it weak curled up embedding}
of $Q \times T^*\T$ consists of a neighborhood $N$ of $Q \times T^* \T$ in $\check{M}$
and a symplectic embedding
$\iota : N \hookrightarrow \check{M}_{<1}$ 
so that $\iota(y) = y$
for all $y$ in a neighborhood of $N \cap (M \times \T) \subset N$.
A {\it curled up embedding}
of $Q \times T^*\T$ is a weak curled up embedding as above
with the additional property that the map
$H^1(N;\R) \lra{} H^1(N \cap (M \times \T);\R)$ induced by the inclusion map is injective.
\end{defn}

\begin{lemma} \label{lemma cureld embedding implies stably displaceable}
Let $Q \subset M$ be a compact subset of $M$ so that $Q \times T^*\T$ admits a curled up embedding.
Then $Q$ is stably displaceable.
\end{lemma}
\begin{proof}
Let
$\iota$, $N$ be as in Definition \ref{definition curled up embedding} above.
Let $V := \frac{\partial}{\partial \sigma}$.
We wish to construct $V_t$, $t \in \R$ as described above.
Define $\widetilde{V} := \iota_*(V)$.
Let $\phi^V_t$ be the time $t$ flow of $V$ for each $t \in \R$
and similarly define
$\phi^{\widetilde{V}}_t$ (when defined).
Since $\iota(y) = y$ for all $y \in \check{M}$ near $N \cap (M \times \T)$,
we get that the closed $1$-form
$\beta(-) := \check{\omega}(V|_{\iota(N)} - \widetilde{V},-)$
vanishes near $N \cap (M \times \T)$.
Combining this with the fact that
the map
$H^1(N;\R) \lra{} H^1(N \cap (M \times \T);\R)$ induced by the inclusion map is injective,
we have that
$\beta$
is exact.
Hence
$V|_{\iota(N)} - \widetilde{V}$
is a Hamiltonian vector field on $\iota(N) \subset \check{M}$.
Let $\check{H} : \iota(N) \lra{} \R$
be the corresponding Hamiltonian.
Let
$\rho : \iota(N) \lra{} [0,1]$
be a smooth compactly supported function
whose restriction to a small neighborhood of the relatively compact set
$\widetilde{Q} := \iota((Q \times T^* \T) \cap \check{M}_{<3})$ is $1$
and define
$$\widetilde{H} : \check{M} \lra{} \R, \quad \widetilde{H}(y) := \left\{
\begin{array}{ll}
\rho(y)\check{H}(y) & \textnormal{if} \ y \in \iota(N) \\
0 & \textnormal{otherwise}.
\end{array}
\right.$$
Define
$H_t := (\phi^V_t)_*(\widetilde{H})$
for each $t \in \R$.
Note that the vector field $V_t$ in the proof sketch above is just $V - X_{H_t}$.
Our claim is that $H := (H_t)_{t \in \R}$ is the Hamiltonian which displaces $Q$.
Let $\phi^H_t$ be the time $t$ flow of $H$ for each $t \in \R$.
Near $\phi^V_t(\widetilde{Q})$, we have
$$X_{H_t} = (\phi^V_t)_*(X_{\widetilde{H}})
= (\phi^V_t)_*(V - \widetilde{V})
= V + (\phi^V_t)_*(-\widetilde{V})$$
for all $t \in \R$.
Hence for each $y \in Q \times \T$,
$\phi^H_t(y) = \phi^V_t(\phi^{-\widetilde{V}}_t(y))$
for all $y \in Q \times \T$ and $t \in [0,3]$.
Combining this with the fact
that
$\phi^V_t(\phi^{-\widetilde{V}}_t(y)) \in \phi^V_t(\check{M}_{<1})$ for all $y \in Q \times \T$,
and $\phi^V_{3}(\check{M}_{<1}) \cap \check{M}_{<1} = \emptyset$ 
we get
$\phi^H_{3}(Q \times \T) \cap (Q \times \T) = \emptyset$.
Hence $Q$ is stably displaceable.
\end{proof}

\begin{defn} \label{defn less than function}
Let $A \subset M \times \T$ be a subset.
For any function
$f : A \lra{} \R \cup \{\infty\}$,
define
$$\check{M}_{< f}|_A := \{(x,(\sigma,\tau)) \in A \times T^* \T \ : \ |\sigma| < f(x,\tau) \} \subset \check{M}.$$

Let $N \subset \check{M}$ be an open subset
containing $A$.
Define
\begin{equation} \label{equation lower semi subset}
f_N : A \lra{} \R \cup \{\infty\},
\quad f_N(a) := \sup \{\sigma > 0 \ : \ \{a\} \times [-\sigma,\sigma] \subset N \}.
\end{equation}
See Figure \ref{fig:checkmlessthanN}.
We define
$$\check{M}_{<N}|_A := \check{M}_{<f_N}|_A.$$
If $A = M \times \T$,
we define
$$\check{M}_{<N} := \check{M}_{<N}|_A.$$
\end{defn}

\begin{center}
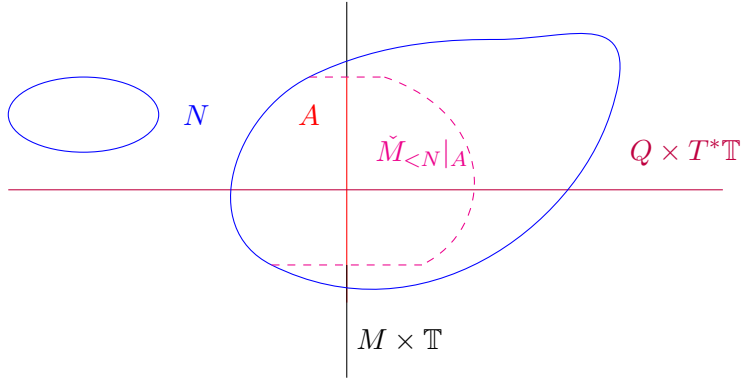
\begin{figure}[h]
\begin{tikzpicture}
\draw[red] (-1,0.5) -- (-1,-2.5);
\draw[purple] (-5.5,-1) -- (4,-1);
\draw (-1,1.5) -- (-1,0.5);
\draw (-1,-3.5) -- (-1,-2);

\node[red] at (-1.5,0) {$A$};

\node at (-0.3,-3) {$M \times \T$};

\node[purple] at (3.5,-0.5) {$Q \times T^*\T$};
\draw[blue]  (-4.5,0) ellipse (1 and 0.5);
\draw[blue] (-1.5,0.5) .. controls (-2.5,0) and (-3,-1.5) .. (-2,-2) .. controls (0,-3) and (2,-1.5) .. (2.5,0) .. controls (3,1.5) and (2,1) .. (1,1) .. controls (0.5,1) and (-0.5,1) .. (-1.5,0.5);
\node at (-3,0) {\color{blue} $N$};

\draw [dashed,magenta](-2,-2) -- (0,-2);
\draw [dashed,magenta](-1.5,0.5) -- (-0.5,0.5);
\draw [dashed,magenta](-0.5,0.5) .. controls (1,0) and (1,-1.5) .. (0,-2);
\node at (0,-0.5) {\color{magenta} $\check{M}_{<N}|_A$};
\end{tikzpicture}
\caption{Picture of $\check{M}_{<N}|_A$.} \label{fig:checkmlessthanN}
\end{figure}
\end{center}

\begin{lemma} \label{lemma open subset}
The function $f_N$ from Definition \ref{defn less than function}
is lower semi-continuous.
Hence $\check{M}_{<N}|_A$
is an open subset of $\check{M}$ if $A \subset M \times \T$ is open.
\end{lemma}
\begin{proof}[Proof of Lemma \ref{lemma open subset}]
Let $a \in A$ and let $(a_i)_{i \in \N}$
be a sequence of points in $A$ converging to $a$.
Let $\sigma > 0$
satisfy $\sigma < f_N(a)$.
Since $N \subset \check{M}$
is an open subset
and since $\{a\} \times [-\sigma,\sigma]$
is compact,
we have that $\{a_j\} \times [-\sigma,\sigma] \subset N$
for all $j$ sufficiently large.
Hence $\sigma < f_N(a_j)$ for all $j$ sufficiently large.
Hence
$$\liminf_{j \to \infty} f_N(a_j)
\geq \liminf_{\sigma < f_N(a)} \sigma
= f_N(a).$$
Hence $f_N$ is lower semi-continuous.
\end{proof}

\begin{lemma} \label{lemma weak curled implies curled}
Let $Q \subset M$ be a compact subset of $M$ so that $Q \times T^*\T$ admits a weak curled up embedding.
Then $Q \times T^* \T$ admits a curled up embedding.
\end{lemma}
\begin{proof}
Let
$\iota : N \hookrightarrow \check{M}_{<1}$, be a weak curled up embedding of $Q \times T^*\T$.
The claim is that $\iota|_{(\check{M}_{<N})}$
is a curled up embedding.
%
%
%
%
%
To prove this, 
it is sufficient to show that
$\check{M}_{<N}$ deformation retracts onto $E$.
This deformation retraction
is given by
$$\Phi_t : \check{M}_{<N} \lra{} \check{M}_{<N}, \quad \Phi(x,(\sigma,\tau)) := (x,((1-t)\sigma,\tau)), \quad t \in [0,1].$$
\end{proof}

We need the following Moser lemma in order to construct appropriate weak curled up embeddings of stratified symplectic subsets.
The proof of this lemma is a slight modification of \cite[Lemma 3.14]{McduffSalamon:sympbook}.

\begin{lemma} \label{lemma Moser lemma}
Let $(W,\omega_W)$, $(\check{W},\omega_{\check{W}})$ be symplectic manifolds, let $Q \subset W$ be a symplectic submanifold and let $U, N \subset W$ open sets satisfying $\overline{U} \subset N$.
Let
$\iota_N : N \hookrightarrow \check{W}$,
$\iota_Q : Q \hookrightarrow \check{W}$
be symplectic embeddings
so that
\begin{enumerate}
\item $\iota_N$ is a codimension $0$ symplectic embedding,
\item $\iota_N|_{N \cap Q} = \iota_Q|_{N \cap Q}$ and
\item \label{item:isomorphism between bundles} the pullback via $\iota_Q$ of the normal bundle of $\iota_Q(Q)$
is isomorphic as a symplectic vector bundle to the normal bundle of $Q$ in $W$ and this isomorphism coincides with the isomorphism induced by $\iota_N$ along $N \cap Q$.
\end{enumerate}
Then there is a neighborhood $V \subset W$ of $Q$ and a codimension $0$ symplectic embedding
$\iota : U \cup V \hookrightarrow \check{W}$
so that
$\iota|_U = \iota_U$ and $\iota|_Q = \iota_Q$ (See Figure \ref{fig:uvembedding}).
\end{lemma}

\begin{center}
\begin{figure}[h]
	\begin{tikzpicture}
	\draw (5,1.5) .. controls (4,0.5) and (4.5,-1) .. (6,-0.5) .. controls (7.5,0) and (7,1.5) .. (6.5,2.5) .. controls (6,3) and (5.5,2) .. (5,1.5);
	\draw[purple] (5.5,1) .. controls (5,0.5) and (5.5,0) .. (6.5,0.5) .. controls (7.5,1) and (6,1.5) .. (5.5,1);
	\draw[blue] (3.5,2) .. controls (5.5,1.5) and (6,0.5) .. (7,-1.5);
	\draw[red] (3,2) .. controls (5.5,0.5) and (5.5,0.5) .. (7,-2) .. controls (7,-2.5) and (8,-1.5) .. (6,1) .. controls (4,3) and (2,2.5) .. (3,2);
	\node at (7,2.1) {$N$};
	\node at (6.4,1.4) {\color{purple} $U$};
	\node at (4,1.6) {\color{blue} $Q$};
	\node at (2.7,1.8) {\color{red} $V$};
	\end{tikzpicture}
\caption{Embedding of $U \cup V$ into $\check{W}$.} \label{fig:uvembedding}
	\end{figure}
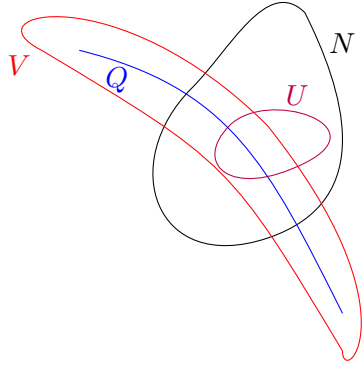
\end{center}

Lemma 3.1 in \cite{McduffSalamon:sympbook}
requires that $Q$ be compact.
However this is really not needed.

\begin{proof}[Proof of Lemma \ref{lemma Moser lemma}.]
Choose complete metrics $g$ and $\check{g}$ on $W$ and $\check{W}$ respectively
so that
$\iota_N^* \check{g} = g$
near $\overline{U}$ and
let $NQ$ be the symplectic normal bundle of $Q$ in $W$ and let $E \lra{} \iota_Q(Q)$ be the symplectic normal bundle of $\iota_Q(Q)$
in $\check{W}$.
Let
$\exp_W : NQ \lra{} W$
and
$\exp_{\check{W}} : E \lra{} \check{W}$
be the exponential maps using $g$ and $\check{g}$ respectively.
Let
$\phi : NQ \lra{} \iota_Q^* E$
be the isomorphism
described in 
(\ref{item:isomorphism between bundles}) above.
Choose small tubular neighborhoods
$A \subset NQ$ and $B \subset E$ of the zero section so that
$\exp_W|_A$ and
$\exp_{\check{W}}|_B$ are smooth embeddings and so that
$\phi(A) = B$.
We have a smooth family of symplectic forms
$$\omega_t := (1-t) ((\exp_{\check{W}})_* \phi_* \exp_W^* \omega_W) + t\omega_{\check{W}}, \quad t \in [0,1]$$
on $\widetilde{B} := \exp_{\check{W}}(B)$ after shrinking $A$ and $B$.
Now $\omega_t$ is equal to $\omega_{\check{W}}$ along $T\check{W}|_{\iota_Q(Q)}$ and also 
along some neighborhood $U' \subset \widetilde{B}$ of $\overline{U} \cap \widetilde{B}$ for each $t \in [0,1]$.
After shrinking $U'$, $A$ and $B$,
we can assume that $U'$ deformation retracts onto $U' \cap \iota_Q (Q)$.
Hence we can find a smooth family of $1$-forms $\sigma_t \in \Omega^1(\widetilde{B})$, $t \in [0,1]$
satisfying
\begin{itemize}
\item $\frac{d}{dt} \omega_t = d\sigma_t$ and
\item $\sigma_t(v) = 0$ for each $v \in T_x \widetilde{B}$, $x \in U' \cup \iota_Q(Q)$ and $t \in [0,1]$
\end{itemize}
after shrinking $U'$ slightly again.
Let $(X_t)_{t \in [0,1]}$ be a smooth family of vector fields
satisfying $\sigma_t + \iota_{X_t} \omega_t = 0$ for each $t \in [0,1]$.
Let $\psi_t$ be the time $t$ flow of $(X_t)_{t \in [0,1]}$ for each $t \in [0,1]$ (if it exists).
For each $x \in \iota_Q(Q)$,
we have that $X_t = 0$ at $x$ since $\sigma_t = 0$ at $x$.
Therefore for each $x \in \iota_Q(Q)$,
there is a small neighborhood
$V_x \subset \widetilde{B}$ of $x$ so that the time $t$
flow $\psi_t(y)$ is well defined
for each $y \in V_x$
and $t \in [0,1]$.
Hence the time $t$ flow
$\psi_t(y)$ is well defined
for each $y \in V := \cup_{x \in Q} V_x$
and each $t \in [0,1]$.
%
%
Define
$$\iota : U \cup V \hookrightarrow \check{W}, \quad \iota(x) := 
\left\{
\begin{array}{ll}
\iota_N(x) & \textnormal{if} \ x \in U \\
\psi_1(\exp_{\check{W}}(\phi(\exp_W^{-1}(x))) & \textnormal{if} \ x \in V.
\end{array}
\right.$$
This is a symplectic embedding.
Also since $\psi_1(x) = x$
for all $x \in (\overline{U} \cap \widetilde{B}) \cup \iota_Q(Q)$
we have
$\iota|_U = \iota_U$ and $\iota|_Q = \iota_Q$.
\end{proof}

The proof of the following lemma
is identical to the proof of
\cite[Lemma 12.1.2]{eliashbergmishchevhprinciple}
and so we will omit it.
However, we will give an idea of the proof since it will be a crucial ingredient in the proof
of Proposition \ref{proposition stably displaceable partly stratified symplectic subset} below.
\begin{lemma} \label{lemma isotopy submanifold}
Let $Q$ be a
symplectic submanifold of a symplectic manifold
$(W,\omega_W)$ of positive codimension and let $U, N \subset Q$ be open sets so that
$\overline{U} \subset N$.
Let $\alpha_t \in \Omega^1(Q)$, $t \in [0,1]$
be a smooth family of $1$-forms
so that $\omega_t := \omega_W|_{TQ} + d\alpha_t$, $t \in [0,1]$ is a smooth family of symplectic forms on $Q$ and so that $\alpha_t|_N = 0$ for all $t \in [0,1]$ and $\alpha_0=0$.
Then there is a smooth family of symplectic embeddings
$\iota_t : Q \hookrightarrow (W,\omega_W)$, $t \in [0,1]$ which are arbitrarily $C^0$
close to $\iota_0$
so that $\iota_1^* \omega_W = \omega_1$,
$\iota_0 = \id_Q$
and so that $\iota_0|_U = \iota_t|_U$ for all $t \in [0,1]$.
\end{lemma}

We will now give a hint at why the lemma above is true. The details are contained
in \cite[Lemma 12.1.2]{eliashbergmishchevhprinciple}.
First of all,
we can $C^1$ approximate
the family of $1$-forms
$\alpha_t$, $t \in [0,1]$
by a new
family of $1$-forms
$\alpha'_t$, $t \in [0,1]$
so that
\begin{itemize}
	\item $\alpha'_0 = \alpha_0$, $\alpha'_1 = \alpha_1$,
	\item
for each compact codimension $0$ submanifold $D \subset Q$,
the function 
$$[0,1] \lra{} \Omega^1(D), \quad t \to \alpha'_t|_D$$
is piecewise linear and
\item for each $t \in [0,1]$
where $\frac{d}{dt}\alpha'_t|_D$ is well defined,
there is a relatively compact open neighborhood $N_D$ of $D$
so that
$\frac{d}{dt}\alpha'_t|_{N_D}$
is equal to $r ds$
where $r, s$ are smooth functions on $Q$ with compact support in $N_D - \overline{U}$.
\end{itemize}
See \cite[12.1.3]{eliashbergmishchevhprinciple}) for more details.
After a bit more work, it is then sufficient to prove
Lemma \ref{lemma isotopy submanifold}
when $\alpha_t = t rds$
for some smooth compactly supported functions $r, s$ on $Q$.
In this special case,
the embeddings $\iota_t$, $t \in [0,1]$ are obtained by
using the following Lemma:
\begin{lemma} \cite[12.1.5]{eliashbergmishchevhprinciple} (Symplectic Twisting Lemma).
Let $(Q,\omega_Q)$
be a symplectic manifold and let $\D^2(\epsilon) \subset \C$
be the open disk of radius $\epsilon$ with the standard symplectic form $dx \wedge dy$.
Then for any compactly supported smooth functions $r,s$ on $Q$,
there is a smooth function $\phi : Q \lra{} \D^2(\epsilon)$ so that
the function
$$\Phi : Q \lra{} Q \times \D^2(\epsilon), \quad \Phi(q) := (q,\phi(q))$$
satisfies
$\Phi^*(\omega_Q + dx \wedge dy) = \omega_Q + dr \wedge ds$.
\end{lemma}
\proof
In order to prove this lemma,
one just needs to find $\phi$ so that
$\phi^* dx \wedge dy = dr \wedge ds$.
If $\epsilon>0$ was really large,
then we could just choose $\phi = (r,s)$.
However, $\epsilon$ could be really small
and such a map would not be well defined.
But this can be corrected by first finding a smooth
area preserving immersion
$\tau : \D^2(R) \looparrowright \D^2(\epsilon)$ for $R$ large and then letting $\phi = \tau \circ (r,s)$.
Such an immersion is illustrated in Figure \ref{fig:diskimmersion} below.
Here we should think of the disk $\D^2(R)$ as being `curled up'
inside the disk $\D^2(\epsilon)$.
Then $\phi = \tau \circ (r,s)$ for $R$ large has the properties we want.
\qed
\begin{center}
\begin{figure}[h]
	\begin{tikzpicture}
	
	\tikzset{
		ribbon/.style={
			preaction={
				preaction={
					draw,
					line width=0.85cm,
					white
				},
				draw,
				line width=0.8cm,
				black
			},
			line width=0.75cm,
			yellow!20!
		},
		ribbon/.default=white
	}

	\draw [fill=yellow!20!,thick] (0.1,0) ellipse (0.4 and 0.4);
	
	\draw  (0.1,-1.2) ellipse (2.7 and 2.3);

	\draw [ribbon](0.1,0) .. controls (2.1,0) and (2.4,-2.8) .. (0.2,-2.8);
	\draw [ribbon](0.2,-2.8) .. controls (-2.4,-2.8) and (-2.2,-0.1) .. (-0.2,-0.1);
	\draw [ribbon](-0.2,-0.1) .. controls (1.7,-0.1) and (2.1,-2.5) .. (0,-2.5);
	\draw [ribbon](0,-2.5) .. controls (-2.1,-2.5) and (-1.4,-0.2) .. (-0.2,-0.2) .. controls (1,-0.2) and (1.5,-2.3) .. (0.2,-2.3);
	\draw [ribbon](-0.2,-2.3) -- (0.2,-2.3);
	\draw [fill=yellow!20!,thick] (0.2,-2.1) ellipse (0.4 and 0.4);
	\draw [ribbon](-0.2,-2.3) .. controls (-1.1,-2.3) and (-1,-0.3) .. (-0.2,-0.3) .. controls (0.6,-0.3) and (1.3,-2.1) .. (0.2,-2.1);

	\node at (-3.3,-1.2) {$D^2(\epsilon)$};
	\node at (-3.6,0.2) {$D^2(R)$};
	\draw [->](-3,0.1) -- (-0.8,-1.1);

	\end{tikzpicture}
\caption{Immersion $\tau$.} \label{fig:diskimmersion}
\end{figure}
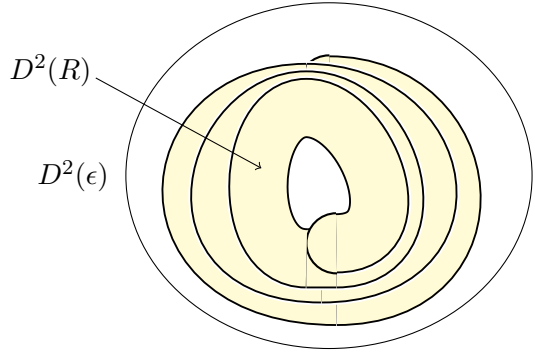
\end{center}

\begin{defn}
Let $Q$ be a manifold and let $U \subset N \subset Q$ be open subsets.
A {\it smooth embedded homotopy to $N$ rel $U$}
is a smooth family of embeddings
$\iota_t : Q \hookrightarrow Q$, $t \in [0,1]$
so that
\begin{enumerate}
\item $\iota_0 = \id$,
\item $\iota_1$ is a diffeomorphism onto $N$ and
\item $\iota_t(u) = u$ for all $u \in U$ and $t \in [0,1]$.
\end{enumerate}
\end{defn}

\begin{lemma} \label{h principle extension lemma}
Let $(W,\omega_W)$, $(\check{W},\omega_{\check{W}})$ be symplectic manifolds and let $Q \subset W$ be a symplectic submanifold of positive codimension.
Let $U, N \subset W$ be open sets so that
$\overline{U} \subset N$ 
and define
$U' := U \cap Q$ 
and $N' := N \cap Q$.
Suppose also that $\overline{U'}$ is a codimension $0$ submanifold
of $Q$ with the property that
$H^2(Q;\R) \lra{}H^2(U';\R)$
is an injection
and suppose
$Q$ admits a smooth embedded homotopy to $N'$ rel $U'$.
Let
$\iota_N : N \hookrightarrow \check{W}$
be a codimension $0$ symplectic embedding.
Then there is a neighborhood $V \subset W$ of $Q - \overline{U'}$ and a codimension $0$ symplectic embedding
$\iota : U \cup V \hookrightarrow \check{W}$
so that
$\iota|_U = \iota_N|_U$.
\end{lemma}
\begin{proof}
Let $\iota_t : Q \lra{} Q$, $t \in [0,1]$
be our smooth embedded homotopy to $N'$ rel $U'$.
Since $\overline{U'} \subset N'$
is a codimension $0$ submanifold,
we can modify $(\iota_t)_{t \in [0,1]}$
so that it is a smooth embedded homotopy to $N'$ rel $Y$ where $Y \subset N'$ is a neighborhood of $\overline{U'}$.
After shrinking $Y$,
we can also ensure that $\overline{Y} \subset N'$ is a codimension $0$ submanifold and that
the inclusion map $U' \hookrightarrow Y$ is a homotopy equivalence.
Let $\omega_t := \iota_{1-t}^* \omega_W$ for each $t \in [0,1]$.
Since $\overline{Y}$ is a submanifold of $Q$ and since $H^2(Q;\R) \lra{} H^2(\overline{Y};\R)$ is an injection, there is a smooth family of $1$-forms $(\alpha_t)_{t \in [0,1]}$
on $Q$ satisfying
\begin{itemize}
\item $\omega_t = \omega_0 + d\alpha_t$ and
\item $\alpha_t|_Y = 0$ for all $t \in [0,1]$ and $\alpha_0=0$.
\end{itemize}
Choose an open subset $\check{U} \subset Q$
so that $\overline{U'} \subset \check{U}$ and $\overline{\check{U}} \subset Y$.
Then by Lemma
\ref{lemma isotopy submanifold},
there is a smooth family of symplectic embeddings
$$\nu_t : Q - \overline{U'} \hookrightarrow (\check{W} - \overline{\iota_N(U)}, \omega_{\check{W}}), \quad t \in [0,1]$$
so that
$\nu_0(y) = \iota_N(\iota_1(y))$ for each $y \in Q$,
$\nu_t(y) = \iota_N(y)$ for all $y \in \check{U} - \overline{U'}$
and $\nu_1^* \omega_{\check{W}} = \omega_1$.
Hence by Lemma
\ref{lemma Moser lemma},
there is a neighborhood $V \subset W$
of $Q - \overline{U'}$
and a codimension $0$ symplectic embedding
$\iota : U \cup V \lra{} \check{W}$ so that
$\iota|_U = \iota_N|_U$
and
$\iota|_Q = \nu_1$.
\end{proof}

\begin{lemma} \label{lemma partly stratified admits weak curled}
Every partly stratified symplectic subset $S \subset M$ admits a weak curled up embedding.
\end{lemma}
\begin{proof}
Let $S_1,\cdots,S_l$ be the strata of $S$.
Suppose (by induction) that
$S_{<j} := \cup_{i <j} S_i$ admits a weakly curled up embedding $\check{\iota} : \check{N} \hookrightarrow \check{M}_{<1}$ for some $j \in \{1,\cdots,l\}$.
Since $\check{\iota}(y) = y$ for each $y$ in a neighborhood of $\check{N} \cap (M \times \T) \subset \check{M}$,
there is a neighborhood $\check{U} \subset \check{M}_{<1}$
of $M \times \T$ with the property that
$$\widetilde{\iota} : \check{U} \cup \check{N} \lra{} \check{M}, \quad \widetilde{\iota}(y) := \left\{
\begin{array}{ll}
\check{\iota}(y) & \textnormal{if} \ y \in \check{N} \\
y & \textnormal{if} \ y \in \check{U}
\end{array}
\right.
$$
is a smooth map after shrinking $\check{N}$ by an arbitrarily small amount.
Let $f_{\check{U} \cup \check{N}} : M \times \T \lra{} \R \cup \{\infty\}$
be the lower semi-continuous function
given in Equation (\ref{equation lower semi subset}) with $N$ replaced by $\check{U} \cup \check{N}$
(see Lemma \ref{lemma open subset}).
Now $f_{\check{U} \cup \check{N}} = \infty$
along $S_{<j}$.
Combining this with the fact that $f_{\check{U} \cup \check{N}}$ is positive and lower semi-continuous and the fact that $S_{<j} \times \T$ is compact,
we can find a smooth function
$\widetilde{f} : (M - S_{<j}) \times \T \lra{} \R_{>0}$
satisfying $\widetilde{f} < f_{\check{U} \cup \check{N}}$
so that for any sequence of points $(x_j)_{j \in \N}$ in $(M - S_{<j}) \times \T$ converging to a point in $S_{<j} \times \T$,
we have $\widetilde{f}(x_j) \to \infty$ as $j \to \infty$.
In particular,
$$U := (S_{<j} \times T^* \T) \cup \check{M}_{<\widetilde{f}}|_{(M - S_{<j}) \times \T}$$
is an open neighborhood of
$(M \cup \T) \cup (S_{<j} \times T^* \T)$.
See Figure \ref{fig:mlessthan2f}.
%
\begin{center}
\begin{figure}[h]
\usetikzlibrary{decorations.pathreplacing}
\begin{tikzpicture}

\draw[orange] (5,-0.6) .. controls (-1.6,-0.6) and (-1.3,-2.4) .. (5,-2.4);

\draw (-0.5,-1.1) -- (5,-1.1);
\draw (-0.5,-2) -- (5,-2);
\draw[purple] (-0.5,0) -- (-0.5,-3);
\draw[red] (-0.5,-0.5) -- (5,-0.5);
\draw[red] (-0.5,-2.5) node (v1) {} -- (5,-2.5);

\draw (-0.5,0) -- (1.3,0);
\draw (-0.5,-3) -- (1.3,-3);

\draw (1.3,0) -- (1.3,-3);
\draw[decorate,decoration={brace,amplitude=4pt},xshift=0pt,yshift=0pt] (1.3,-3.2) -- (-0.5,-3.2);

\node at (0.4,-3.7) {$\check{U}$};
\draw[decorate,decoration={brace,amplitude=3pt},xshift=-4pt,yshift=0pt] (5.4,-0.5) -- (5.4,-1.1);
\node at (5.9,-0.8) {$\check{N}$};
\draw[decorate,decoration={brace,amplitude=3pt},xshift=-4pt,yshift=0pt] (5.4,-2) -- (5.4,-2.5);
\node at (5.9,-2.2) {$\check{N}$};
\draw [->](2.0,-1.5) -- (0.3,-1.6);
\node[orange] at (3.2,-1.5) {$\sigma = \widetilde{f}(x,\tau)$};
\node at (-1.3,-1.4) {$U$};
\draw [->](-0.8,-1.45) -- (-0.2,-1.5);
\node[red] at (2.9,-3) {$S_{<j}$};
\node[red] at (2.9,-0.1) {$S_{<j}$};

\node[purple] at (-0.5,0.4) {$M \times \T$};
\end{tikzpicture}
\caption{Picture of $\check{M}_{<f}|_{S_j \times \T}$.} \label{fig:mlessthan2f}
\end{figure}
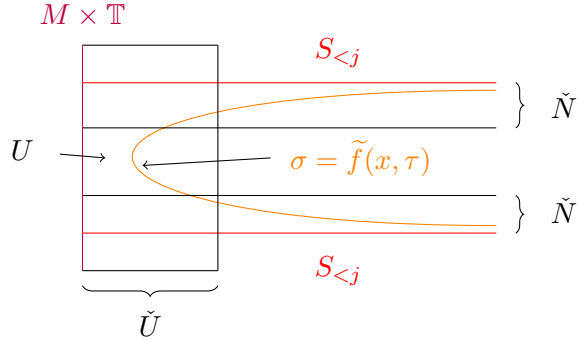
\end{center}
Define $Q := S_j \times T^*\T$
and $N := \check{M}_{<\check{U} \cup \check{N}}$.
Define $U' := U \cap Q$, $N' := N \cap Q$.
Then $\overline{U'}$
is a codimension $0$ submanifold of $Q$
with the property that $H^2(Q;\R) \lra{} H^2(U';\R)$ is injective.
Also $\overline{U} \subset N$ and
the manifold $Q$ admits a smooth embedded homotopy to $N'$ rel $U'$.
Hence by Lemma
\ref{h principle extension lemma} with $U$, $N$, $U'$, $N'$ and $Q$ as above and with $W = \check{M}$, $\check{W} = \check{M}_{< 1}$ and $\iota_N = \widetilde{\iota}|_N$,
there is a neighborhood $V \subset \check{M}$ of
$Q - \overline{U'}$
and a codimension $0$ symplectic embedding
$\iota : U \cup V \lra{} \check{M}_{<1}$ so that $\iota|_U = \widetilde{\iota}|_U$.
Since $\iota(y)=y$ for all $y$ sufficiently near $M \times \T$, we get that $\iota$ is a weakly curled up embedding of $S_{<j+1} = S_{<j} \cup S_j$.
Hence we are done by induction on $j \in \{1,\cdots,l\}$.
\end{proof}

\begin{proof}[Proof of Proposition \ref{proposition stably displaceable partly stratified symplectic subset}]
Let $S$ be our partly stratified symplectic subset.
By Lemma \ref{lemma partly stratified admits weak curled}
$S \times T^*\T$ admits a weak curled up embedding.
Hence by Lemma \ref{lemma weak curled implies curled},
$S \times T^*\T$ admits a curled up embedding.
Therefore by Lemma \ref{lemma cureld embedding implies stably displaceable},
$S$ is stably displaceable.
\end{proof}

\section{Proof of the Main Theorem} \label{section proof of main theorem}

By Calabi-Yau manifold,
we will just mean a smooth complex projective variety
with trivial first Chern class.
We have the following lemma:
\begin{lemma} \label{lemma identification of second homology groups}
(\cite[Proposition 3.1]{batyrev1999birational}).
Let $\widehat{\Phi} : X \dashrightarrow \widehat{X}$ be a birational equivalence between smooth projective Calabi-Yau manifolds.
Then there are complex
codimension $\geq 2$ subvarieties
$V_X \subset X$ and $V_{\widehat{X}} \subset \widehat{X}$
and an isomorphism
$$\Phi : X - V_X \lra{\cong} \widehat{X} - V_{\widehat{X}}$$
equal to $\widehat{\Phi}|_{X - V_X}$.
\end{lemma}

Note that the Calabi-Yau condition here
is crucial.
For instance, if we take a smooth projective variety and blow it up at a point,
then we create a birational variety
whose second Betti number is strictly larger.
Even though the proof of lemma
\ref{lemma identification of second homology groups}
is contained in \cite[Proposition 3.1]{batyrev1999birational},
we give a more detailed version of the same proof for
readers who may not be experts
in algebraic geometry.
There is also a more general version of this lemma (see \cite[Theorem 3.52]{KollarMori:birational}),
which proves the same statement for
varieties which are `minimal' in some sense.

\begin{proof}[Proof of Lemma \ref{lemma identification of second homology groups}.]
Morally, the idea of the proof is as follows.
If the region where $\widehat{\Phi}$ is not a submersion had complex codimension $1$ then $\widehat{\Phi}$
would contract a codimension $1$ subvariety $Z$ to a lower dimensional variety (See Definition \ref{definition divisors and line bundle etc}). One can then show that
such a contraction would ensure
that $c_1(X)$ or $c_1(\widehat{X})$ is non-trivial giving us a contradiction.
A similar argument also applies to the inverse $\widehat{\Phi}^{-1}$ and this gives us our result.
Therefore, our proof comes in two parts:
\begin{enumerate}
\item[Part (1)] \label{item:contractingdivispr} Showing that if the region where $\widehat{\Phi}$ is not a submersion has codimension less than $2$, then $\widehat{\Phi}$ contracts a divisor.
\item[Part (2)] \label{item:c1nontrivialneardivisor} Showing that $c_1(X)$ or $c_1(\widehat{X})$ is non-trivial due to the existence of such a divisor.
\end{enumerate}

Part (1): Since $\widehat{X}$ is a projective variety, we have that
$\widehat{X}$ is a closed subvariety
of $\C P^N$ for some $N  \in \N$.
Let $\iota : \widehat{X} \hookrightarrow \C P^N$ be the inclusion map.
Consider the rational map
$f = \iota \circ \widehat{\Phi}$.
In any holomorphic chart
$z_1,\cdots,z_n$ on $X$,
the map $f$ has the form
$[\phi_0(z_1,\cdots,z_n),\cdots,\phi_N(z_1,\cdots,z_n)]$
in homogeneous coordinates on $\C P^N$
where $\phi_0,\cdots,\phi_N$
are holomorphic functions.
This local description of $f$
does not change if we multiply
$\phi_0,\cdots,\phi_N$ by a common
holomorphic function.
In particular, if $\phi_0,\cdots,\phi_N$
all vanish to the same order
along some codimension $1$
subvariety, then we can remove a common factor from all of these holomorphic functions so that
$\phi_0^{-1}(0) \cap \cdots \cap \phi_N^{-1}(0)$ has codimension $2$ or higher after shrinking the chart
(\cite[Chapter 0, Section 1]{GriffithsHarris:algeraicgeometry}).
As a result, we get that there is a codimension $\geq 2$ subvariety
$V_X \subset X$
so that $\widehat{\Phi}$ is well defined
outside $V_X$.
Now suppose that $\widehat{\Phi} : X - V_X \lra{} \widehat{X}$
is not an isomorphism onto its image.
Then by \cite[Theorem 2.16]{Shafarevichbasicalgebraic},
there is a codimension $1$
subvariety $Z \subset X -  V_X$
which is contracted by
$\widehat{\Phi}$
as in Definition \ref{definition divisors and line bundle etc}.

\bigskip

Part (2): Let $V_X$ and $Z$ be as above.
We now wish to show that the existence of the variety $Z$ implies that either $c_1(X)$ or $c_1(\widehat{X})$ is non-trivial, and therefore giving us a contradiction.
Consider the relative canonical bundle
$\kappa_{X/\widehat{X}} := \kappa_X \otimes \widehat{\Phi}^* \kappa_{\widehat{X}}$
on $X - V_X$
where $\kappa_X$ and $\kappa_{\widehat{X}}$
are the canonical bundles of $X$ and $\widehat{X}$ respectively (Definition \ref{definition divisors and line bundle etc}).
Then
$\kappa_{X/\widehat{X}}$
can be identified with the
bundle
$\Hom(\kappa^*_X,\widehat{\Phi}^* \kappa^*_{\widehat{X}})$
where $\kappa^*_X$
and $\kappa^*_{\widehat{X}}$ are the anti-canonical bundles of $X$ and $\widehat{X}$ respectively (See Definition \ref{defn anti canonical bundle}).
Under this identification,
we have a section $s$ of
$\kappa_{X/\widehat{X}}$
given by the {\it Jacobian} of $\widehat{\Phi}$,
which is the map sending
$v_1 \wedge \cdots \wedge v_1 \in \kappa^*_X$
to $D\pi(v_1) \wedge \cdots \wedge D\pi(v_n) \in \widehat{\Phi}^* \kappa^*_{\widehat{X}}$.
This holomorphic section $s$
vanishes along $Z$.
Now since $V_X$
has complex codimension $\geq 2$,
we have by using
Harthog's theorem and the Cauchy integral formula \cite[Chapter 0, Section 1]{GriffithsHarris:algeraicgeometry},
that the bundle $\kappa_{X/\widehat{X}}$
and the section $s$
extends to a holomorphic
bundle $L$ over $X$ together with a
holomorphic section $\check{s}$ of $L$.
We have that $c_1(L)$
is Poincar\'{e} dual to $[(\check{s})]$
by (\cite[Chapter 1, Section 1]{GriffithsHarris:algeraicgeometry}),
where $(\check{s})$ and the homology class
$[(\check{s})]$
are given in Definition \ref{definition divisors and line bundle etc}.
Choose a K\"{a}hler form $\omega$ on $X$.
Then $\omega^{n-1}$ is a positive
volume form
on the smooth locus of any codimension
$1$ subvariety.
Therefore the de-Rham cohomology class
$[\omega^{n-1}]$
on $X$ pairs non-trivially with the homology class $[(\check{s})$] since
$(\check{s})$ is a non-trivial
effective divisor (See Definition \ref{definition divisors and line bundle etc}).
Hence $c_1(L)$ is non-zero.
Since $V_X$ is a complex
codimension $\geq 2$ subvariety,
we get that the natural restriction map
$H^2(X;\Z) \lra{} H^2(X - V_X;\Z)$
is an isomorphism.
This implies that $c_1(\kappa_{X/\widehat{X}}) \neq 0$
in $X - V_X$.
Since $c_1(\kappa_{X/\widehat{X}})
= c_1(\kappa_X|_{X - V_X})
- \widehat{\Phi}^* c_1(\kappa_{\widehat{X}})$,
either $c_1(X)$ or $c_1(\widehat{X})$ is non-zero giving us a contradiction.
Hence $\widehat{\Phi}$ maps $X - V_X$ isomorphically to its image.

Similar reasoning can be used to show that
there is a codimension $\geq 2$ subvariety
$\check{V}_{\widehat{X}} \subset \widehat{X}$
so that $\widehat{\Phi}^{-1}$
maps $\widehat{X} - \check{V}_{\widehat{X}}$
isomorphically to an open subset
of $X$ containing $X - V_X$.
Hence $V_{\widehat{X}} := \widehat{X} - \widehat{\Phi}(X - V_X)$
has codimension $\geq 2$
and our lemma holds.
\end{proof}

\bigskip

We have the following immediate corollary of Lemma \ref{lemma identification of second homology groups}.

\begin{corollary} \label{corollary identification of H2}
Let $X$, $\widehat{X}$, $\widehat{\Phi}$, $V_X$, $V_{\widehat{X}}$ and $\Phi$
be as in Lemma
\ref{lemma identification of second homology groups}.
Let $D \subset X - V_X$
be a compact submanifold possibly with boundary and let $R$ be any ring.
Then there is
a natural identification of homology
and cohomology
groups
\begin{equation} \label{definition explicit isomorphism of H2 groups}
\Phi_* : H_2(X - D;R) \stackrel{\cong}{\longleftarrow} H_2(X - V_X - D;R) \stackrel[\Phi_*]{\cong}{\longrightarrow} H_2(\widehat{X} - V_{\widehat{X}} - \Phi(D);R) \lra{\cong} H_2(\widehat{X} - \Phi(D);R).
\end{equation}
\begin{equation} \label{definition explicit isomorphism of H hat 2 groups}
\Phi^* : H^2(\widehat{X},\Phi(D);R) \stackrel{\cong}{\longrightarrow} H^2(\widehat{X} - V_{\widehat{X}},\Phi(D);R) \stackrel[\Phi^*]{\cong}{\longrightarrow} H^2(X - V_X,D;R) \stackrel{\cong}{\longleftarrow} H^2(X,D;R).
\end{equation}
\end{corollary}

\begin{remark}
From now on we will fix the notation $\widehat{\Phi}, \Phi, X, \widehat{X}, V_X$ and $V_{\widehat{X}}$.
Also we will not distinguish between $H_2(X -D;R)$ and $H_2(\widehat{X} -\Phi(D);R)$
(resp. $H^2(X,D;R)$, $H^2(\widehat{X},\Phi(D);R)$) for each compact submanifold with boundary $D \subset X - V_X$ and each ring $R$.
\end{remark}

The following lemma finds for us an appropriate common Zariski dense affine variety on $X$ and $\widehat{X}$.
This lemma will also be used to
construct appropriate K\"{a}hler forms on $X$ and $\widehat{X}$ which agree on a large compact subset of this common affine subvariety.
We will use the notation from Definition \ref{definition divisors and line bundle etc}.

\begin{lemma} \label{lemma effective divisors on X and widehat X}
Let $\omega_X$ and $\omega_{\widehat{X}}$ be K\"{a}hler forms on $X$ and $\widehat{X}$
and suppose that their corresponding De Rham
cohomology classes
admit lifts to integral cohomology classes
$[\omega_X] \in H^2(X;\Z)$
and $[\omega_{\widehat{X}}] \in H^2(\widehat{X};\Z)$ respectively.
Then there are effective divisors
$\Delta$ on $X$ and $\widehat{\Delta}$ on $\widehat{X}$
and an integer $\mu>0$
so that
\begin{enumerate}
\item \label{item ample condition}
$[\Delta]^* = \mu [\omega_X]$
and $[\widehat{\Delta}]^* = \mu [\omega_{\widehat{X}}]$,
\item \label{item containing exceptional sets}
$V_X \subset \textnormal{supp}(\Delta)$,
$V_{\widehat{X}} \subset \textnormal{supp}(\widehat{\Delta})$,
\item \label{item: same support} $\overline{\Phi(\textnormal{supp}(\Delta) - V_X)} = \textnormal{supp}(\widehat{\Delta})$ and
\item $A := X - \textnormal{supp}(\Delta)$
and $\widehat{A} := \widehat{X} - \textnormal{supp}(\widehat{\Delta})$
are affine varieties.
\end{enumerate}
\end{lemma}
\proof
By \cite[Chapter 1, Section 2]{GriffithsHarris:algeraicgeometry},
we can find a collection
$\Delta_1,\cdots,\Delta_l$
of irreducible codimension $1$ subvarieties of $X$
so that
$[\Delta_1]^*,\cdots,[\Delta_l]^*$
generates $H^{1,1}(X) \cap H^2(X;\Z)$
and so that
$[\widehat{\Delta}_1]^*,\cdots,[\widehat{\Delta}_l]^*$
generates $H^{1,1}(\widehat{X}) \cap H^2(\widehat{X};\Z)$
where $\widehat{\Delta_i} = \overline{\Phi(\Delta_i - V_X)}$ for each $i=1,\cdots,l$.
We can enlarge this collection of varieties
so that {\it (\ref{item containing exceptional sets})} is satisfied.
Choose a positive integer $N>0$
large enough so that
$$[\omega_X] - \frac{1}{N} \sum_{i=1}^l [\Delta_i]^*, \quad [\omega_{\widehat{X}}] - \frac{1}{N} \sum_{i=1}^l [\widehat{\Delta}_i]^*$$
represent K\"{a}hler forms.
Such an $N$ exists since being K\"{a}hler is an open condition among $(1,1)$-forms with respect to the $C^0$-topology.
By the Kodaira embedding theorem
\cite[Chapter 1, Section 4]{GriffithsHarris:algeraicgeometry},
we can find a positive integer $m$
and codimension $1$ subvarieties
$\Upsilon$ on $X$ and $\widehat{\Upsilon}'$ on $\widehat{X}$
so that
\begin{equation} \label{equation ample perturbed sum}
[\Upsilon]^* = mN[\omega_X] - m\sum_{i=1}^l [\Delta_i]^*, \quad [\widehat{\Upsilon}']^* = mN[\omega_{\widehat{X}}] - m\sum_{i=1}^l [\widehat{\Delta}_i]^*.
\end{equation}
Let $\widehat{\Upsilon} := \overline{\Phi(\Upsilon-V_X)}$
and
$\Upsilon' = \overline{\Phi^{-1}(\widehat{\Upsilon}' - V_{\widehat{X}})}$.
Since $[\Delta_1]^*,\cdots,[\Delta_l]^*$
generates $H^{1,1}(X) \cap H^2(X;\Z)$,
and
$[\widehat{\Delta}_1]^*,\cdots,[\widehat{\Delta}_l]^*$
generates $H^{1,1}(\widehat{X}) \cap H^2(\widehat{X};\Z)$,
we have by
\cite[Chapter 1, Section 2]{GriffithsHarris:algeraicgeometry}
that
there are integers $(a_i)_{i=1}^l$,
$(\widehat{a}_i)_{i=1}^l$
so that
\begin{equation} \label{equation making divisors in span}
\sum_{j=1}^l a_i[\Delta_i]
= [\Upsilon'], \quad
\sum_{j=1}^l \widehat{a}_i[\widehat{\Delta}_i]
= [\widehat{\Upsilon}].
\end{equation}
Now choose a positive integer $m'$
greater than $\max_{i=1}^l a_i$
and $\max_{i=1}^l \widehat{a}_i$.
Define
$$\Delta := m'\Upsilon + \Upsilon' + \sum_{i=1}^l (m'm - a_i) \Delta_i, \quad
\widehat{\Delta} := \widehat{\Upsilon} + m'\widehat{\Upsilon}' + \sum_{i=1}^l (m'm - \widehat{a}_i) \widehat{\Delta}_i.
$$
Then by Equations (\ref{equation ample perturbed sum})
and
(\ref{equation making divisors in span}),
we have that
$[\Delta]^* = m'mN[\omega_X]$
and $[\widehat{\Delta}]^* = m'mN [\omega_{\widehat{X}}]$ and hence 
{\it (\ref{item ample condition})} holds with $\mu = m'm N$.
Also $\Delta$ and $\widehat{\Delta}$ are effective divisors with support
$$\textnormal{supp}(\Delta)
= \Upsilon \cup \Upsilon' \cup \bigcup_{i=1}^l \Delta_i, \quad
\textnormal{supp}(\widehat{\Delta}) =
\widehat{\Upsilon} \cup \widehat{\Upsilon}' \cup \bigcup_{i=1}^l \widehat{\Delta}_i
$$
and hence
{\it (\ref{item containing exceptional sets})} and
{\it (\ref{item: same support})} are satisfied.
Finally $A$ and $\widehat{A}$
are affine varieties
by the Kodaira embedding
theorem
\cite[Chapter 1, Section 4]{GriffithsHarris:algeraicgeometry}.
\qed

\bigskip

We will now prove our main result (Theorem \ref{theorem main theorem}). Here is a statement of this theorem:

{\it 
Let $\omega_X$ and $\omega_{\widehat{X}}$
be K\"{a}hler forms on $X$ and $\widehat{X}$ respectively whose cohomology classes
lift to integer cohomology classes.
Then there exists a graded
$\Lambda_\K^{\omega_X,\omega_{\widehat{X}}}$-algebra
$Z$ and algebra isomorphisms
$$Z \otimes_{\Lambda_\K^{\omega_X,\omega_{\widehat{X}}}} \Lambda_\K^{\omega_X} \lra{\cong}
QH^*(X;\Lambda_\K^{\omega_X}), \quad Z \otimes_{\Lambda_\K^{\omega_X,\omega_{\widehat{X}}}} \Lambda_\K^{\omega_{\widehat{X}}} \lra{\cong}
QH^*(X;\Lambda_\K^{\omega_{\widehat{X}}})$$
over the Novikov rings $\Lambda^{\omega_X}_\K$ and $\Lambda^{\omega_{\widehat{X}}}_\K$
respectively
where
$\Lambda_\K^{\omega_X,\omega_{\widehat{X}}}$,
$\Lambda_\K^{\omega_X}$ and
$\Lambda_\K^{\omega_{\widehat{X}}}$
are the Novikov rings given in Example
\ref{example Novikov ring associated to some forms}.
}
\begin{proof}[Proof of Theorem \ref{theorem main theorem}]
Let $n$ be the complex dimension of $X$.
Let $\Delta$, $\widehat{\Delta}$, $\mu$, $A$ and $\widehat{A}$ be as in Lemma \ref{lemma effective divisors on X and widehat X}.
By the divisor line bundle correspondence
\cite[Chapter 1, Section 1]{GriffithsHarris:algeraicgeometry},
there are holomorphic line bundles
$L \lra{} X$, $\widehat{L} \lra{} \widehat{X}$
with holomorphic sections
$s$ and $\widehat{s}$ respectively
so that we have an equality of divisors $(s^{-1}(0)) = \Delta$ and $(\widehat{s}^{-1}(0)) = \widehat{\Delta}$ respectively.
Choose Hermitian metrics
$|\cdot|$ and $|\cdot|'$
on $L$ and $\widehat{L}$ respectively
so that
$-dd^c \rho = \mu\omega_X|_A$ and
$-dd^c \widehat{\rho} = \mu \omega_{\widehat{X}}|_{\widehat{A}}$
where $\rho := -\log(|s|)$
and $\widehat{\rho} := -\log(|\widehat{s}|')$.

By Corollary \ref{corollary divisor stably displaceable},
there exists a constant
$C>0$ so that $X - K$ is stably displaceable inside the symplectic manifold
$(X,\mu \omega_X)$ where
$K := \rho^{-1}((-\infty,C])$.
Also by Lemma \ref{lemma gradient large enough at infinity}
we can enlarge $C$ so that
the interior of $K$
contains the skeleton of $\rho$
(Definition \ref{definition skeleton of plurisubharmonic function}).
Therefore by a Moser argument (\cite[Exercise 3.36]{McduffSalamon:sympbook}),
there is a contact cylinder
$\check{C}_4 \subset A$
inside the symplectic manifold
$(X,\mu \omega_X)$
whose associated Liouville domain
is $D_4 := K$.
Define $\widehat{C}_4 := \Phi(\check{C}_4)$
and $\widetilde{D}_4 := \Phi(D_4)$.

By Proposition
\ref{proposition constructino of index bounded contact cylinder on appropriate affine variety},
there is an index
bounded contact cylinder
$\check{C}_3 = [1-\epsilon_3,1+\epsilon_3] \times C_3 \subset A$
inside the symplectic manifold
$(X,\mu \omega_X)$
so that the interior of the associated Liouville domain $D_3$
contains $D_4$ and is contained in $A$.
Define
$\widehat{C}_3 := \Phi(\check{C}_3)$
and $\widetilde{D}_3 := \Phi(D_3)$.

By Corollary
\ref{lemma making forms coincide at infinity},
there exists a partially algebraic plurisubharmonic function
$\check{\rho} : \widehat{A} \lra{} \R$,
a compact set
$K'$ of $A$ containing $D_3 \cup \check{C}_3$ 
so that
\begin{enumerate}
	\item $\Phi^*(\check{\rho})|_{D_3 \cup \check{C}_3} = \rho|_{D_3 \cup \check{C}_3}$,
	\item $\check{\rho} = \kappa_1( \widehat{\rho} - \log(\kappa_2))$ outside $\Phi(K')$ for some large $\kappa_1,\kappa_2 \in \N$ and
	\item \label{items skeletons are the same up to isomorphism}
	the skeleton of $\Phi^*(\check{\rho})$ is equal to the skeleton of $\rho$.
\end{enumerate}
Define
$\omega'_X := \mu \omega_X$
and
$$\omega'_{\widehat{X}} := 
\left\{
\begin{array}{ll}
\kappa_1 \mu \omega_{\widehat{X}} & \text{inside} \ X - \Phi(K') \\
-dd^c \check{\rho} & \text{otherwise}.
\end{array}
\right\}
$$

By Corollary \ref{corollary divisor stably displaceable},
there is a compact subset
$\widehat{K} \subset \widehat{X}$
whose interior contains $\Phi(K')$
so that $X - \widehat{K}$ is stably displaceable in $(\widehat{X},\omega'_{\widehat{X}})$.
By Proposition
\ref{proposition constructino of index bounded contact cylinder on appropriate affine variety},
there exists an index bounded contact cylinder
$\widehat{C} = [1-\epsilon,1+\epsilon] \times C$
in
the symplectic manifold
$(\widehat{X},\omega'_{\widehat{X}})$
which is disjoint from $\cup_{j \in S} \widehat{\Delta}_j$ and $\widehat{K}$
and whose associated Liouville domain
$\widehat{D}$
is contained in $\widehat{A}$ and
whose interior contains $\widehat{K}$.
By the same proposition, we can also assume that $\widehat{C}$ contains a contact cylinder $\widehat{C}_1$
whose associated Liouville domain
 $\widetilde{D}_1$ contains $\widehat{K}$ and
so that a Liouville form associated to $\widehat{C}_1$ is the restriction of $-d^c \check{\rho}$ to $\widetilde{D}$.
By a Moser argument (\cite[Exercise 3.36]{McduffSalamon:sympbook}) we can construct
contact cylinders
$\widehat{C}_0 := [1-\epsilon_0,1+\epsilon_0] \times C_0$
and
$\widehat{C}_2$
inside
$(\widehat{X},\omega'_{\widehat{X}})$
whose associated Liouville domains
are $\widehat{D}_0 = \widehat{D} \cup \widehat{C}$
and 
$\widetilde{D}_2 = \overline{\widehat{D} - \widehat{C}}$.
I.e. they
have boundary equal to
$\{1+\epsilon\} \times C$
and $\{1-\epsilon\} \times C$ respectively.
Also we can assume
$\widetilde{D}_3 \subset \widetilde{D}_2 \subset \widetilde{D}_1 \subset \widetilde{D}_0$.
We define $\check{C}_i := \Phi^{-1}(\widehat{C}_i)$, $D_i := \Phi^{-1}(\widetilde{D}_i)$ for $i=0,1,2$ (See Figure \ref{fig:liouvilledomainsinxandwidehatx}).

\begin{center}
\begin{figure}[h]
	\begin{tikzpicture}[scale=1.5]

	\draw[red]  (1.2,-1.8) ellipse (2 and 1);
	\draw[orange] (1.2,-1.8) ellipse (1.8 and 0.8);
	\draw[magenta]  (1.2,-1.8) ellipse (1.6 and 0.6);
	\draw[purple]  (1.2,-1.8) ellipse (1.4 and 0.4);
	\draw [blue] (1.2,-1.8) ellipse (1.2 and 0.2);
	
	\draw (1.1,-0.5) -- (-0.9,-1.1);
	\draw (0.5,-0.5) -- (2.9,-0.9);
	\node at (1.9,-0.6) { \footnotesize $\Delta$};
	\node at (0.1316,-2.5486) {\color{red} \footnotesize $D_0$};
	\node at (0.6205,-2.4748) {\color{orange} \footnotesize $D_1$};
	\node at (0.9127,-2.315) {\color{magenta} \footnotesize $D_2$};
	\node at (1.2771,-2.1301) {\color{purple} \footnotesize $D_3$};
	\node at (1.5,-1.9) {\color{blue} \footnotesize $D_4$};
	\node at (-1.3271,-0.3782) {$X$};

	\node at (1.2+2.5,-1.6) {$\Phi$};
	\draw[->] (1.2+2.5-0.3,-1.8) -- (1.2+2.5+0.3,-1.8);

	\draw[red]  (1.2+5,-1.8) ellipse (2 and 1);
	\draw[orange] (1.2+5,-1.8) ellipse (1.8 and 0.8);
	\draw[magenta]  (1.2+5,-1.8) ellipse (1.6 and 0.6);
	\draw[purple]  (1.2+5,-1.8) ellipse (1.4 and 0.4);
	\draw [blue] (1.2+5,-1.8) ellipse (1.2 and 0.2);
	
	\draw (1.1+5,-0.5) -- (-0.9+5,-1.1);
	\draw (0.5+5,-0.5) -- (2.9+5,-0.9);
	\node at (1.9+5,-0.6) { \footnotesize $\widehat{\Delta}$};
	\node at (0.1316+5,-2.5486) {\color{red} \footnotesize $\widetilde{D}_0$};
	\node at (0.6205+5,-2.4748) {\color{orange} \footnotesize $\widetilde{D}_1$};
	\node at (0.9127+5,-2.315) {\color{magenta} \footnotesize $\widetilde{D}_2$};
	\node at (1.2771+5,-2.1301) {\color{purple} \footnotesize $\widetilde{D}_3$};
	\node at (1.5+5,-1.9) {\color{blue} \small $\widetilde{D}_4$};
	\node at (-1.3271+10,-0.3782) {$\widehat{X}$};
	
	\end{tikzpicture}
\caption{Liouville domains in $X$ and $\widehat{X}$.} \label{fig:liouvilledomainsinxandwidehatx}
\end{figure}
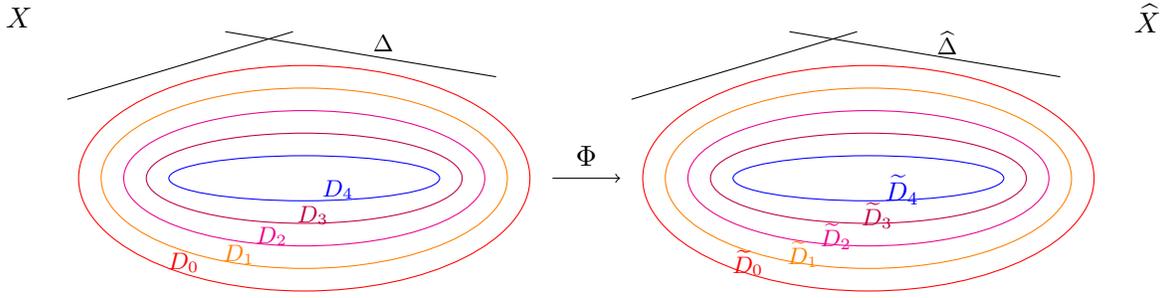
\end{center}

Let $\phi_t : \widehat{A} \lra{}  \widehat{A}$ be the time $t$ flow of $-\nabla_{\check{\rho}} \check{\rho}$
for all $t \geq 0$.
Since the skeleton of $\rho$ is contained in the interior of $D_4$,
we get that the skeleton of $\check{\rho}$ is contained in the interior of $\widetilde{D}_4$ by 
(\ref{items skeletons are the same up to isomorphism}).
Therefore
we have that there is a constant $T>0$
so that $\phi_T(\widetilde{D}_3) \subset \widetilde{D}_4$ for some $T>0$.
Since the inclusions
$\widetilde{D}_4 \subset \widetilde{D}_1$,
$\phi_T(\widetilde{D}_3) \subset \widetilde{D}_3$
and
$\widetilde{D}_2 \subset \widetilde{D}_0$
induce isomorphisms
$$H^2(\widehat{X},\widetilde{D}_1;\R)
\cong H^2(\widehat{X},\widetilde{D}_4;\R), \quad
H^2(\widehat{X},\widetilde{D}_3;\R)
\cong H^2(\widehat{X},\phi_T(\widetilde{D}_3);\R),$$
$$
H^2(\widehat{X},\widetilde{D}_0;\R)
\cong H^2(\widehat{X},\widetilde{D}_2;\R),
$$
we get that the restriction map
$H^2(\widehat{X},\widetilde{D}_i;\R)
\lra{} H^2(\widehat{X},\widetilde{D}_j;\R)
$
is an isomorphism for each $0 \leq i \leq j \leq 4$.
Also we have natural isomorphisms
$H^2(X,D_i;\R)
\lra{} H^2(\widehat{X},\widetilde{D}_i;\R)
$
for each $i=0,1,2,3,4$
by Corollary \ref{corollary identification of H2}.
So from now on we will identify all of these cohomology groups.

%

Since the quantum cohomology groups
and the associated Novikov rings
in the statement of our theorem
only depend on $\omega_X$ and $\omega_{\widehat{X}}$ up to scalar multiplication and up to adding an exact $2$-form,
we can just replace $\omega_X$
and $\omega_{\widehat{X}}$
with $\omega'_X$
and $\omega'_{\widehat{X}}$.
Hence from now on, we'll assume
$\omega_X = \omega'_X$
and
$\omega_{\widehat{X}} = \omega'_{\widehat{X}}$.

Let $\widetilde{\omega} \in \Omega^2(X)$
be a $\check{C}_2$-compatible
$2$-form which is equal to $0$
inside $D_3$
and $\omega_X$ outside $D_3 \cup ([1,1+\epsilon_3/2] \times C_3)$.
Let $\widetilde{\omega}' \in \Omega^2(\widehat{X})$
be a $\check{C}_0$-compatible $2$-form
which is equal to $0$ inside $\widetilde{D}_0$
and $\omega_{\widehat{X}}$ outside $\widetilde{D}_0 \cup ([1,1+\epsilon_0/2] \times C_0)$.
In the unlikely event that $[\widetilde{\omega}] \in H^2(X;\R)$ is proportional
to $[\widetilde{\omega}'] \in H^2(\widehat{X};\R)$, we rescale
$\widetilde{\omega}'$ so that these cohomology classes are equal (this is done to ensure that the cones constructed below satisfy condition (\ref{item:action interval thin}) of Definition
\ref{defn chain complex}).
Let $Q_+ \subset H^2(X,D_3;\R) \times \R \times \R= H^2(\widehat{X},\widetilde{D}_0;\R) \times \R\times \R$
be the cone spanned
by $([\widetilde{\omega}],1,1)$
and $([\widetilde{\omega}'],1,1)$.
Let $Q_- \subset H^2(X,D_3;\R) \times \R \times \R$
be the cone spanned by
$([\widetilde{\omega}],0,1)$,
$([\widetilde{\omega}],1,1)$
$([\widetilde{\omega}'],0,1)$
and $([\widetilde{\omega}'],1,1)$.
Let $Q^{\check{C}_3}_+ \subset H^2(X,D_3;\R) \times \R \times \R$,
(resp. $Q^{\widehat{C}_0}_+ \subset H^2(X,D_3;\R) \times \R \times \R$)
be the one dimensional cone spanned by
$([\widetilde{\omega}],1,1)$
(resp. 
$([\widetilde{\omega}'],1,1)$).
Let
$Q^{\check{C}_3}_- \subset H^2(X,D_3;\R) \times \R \times \R$,
(resp. $Q^{\widehat{C}_0}_- \subset H^2(X,D_3;\R) \times \R \times \R$)
be the two dimensional cone spanned by
$([\widetilde{\omega}],0,1)$ and
$([\widetilde{\omega}],1,1)$
(resp.
$([\widetilde{\omega}'],0,1)$,
$([\widetilde{\omega}'],1,1)$).

Since the interior of $D_3$ contains $D_4=K$,
we have that $\overline{X - D_3}$ is stably
displaceable inside $(X,\omega_X)$ and
hence
we have an isomorphism
of $\Lambda^{Q^{\check{C}_3}_+}_\K = \Lambda^{\omega_X}_\K$-algebras
\begin{equation} \label{equation isomorphic to QHX}
SH^*_{\check{C}_3,Q^{\check{C}_3}_-,Q^{\check{C}_3}_+}(D_3 \subset X) \cong QH^*(X,\Lambda_\K^{\omega_X}),
\end{equation}
and
\begin{equation} \label{equation lim lim1 is zero for X}
\varinjlim {\varprojlim}^1 \SH^*_{\check{C}_3,Q^{\check{C}_3}_-,Q^{\check{C}_3}_+}(D_3 \subset X) = 0.
\end{equation}
by Theorems \ref{theorem isomorphic to quantum cohomology}, \ref{theorem stably displaceable complement} and Propositions 
\ref{proposition forgetting contact cylinder} and \ref{proposition alternative filtrations defining symplectic cohomology}.
Similarly, we have an isomorphism
of $\Lambda^{Q^{\widehat{C}_0}_+}_\K = \Lambda^{\omega_{\widehat{X}}}_\K$-algebras
\begin{equation} \label{equation isomorphic to HWHX}
SH^*_{\widehat{C}_0,Q^{\widehat{C}_0}_-,Q^{\widehat{C}_0}_+}(\widetilde{D}_0 \subset \widehat{X}) \cong QH^*(X,\Lambda_\K^{\omega_{\widehat{X}}}),
\end{equation}
and
\begin{equation} \label{equation lim lim1 is zero}
\varinjlim {\varprojlim}^1 \SH^*_{\widehat{C}_0,Q^{\widehat{C}_0}_-,Q^{\widehat{C}_0}_+}(\widetilde{D}_0 \subset \widehat{X}) = 0.
\end{equation}
By Proposition \ref{proposition transfer isomorphism between index bounded Liouville domains} applied to
$(\widehat{C}_i)_{i=0,1,2,3,4}$
and
$(\widetilde{D}_i)_{i=0,1,2,3,4}$
we have that the transfer map
\begin{equation} \label{equation index bounded transfer map is an isomorphism}
\SH^*_{\widehat{C}_0,Q^{\widehat{C}_0}_-,Q^{\widehat{C}_0}_+}(\widetilde{D}_0 \subset \widehat{X}) \lra{\cong}
\SH^*_{\widehat{C}_0,Q^{\widehat{C}_0}_-,Q^{\widehat{C}_0}_+}(\widetilde{D}_3 \subset \widehat{X})
\end{equation}
is an isomorphism.
Also by Proposition \ref{proposition changing contact cylinder},
we have an isomorphism
\begin{equation} \label{equation changing contact cylinders for symplectic cohomology}
\SH^*_{\widehat{C}_0,Q^{\widehat{C}_0}_-,Q^{\widehat{C}_0}_+}(\widetilde{D}_3 \subset \widehat{X}) \cong
\SH^*_{\widehat{C}_3,Q^{\widehat{C}_0}_-,Q^{\widehat{C}_0}_+}(\widetilde{D}_3 \subset \widehat{X}).
\end{equation}

Define
$$Z := SH^*_{\check{C}_3,Q_-,Q_+}(D_3 \subset X).
$$
Since $\Phi(D_3) = \widetilde{D}_3$
and since all $1$-periodic orbits and Floer trajectories used to define
$Z$
can be made to avoid $V_X$ and $V_{\widehat{X}}$ 
by the ideas in Section \ref{section avoiding codimsion geq 4 submanifolds}
due to the fact that they are unions of submanifolds of real codimension $\geq 4$,
we have an isomorphism of $\Lambda_\K^{\omega_X,\omega_{\widehat{X}}} = \Lambda_\K^{Q_+}$-algebras
\begin{equation} \label{equation main isomorphism between Zs}
Z \cong SH^*_{\widehat{C}_3,Q_-,Q_+}(\widetilde{D}_3 \subset \widehat{X}).
\end{equation}
Since
both $\Lambda_\K^{\omega_X}$
and $\Lambda_\K^{\omega_{\widehat{X}}}$
are flat $\Lambda_\K^{\omega_X,\omega_{\widehat{X}}}$-modules
by Proposition \ref{flatness of rational polyhedral novikov rings},
we have by
Theorem \ref{theorem changing novikov ring}
combined with Equations
(\ref{equation lim lim1 is zero for X}),
(\ref{equation lim lim1 is zero}),
(\ref{equation index bounded transfer map is an isomorphism}), (\ref{equation changing contact cylinders for symplectic cohomology})
and
(\ref{equation main isomorphism between Zs})
isomorphisms
\begin{equation} \label{equation fmi}
Z \otimes_{\Lambda_\K^{\omega_X,\omega_{\widehat{X}}}} \Lambda_\K^{\omega_X} \cong SH^*_{\check{C}_3,Q^{\check{C}_3}_-,Q^{\check{C}_3}_+}(D_3 \subset X),
\end{equation}
\begin{equation} \label{equation smi}
Z \otimes_{\Lambda_\K^{\omega_X,\omega_{\widehat{X}}}} \Lambda_\K^{\omega_{\widehat{X}}} \cong
SH^*_{\widehat{C}_3,Q^{\widehat{C}_3}_-,Q^{\widehat{C}_3}_+}(\widetilde{D}_3 \subset \widehat{X}).
\end{equation}
Therefore our theorem now follows from Equations
(\ref{equation fmi}),
(\ref{equation smi}),
(\ref{equation isomorphic to QHX}),
(\ref{equation isomorphic to HWHX}),
(\ref{equation index bounded transfer map is an isomorphism}) and (\ref{equation changing contact cylinders for symplectic cohomology}).
\end{proof}

\section{Appendix A: Hamiltonians and Almost Complex Structures Compatible with Contact Cylinders.}

This section contains some lemmas allowing us to perturb Hamiltonians so that they become non-degenerate, while retaining certain properties.
It also has a lemma telling us that a certain action spectrum has measure $0$.

\begin{defn}
Let $H$ be a Hamiltonian on $M$.
Let $\Gamma$ be a collection of $1$-periodic orbits of $H$.
The {\it associated fixed points of $\Gamma$}
is a subset of $M$ denoted by
$$\Gamma(0) := \{x \in M \ : \ \exists \ \gamma \in \Gamma \ \text{such that} \ x = \gamma(0) \}.$$

We say that $\Gamma$ is {\it isolated}
if there is a neighborhood $N_\Gamma$ of $\Gamma(0)$ so that for each $1$-periodic orbit $\gamma$ satisfying $\gamma(0) \in N_\Gamma$,
we have $\gamma \in \Gamma$.
We call $N_\Gamma$ an {\it isolating neighborhood of $\Gamma$}.
\end{defn}

\begin{lemma} \label{lemma generic Hamiltonian constant elsewhere}

Let $H = (H_t)_{t \in \T}$ be a Hamiltonian on $M$ and let
$\Gamma$ be a set of $1$-periodic orbits of $M$ which is isolated with isolating neighborhood
$N_\Gamma$ and let $N$ be a neighborhood of $\Gamma(0)$ whose closure is contained in $N_\Gamma$.
Let $\ccH(N,H)$ be the space of Hamiltonians $K = (K_t)_{t \in \T}$
on $M$ satisfying
$H_t|_{\phi^H_t(N)} = K_t|_{\phi^H_t(N)}$ for all $t \in [0,1]$
equipped with the $C^\infty$ topology
and let $\ccH^\reg(N,H) \subset \ccH(N,H)$ be the subset of those Hamiltonians with the property that every $1$-periodic
orbit $\gamma$ not in $\Gamma$ is non-degenerate.
Then there exists a sequence $(H_i)_{i \in \N}$ of elements in $\ccH^\reg(N,H)$ converging to $H$.
\end{lemma}
\proof

Let $N' \subset N_\Gamma$ be an open set
so that $\overline{N} \subset N'$ and $\overline{N'} \subset N_\Gamma$.
By \cite[Theorem 3.1]{HoferSalamonNovikov},
there is a sequence of non-degenerate Hamiltonians $(\check{H}_i)_{i \in \N}$ $C^\infty$ converging to $H$ where
$\check{H}_i = (\check{H}_{i,t})_{t \in \T}$ for all $i \in \N$.
Let $\rho : M  \lra{} [0,1]$ be a smooth function
equal to $0$ inside $N$ and $1$ outside $N'$.
Define $H_{i,t} := (\phi^H_t)_*(\rho) \check{H}_{i,t} + (1-(\phi^H_t)_*(\rho)) H_t$ for all $t \in \T$
and define $H_i := (H_{i,t})_{t \in \T}$.
By a compactness argument, we have
for all $i$ sufficiently large that
any $1$-periodic orbit $\gamma$ of $H_i$ satisfying $\gamma(t) \in \phi^H_t(N')$ for some $t \in \T$ also satisfies $\gamma \in \Gamma$.
Therefore all $1$-periodic orbits $\gamma$ of $H_i$ not contained in $\Gamma$ 
satisfy $\gamma(t) \notin \phi^H_t(N')$ for each $t \in \T$
and hence
are non-degenerate orbits of $\check{H}_{i,t}$ for all $i$ sufficiently large.
Therefore such orbits are non-degenerate
orbits of $H_{i,t}$ for all $i$ sufficiently large.
\qed

\begin{lemma} \label{lemma regular Hamiltonians}
Let $\check{C}$ be a contact cylinder
and let $(a_-,a_+)$ be a
$\check{C}$-action interval.
Then $\ccH^{\T,\reg}(\check{C},a_-,a_+)$
is an open dense subset
of $\ccH^\T(\check{C},a_-,a_+)$
(See Definition \ref{defn chain complex}).
\end{lemma}
\proof
Openness is clear, so we just need to prove density.
Let $\ccH'(\check{C},a_-,a_+) \subset \ccH^\T(\check{C},a_-,a_+)$
be the subset consisting of Hamiltonians whose average slope along $\check{C}$ is not in the period spectrum of $\check{C}$.
This is a dense subset
since the period spectrum of both contact cylinders has measure $0$ in $\R$ (see \cite[Proposition 3.2]{popov1993length}).

Now there are two cases to consider.
The first case is when $(a_-,a_+)$ is small and the second case is when $(a_-,a_+)$ is not small.
In the first case,
$\ccH'(\check{C},a_-,a_+) \cap \ccH^{\T,\reg}(\check{C},a_-,a_+)$
is dense in
$\ccH'(\check{C},a_-,a_+)$
by Lemma
\ref{lemma generic Hamiltonian constant elsewhere}
applied to each $H \in \ccH'(\check{C},a_-,a_+)$
with $\Gamma = \emptyset$ and $N_\Gamma$, $N$
small neighborhoods of $[1+\epsilon/8,1+\epsilon/2] \times C$ containing no fixed points of $\phi^H_1$.
Hence our lemma is true if $(a_-,a_+)$ is small.

Now suppose $(a_-,a_+)$ is not small.
Let $D \subset M$ be the Liouville domain associated to $\check{C}$.
Then
${\ccH}'(\check{C},a_-,a_+) \cap {\ccH}^{\T,\reg}(\check{C},a_-,a_+)$
is dense in
${\ccH}'(\check{C},a_-,a_+)$
by applying Lemma
\ref{lemma generic Hamiltonian constant elsewhere}
to each $H \in  {\ccH}'(\check{C},a_-,a_+)$
with $\Gamma = M - (D \cup \check{C})$
 and $N$, $N_\Gamma$ a small neighborhood of $\Gamma \cup ([1+\epsilon/8,1+\epsilon/2] \times C)$ combined with the fact that the capped $1$-periodic orbits whose associated $1$-periodic orbit has image in $\Gamma$ are not contained in $\Gamma^\Z_{\check{C},a_-,a_+}(H)$ by Equation (\ref{equation high enough hamiltonians}) combined with Equation
 (\ref{equation contact cylinder action}).
\qed

\begin{defn} \label{definition action spectrum}
Let $H$ be a Hamiltonian.
The {\it action spectrum} of $H$ is the set:
$$\{\cA_{H,\emptyset}(\gamma)([\omega],1,1) \ : 
\gamma \ \text{is a capped 1-periodic orbit of} \ H  
\}$$
(see Example \ref{example usual action filtraiton}).
\end{defn}

\begin{lemma} \label{lemma action spectrum}

The action spectrum of any Hamiltonian is closed and has measure $0$.
\end{lemma}
\proof
We will first prove that the action spectrum of a Hamiltonian $H$ has measure $0$.
Let $\Delta \subset M \times M$
be the diagonal
and let
$N \subset M \times M$
be a neighborhood of $\Delta$
so that $(N,-\omega \oplus \omega)$
is symplectomorphic to a neighborhood
of $\Delta$ in $T^*\Delta$ where such a symplectomorphism is the identity along $\Delta$ (see \cite[Theorem 3.33]{McduffSalamon:sympbook}).
Let $\lambda$ be the canonical $1$-form on $T^* \Delta$
and let $\lambda_N$ be its restriction to $N$.
Let $\gamma$ be a capped $1$-periodic orbit of $H$ with associated $1$-periodic orbit $\widehat{\gamma}$
and let $U \subset M$
be a simply connected open neighborhood of
$\widehat{\gamma}(0)$
 so that the graph $\Gamma$ of $\phi^H_1|_U$
 in $M \times M$ is contained in $N$
 and so that $\Gamma$ is transverse to the fibers of $T^* \Delta$ (this can be done after perturbing the fibration $T^*\Delta$ by a small Hamiltonian flow fixing $\Delta$).
%
%
Then since $\phi^H_1$ is a symplectomorphism and $U$ is simply connected,
there is a unique smooth function $f : \Gamma \lra{} \R$
so that
\begin{itemize}
	\item $f(\widehat{\gamma}(0),\widehat{\gamma}(0)) = 0$,
	\item $df = \lambda_N|_{\Gamma}$.
\end{itemize}
Let $O_U$ be the set of capped $1$-periodic orbits of $H$ whose associated $1$-periodic orbit starts in $U$.
Then each orbit in $O_U$ starts at a critical point of $f$.
Let $C_U \subset \R$ be the set of critical values of $f$.
Then for any $\gamma_0 \in O_U$
we have that $\cA_{H,\emptyset}(\gamma_0)([\omega],1,1) = -f(\widehat{\gamma}_0(0),\widehat{\gamma}_0(0)) + \cA_{H,\emptyset}(\gamma)([\omega],1,1) + k$ for some $k \in \Z$ (this is because the de Rham cohomology class of $\omega$ admits an integral lift).
As a result, every capped $1$-periodic orbit of $H$ in $O_U$
has action contained in a translation of $\cup_{k \in \Z} (C_U + k)$ which has measure $0$ by Sard's theorem.
Since the set of fixed points of $\phi^H_1$ can be covered by a finite number of such subsets $U$, we get that the action spectrum of $H$ has measure $0$.

Now we wish to show that the action spectrum of $H$ is closed.
Let $a_i$ be a sequence of points in the action spectrum of $H$
which converge to $a \in \R$.
Since all $1$-periodic orbits of $H$ are contained in a compact set
we have
a $1$-periodic orbit
$\widehat{\gamma}_\infty$ of $H$
and a sequence of capped $1$-periodic orbits $(\gamma_i)_{i \in \N} = (\widetilde{\gamma}_i,\check{\gamma}_i)_{i \in \N}$
so that
$\cA_{H,\emptyset}(\gamma_i)([\omega],1,1) = a_i$
and so that
$\widetilde{\gamma}_i \circ \check{\gamma}$
$C^\infty$ converges to $\widehat{\gamma}_\infty$.
Since the actions of these capped orbits $(\gamma_i)_{i \in \N}$
are bounded and since their associated $1$-periodic orbits converge to $\widehat{\gamma}_\infty$,
there is a capped $1$-periodic orbit $\gamma_\infty$
and a sequence of $H_2(M;\Z)$ classes
$(v_i)_{i \in \N}$ satisfying $\omega(v_i) = 0$ for all $i \in \N$
so that
$(\gamma_i \# v_i)_{i \in \N}$ converges to $\gamma_\infty$ in the natural topology on capped loops as described in Definition \ref{defn capped 1-periodic orbit} and where $\#$ is given in Definition \ref{definition connect sub of orbit}.
Hence $(a_i)_{i \in \N}$
converges to $\cA_{H,\emptyset}(\gamma_\infty)([\omega],1,1) = a$.
Hence the action spectrum is closed.
\qed


\section{Appendix B: Avoiding Codimension $\geq 4$ Submanifolds.} \label{section avoiding codimsion geq 4 submanifolds}

In this section we will show that families of Floer trajectories intersect any countable collection of submanifolds transversely.
Also since we are working in the semi-positive
setting, we also need low dimensional families of such Floer trajectories to avoid holomorphic spheres (this is needed for compactness to prevent bubbling).
These are standard arguments whose main ideas are contained in \cite[Section 6.3 and 6.7]{McDuffSalamon:Jholomorphiccurves}
for instance.
We will mainly cite and use the Floer theoretic machinery developed in \cite{schwarz1995cohomology}.
Throughout this section, we will fix:
\begin{itemize}
\item  a (possibly empty) contact
cylinder $\check{C} = [1-\epsilon,1+\epsilon] \times C \subset M$ (Definition \ref{definition contact cylinder}),
\item a Riemann surface $\Sigma$
with $n_-$ negative cylindrical ends 
and $n_+$ positive cylindrical ends
labeled by finite sets $I_-$ and $I_+$ respectively (Definition \ref{defn Riemann Surface}),
\item a $\Sigma$-compatible $1$-form $\beta \in \Omega^1(\Sigma)$ with weights
$(\kappa_j)_{j \in I_- \sqcup I_+}$ at the corresponding cylindrical ends
(Definition \ref{definition riemann surface admissible}),
\item a $(\Sigma,\check{C})$-compatible family of Hamiltonians $H := (H_\sigma)_{\sigma \in \Sigma}$
with limits $H^\# := (H^j)_{j \in I_- \sqcup I_+}$
as in Definition \ref{definition riemann surface admissible},
\item a tuple $J^\# := (J^j)_{j \in I_- \sqcup I_+}$ of elements of $\ccJ^\T(\check{C})$ (Definition \ref{definition wspaces of compatible objects}) and
\item $\Gamma(H)$ the set of tuples
$(\gamma^j)_{j \in I_- \sqcup I_+}$
of non-degenerate capped $1$-periodic orbits
of $(\kappa_j H^j)_{j \in I_- \sqcup I_+}$ respectively
whose associated $1$-periodic orbits are disjoint from $V$
so that if
$\Sigma \neq \R \times \T$ then at least two such $1$-periodic orbits have distinct images (Definition \ref{defn capped 1-periodic orbit}).
\end{itemize}

\begin{defn} \label{defn compactification of Riemann surface}
(\cite[Definition 2.1.2]{schwarz1995cohomology})
Let
$\overline{\Sigma}$
be a
manifold with boundary obtained
by extending each cylindrical end
$\iota_i : \bbI_- \times \T \hookrightarrow \Sigma$
(resp.
$\iota_i : \bbI_+ \times \T \hookrightarrow \Sigma$)
of $\Sigma$
to $\overline{\iota}_i : [-\infty,0] \times \T \hookrightarrow \overline{\Sigma}$,
$i \in I_-$
(resp. $\overline{\iota}_i : [0,\infty] \times \T \lra{} \Sigma$, $i \in I_+$).
We get a smooth structure on this manifold
by extending the smooth charts
$$\Im(\iota_i) \lra{} (-1,0] \times \T, \quad (s,t) \lra{} \left(\frac{s}{\sqrt{1+s^2}},t\right), \ i \in I_-$$
$$\Im(\iota_i) \lra{} [0,1) \times \T, \quad (s,t) \lra{} \left(\frac{s}{\sqrt{1+s^2}},t\right), \ i \in I_+$$
on $\Sigma$
to charts
$$\Im(\iota_i) \lra{} [-1,0] \times \T, \quad
\Im(\iota_i) \lra{} [0,1] \times \T$$
respectively on $\overline{\Sigma}$.
\end{defn}

\begin{defn} \label{defn sobolov spaces of sections}
(\cite[Definition 2.1.5]{schwarz1995cohomology})
Let $k = 0$ or $1$.
Let $\pi : E \lra{} \overline{\Sigma}$
be a $C^k$ vector bundle and let
$\phi_i : \iota_i^* E \lra{} (\bbI_\pm \times \T) \times \R^k$
be a $C^k$ trivialization
which extends to a $C^k$ trivialization of $\overline{\iota}_i^* E$
 and let $\Pi_i : (\bbI_\pm \times \T) \times \R^k \twoheadrightarrow \R^k$ be the natural projection map for each $i \in I_\pm$.
Define
$$W^{k,p}_\Sigma(E) := \left\{\sigma \in W^{k,p}_{\text{loc}}(E|_\Sigma) \ : \
\Pi_i \circ \phi_i \circ \sigma \circ \iota_i \in W^{k,p}(\bbI_\pm \times \T,\R^k), \   \forall \ i \in I_\pm 
\right\}.$$
We define
$L^p_\Sigma(E) := W^{0,p}_\Sigma(E).$
If $D \subset E$ is a subset of $E$
then we define
$W^{k,p}_\Sigma(D) := \{ \sigma \in W^{k,p}_\Sigma(E) \ : \ \Im(\sigma) \subset D \}$
and
$L^p_\Sigma(D) := W^{0,p}_\Sigma(D)$.
\end{defn}

These are Banach spaces which do not depend on the choice of trivializations $\phi_i, \ i \in I_- \sqcup I_+$
by 
the paragraph after \cite[Definition 2.1.5]{schwarz1995cohomology}.

\begin{defn} \label{defn Banach manifold of maps}
(\cite[Definition 2.1.6]{schwarz1995cohomology}).
Let $\gamma := (\gamma_i)_{i \in I_- \sqcup I_+} = ((\widetilde{\gamma}_i,\check{\gamma}_i))_{i \in I_- \sqcup I_+}$
be finite collection of capped loops as in Definition \ref{defn capped 1-periodic orbit}.
A smooth
map $u : \overline{\Sigma} \lra{} M$
{\it converges to $(\gamma_i)_{i \in I_- \sqcup I_+}$} if
$\iota_i(\pm \infty, t) = \widetilde{\gamma}_i(\check{\gamma_i}(t))$ for all $i \in I_\pm$, $t \in \T$
and the surface obtained by gluing
the surfaces $\widetilde{\gamma}_i$, $i \in I_- \sqcup I_+$ to
$u$ is null-homologous.
Let
$C^\infty_\gamma(\overline{\Sigma},M)$ be the space of such maps equipped with the $C^\infty$ topology.

Fix a Riemannian metric on $M$
and let
$D \subset TM$ be an open neighborhood of $M$ so that
$\exp|_{D \cap T_x M}$ is a diffeomorphism on to its image for all $x \in M$ where $\exp$ is the exponential map with respect to this metric.
Define
\begin{equation} \label{equation banach space of maps}
{\mathcal P}^{1,p}_\gamma(\Sigma,M)
:= \{\exp \circ v \in C^0(\overline{\Sigma},M) \ : \ v \in W^{1,p}_\Sigma(h^* D), \ h \in C^\infty_\gamma(\overline{\Sigma},M) \}
\end{equation}
for all $p > 2$.
\end{defn}

The space (\ref{equation banach space of maps}) is a Banach manifold by \cite[Theorem 2.1.7]{schwarz1995cohomology}
with charts mapping to $W^{1,p}_\Sigma(h^* D)$
for all $h \in C_\gamma^\infty(\overline{\Sigma},M)$.

\begin{defn}
Let $\pi : E \lra{} B$ be a Banach vector bundle.
Let $s : B \lra{} E$ be a $C^1$ section
and let $b \in s^{-1}(0)$,
We say that $s$
is {\it transverse to zero at $b$}
if
the linear map
\begin{equation} \label{equation linearization of section}
T_bB \lra{s} T_{s(b)}E \lra{\textnormal{pr}} \ker{D\pi}|_{s(b)}
\end{equation}
is surjective
where 
$$\text{pr} : TE|_B = TB \oplus (\ker{D\pi}|_B) \lra{} \ker{D\pi}|_B$$ is the natural projection map.
We say that $s$ is {\it transverse to $0$} if it is transverse to $0$
at every point $b \in s^{-1}(0)$.
Such a section is {\it Fredholm}
if the map (\ref{equation linearization of section}) is Fredholm.
\end{defn}

\begin{defn}
For each $J = (J_\sigma)_{\sigma \in \Sigma} \in \ccJ^\Sigma(J^\#,\check{C})$,
we define $M^J$ to
be the vector bundle over
$\overline{\Sigma} \times M$
whose fiber at $(\sigma,x)$
is the space of
$J|_{(\sigma,x)}$ anti-linear
maps from $T_\sigma \overline{\Sigma}$ to
$T_xM$.
\end{defn}

\begin{prop} \label{proposistion bundle and Fredholm section}
\cite[Theorem 2.2.5, Proposition 2.3.1, Theorem 3.1.31]{schwarz1995cohomology}.
Let
$\gamma \in \Gamma(H)$ and $J = (J_\sigma)_{\sigma \in \Sigma} \in \ccJ^\Sigma(J^\#,\check{C})$.
Then the set
\begin{equation} \label{equation banach bundle over B times P at J}    
{\mathcal E}_{J,\gamma} := \sqcup_{u \in {\mathcal P}^{1,p}_\gamma(\Sigma,M)}
\{u\} \times L^p_\Sigma((\id_{\overline{\Sigma}},u)^* M^J)
\end{equation}
is naturally a smooth Banach bundle over
${\mathcal P}^{1,p}_\gamma(\Sigma,M)$ for all $p > 2$.
Also if ${\bf j}$ is the natural complex structure on $\Sigma$ then the section
\begin{equation} \label{equation d bar from gamma to epsilon J}
\overline{\partial}_{J,\gamma} : {\mathcal P}^{1,p}_\gamma(\Sigma,M) \lra{} {\mathcal E}_{J,\gamma},
\end{equation}
$$
\overline{\partial}_{J,\gamma}(u)(\sigma) := 
(du + X_{H_\sigma} \otimes \beta) + J_\sigma \circ (du + X_{H_\sigma} \otimes \beta) \circ {\bf j}, \quad \forall \ \sigma \in \Sigma
$$
is well defined (I.e. $\overline{\partial}_{J,\gamma}(u)$ is an element of 
$L^p_\Sigma((\id_{\overline{\Sigma}},u)^* M^J)$ for all $u \in {\mathcal P}$) and $\overline{\partial}_{J,\gamma}$ is a smooth Fredholm section
of ${\mathcal E}_{J,\gamma}$.
\end{prop}

\begin{defn} \label{definition regular almost complex structure}
An element $J \in \ccJ^\Sigma(J^\#,\check{C})$ is {\it $H$-regular} if the Fredholm section $\overline{\partial}_{J,\gamma}$ is transverse to zero for each $\gamma \in \Gamma(H)$.
\end{defn}

Note that if the Hamiltonian $H$ only has degenerate one periodic orbits then every element $J \in \ccJ^\Sigma(J^\#,\check{C})$
is $H$-regular by the definition above since $\Gamma(H)$ is empty.
We only care about Floer trajectories connecting non-degenerate $1$-periodic orbits.

\begin{theorem} \cite[Theorem 4.2.2 and 3.3.11]{schwarz1995cohomology}.
If $J \in \ccJ^\Sigma(J^\#,\check{C})$
is $H$-regular then
$\ccM(H,J,\gamma) = \overline{\partial}_{J,\gamma}^{-1}(0)$
and this set is a smooth finite dimensional submanifold of
${\mathcal P}^{1,p}_\gamma(\Sigma,M)$
of dimension
$k$ where $k$ is defined in Equation (\ref{equation conley zehnder index sum})
for each $\gamma \in \Gamma(H)$.
\end{theorem}

\begin{defn}
Let $f_j : W_j \lra{} W$, $j=0,1$
be two smooth maps.
We say that $f_0$ is {\it transverse}
to $f_1$
if, for each $w_0 \in W_0$, $w_1 \in W_1$ satisfying $f_0(w_0) = f_1(w_1)$,
we have that the span of the subspaces
$Df_0(T_{w_0} W_0)$,
$Df_1(T_{w_1} W_1)$
is $T_{f_0(w_0)} W$.
\end{defn}

\begin{defn} \label{definition H V regular}
For each $\gamma \in \Gamma(H)$,
we say that $J = (J_\sigma)_{\sigma \in \Sigma} \in \ccJ^\Sigma(J^\#,\check{C})$
is {\it $(H,V,\gamma)$-regular}
if it is $H$-regular and
if there is a countable collection of smooth maps $f_i : W_i \lra{} \Sigma \times M$, $i \in \N$ where $\dim_\R(W_i) \leq 2n-2$ for each $i$
so that
\begin{enumerate}
\item $V \subset \cup_{i \in \N} f_i(W_i)$,
\item every $J_\sigma$-holomorphic map
$u : \P^1 \lra{} M$ has image contained in
$(\{\sigma\} \times M) \cap (\cup_i f_i(W_i))$
after identifying $\{\sigma\} \times M$ with $M$ in the natural way for each $\sigma \in \Sigma$ and
\item the evaluation map
$$\ev : \Sigma \times \ccM(H,J,\gamma) \lra{} \Sigma \times M, \quad (\sigma,u) \lra{} (\sigma,u(\sigma))$$
is smooth and transverse to $f_i$ for each $i \in \N$.
\end{enumerate}
We say that $J$ is {\it $(H,V)$-regular}
if it is $(H,V,\gamma)$-regular for each $\gamma \in \Gamma(H)$.
We define $\ccJ^{\Sigma,\reg}(H,J^\#,\check{C}) \subset \ccJ(J^\#,\check{C})$
to be the subspace of $(H,V)$-regular families of almost complex structures.
\end{defn}

We wish to show that $\ccJ^{\Sigma,\reg}(H,J^\#,\check{C})$
is ubiquitous in $\ccJ(J^\#,\check{C})$
(Proposition \ref{proposition regular subset for surface}).
In order to do this we
need to define an appropriate
space of almost complex structures.
\begin{defn} \label{Definition almost copmlex structures}
	\cite[Definitions 4.2.6, 4.2.10 and 4.2.11]{schwarz1995cohomology}.
%
Let $J \in \ccJ^\Sigma(J^\#,\check{C})$.
Let $U \subset \Sigma$ be a relatively compact open subset, define
$\widetilde{U} := U \times (M-([1+\epsilon/8,1+\epsilon/2] \times C))$
and
let $\pi : \Sigma \times M \lra{} M$ be the natural projection map.
Let
$$S_J := \{A \in \End(\pi^* TM) \ : \ AJ + JA = 0\}
$$
be a bundle over $\Sigma \times M$
and let
$C^\infty_U(S_J)$ be the space of $C^\infty$ sections $A$ of $S_J$
so that 
all the derivatives of $A$ vanish
along $\Sigma \times V$ and $\Sigma \times M - \widetilde{U}$.
Let $\|\cdot\|$ be a metric on $M$ given by $\frac{1}{2}(\omega(\cdot,J(\cdot)) + \omega(J(\cdot),\cdot)$.
Let $\epsilon := (\epsilon_i)_{i \in \N}$
be a sequence of rapidly decreasing positive constants
and define
$$C^\infty_{\epsilon,U}(S_J) :=
\{
A \in C^\infty_U(S_J) \ : \  \|A\|_\epsilon < \infty
\}
$$
where $\|A\|_\epsilon := \sum_{i \in \N} \epsilon_i \|\nabla^k A\|$
coming from a product metric on $\Sigma \times M$.
Define
$$\Phi_J : C^\infty_{\epsilon,U}(S_J) \lra{} \ccJ^\Sigma(J^\#,\check{C}), \quad \Phi_J(A) := J e^A$$
and define
$$\ccJ_\epsilon(J|_U) := \{\Phi_J(A) \ : \ A \in C^\infty_{\epsilon,U}(S_J), \ \|A\| < \epsilon_0 \} \subset \ccJ^\Sigma(J^\#,\check{C}).$$
\end{defn}

By the ideas in
\cite[4.2.7, 4.2.9, 4.2.10]{schwarz1995cohomology}
we have that
$C^\infty_{\epsilon,U}(S_J)$ is a Banach space with Banach norm $\|\cdot\|_\epsilon$ and $\Phi_J^{-1}$ embeds
$\ccJ_\epsilon(J|_U)$ as an open subset of this Banach space making it into a Banach manifold for $\epsilon$ small enough.
There are a few minor differences between
Definition \ref{Definition almost copmlex structures} and \cite[Definition 4.2.11]{schwarz1995cohomology}.
\begin{enumerate}
\item Our almost complex structures
are not necessarily compatible with $\omega$,
but they do tame $\omega$,
\item Our almost complex structures
and all of their derivatives agree with those of $J$ along some regions of $\Sigma \times M$ and
\item the formula for the map $\Phi_J$ is different.
\end{enumerate}
These differences play no important role
in the proof of the fact that
$\ccJ_\epsilon(J|_U)$
is a Banach manifold.
Also it
makes no difference in the proofs of Propositions
\ref{proposition bundle for d bar equation},
\ref{proposition transverse to zero} and
\ref{proposition regular subset of J}
which are just modified versions of
\cite[Proposition 4.2.4]{schwarz1995cohomology},
\cite[4.2.18]{schwarz1995cohomology}
and
\cite[Proposition 4.2.5]{schwarz1995cohomology}
respectively.
Now one of the issues with $\ccJ_\epsilon(J|_U)$ is that it is not a topological subspace of $\ccJ^\Sigma(J^\#,\check{C})$.
However we wish to
prove theorems with respect to the topology of
$\ccJ^\Sigma(J^\#,\check{C})$.
The following lemma addresses this issue.
If $\epsilon = (\epsilon_i)_{i \in \N}$
and $\epsilon' = (\epsilon'_i)_{i \in \N}$ are sequences of positive constants, we write $\epsilon \leq \epsilon'$ if $\epsilon_i \leq \epsilon'_i$ for all $i \in \N$.
\begin{lemma} \label{lemma ubiquitous from Banach ubiquitous}
Let $\epsilon^{U,J}$ be a sequence of rapidly decreasing constants for each $J \in \ccJ^\Sigma(J^\#,\check{C})$ and each relatively compact subset $U \subset \Sigma$.
Let $W_{U,J,\epsilon} \subset \ccJ_\epsilon(J|_U)$
be a ubiquitous subset for each $J \in J^\Sigma(J^\#,\check{C})$, relatively compact open subset $U \subset \Sigma$ and each $\epsilon \leq \epsilon^{U,J}$.
Then $\cup_{U,J,\epsilon} W_{U,J,\epsilon}$ is ubiquitous in $J^\Sigma(J^\#,\check{C})$.
\end{lemma}
\proof
Choose a sequence $U_i \subset \Sigma$, $i \in \N$ of relatively compact open subsets whose union is $\Sigma$.
Choose a sequence of constants
$\epsilon^{J,i} \leq \epsilon^{J,U_i}$, $i \in \N$
and an open neighborhood
$W_J \subset \ccJ^\Sigma(J^\#,\check{C})$ of $J$
so that
$V_J := \cup_{i \in \N} \ccJ_{\epsilon^{J,i}}(J|_{U_i})$
contains $W_J$
for each $J \in \ccJ^\Sigma(J^\#,\check{C})$.
Choose an open neighborhood $\check{W}_J \subset \ccJ^\Sigma(J^\#,\check{C})$
of $J$ whose closure is contained in $W_J$ for each
$J \in \ccJ^\Sigma(J^\#,\check{C})$.
Separability of $\ccJ^\Sigma(J^\#,\check{C})$ implies
that there exists a sequence
$(J_i)_{i \in \N}$ of elements
in $\ccJ^\Sigma(J^\#,\check{C})$
so that
$\ccJ^\Sigma(J^\#,\check{C}) = \cup_{i \in \N} \check{W}_{J_i}$.
Define 
$$W^{i,j} := W_{U_j,J_i,\epsilon^{J_i,j}}
\cup (\ccJ^\Sigma(J^\#,\check{C}) - \overline{W_J}) \subset \ccJ^\Sigma(J^\#,\check{C}).$$
Then $W^{i,j}$ is ubiquitous in
$\ccJ^\Sigma(J^\#,\check{C})$
and
$W^{i,j} \cap \ccJ_{\epsilon^{J_i,j}}(J_i|_{U_j})
\cap \check{W}_{J_i} = W_{U_j,J_i,\epsilon^{J_i,j}} \cap \check{W}_{J_i}$
for each $i,j \in \N$.
Hence
$\cap_{i,j \in \N} W^{i,j}$
is ubiquitous in $\ccJ^\Sigma(J^\#,\check{C})$ and is contained in
$\cup_{U,J,\epsilon} W_{U,J,\epsilon}$ and therefore
$\cup_{U,J,\epsilon} W_{U,J,\epsilon}$
is ubiquitous.
\qed

\bigskip

By combining Lemma \ref{lemma ubiquitous from Banach ubiquitous}
above with
\cite[Proposition 4.2.5]{schwarz1995cohomology}
 (see also Proposition \ref{proposition transverse to zero} below),
we get the following proposition:

\begin{prop} \label{proposition ubiquitous regular}
The subspace of $H$-regular almost complex structures in
$\ccJ^\Sigma(J^\#,\check{C})$
as in Definition \ref{definition regular almost complex structure}
is ubiquitous.
\end{prop}

We now wish to prove the same thing for $(H,V)$-regular almost complex structures.
In order to do this we need some more propositions and lemmas.
Note that every element
$J = (J_\sigma)_{\sigma \in \Sigma}$
in
$\ccJ^\Sigma(J^\#,\check{C})$ extends to a smooth family of almost complex structures
$(J_\sigma)_{\sigma \in \overline{\Sigma}}$
since $J$ is translation invariant in the cylindrical ends at infinity.
Hence, from now on we will define
$J_\sigma$ to be the limit as $\sigma' \in \Sigma$ tends to $\sigma$ of $J_{\sigma'}$ for each $\sigma \in \partial \overline{\Sigma}$.
We will also use such conventions for
other families of objects over $\Sigma$
such as $\beta$ and $H$.

\begin{prop} \label{proposition bundle for d bar equation}
	\cite[Proposition 4.2.4 and Theorem 3.1.31]{schwarz1995cohomology}.
Let $U \subset \Sigma$ be a relatively compact open set, $\epsilon$ a sequence of rapidly decreasing constants, $J \in \ccJ^\Sigma(J^\#,\check{C})$ and
$\gamma \in \Gamma(H)$.
Let ${\mathcal B} := \ccJ_\epsilon(J|_U)$
and ${\mathcal P} :=
{\mathcal P}^{1,p}_\gamma(\Sigma,M)$, $p > 2$
where
${\mathcal P}^{1,p}_\gamma(\Sigma,M)$
is defined as in Equation
(\ref{equation banach space of maps}).
Then the set
\begin{equation} \label{equation banach bundle over B times P}    
{\mathcal E} := \sqcup_{(J',u) \in {\mathcal B} \times {\mathcal P}}
\{(J',u)\} \times L^p_\Sigma((\id_{\overline{\Sigma}},u)^* M^{J'})
\end{equation}
is naturally a smooth Banach bundle
over
${\mathcal B} \times {\mathcal P}$
whose fiber over
$(J',u)$ is $L^p_\Sigma((\id_{\overline{\Sigma}},u)^* M^{J'})$.
%
Also
$$\overline{\partial}_{{\mathcal B} \times {\mathcal P}} : {\mathcal B} \times {\mathcal P} \lra{} {\mathcal E}, \quad \overline{\partial}_{{\mathcal B} \times {\mathcal P}}(J',u) := \overline{\partial}_{J',\gamma}(u)$$
is a smooth section
of ${\mathcal E}$
where $\overline{\partial}_{J,\gamma}$ is defined in Equation (\ref{equation d bar from gamma to epsilon J}).
\end{prop}

From now on, until the proof of Proposition \ref{proposition regular subset for surface} below,
we will fix $U$, $\epsilon$, $J$,
$\gamma$, ${\mathcal B}$, ${\mathcal P}$, ${\mathcal E}$ and $\overline{\partial}_{{\mathcal B} \times {\mathcal E}}$
from the proposition above.
\begin{defn}
Define
${\mathcal S} \subset {\mathcal B} \times {\mathcal P}$ to be
the open subset of pairs
$(J',u) \in {\mathcal B} \times {\mathcal P}$
satisfying
\begin{equation} \label{equation U large enough}
U \nsubseteq u^{-1}(V \cup ([1+\epsilon/8,1+\epsilon/2] \times C)) \cup (\Sigma - \text{supp}(du - \beta \otimes X_H))
\end{equation}
where $\text{supp}(du - \beta \otimes X_H)$
is the support of the distribution
$du - X_H \otimes \beta \in L^p_{\text{loc}}(\Sigma \otimes u^* TM)$
where
$$X_H \otimes \beta|_\sigma  :=  X_{H_\sigma}|_{u(\sigma)} \otimes  \beta|_\sigma \ \forall \ \sigma  \in \Sigma.$$
We also define
 ${\mathcal D} := \overline{\partial}_{{\mathcal B} \times {\mathcal P}}^{-1}(0) \cap {\mathcal S}$ and let
 \begin{equation} \label{equation projection map to almost complex structures}
 \Pi_{\mathcal B} : {\mathcal D} \lra{} {\mathcal B}.
 \end{equation}
 be the natural projection map.
\end{defn}

\begin{prop} \label{proposition transverse to zero}
\cite[Proposition 4.2.18 and its proof]{schwarz1995cohomology}.
Suppose that $(J',u) \in {\mathcal D}$.
Then the section
$\overline{\partial}_{{\mathcal B} \times {\mathcal P}}$
above is transverse to $0$ at $(J',u)$.
Also the natural linear map
\begin{equation} \label{equation linearization restricted to complex space}
D_{J',u}|_{T_{J'}{\mathcal B}} :  T_{J'}{\mathcal B} = T_{J'}{\mathcal B} \times 0 \lra{D\overline{\partial}_{{\mathcal B} \times {\mathcal P}}} T_{(J',u)} {\mathcal E} \lra{\textnormal{pr}} L^p_\Sigma((\id_{\overline{\Sigma}},u)^* M^{J'})
\end{equation}
linearizing $\overline{\partial}_{{\mathcal B} \times {\mathcal P}}$ at $(J',u)$
has dense image.
\end{prop}

Hence by combining the proposition above with the last part of Proposition \ref{proposistion bundle and Fredholm section} and the implicit function theorem, we get the following corollary.

\begin{corollary} \label{corollary banach manifold}
${\mathcal D}$
is a Banach submanifold
of ${\mathcal S}$.
\end{corollary}

\begin{prop} \label{proposition regular subset of J}
	\cite[Theorem 3.3.11, Theorem 4.2.2, Proposition 4.2.5 together with its proof]{schwarz1995cohomology}.
	The map $\Pi_{\mathcal B}$ in Equation (\ref{equation projection map to almost complex structures}) is Fredholm.
	The subset
	$(\Pi_{\mathcal B})^\reg \subset {\mathcal B}$ of regular values of
	$\Pi_{\mathcal B}$
	is ubiquitous as in Definition \ref{defnition ubiquitous} and
	$\Pi_{\ccJ}^{-1}(J')$
	is a smooth manifold of dimension $k$ where $k$ is defined in Equation (\ref{equation conley zehnder index sum})
	for all $J' \in \Pi_{\mathcal B}$.	
\end{prop}

By the Sobolev embedding theorem
we can think of the tangent space
$T_u {\mathcal P}$
at a point $u \in {\mathcal P}$
naturally as a subspace of
$C^0(u^* TM)$.
Let $\iota_{\mathcal P} : T_u {\mathcal P} \lra{} C^0(u^* TM)$
be the natural inclusion map.
Then we have the following definition:
\begin{defn} \label{defninition tangent space for P}
For each $u \in {\mathcal P}$ and $\sigma \in \Sigma$,
define the Banach subspace
$T_{u,\sigma} {\mathcal P} \subset T_u {\mathcal P}$
to be the subspace consisting of elements $v \in T_u {\mathcal P}$
satisfying $\iota_{\mathcal P}(v) (\sigma) = 0$.
\end{defn}

\begin{lemma} \label{lemma surjectivity of smaller space}
Let $\sigma \in \Sigma$
and let
$(J',u) \in {\mathcal D}$.
%
Then the map
$$D_{J',u,\sigma} : T_{J'} {\mathcal B} \times T_{u,\sigma} {\mathcal P} 
\lra{D\overline{\partial}_{{\mathcal B} \times {\mathcal P}}} T_{(J',u)} {\mathcal E} \lra{\textnormal{pr}} L^p_\Sigma((\id_{\overline{\Sigma}},u)^* M^{J'})
$$
is surjective.
\end{lemma}
\proof
The map
$$D_{J',u}|_{0 \times T_u{\mathcal P}} : 
{0 \times T_u{\mathcal P}}
\lra{D\overline{\partial}_{{\mathcal B} \times {\mathcal P}}} T_{(J',u)} {\mathcal E}  \lra{\text{pr}} L^p_\Sigma((\id_{\overline{\Sigma}},u)^* M^{J'})
$$
is Fredholm by
Proposition \ref{proposition bundle for d bar equation}
and hence the map
$D_{J',u,\sigma}|_{0 \times T_{u,\sigma} {\mathcal P}}$
is Fredholm since $T_{u,\sigma}{\mathcal P} \subset T_u {\mathcal P}$ is a subspace of finite codimension.
This implies that the image
of $D_{J',u,\sigma}$ is closed.
Such an image is also dense by the last part of
Proposition
\ref{proposition transverse to zero}
and hence is surjective.

\qed

\begin{lemma} \label{lemma smoothness of evaluation map}
	
	The natural map
	$$E : \Sigma \times {\mathcal D} \lra{} M, \quad E(\sigma,(J',u)) := u(\sigma)$$
	is $C^\infty$.
\end{lemma}
Note that if we extend  this map
to ${\mathcal B} \times {\mathcal P}$
in the natural way
then such a map is not even $C^1$.

\proof

By \cite[Proposition 2.5.7]{schwarz1995cohomology},
we have that
$u \in C^\infty(\Sigma,M)$
for all $(J',u) \in {\mathcal D}$.
By \cite[Proposition B.4.9]{McDuffSalamon:Jholomorphiccurves} combined with the Sobolov embedding theorem \cite[Proposition B.1.11]{McDuffSalamon:Jholomorphiccurves},
we have that for each
$(J',u) \in {\mathcal D}$,
and each compact codimension
$0$ submanifold $K \subset \Sigma$,
the natural map
from $T_{(J',u)} {\mathcal D}$
to $T_u C^r(K,M)$
is well defined
continuous map between
Banach spaces.
%
%
Since charts on mapping spaces are constructed using the exponential map of a metric on $M$,
this implies that the natural map
from ${\mathcal D}$ to
$C^r(K,M)$
is smooth
for each compact codimension $0$ submanifold $K \subset \Sigma$ and each $r \geq 0$.
%
%
%
%
Therefore,  our lemma follows
from the fact that the evaluation map
$$K \times C^r(K,M) \lra{} M$$
is $C^r$ for all compact codimension
$0$ submanifolds $K \subset \Sigma$ and all $r \geq 0$ (\cite[Page 78]{krikorian1972differentiable}).
\qed

\begin{lemma} \label{lemma submersion of evaluation at sigma}
Let $\sigma \in \Sigma$.
Then the map
$$E|_{\{\sigma\} \times {\mathcal D} } : {\mathcal D} \lra{} M, \quad E|_{\{\sigma\} \times {\mathcal D} }(J',u) = u(\sigma)$$
is a submersion.
\end{lemma}
\proof
Let 
$(J',u) \in {\mathcal D}$ and let
$W \in T_{u(\sigma)}M$.
Choose $w \in T_u {\mathcal P}$ so that
$w(\sigma) = W$.
Since $D_{J',u,\sigma}$
from Lemma \ref{lemma surjectivity of smaller space}
is surjective,
there exists $(Y_1,w_1) \in T_{J'}{\mathcal B} \times T_{u,\sigma} {\mathcal P}$
so that
$D_{J',u}(Y_1,w_1) = D_{J',u}(0,w)$
where $T_{u,\sigma} {\mathcal P}$
is defined in Definition
\ref{defninition tangent space for P} and where $D_{J',u}$ is the composition
$$D_{J',u} : 
{T_{J'} {\mathcal B} \times T_u{\mathcal P}}
\lra{D\overline{\partial}_{{\mathcal B} \times {\mathcal P}}} T_{(J',u)} {\mathcal E}  \lra{\text{pr}} L^p_\Sigma((\id_{\overline{\Sigma}},u)^* M^{J'}).
$$
Therefore
$(- Y_1,w - w_1) \in T_{J',u}{\mathcal D}$
and
$(w - w_1)(\sigma) = W$ and hence the map
$E|_{\{\sigma\} \times {\mathcal D} }$ is a submersion.
\qed

\begin{lemma} \label{lemma surjectivity of evaluation map}
The map
\begin{equation} \label{equation ev}
\ev : \Sigma \times {\mathcal D} \lra{} \Sigma \times M, \quad F(\sigma,(J',u)) := (\sigma,u(\sigma))
\end{equation}
is smooth and a submersion.
\end{lemma}
\proof
This follows directly from
Lemmas
\ref{lemma smoothness of evaluation map} and
 \ref{lemma submersion of evaluation at sigma}.
\qed

\begin{lemma} \label{lemma transverse ubiquitous}
Let $f := (f_i)_{i \in \N}$ be a countable collection of smooth maps
$f_i : Q_i \lra{} \Sigma \times M$,
$i \in \N$.
Let
$\ccJ_\epsilon(J|_U,f) \subset {\mathcal B}$ be the subset consisting of elements $J'$ 
that are $H$-regular
with the property that
$\ev|_{\Pi_{\mathcal B}^{-1}(J')}$
is transverse to $f_i$ for each $i \in \N$
where $\ev$ and $\Pi_{\mathcal B}$ are defined in Equations
(\ref{equation ev}) and (\ref{equation projection map to almost complex structures}) respectively.
Then $\ccJ_\epsilon(J|_U,f)$ is ubiquitous in ${\mathcal B}$.
\end{lemma}
\proof
Let $\widetilde{\Pi}_i := \ev^* f_i : \widetilde{Q}_i \lra{} \Sigma \times {\mathcal D}$
be the pullback of $f_i$ for each $i \in \N$ (this exists by Lemma \ref{lemma surjectivity of evaluation map}).
%
Since $\Sigma \times M$ and $Q_i$ is finite dimensional, we have that $\widetilde{\Pi}_i$ is a Fredholm map. Hence the composition
$P \circ \widetilde{\Pi}_i$
is Fredholm by Proposition \ref{proposition regular subset of J}
where
$P : \Sigma \times {\mathcal D} \lra{} {\mathcal B}$ is the natural projection map.
Hence the set of regular values
${\mathcal R}_i$ of $P \circ \widetilde{\Pi}_i$ is ubiquitous in ${\mathcal B}$.
The subset of $H$-regular almost complex structures $\ccJ^\reg_H$
is ubiquitous by Lemma \ref{proposition ubiquitous regular}.
Our lemma now follows from
the fact that
$\ccJ_\epsilon(J|_U,f)$
contains the ubiquitous set $\ccJ^\reg_H \cap \cap_{i \in \N} {\mathcal R}_i$.
\qed

\begin{proof}[Proof of Proposition \ref{proposition regular subset for surface}.]
Let $S_i, U_i \subset \Sigma$, $i \in \N$ be non-empty relatively compact open subsets of $\Sigma$
so that $\cup_i S_i = \Sigma$, $S_i \subset S_{i+1}$ and 
$\overline{U_i} \cap \overline{S_i} = \emptyset$ 
for all $i \in \N$.
For each $J  = (J_\sigma)_{\sigma \in \Sigma} \in \ccJ^\Sigma(J^\#,\check{C})$,
let
$\ccM^{\textnormal{vert}}(J)$
be the set of somewhere injective maps
$v : \P^1 \lra{} \Sigma \times M$
where
\begin{enumerate}
\item the image of $v$ is contained in $\{\sigma_v\} \times M$ for some $\sigma_v \in \Sigma$ and not contained in $\Sigma \times V$ and
\item $v$ is $J_{\sigma_v}$-holomorphic after identifying $\{\sigma_v\} \times M$ with $M$ in the natural way.
\end{enumerate}
Then by the methods in
\cite[Section 3.2]{McDuffSalamon:Jholomorphiccurves}
we have that there is a ubiquitous subset
$\ccJ^{\reg}_{\P^1} \subset \ccJ^\Sigma(J^\#,\check{C})$
so that $\ccM^{\textnormal{vert}}(J)$ is a manifold whose connected components are of dimension
at most $2n-2$ where $2n$ is the dimension of $M$ and so that the
evaluation map
$$\ev_{\P^1} :  \ccM^{\textnormal{vert}}(J) \lra{} \Sigma \times M, \quad \ev_{\P^1}(u) := u(0)$$
is smooth.
Then
for each $J \in \ccJ^\reg_{\P^1}$,
there exists a countable collection of smooth maps $f_J^i : W_J^i \lra{} \Sigma \times M$, $i \in \N$ where $\dim(W_J^i) \leq 2n-2$
so that
for each $u \in \ccM^{\textnormal{vert}}(J)$, we have that
$\textnormal{image}(u) \subset \cup_{i \in \N} f_J^i(W^i_J)$.
We can assume that the functions $f^i_J$, $i \in \N$, $J \in \ccJ^\reg_{\P^1}$
have the property that if there exists $i \in \N$ so that $J = (J_\sigma)_{\sigma \in \Sigma}, J' = (J'_{\sigma})_{\sigma \in \Sigma} \in \ccJ^\reg_{\P^1}$ satisfies $J_\sigma = J'_\sigma$ for all $\sigma \in S_i$
then $E_{i,k,J} := (f^k_J)^{-1}(S_i \times M)$
is diffeomorphic to $(f^k_{J'})^{-1}(S_i \times M)$
and $f^k_J|_{E_{i,k,J}} = f^k_{J'}|_{E_{i,k,J}}$
under this identification
for all $k \in \N$.
Since any manifold is a countable
union of compact codimension $0$ submanifolds with boundary,
we can assume that $W^i_J$ is a compact manifold with boundary for each $i \in \N$ as well.

We now wish to write the moduli spaces
$\ccM(H,J,\gamma)$ as a union of compact sets for each $H,J,\gamma$ in a consistent way.
We will use Gromov compactness ideas to do this.
For each $j \in I_\pm$ and
each non-degenerate fixed point
$p$ of $\phi^{\kappa_j H^j}_1$,
let $N_{j,p}$ be a neighborhood of $p$
whose closure $\overline{N_{j,p}}$
does not contain any other fixed points of $\phi^{\kappa_j H^j}_1$.
For each $J \in \ccJ^\Sigma(J^\#,\check{C})$
and each $\gamma = (\gamma^j)_{j \in I_- \sqcup I_+} \in \Gamma(H)$,
let
$K(i,J,\gamma) \subset \ccM(H,J,\gamma)$
be the subset of maps
$u : \Sigma \lra{} M$
satisfying $|du| \leq i$
with respect to a fixed metric on $\Sigma$ which is translation invariant on the cylindrical ends
and
so that $u(\iota_j(s,t)) \in \phi^{\kappa_j H^j}_t(\overline{N_{j,\widehat{\gamma}^j(0)}})$
for each $\pm s \geq i$ and $t \in \T$
where $\widehat{\gamma}^j$
is the $1$-periodic associated to $\gamma^j$
for each $j \in I_\pm$.
Then a Gromov compactness argument
(e.g. \cite[Theorem 4.3.22]{schwarz1995cohomology})
tells us that for each $J \in \ccJ^\Sigma(J^\#,\check{C})$ and $\gamma \in \Gamma(H)$,
$K(i,J,\gamma)$ is compact for each $i \in \N$ and the union of such subsets over all $i$ is $\ccM(H,J,\gamma)$.

Let $\ccJ^\reg_H \subset \ccJ^\Sigma(J^\#,\check{C})$
be the subspace of $H$-regular almost complex structures as in Definition
\ref{definition regular almost complex structure}.
This is a ubiquitous subset by Proposition \ref{proposition ubiquitous regular}.
Let $(N_i)_{i \in \N}$ be open subsets of $M$ satisfying $\cap_{i \in \N} N_i = V \cup ([1+\epsilon/8,1+\epsilon/2] \times C)$.
For each $i \in \N$, $J \in \ccJ^\Sigma(J^\#,\check{C})$
and $\gamma \in \Gamma(H)$
let
$\ccM(i,J,\gamma) \subset K(i,J,\gamma)$
be the subset consisting of maps $u : \Sigma \lra{} M$
satisfying $u(U_i) \cap N_i = \emptyset$.
Let
$$
\ev : \Sigma \times \ccM(H,J,\gamma) \lra{} \Sigma \times M, \quad \ev(\sigma,u) := (\sigma,u(\sigma))
$$
be the natural evaluation map.
For each $i \in \N$,
$J \in \ccJ^\reg_H$
and $\gamma \in \Gamma(H)$,
let $\ccM^{\textnormal{tr}}(i,J,\gamma) \subset \ccM(H,J,\gamma)$ be the open subset
consisting of maps $u$
for which
there exists a neighborhood $N'_u$ of $u$
in $\ccM(H,J,\gamma)$ so that
$\ev|_{S_i \times N'_u}$ is transverse
to $f^k_J|_{E_{i,k,J}}$
for each $k \leq i$.
%
Let $\ccJ^\reg_{i,\gamma} \subset \ccJ^\reg_H$
be the subset of 
almost complex structures $J$ satisfying
$\ccM(i,J,\gamma) \subset \ccM^{\textnormal{tr}}(i,J,\gamma)$ for each $i \in \N$ and $\gamma \in \Gamma(H)$.
%
Since 
$\ccM(i,J,\gamma)$
is compact
for each $i \in \N$, $J \in \ccJ^\reg_{i,\gamma}$ and $\gamma \in \Gamma(H)$
and since transversality is an open condition so long as the corresponding domains are compact,
we have that $J^\reg_{i,\gamma} \subset \ccJ^\reg_H$ is open.
It is also dense by Lemma \ref{lemma transverse ubiquitous}.
Hence
$\ccJ^\reg_{i,\gamma}$ is a ubiquitous subset of $J^\Sigma(J^\#,\check{C})$
for each $i \in \N$, $\gamma \in \Gamma(H)$ since modifying $J$ inside $U_i \times M$
does not change $f^k_J|_{E_{i,k,J}}$ for all $k \leq i$.
Hence $\ccJ^\reg := \cap_{i,\gamma} \ccJ^\reg_{i,\gamma}$
is ubiquitous in $J^\Sigma(J^\#,\check{C})$.

Now let $J \in \ccJ^\reg$
and let $u \in \ccM(H,J,\gamma)$
for some $\gamma \in \Gamma(H)$.
Since $\gamma \in \Gamma(H)$ and since $U_i$ eventually becomes disjoint from any compact subset of $\Sigma$ for $i$ large enough,
there exists $i_u \in \N$
so that $u \in \ccM(i,J,\gamma)$ for each $i \geq i_u$.
Since $J \in \ccJ^\reg_{i,\gamma}$ for each $i \geq i_u$,
there is a neighborhood $N'_{u,i}$ of $u$ in $\ccM(H,J^\#,\gamma)$
so that
the evaluation map $\ev|_{\Sigma \times N'_{u,i}}$
is transverse to $f^k_J|_{E_{i,k,J}}$
for each $k \leq i$ and each $i \geq i_u$.
Therefore since $S_j \subset S_{j+1}$ for all $j \in \N$,
we get that $\ev$ is transverse to
$f^k_J|_{E_{i,k,J}}$ for each $i,k \in \N$ satisfying $k \leq i$
at each point $u \in \ccM(H,J,\gamma)$.
Hence $\ccJ^\reg \subset \ccJ^\reg(H,J^\#,\check{C})$
and so
$\ccJ^\reg(H,J^\#,\check{C})$
is ubiquitous.

Also $M(H,J,\gamma^\#)$
is a manifold of dimension $k$
where $k$ is defined in (\ref{equation conley zehnder index sum})
for each $\gamma^\# \in \Gamma(H)$
and each $J \in \ccJ^{\Sigma,\reg}(H,J^\#,\check{C})$
by Proposition \ref{proposition regular subset of J}.
\end{proof}

\section{Appendix C: Floer trajectories, Filtrations and Compactness.}

Throughout this section we will use the following notation (see Definitions \ref{defn Riemann Surface} and \ref{definition riemann surface admissible}):
\begin{itemize}
\item $\Sigma$ is a Riemann surface with $n_-$ negative cylindrical ends and $n_+$ positive cylindrical ends labeled by finite sets $I_-$, $I_+$ respectively,
\item $\iota_j : \bbI_\pm \times \T \hookrightarrow \Sigma$
is the cylindrical end corresponding to $j$ for each $j \in I_\pm$,
\item $\beta$ is a $\Sigma$-compatible $1$-form and $(\kappa_j)_{j \in I_- \sqcup I_+}$ are the weights of $\beta$ at each cylindrical end,
\item $\check{C}$ is a contact cylinder with associated Liouville domain $D$,
\item $H^\# := (H^j)_{j \in I_- \sqcup I_+}$
is a tuple of Hamiltonians and
\item $J^\# := (J^j)_{j \in I_- \sqcup I_+}$
is a tuple of families of almost complex structures
in $\ccJ^\T(J_0,V,\omega)$.
\end{itemize}

\begin{defn} \label{defn partially converging to capped orbits}
	Let
	$(\gamma_j)_{j \in I_- \sqcup I_+} = (\widetilde{\gamma}_j,\check{\gamma}_j)_{j \in I_- \sqcup I_+}$
	be capped loops where $\widetilde{\gamma}_j : \Sigma_j \lra{} M$ for each $j \in I_- \sqcup I_+$ (Definition \ref{defn capped 1-periodic orbit}).
	A smooth map $u : \Sigma \lra{} M$
	{\it partially converges to $(\gamma_j)_{j \in I_- \sqcup I_+}$}
	if there is a sequence $a^j_1,a^j_2,a^j_3,\cdots \in (0,\infty)$ tending to 
	$a^j_\infty \in (0,\infty]$
	for each $j \in I_\pm$
	and a
	sequence of capped loops $$(\gamma^k_j)_{j \in I_- \sqcup I_+, \ k \in \N} = (\widetilde{\gamma}^k_j,\check{\gamma}^k_j)_{j \in I_- \sqcup I_+, k \in \N}$$
	where $\widetilde{\gamma}_j^k : \Sigma_j^k \lra{} M$ for each $j \in I_- \sqcup I_+$, $k \in \N$
	so that
	\begin{itemize}
		\item 	$\widetilde{\gamma}^k_j(\check{\gamma}^k_j(t)) = u(\pm a^j_k,t)$ for each $t \in \T$, $j \in I_\pm$,  and $k \in \N$,
		\item the surface $u^k : S_k \lra{} M$ obtained by gluing
		$$u|_{\Sigma - \cup_{j \in I_-} \iota_j((-\infty,-a^j_k) \times \T)
		 - \cup_{j \in I_+} \iota_j((a^j_k,\infty) \times \T )}$$
		 to each oriented surface $\widetilde{\gamma}^k_j$, $j \in I_- \sqcup I_+$
		 is null-homologous for each $k \in \N$ and
		 \item $\gamma^k_j$ $C^0$ converges to $\gamma_j$ in the space of capped loops as $k \lra{} \infty$ for each $j \in I_- \sqcup I_+$.
	\end{itemize}

\end{defn}

Note that if $u$ converges to
capped $1$-periodic orbits $(\gamma_j)_{j \in I_- \sqcup I_+}$ as in Definition \ref{definition riemann surface admissible}
then it partially converges to these capped $1$-periodic orbits. Also note that $a^j_\infty$ does not have to be equal to $\infty$.

\begin{defn}
Let $V$ be a vector space over $\R$
and $\omega_V \in \bigwedge^2 V^*$.
A linear complex structure $J_V : V \lra{} V$
is {\it $\omega_V$-semi tame}
if
$\omega_V(v,J_V(v)) \geq 0$ for all $v \in V$.
An almost complex structure
$J$ on $M$ is {\it $\widetilde{\omega}$-semi tame}
for some $\widetilde{\omega} \in \Omega^2(M)$
if $J|_x$ is $\widetilde{\omega}|_x$-semi tame
for all $x \in X$.
Similarly,
a smooth family of almost complex structures
$(J_\sigma)_{\sigma \in \Sigma'}$
is {\it $\widetilde{\omega}$-semi tame}
if $J_\sigma$ is $\widetilde{\omega}$-semi tame for each $\sigma \in \Sigma'$.
\end{defn}

\begin{lemma} \label{lemma filtration}
	Let $\widetilde{\omega} \in \Omega^2(M)$ be closed $2$-form.
	Let
	\begin{enumerate}
		\item $H := (H_\sigma)_{\sigma \in \Sigma}$
		be a $\Sigma$-compatible family of smooth functions (as in Definition \ref{definition riemann surface admissible})
		with limits $H^\#$,
		\item let $F := (F_\sigma)_{\sigma \in \Sigma}$
		be a smooth family of functions which is $\Sigma$-compatible
		with limits $F^\# = (F^j)_{j \in I_- \sqcup I_+}$ and
		\item let $J := (J_\sigma)_{\sigma \in \Sigma} \in \ccJ^\Sigma(V,J_0,\omega)$
		be a $\Sigma$-compatible family of almost complex structures with limits $J^\#$.
	\end{enumerate}
	Suppose
	\begin{itemize}
		\item $H$ is $\widetilde{\omega}$-compatible where $F$ is the primitive associated to $(H,\widetilde{\omega})$ as in Definition \ref{definition action}, $J$ is $\widetilde{\omega}$-semi tame and
		\item $d(f^x \beta) \leq 0$ for all $x \in M$
		where $$f^x : \Sigma \lra{} \R, \quad f^x(\sigma) := F_\sigma(x), \ \forall \sigma \in \Sigma.$$
	\end{itemize}
	
	Then for any solution $u : \Sigma \lra{} M$
	of the $(H,J)$-Floer equation which partially converges capped loops $(\gamma_j)_{j \in I_- \sqcup I_+}$,
	we have
	\begin{equation}
	\sum_{j \in I_-} \cA_{\kappa_j H^j, \widetilde{\omega},\kappa_j F^j}(\gamma_j) \geq \sum_{j \in I_+} \cA_{\kappa_j H^j, \widetilde{\omega},\kappa_j F^j}(\gamma_j)
	\end{equation}
	(see Equation (\ref{eqn:fomegahactionfunctional})).
	
\end{lemma}
\begin{proof}
	The above inequality will follow from Stokes' formula.
	Let $\widehat{\omega}_\beta$ be a  $2$-form on $\Sigma$ defined by
	$$\widehat{\omega}_\beta(Z_1,Z_2) := \widetilde{\omega}(du(Z_1) - \beta(Z_1)X_{H_\sigma},du(Z_2) - \beta(Z_2)X_{H_\sigma}), \quad \forall \ Z_1,Z_2 \in T_\sigma \Sigma, \ \sigma \in \Sigma.$$

	In order to prove our Lemma, we will show
	\begin{enumerate}
		\item \label{item:integralgeq0}
		$\int_\Sigma \widehat{\omega}_\beta \geq 0$ and
		\item \label{item:stokesformula}
		$\int_\Sigma \widehat{\omega}_\beta \leq
		\sum_{j \in I_-} \cA_{\kappa_j H^j, \widetilde{\omega},\kappa_j F^j}(\gamma_j) - \sum_{j \in I_+} \cA_{\kappa_j H^j, \widetilde{\omega},\kappa_j F^j}(\gamma_j).$
	\end{enumerate}
	
	We will now prove (\ref{item:integralgeq0}).
	If ${\bf j}$ is the complex structure on $\Sigma$, it is sufficient for us to show
	$\widehat{\omega}_\beta(Z,{\bf j} Z) \geq 0$ for all $Z \in T\Sigma$.
	Fix $z \in \Sigma$ and $Z \in T_z \Sigma$.
	Then $\widehat{\omega}_\beta(Z,{\bf j}(Z)) = \widetilde{\omega}(du(Z) - \beta(Z)X_{H_\sigma},du({\bf j}(Z)) - \beta({\bf j}(Z)X_{H_\sigma}))$.
	Since $u$ satisfies Floer's equation
	(\ref{eqn:floer}),
	we get that the above expression is equal to:
	$\widetilde{\omega}(du(Z) - \beta(Z)X_{H_\sigma},J_\sigma (du(Z) - \beta(Z)X_{H_\sigma}))$
	which is $\geq 0$ since $J$ is $\widetilde{\omega}$-semi tame.
	
	We now need to prove (\ref{item:stokesformula}).
	We will do this by first modifying $\widehat{\omega}_\beta$ and then applying Stokes' formula.
	Let us look at a holomorphic
	coordinate chart $U$ inside $\Sigma$
	with holomorphic coordinate $z = s+it$
	and consider the vectors
	$\partial_s := \frac{\partial}{\partial s}$,
	$\partial_t := \frac{\partial}{\partial t}$  at this point.
	Then
	$$\widehat{\omega}_\beta(\partial_s,\partial_t)
	= \widetilde{\omega}(du(\partial_s) -\beta(\partial_s) X_{H_z} ,
	du(\partial_t) - \beta(\partial_t)X_{H_z} ) = $$
	$$
	\widetilde{\omega}
	(du(\partial_s),du(\partial_t))
	+ 
	\beta(\partial_s)
	\widetilde{\omega}
	(-X_{H_z} ,
	du(\partial_t))
	-
	$$
	$$ 
	\beta(\partial_t)
	\widetilde{\omega}
	(du(\partial_s),X_{H_z} )
	+
	\beta(\partial_s)
	\beta(\partial_t)
	\widetilde{\omega}
	(X_{H_z} , X_{H_z} ) \stackrel{(\ref{equation H compatible with omega})}{=} 
	$$
	$$
	\widetilde{\omega}
	(du(\partial_s),du(\partial_t))
	+ 
	\beta(\partial_s)
	dF_z(du(\partial_t))
	-
	\beta(\partial_t)
	dF_z
	(du(\partial_s))=
	$$
	$$
	\left(u^* \widetilde{\omega}
	+
	\beta \wedge u^* dF_z \right)(\partial_s,\partial_t).
	$$
	Therefore
	\begin{equation} \label{eqn:omegabetaformula}
	\widehat{\omega}_\beta|_\sigma = u^* \widetilde{\omega}
	+
	\beta \wedge u^* dF_\sigma, \quad \forall \ \sigma \in \Sigma.
	\end{equation}
	
	Define $\widehat{F} : \Sigma \lra{} \R$, $\widehat{F}(\sigma) := F_\sigma(u(\sigma))$.
	Since $d(f^x \beta) \leq 0$ for all $x \in M$,
	we get by Equation (\ref{eqn:omegabetaformula}),
	\begin{equation} \label{eqn:inequalityone}
	\widehat{\omega}_\beta|_\sigma \leq u^* \widetilde{\omega} - d(\widehat{F} \beta).
	\end{equation}

	Let $(\overline{\gamma}_j)_{j \in I_- \sqcup I_+}$
	be the associated loops of the capped loops $(\gamma_j)_{j \in I_- \sqcup I_+}$.
	Since the maps
	$\gamma^k_j$ from Definition \ref{defn partially converging to capped orbits} $C^0$ converge to $\gamma_j$ as $k$ tends to infinity, we have
	by Stokes' formula,
	\begin{equation} \label{eqn:stokesformulaapplied}
	\int_\Sigma d(\widehat{F} \beta) = \sum_{j \in I_+} \int_0^1 \kappa_j F^j_t(\overline{\gamma}_j(t)) dt - \sum_{j \in I_-} \int_0^1 \kappa_j F^j_t(\overline{\gamma}_j( t)) dt.
	\end{equation}
	Let $\gamma_j = (\widetilde{\gamma}_j,\check{\gamma}_j)$ be our capped loop where
	$\widetilde{\gamma}_j : \Sigma_j \lra{} M$
	for each $j \in I_- \sqcup I_+$.
	Since the surface $u^k$ from Definition \ref{defn partially converging to capped orbits}
	is null-homologous for each $k$,
	we get the following equation
	\begin{equation} \label{eqn:cappingzeroequation}
	\sum_{j \in I_-} \int_{\Sigma_j} (\widetilde{\gamma}_j)^*
	 \widetilde{\omega} + \int_\Sigma u^*\widetilde{\omega} = \sum_{j \in I_+} \int_{\Sigma_j} (\widetilde{\gamma}_j)^*\widetilde{\omega}.
	\end{equation}
	Therefore by Equations (\ref{eqn:inequalityone}),
	(\ref{eqn:stokesformulaapplied}),
(\ref{eqn:cappingzeroequation}) and
(\ref{eqn:fomegahactionfunctional})
	we have
	$$\int_\Sigma \widehat{\omega}_\beta \leq
	\sum_{j \in I_-} \cA_{\kappa_j H^j, \widetilde{\omega},\kappa_j F^j}(\gamma_j) - \sum_{j \in I_+} \cA_{\kappa_j H^j, \widetilde{\omega},\kappa_j F^j}(\gamma_j).$$
	Therefore
	(\ref{item:integralgeq0})
	and (\ref{item:stokesformula})
	hold and we are done.
\end{proof}

\begin{defn} \label{defn covergence degree one}
	(see \cite[Definition 4.3.20]{schwarz1995cohomology}).
%
Let $\gamma := (\gamma_j)_{j \in I_- \sqcup I_+}$
be non-degenerate capped $1$-periodic
orbits of $(\kappa_j H_j)_{j \in I_- \sqcup I_+}$.
Let $H \in \ccH^{\Sigma}(H^\#,\check{C})$
and $J \in \ccJ^{\Sigma}(J^\#,\check{C})$ as in Definition \ref{definition riemann surface admissible}.
A sequence
$(u_k)_{k \in \N}$ in $\ccM(H,J,\gamma)$
(Definition \ref{definition riemann surface admissible})
{\it geometrically converges to a broken solution $(u,v)$ of degree $1$} where
$u : \Sigma \lra{} M$,
$v : \R \times \T \lra{} M$ if there exists
\begin{itemize}
\item $m \in I_\pm$,
\item a non-degenerate capped $1$-periodic orbit $\widehat{\gamma}$ of $\kappa_m H^m$ and
\item a sequence $(s_k)_{k \in \N}$ tending to $\infty$ if $m \in I_+$ and $-\infty$ if $m \in I_-$
\end{itemize}
so that $u_k$ converges in $C^\infty_{\text{loc}}$
to $u$
and $u_k \circ \iota_m \circ \tau_k$
converges in $C^\infty_{\text{loc}}$
to $v$
where
\begin{itemize}
\item  $u \in \ccM(H,J,\check{\gamma})$ where
$\check{\gamma} = (\check{\gamma}_j)_{j \in I_- \sqcup I_+}, \quad \check{\gamma}_j =
\left\{
\begin{array}{cc}
\widehat{\gamma} & \text{if} \ j = m \\
\gamma_j & \text{otherwise}
\end{array},
\right.
$
\item $v \in \ccM(\kappa_m H_m, J^m,\gamma')$ 
(Definition \ref{definition space of Floer cylinders})
where $\gamma' = (\gamma_m,\widehat{\gamma})$
 if $m \in I_-$ and
$\gamma' = (\widehat{\gamma},\gamma_m)$ if $m \in I_+$,
\item $\tau_k$ is the map $\tau_k : \bbI_{s_k} \times \T \lra{} \bbI_\pm \times \T, \quad \tau_k(s,t) := (s + s_k,t)$ for each $k \in \N$
where $\bbI_{s_k} := [-s_k,\infty)$ if $m \in I_+$
and $\bbI_{s_k} := (-\infty,-s_k]$ if $m \in I_-$,
\item $|(\check{\gamma}^j)_{j \in I_-}| -  |(\check{\gamma}^j)_{j \in I_+}| = 0$ and
 $|\widehat{\gamma}| - |\gamma_m| = 1$
 if $m \in I_-$
 and $|\gamma_m| - |\widehat{\gamma}| = 1$ if $m \in I_+$.
\end{itemize}
We will call the capped $1$-periodic orbit $\widehat{\gamma}$ the {\it connecting orbit}
and $m \in I_- \sqcup I_+$ the {\it connecting index}.
%
\end{defn}

The following proposition
is inspired by ideas from
\cite[Section 10.1]{cieliebak2015symplectic}.

\begin{prop} \label{proposition compactness result}
Suppose $I_- = \{\star\}$ is a single element set and
(see Definitions \ref{defn chain complex} and \ref{definition space of Floer cylinders})
\begin{itemize}
\item 
$(a^j_-,a^j_+) \in \Sc(Q^j_-) \times \Sc(Q^j_+)$ is a $\check{C}$-action interval,
\item $\kappa_j H^j \in \ccH^{\T,\reg}(\check{C},a^j_-,a^j_+)$,
 $J^j \in \ccJ^{\T,\reg}(\kappa_j H^j,\check{C})$ and
\item $\gamma_j \in \Gamma^\Z_{\check{C},a^j_-,a^j_+}(H^j)$
\end{itemize}
for each $j \in I_- \sqcup I_+$.
Suppose $Q^\star_\pm \subset Q^j_\pm$ for each $j \in I_+$ and
\begin{equation} \label{equation action lower bound}
a^\star_- \leq \sum_{j \in I_+} a^j_-|_{Q_-^\star},  \quad
a^\star_+  \leq a^j_+|_{Q_+^\star} + \sum_{j' \in I_+ - j}
a^{j'}_-|_{Q_+^\star}, \quad \forall \ j \in I_+.
\end{equation}
Suppose
$|\gamma_\star| -  |(\gamma_j)_{j \in I_+}| = 1$.
Let $H \in  \ccH^\Sigma(H^\#,\check{C})$
(Definition \ref{definition riemann surface admissible})
, $J \in  J^{\Sigma,\reg}(H,J^\#,\check{C})$
(Proposition \ref{proposition regular subset for surface})
and let $(u_k)_{k \in \N}$
be a sequence in $\ccM(H,J,\gamma)$ where $\gamma = (\gamma_j)_{j \in I_- \sqcup I_+}$.

Then
there is a subsequence
$(u_{k_j})_{j \in \N}$
which geometrically converges to a broken solution $(u,v)$ of degree $1$ as in Definition \ref{defn covergence degree one} so that 
the connecting orbit $\widehat{\gamma}$
is an element of
$\Gamma^\Z_{\check{C},a^m_-,a^m_+}(H^m)$
where $m$ is the connecting index.
\end{prop}
\proof
Suppose that
there is some $m \in I_\pm$ and some sequence
$s_k \in \bbI_\pm$
so that the sequence of maps
$$l_k : \T \lra{} M, \quad t \to \iota_m(s_k,t)$$ $C^0$ converges to
a $1$-periodic orbit $l : \T \lra{} M$.
Let $\eta^k = (\widetilde{\eta}^k,\check{\eta}^k)$ be the unique capped loop with the property that
its associated $1$-periodic orbit is $l_k$
and
where $\widetilde{\eta}^k$ is given by 
\begin{itemize}
\item 
the catenation of
$\iota_m|_{(-\infty,s_m] \times \T}$ and $\widetilde{\gamma}^m$ if $m \in I_-$
\item or the catenation of
$\iota_j|_{[s_j,\infty) \times \T}$ (with the opposite orientation) and
$\widetilde{\gamma}^m$ if $m \in I_+$ where $\gamma_m = (\widetilde{\gamma}^m,\check{\gamma}^m)$. 
\end{itemize}
By a repeated Gromov compactness argument (e.g. \cite[Theorem 4.3.22]{schwarz1995cohomology}) applied to cylinders and half cylinders,
 $\eta^k$ converges to a capped $1$-periodic orbit $\widehat{\gamma}$ whose associated $1$-periodic orbit is $l$
after passing to a subsequence.
Now by Lemma \ref{lemma filtration},
$$\sum_{j \in I_+} \cA_{H,\check{C}}(\gamma_j) \leq 
\cA_{H,\check{C}}(\eta^k)
\leq \cA_{H,\check{C}}(\gamma_\star)
$$
if $m \in I_-$ and
$$
\cA_{H,\check{C}}(\gamma_m) \leq
\cA_{H,\check{C}}(\eta^k)
\leq \cA_{H,\check{C}}(\gamma_\star) - \sum_{j \in I_+ - m} \cA_{H,\check{C}}(\gamma_j)
$$
if $m \in I_+$
for each $k \in \N$.
Hence by Equation (\ref{equation action lower bound}), 
$a_-^m \leq \cA_{H,\check{C}}(\eta^k)|_{Q^m_-}$ and $a_+^m \nleq \cA_{H,\check{C}}(\eta^k)|_{Q^m_+}
$
for each $k \in \N$ and hence
\begin{equation} \label{equation middle part action bounds}
a_-^m \leq \cA_{H,\check{C}}(\widehat{\gamma})|_{Q^m_-}, \quad a_+^m \nleq \cA_{H,\check{C}}(\widehat{\gamma})|_{Q^m_+}.
\end{equation}
Therefore $\widehat{\gamma}$ has to be non-degenerate.
Hence by using the argument above
one can show by \cite[Theorem 4.3.21]{schwarz1995cohomology}
that (after passing to a subsequence)
$(u_k)_{k \in \N}$
geometrically converges to a broken solution $(u,v)$ of degree $1$ with connecting orbit
$\widehat{\gamma}$
satisfying Equation (\ref{equation middle part action bounds}).
%
%
\qed

\begin{remark} \label{remark parameterized version of breaking action}
There is also a parameterized version of
Proposition \ref{proposition compactness result}
where we now have $H = (H_{s,\sigma})_{s \in [0,1], \sigma \in \Sigma_s} \in  \ccH^{\Sigma_\bullet}(H^\#,\check{C})$ 
(Definition \ref{definition parameterized moduli spaces})
and $J \in  J^{\Sigma_\bullet,\reg}(H,(Y_0,Y_1),J^\#,\check{C})$
(Proposition \ref{proposition parameterized regular subset for surface})
for some smooth family of Riemann surfaces
$\Sigma_\bullet := (\Sigma_t)_{t \in [0,1]}$
where
$Y_j = (H_{j,s})_{s \in \Sigma_j}$ for $j=0,1$ and
$|\gamma_\star| -  |(\gamma_j)_{j \in I_+}| = 0$.
The proof is also identical.
\end{remark}

\section{Appendix D: Flatness of Novikov rings.}

Throughout this section we will fix a finitely generated abelian group $(A,\cdot)$.

\begin{defn}
A {\it rational polyhedral cone in $(A \otimes_\Z \R)^*$}
is a cone of the form
$$Q := \left\{
\sum_{i=0}^k r_i w_i \ : \ r_0,\cdots,r_k \geq 0
\right\}
\subset (A \otimes_\Z \R)^*$$
for some fixed elements
$w_0,\cdots,w_k \in (A \otimes_\Z \Q)^* \subset (A \otimes_\Z \R)^*$
called {\it generators of $Q$}.
\end{defn}

Such a cone is closed and hence if, in addition, this cone is salient 
(Definition \ref{definition convex cone})
then we can define the Novikov rings
$\Lambda_\K^{A,Q}$ and $\Lambda_\K^{A,Q,+}$ as in Definition \ref{definition novikov ring associated to salient cone}. The aim of this section is to prove:
\begin{prop} \label{flatness of rational polyhedral novikov rings}
Suppose $\K$ is Noetherian and let $Q_0, Q_1 \subset (A \otimes_\Z \R)^*$ be salient rational polyhedral cones satisfying $Q_1 \subset Q_0$.
Then $\Lambda_\K^{A,Q_1}$ is a flat $\Lambda_\K^{A,Q_0}$-module.
\end{prop}

The key idea of the proof is
to show that there are appropriate subalgebras of $\Lambda^{A,Q_j}$, $j=0,1$
so that one is the completion of the other (Lemma \ref{lemma completion of tensor product induces isomorphism}) and so that appropriate localizations of them
recover $\Lambda^{A,Q_j}$, $j=0,1$ (Lemma \ref{lemma inclusion map tensor localization isomorphism}).
Before we prove Proposition \ref{flatness of rational polyhedral novikov rings},
we need a few definitions and lemmas.
Throughout this section we will assume that our ring $\K$ is Noetherian.

\begin{defn} \label{definition monoid ring}
Let $Q \subset (A \otimes_\Z \R)^*$ be a closed salient cone.
	For each $x \in A$,
	let $F^Q_x$
	be the free $\K$-module generated
	by elements of the set
	$S^Q_x := \{a \in A \ : \ x \preceq_Q a \}$ (Definition \ref{definition novikov ring associated to salient cone}).
\end{defn}

\begin{remark}
	When $x = 0$, $F^Q_0$ is a $\K$-algebra with multiplication induced from the product $\cdot$ on $A$ and
	$F^Q_x$ is an ideal in $F^Q_0$
	for each $x \in A$ satisfying $0 \preceq_Q x$.
\end{remark}

\begin{defn} \label{defn ideal}
If $R$ is a ring and $x \in R$, then we write $(x)_R$ for the ideal generated by $x$. If it is clear which ring $x$ lives in, then we write $(x) = (x)_R$.
\end{defn}

\begin{lemma} \label{lemma multiplying ideals}
Let $Q \subset (A \otimes_\Z \R)^*$
be a closed salient cone.
Then we have an equality of ideals
$(F^Q_x)^m = (x^m)$
in the $\K$-algebra
$F^Q_0$ for each $x \in A$, $m > 0$ satisfying $0 \preceq_Q x$.
\end{lemma}
\proof
Let $x \in A$ satisfy $0 \preceq_Q x$ and let $m > 0$.
Since $x \in F^Q_x$, we have
$(x^m) \subset (F^Q_x)^m$.
Now let $a \in A$ satisfy $x \preceq_Q a$.
Then $0 \preceq_Q a \cdot x^{-1}$ which means that $a \cdot x^{-1} \in F^Q_0$.
Hence
$a^m = (a \cdot x^{-1})^m x^m \in (x^m)$.
Since $(F^Q_x)$ is an ideal generated by elements $a^m$ satisfying $x \preceq_Q a$,
we then get that $(x^m) = (F^Q_x)^m$.
\qed

\begin{defn} \label{definition cofinal element}
Let $Q \subset (A \otimes_\Z \R)^*$ be a closed salient cone.
A {\it $Q$-cofinal element} is an element $y \in A$ satisfying $0 \preceq_Q y$
so that the sequence $(y^n)_{n \in \N}$
is cofinal in $(A,\preceq_Q)$.
\end{defn}

\begin{lemma} \label{lemma cofinal element}
Let $Q_0,Q_1 \subset (A \otimes_\Z \R)^*$
be closed salient cones so that $Q_1 \subset Q_0$.
Then there exists $y \in A$ which is both a $Q_0$ and $Q_1$-cofinal element.
\end{lemma}
\proof
Since $Q_0$ is closed and salient and $A$ is finitely generated, there exist $y_0,\cdots,y_k \in A$ generating $A$ as a group and satisfying $0 \preceq_{Q_0} y_i$ for each $i = 0,\cdots, k$.
Then $0 \preceq_{Q_1} y_i$ for each $i =0,\cdots,k$ as well since $Q_1 \subset Q_0$.
Define $y := \prod_{i=0}^k y_i$.
Now suppose $x \in A$ satisfies $0 \preceq_{Q_j} x$ for some $j = 0,1$. Then there exists $l_0,\cdots,l_k \in \Z$ so that
$x = \prod_{i=0}^k y_i^{l_i}$.
Then $0 \preceq_{Q_j} x^{-1} \cdot y^{\max_{i=0}^k |l_i|}$
and hence $x \preceq_{Q_j} y^{\max_{i=0}^k |l_i|}$. Therefore $y$ is a $Q_j$-cofinal element.
\qed

\begin{defn}
Let $R$ be a ring and $I \subset R$ an ideal.
We define the {\it completion of $R$ along $I$} to be the inverse limit of rings
$$\widehat{R}_I := \varprojlim_{m \in \N} R/I^m.$$
\end{defn}

%

\begin{lemma} \label{lemma completion definition of positive Novikov ring}
Let $Q \subset (A \otimes_\Z \R)^*$ be a closed salient cone and let $y$ be a $Q$-cofinal element.
The natural inclusion map of $\K$-algebras
$F^Q_0 \hookrightarrow \Lambda_\K^{A,Q,+}$
extends to an isomorphism
$$\widehat{F^Q_0}_{(y)} \lra{\cong} \Lambda_\K^{A,Q,+}.$$
\end{lemma}
\proof
By Lemma \ref{lemma multiplying ideals}
combined with the fact that $y$ is a $Q$-cofinal element, we have
$$
	\Lambda_\K^{A,Q,+} =
	\varprojlim_{m \in \N} F^Q_0/(y^m).
$$
by Equation (\ref{equation positive Novikov ring}).
This proves our lemma since $(y^m) = (y)^m$ for each $m > 0$.
\qed

\begin{defn}
A multiplicative subset in a commutative ring $R$ is a subset closed under multiplication.
	For a multiplicative subset $S \subset R$, we define the
	{\it localization $S^{-1}R$ of $R$ along $S$}
	to be the set of equivalence classes of pairs $(r,s) \in R \times S$ so that $(r_1,s_1) \sim (r_2,s_2)$ if and only if there exists a $t \in S$ so that
	$t(r_1s_2 - r_2s_1) = 0$, where addition and multiplication are defined via the formulas
	\begin{equation} \label{equation addition and multiplcation of fractions}
	(r_1,s_1) + (r_2,s_2) = (r_1s_2 + r_2 s_1, s_1s_2), \quad (r_1,s_1)(r_2,s_2) = (r_1r_2,s_1s_2).
	\end{equation}
	If $x \in R$, we define $S_x \subset R$ to be the smallest multiplicative subset containing $x$.
\end{defn}

\begin{lemma} \label{lemma localization definition}
Let $Q \subset (A \otimes_\Z \R)^*$ be a closed salient cone and let $y$ be a $Q$-cofinal element.
Then the natural inclusion map
$\Lambda_\K^{A,Q,+} \hookrightarrow \Lambda_\K^{A,Q}$
extends to an isomorphism of $\K$-algebras
\begin{equation} \label{equation isomorphism}
S_y^{-1} \Lambda_\K^{A,Q,+} \lra{\cong} \Lambda_\K^{A,Q}.
\end{equation}
\end{lemma}
\proof
Since $F^Q_{x_0x_1} = x_0F^Q_{x_1}$ for each $x_0,x_1 \in A$ and since $y^{-1}$ is $-Q$-cofinal,
we have a natural isomorphism
\begin{equation} \label{eqn direct limit ring}
\Lambda_\K^{A,Q} = \varinjlim_{m \in \N} y^{-m} \Lambda_\K^{A,Q,+}
\end{equation}
by Equations (\ref{equation Novikov ring definition}) and (\ref{equation positive Novikov ring}) where the morphisms in the corresponding directed system are the inclusion maps.
Therefore since $y$ is invertible in $\Lambda_\K^{A,Q}$, we get by the definition of direct limit
(Definition \ref{definition direct limit})
that
elements of the ring (\ref{eqn direct limit ring}) are equivalence classes of pairs $(x,y^m)$, $x \in \Lambda_\K^{A,Q,+}$, $m > 0$ where 
$(x,y^m) \sim (x',y^{m'})$ iff $xy^{m'+k} - x' y^{m+k} = 0$ for some $k \geq 0$ and where
addition and multiplication satisfy formulas which are similar to (\ref{equation addition and multiplcation of fractions}).
This proves our lemma.
\qed

\smallskip

We get the following immediate corollary of Lemmas \ref{lemma completion definition of positive Novikov ring} and
\ref{lemma localization definition}.

\begin{corollary} \label{corollary isomorphism test}
Let $Q_0, Q_1 \subset (A \otimes_\Z \R)^*$ be closed salient cones so that $Q_1 \subset Q_0$ and let $y \in A$ be a $Q_j$-cofinal element for $j = 0,1$.
Then we have the following commutative diagram of $\K$-algebras:
\begin{center}
\begin{tikzpicture}

\node at (-1,-1) {$S_y^{-1}(\widehat{F^{Q_0}_0}_{(y)})$};

\node at (1,-1) {$\Lambda_\K^{A,Q_0}$};

\node at (-1,-2) {$S_y^{-1}(\widehat{F^{Q_1}_0}_{(y)})$};

\node at (1,-2) {$\Lambda_\K^{A,Q_1}$};

\draw [->](-1,-1.4) -- (-1,-1.6);

\draw [->](1,-1.3) -- (1,-1.7);

\draw [->](0.1,-1) -- (0.4,-1);

\draw [->](0.1,-2) -- (0.4,-2);

\node at (0.2,-0.8) {$\cong$};

\node at (0.2,-1.8) {$\cong$};

\end{tikzpicture}
\end{center}

\end{corollary}

From now on, until the end of this section,
we will let $Q_0, Q_1 \subset (A \otimes_\Z \R)^*$ be closed salient cones so that $Q_1 \subset Q_0$ and
let $y$ be a $Q_j$-cofinal element for $j=0,1$ (see Lemma \ref{lemma cofinal element}).
We will also define
$R^j := F^{Q_j}_0$
and $I_j := (y)_{R^j}$ (Definition \ref{defn ideal}) for $j= 0,1$.

\begin{lemma} \label{lemma inclusion map injective}
The natural map
\begin{equation} \label{equation localization inclusion injective}
\K[A] = S_y^{-1} R^0 \lra{} S^{-1}_y \widehat{R^0}_{I_0}
\end{equation}
is injective.
\end{lemma}
\proof
The natural map
$R^0 \lra{} \widehat{R^0}_{I_0}$
is injective
since $\cap_{m \in \N} I_0^m = 0$.
Hence (\ref{equation localization inclusion injective}) is injective by
\cite[Tag 00CS]{stacks-project}.
\qed

\begin{lemma} \label{lemma completion of tensor product induces isomorphism}
The map
\begin{equation} \label{equation tensor product inclusion}
R^1 \otimes_{R^0} \widehat{R^0}_{I_0} \lra{} \widehat{R^1}_{I_1}
\end{equation}
induced by the natural inclusion maps
extends to an isomorphism
\begin{equation} \label{equation completed tensor inclusion}
\widehat{(R^1 \otimes_{R^0} \widehat{R^0}_{I_0})}_{(1 \otimes y)} \lra{\cong} \widehat{R^1}_{I_1}.
\end{equation}
\end{lemma}
\proof
The map (\ref{equation tensor product inclusion}) induces an isomorphism
$$(R^1 \otimes_{R^0} \widehat{R^0}_{I_0}) / (1 \otimes y^m)
= (R^1 \otimes_{R^0} R^0) / (1 \otimes y^m) \cong R^1/(y^m)$$
for each $m > 0$. Taking the inverse limit as $m \to \infty$ of this isomorphism gives us the isomorphism (\ref{equation completed tensor inclusion}).
\qed

\smallskip

Note that we have natural inclusion maps
\begin{equation} \label{eqn natural inclusion map}
R^1 \hookrightarrow S_y^{-1} R^1 = \K[A] = S_y^{-1} R^0 \hookrightarrow S_y^{-1}\widehat{R^0}_{I_0}
\end{equation}
since $y$ is not a zero divisor in $R^0$ or $R^1$ and also by Lemma \ref{lemma inclusion map injective}.

\begin{lemma} \label{lemma inclusion map tensor localization isomorphism}
The map
\begin{equation} \label{equation inclusion to localization R0}
\Psi : R^1 \otimes_{R^0} \widehat{R^0}_{I_0} \lra{} S_y^{-1}\widehat{R^0}_{I_0}
\end{equation}
induced by the natural inclusion map
(\ref{eqn natural inclusion map}) extends to an isomorphism
\begin{equation} \label{localiszation isomorphism from Q0}
\Phi : S_{1 \otimes y}^{-1}(R^1 \otimes_{R^0} \widehat{R^0}_{I_0}) \lra{\cong} S_y^{-1}(\widehat{R^0}_{I_0}).
\end{equation}
\end{lemma}
\proof
Since the map (\ref{equation inclusion to localization R0}) sends $1 \otimes y$ to $y$, we get that the map $\Phi$ is well defined.
Also the composition of the map
$S^{-1}_{1 \otimes y}(R^0 \otimes_{R^0} \widehat{R^0}_{I_0}) \lra{} S^{-1}_{1 \otimes y}(R^1 \otimes_{R^0} \widehat{R^0}_{I_0})$ with $\Phi$ is an isomorphism and hence $\Phi$ is surjective.
Finally, suppose $c \in \ker(\Phi)$.
Then, there exists $m > 0$ so that
$c' := (1 \otimes y)^m c \in R^1 \otimes_{R^0} \widehat{R^0}_{I_0}$.
Since $y$ is a $Q_0$-cofinal element,
we have
\begin{equation} \label{equation multiple of c prime}
(y \otimes 1)^{m'} c' \in R^0 \otimes_{R^0} \widehat{R^0}_{I_0} = \widehat{R^0}_{I_0}
\end{equation}
for some large $m' > 0$.
Since $y$ is not a zero divisor in $\widehat{R^0}_{I_0}$,
we get that the inclusion map
$\widehat{R^0}_{I_0} \hookrightarrow S_y^{-1}\widehat{R^0}_{I_0}$
is an injection and so the element (\ref{equation multiple of c prime}) is zero.
Hence 
$$(1 \otimes y)^{m+m'} c
=
(y \otimes 1)^{m'} c' = 0 \in R^1 \otimes_{R^0} \widehat{R^0}_{I_0}.$$
Since $1 \otimes y$ is invertible
in $S_{1 \otimes y}^{-1}(R^1 \otimes_{R^0} \widehat{R^0}_{I_0})$,
this implies that $c = 0$
and hence $\Phi$ is injective.
%
%
%
%
Hence $\Phi$ is an isomorphism.
\qed

\begin{proof}[Proof of Proposition \ref{flatness of rational polyhedral novikov rings}.]
Since $Q_j$ is a rational polyhedral cone
we get that
$R^j$ is a finitely generated
 $\K$-algebra by Gordan's lemma
\cite[Proposition 1, Section 1.2]{fulton1993introduction}
 and hence is Noetherian by \cite[Tag 00FN]{stacks-project} for $j=0,1$ since $\K$ is Noetherian.
Hence
$R^1 \otimes_{R^0} \widehat{R^0}_{I_0}$
is Noetherian
by
\cite[Tag 0CY6]{stacks-project},
\cite[Tag 00FN]{stacks-project} and
\cite[Tag 05GH]{stacks-project}.
Therefore
$\widehat{(R^1 \otimes_{R^0} \widehat{R^0}_{I_0})}_{(1 \otimes y)}$
is a flat
$R^1 \otimes_{R^0} \widehat{R^0}_{I_0}$-module by
\cite[Tag 00MB]{stacks-project} and \cite[Tag 00HT (1)]{stacks-project}.
Hence
$\widehat{R^1}_{I_1}$
is a flat
$R^1 \otimes_{R^0} \widehat{R^0}_{I_0}$-module by Lemma
\ref{lemma completion of tensor product induces isomorphism}.
Therefore since
$S_y^{-1}\widehat{R^1}_{I_1}$
is a flat
$R^1 \otimes_{R^0} \widehat{R^0}_{I_0}$-module
by
\cite[Tag 00HT (1)]{stacks-project}
combined with
\cite[Tag 00HC]{stacks-project}
we get that
$S_y^{-1}\widehat{R^1}_{I_1}$
is a flat
$S_{1 \otimes y}^{-1}(R^1 \otimes_{R^0} \widehat{R^0}_{I_0})$-module by \cite[Tag 00HT (2)]{stacks-project}.
Hence
$S_y^{-1}\widehat{R^1}_{I_1}$ is a flat
$S_y^{-1}(\widehat{R^0}_{I_0})$-module
by Lemma \ref{lemma inclusion map tensor localization isomorphism}.
Our proposition now follows from Corollary \ref{corollary isomorphism test}.
\end{proof}

\bibliography{references}

\end{document}